\documentclass[11pt,english,letterpaper]{amsart}
\usepackage{amssymb,bm,mathrsfs,stmaryrd,mathabx}
\usepackage[all]{xy}
\usepackage{hyperref}
\hypersetup{colorlinks=true}
\theoremstyle{plain}
\newtheorem{theorem}{Theorem}[subsection]
\newtheorem{bigtheorem}{Theorem}[section]
\newtheorem{proposition}[theorem]{Proposition}
\newtheorem{lemma}[theorem]{Lemma}
\newtheorem{corollary}[theorem]{Corollary}

\theoremstyle{definition}
\newtheorem{definition}[theorem]{Definition}

\theoremstyle{remark}
\newtheorem{remark}[theorem]{Remark}

\numberwithin{equation}{subsection}

\newcommand{\map}[1]{\xrightarrow{#1}}

\newcommand{\iso}{\cong}
\newcommand{\define}{\stackrel{\mathrm{def}}{=}}

\newcommand{\imes}{\ltimes}
\newcommand{\normal}{\lhd}

\newcommand{\Gal}{\mathrm{Gal}} 
\newcommand{\Hom}{\mathrm{Hom}}

\newcommand{\Aut}{\mathrm{Aut}}
\newcommand{\End}{\mathrm{End}}
\newcommand{\Spec}{\mathrm{Spec}}
\newcommand{\Q}{\mathbb Q}
\newcommand{\Z}{\mathbb Z}
\newcommand{\R}{\mathbb R}

\newcommand{\C}{\mathbb C}
\newcommand{\F}{\mathbb F}
\newcommand{\A}{\mathbb A}
\newcommand{\co}{\mathcal O}
\newcommand{\alg}{\mathrm{alg}}
\newcommand{\action}{\bullet}
\newcommand{\Lie}{\mathrm{Lie}}
\newcommand{\Fil}{F}
\newcommand{\Pap}{\mathrm{Pap}}
\newcommand{\Kra}{\mathrm{Kra}}
\newcommand{\art}{\mathrm{art}}

\newcommand{\ord}{\mathrm{ord}}

\newcommand{\GL}{\mathrm{GL}}
\newcommand{\GU}{\mathrm{GU}}
\newcommand{\SO}{\mathrm{SO}}
\newcommand{\SL}{\mathrm{SL}}
\newcommand{\Stab}{\mathrm{Stab}}

\newcommand{\kk}{{\bm{k}}}
\newcommand{\dR}{\mathrm{dR}}
\newcommand{\eee}{\mathrm{e}}
\newcommand{\fff}{\mathrm{f}}

\renewcommand{\Im}{\operatorname{Im}}
\renewcommand{\Re}{\operatorname{Re}}

\newcommand{\pure}{{\bm{\Omega}}}

\newcommand{\zxz}[4]{\begin{pmatrix} #1 & #2 \\ #3 & #4 \end{pmatrix}}
\newcommand{\abcd}{\zxz{a}{b}{c}{d}}

\newcommand{\calB}{\mathcal{B}}

\newcommand{\calD}{\mathcal{D}}

\newcommand{\calO}{\mathcal{O}}

\newcommand{\calS}{\mathcal{S}}

\newcommand{\calZ}{\mathcal{Z}}

\newcommand{\tot}{\mathrm{tot}}
\newcommand{\reg}{\mathrm{reg}}
\newcommand{\TR}{\Theta^\reg}

\newcommand{\bs}{\backslash}
\newcommand{\Ch}{\operatorname{Ch}}
\newcommand{\Cha}{\widehat{\operatorname{Ch}}}

\author[J.~Bruinier]{Jan H.~Bruinier}
\address{Fachbereich Mathematik, Technische Universit\"at Darmstadt, Schlossgartenstrasse 7, D-64289 Darmstadt, Germany}
\email{bruinier@mathematik.tu-darmstadt.de}

\author[B.~Howard]{Benjamin Howard}
\address{Department of Mathematics, Boston College, 140 Commonwealth Ave, Chestnut Hill, MA 02467, USA}
\email{howardbe@bc.edu}

\author[S.~Kudla]{Stephen S.~Kudla}
\address{Department of Mathematics, University of Toronto, 40 St. George St., BA6290, Toronto, ON M5S 2E4, Canada}
\email{skudla@math.toronto.edu}

\author[M.~Rapoport]{Michael Rapoport}
\address{
Mathematisches Institut der Universit\"at Bonn, Endenicher Allee 60, 53115 Bonn, Germany, and 
 Department of Mathematics, University of Maryland, College Park, MD 20742, USA 
}
\email{rapoport@math.uni-bonn.de}

\author[T.~Yang]{Tonghai Yang}
\address{Department of Mathematics, University of Wisconsin Madison, Van Vleck Hall, Madison, WI 53706, USA}
\email{thyang@math.wisc.edu}

\title[Modularity of unitary generating series]{Modularity of generating series of divisors on  unitary Shimura varieties}

\begin{document}

\begin{abstract}
We form  generating series, valued in the Chow group and  the arithmetic Chow group,  of special divisors on  the compactified integral model of a Shimura variety associated to a unitary group of signature $(n-1,1)$, and prove   their modularity. 
The main ingredient in the proof is the calculation of  vertical components appearing in the divisor of a Borcherds product on the integral model.
\end{abstract}



\thanks{J.B.~was  supported in part by DFG grant BR-2163/4-2. 
B.H.~was supported in part by NSF grants DMS-1501583 and DMS-1801905. 
M.R.~was supported in part by   the Deutsche Forschungsgemeinschaft through the grant SFB/TR 45. 
S.K.~was supported in part by an NSERC Discovery Grant. 
 T.Y.~was  supported in part by NSF grant DMS-1500743.}

\maketitle

\setcounter{tocdepth}{1}
\tableofcontents



\section{Introduction}


The goal of this paper is to prove the modularity of a  generating series of special divisors on the compactified integral model of a Shimura variety associated to a unitary group of signature $(n-1,1)$.    
The special divisors in question  were first  studied on the open Shimura variety in \cite{KR1, KR2}, and then on the toroidal compactification in \cite{Ho2}.  

This generating series is an arithmetic analogue of the classical theta kernel used to lift modular forms from $\mathrm{U}(2)$ and $\mathrm{U}(n)$.  
In a similar vein, our modular generating series can be used to define a lift from classical cuspidal  modular forms of weight $n$ to the 
codimension one Chow group of the unitary Shimura variety.


\subsection{Statement of the main result}


Fix a quadratic imaginary field $\kk\subset \C$ of odd discriminant $\mathrm{disc}(\kk)=-D$. 
We are concerned with the arithmetic of a certain unitary Shimura variety, whose definition depends on 
the choices of $\kk$-hermitian spaces $W_0$ and $W$  of signature $(1,0)$ and $(n-1,1)$, respectively, where  $n\ge 3$.   We assume that  $W_0$ and $W$  each admit  an $\co_\kk$-lattice  that is self-dual with respect to the hermitian form.

Attached to this data is a reductive algebraic group
\begin{equation}\label{G def}
G \subset \GU(W_0 ) \times \GU(W)
\end{equation}
over $\Q$, defined as the subgroup on which the unitary similitude characters are equal, and a compact open subgroup $K\subset G(\A_f)$ depending on the above choice of self-dual lattices.  As explained in \S \ref{s:unitary}, there is an associated hermitian symmetric domain $\mathcal{D}$, and a Deligne-Mumford stack $\mathrm{Sh}(G,\mathcal{D})$ over $\kk$ whose complex points are identified with the orbifold quotient
\[
\mathrm{Sh}(G,\mathcal{D}) (\C) = G(\Q) \backslash \mathcal{D} \times G(\A_f) / K.
\]
This is the unitary Shimura variety of the title.

The stack  $\mathrm{Sh}(G,\mathcal{D})$ can be interpreted as a moduli space of pairs $(A_0,A)$ in which 
 $A_0$ is an elliptic curve with complex multiplication by $\co_\kk$, and $A$ is a principally polarized abelian scheme of dimension $n$ endowed with an $\co_\kk$-action.  The pair $(A_0,A)$ is required to satisfy some additional conditions, which need not concern us in the introduction.

Using the moduli interpretation, one can construct an integral model of $\mathrm{Sh}(G,\mathcal{D})$ over $\co_\kk$.  
In fact, following work of Pappas and Kr\"amer, we explain in \S \ref{ss:unitary integral models} that there are two natural integral models  related by a morphism $\mathcal{S}_\Kra \to \mathcal{S}_\Pap.$  
Each integral model has a canonical toroidal compactification whose boundary  is a disjoint union of  smooth Cartier divisors, and the above morphism  extends uniquely to a morphism
\begin{equation}\label{intro models}
\mathcal{S}^*_\Kra \to \mathcal{S}^*_\Pap
\end{equation}
 of compactifications.

Each compactified integral model has its own desirable and undesirable properties.  
For example, $\mathcal{S}^*_\Kra$ is regular, while $\mathcal{S}^*_\Pap$ is not.  
On the other hand, every vertical (\emph{i.e.}~ supported in nonzero characteristic) Weil divisor on $\mathcal{S}^*_\Pap$ has nonempty intersection with the boundary, while $\mathcal{S}^*_\Kra$ has certain \emph{exceptional} divisors in characteristics $p\mid D$ that do not meet the boundary.   
An essential part of our method is to pass back and forth between these two models in order to exploit the best properties of each.
For simplicity,  we will state our main results in terms of the regular model $\mathcal{S}^*_\Kra$.

In \S \ref{s:unitary} we define  a distinguished line bundle $\bm{\omega}$ on $\mathcal{S}_\Kra$, called the \emph{line bundle of weight one modular forms}, and  a family of Cartier divisors $\mathcal{Z}_\Kra(m)$ indexed by integers $m>0$.   These special  divisors were  introduced  in \cite{KR1, KR2}, and studied further in \cite{BHY, Ho1, Ho2}.   For the purposes of the introduction, we note only that one should regard the divisors as arising from embeddings of smaller unitary groups into $G$.

Denote by 
\[
\Ch_\Q^1(\mathcal{S}^*_\Kra) \iso \mathrm{Pic}(\mathcal{S}^*_\Kra)\otimes_\Z\Q
\] 
the Chow group  of rational equivalence classes of divisors with  $\Q$ coefficients.
Each special divisor $\mathcal{Z}_\Kra(m)$ can be extended to a divisor on the toroidal compactification simply by taking its Zariski closure, denoted $\mathcal{Z}_\Kra^*(m)$.   The \emph{total special divisor} is  defined as
\begin{equation}\label{intro total}
\mathcal{Z}_\Kra^\tot(m) = \mathcal{Z}_\Kra^*(m)  +  \mathcal{B}_\Kra(m)\in \Ch_\Q^1(\mathcal{S}^*_\Kra)
\end{equation}
where the boundary contribution is defined, as in  (\ref{m boundary mult}), by
\[
\mathcal{B}_\Kra(m) = \frac{ m }{n-2}   \sum_\Phi  \#\{ x \in L_0 :  \langle x , x \rangle  =m \} \cdot \mathcal{S}^*_\Kra(\Phi).
\]
The notation here is the following: The sum is over the equivalence classes of  \emph{proper cusp label representatives} $\Phi$ as defined in \S \ref{ss:cusp notation}.  
These index the connected components $\mathcal{S}^*_\Kra(\Phi) \subset  \partial \mathcal{S}_\Kra^*$  of the boundary\footnote{After base change to $\C$, each   $\mathcal{S}^*_\Kra(\Phi)$ decomposes into $h$ connected components, where $h$ is the class number of $\kk$.}.   
Inside the sum, $( L_0, \langle \cdot,\cdot \rangle)$ is a hermitian $\co_\kk$-module of signature $(n-2,0)$, which depends on $\Phi$.

The line bundle of modular forms $\bm{\omega}$ admits a canonical extension to the toroidal compactification, denoted the same way. 
 For the sake of notational uniformity,  we extend (\ref{intro total}) to $m=0$ by setting
\begin{equation}\label{intro naked constant}
\mathcal{Z}_\Kra^\tot(0) = \bm{\omega}^{-1} + \mathrm{Exc} \in \Ch_\Q^1(\mathcal{S}^*_\Kra).
\end{equation}
Here $\mathrm{Exc}$ is the exceptional divisor of Theorem \ref{thm:integral comparison}.  It is a reduced effective divisor supported in characteristics $p\mid D$,  disjoint from the boundary of the compactification.
   The following result appears in the text as Theorem \ref{thm:kramer modularity}.

\begin{bigtheorem}\label{ThmA}
Let  $\chi_\kk: (\Z/ D\Z )^\times \to \{\pm 1\}$ be the  Dirichlet character  determined by $\kk/\Q$.
The formal generating series
\[
\sum_{m\ge 0}   \mathcal{Z}^\tot_\Kra(m) \cdot q^m \in \Ch_\Q^1(\mathcal{S}^*_\Kra)[[ q]] 
\]
is modular of weight $n$, level $\Gamma_0(D)$, and character $\chi_\kk^n$ in the following sense:  for every $\Q$-linear functional
$
\alpha : \Ch_\Q^1(\mathcal{S}^*_\Kra) \to \C,
$
 the  series
\[
\sum_{m\ge 0}   \alpha (  \mathcal{Z}^\tot_\Kra(m) ) \cdot q^m \in  \C [[ q]] 
\]
is the $q$-expansion of a classical modular form  of the indicated weight, level,  and character.
\end{bigtheorem}

We can prove a stronger version of Theorem \ref{ThmA}.  
 Denote by $\Cha^1_\Q(\mathcal{S}^*_\Kra)$
 the Gillet-Soul\'e \cite{GS} arithmetic Chow group   of rational equivalence classes of pairs 
$
\widehat{\mathcal{Z}} = (\mathcal{Z} ,  \mathrm{Gr}  ),
$
 where $\mathcal{Z}$ is a divisor on $\mathcal{S}^*_\Kra$ with rational coefficients, 
and $\mathrm{Gr}$ is a Green function for $\mathcal{Z}$.  
We  allow the Green function to have additional $\log$-$\log$ singularities  along the boundary, as in the more general theory  developed in \cite{BKK}. See also \cite{BBK,Ho2}.

In \S \ref{ss:arithmetic modular} we use the theory of regularized theta lifts  to construct  Green functions for the special divisors $\mathcal{Z}_\Kra^\tot(m)$, and hence obtain arithmetic divisors
\[
\widehat{\mathcal{Z}}_\Kra^\tot(m) \in \Cha^1_\Q(\mathcal{S}^*_\Kra)
\]
for $m>0$.   We also endow the line bundle $\bm{\omega}$ with a metric,  and the resulting metrized line bundle $\widehat{\bm{\omega}}$  defines a class
\[
\widehat{\mathcal{Z}}_\Kra^\tot(0) = \widehat{\bm{\omega}}^{-1} + ( \mathrm{Exc} , -\log(D))  \in \Cha^1_\Q(\mathcal{S}^*_\Kra),
\]
where the vertical divisor $\mathrm{Exc}$ has been endowed with the constant Green function $-\log(D)$.
The following result is  Theorem \ref{thm:arithmeticmodularity}  in the text.

\begin{bigtheorem}\label{ThmB}
 The formal generating series
\[
\widehat{\phi}(\tau) = \sum_{m\ge 0}   \widehat{\mathcal{Z}}^\tot_\Kra(m) \cdot q^m \in \Cha_\Q^1(\mathcal{S}^*_\Kra)[[ q]] 
\]
is modular  of weight $n$, level $\Gamma_0(D)$, and character $\chi_\kk^n$, where modularity is understood in the same sense as Theorem \ref{ThmA}.
\end{bigtheorem}

\begin{remark}
As this article was being revised for publication, Wei Zhang announced a proof of his \emph{arithmetic fundamental lemma}, conjectured in \cite{zhangAFL}.  
Although the statement is a purely local result concerning intersections of cycles on unitary Rapoport-Zink spaces, Zhang's proof uses  global calculations on unitary Shimura varieties, and makes essential use of the modularity result of Theorem \ref{ThmB}.  See  \cite{zhangAFLproof}.
\end{remark}

\begin{remark}
Theorem \ref{ThmB} implies that the $\Q$-span of the classes $ \widehat{\mathcal{Z}}^\tot_\Kra(m) $ is finite dimensional.  See Remark \ref{rem:finite span}.
\end{remark}

\begin{remark}
There is a second method of constructing Green functions for the special divisors, based on the methods of \cite{Ku97}, which gives rise to a non-holomorphic variant of $\widehat{\phi}(\tau)$.  It is a recent theorem of Ehlen-Sankaran \cite{ES} that Theorem \ref{ThmB} implies the modularity of this non-holomorphic generating series.  See \S \ref{ss:nonhol}.
\end{remark}

One motivation for the modularity result of Theorem \ref{ThmB} is that it allows one to construct arithmetic theta lifts.
 If $g(\tau)\in S_n(\Gamma_0(D),\chi_\kk^n)$ is a classical scalar valued cusp form, we may form the Petersson inner product
\[
\widehat{\theta}(g) \define \langle   \widehat{\phi} ,g \rangle_\mathrm{Pet} \in \Cha_\C^1(\mathcal{S}^*_\Kra)
\]
as in \cite{Ku04}.  
One expects, as in [\emph{loc.~cit.}],  that the arithmetic  intersection pairing of $\widehat{\theta}(g)$ against other cycle classes should be related to derivatives of $L$-functions, providing generalizations of the Gross-Zagier and Gross-Kohnen-Zagier theorems.  Specific instances in which this expectation is fulfilled can be deduced from \cite{BHY,Ho1,Ho2}.  This will be explained  in the companion paper \cite{BHKRY-2}.

As this paper is rather long, we  explain in the next two subsections  the main ideas that go into the proof of Theorem \ref{ThmA}.   
The proof of Theorem \ref{ThmB} is exactly the same, but one must keep track of Green functions.


\subsection{Sketch of the proof, part I: the generic fiber}
\label{ss:intro sketch I}


In this subsection we sketch the proof of modularity only in the generic fiber.  That is, the modularity of
\begin{equation}\label{intro generic series}
\sum_{m\ge 0}   \mathcal{Z}^\tot_\Kra(m)_{/\kk} \cdot q^m \in \Ch_\Q^1(\mathcal{S}^*_{\Kra/ \kk} )[[ q]] .
\end{equation}
The key to the proof is the study of \emph{Borcherds products} \cite{Bo1,Bo2}.

A Borcherds product is a meromorphic modular form on an orthogonal Shimura variety, whose construction depends on a choice of weakly holomorphic input form, typically of negative weight.  In our case the input form is any 
\begin{equation}\label{intro input form}
f (\tau) = \sum_{  m\gg -\infty } c(m) q^m  \in M^{!,\infty}_{2-n}(D,\chi_\kk^{n-2} ),
\end{equation}
where the superscripts $!$ and $\infty$ indicate that the weakly holomorphic form $f(\tau)$ of weight $2-n$ and level $\Gamma_0(D)$ is allowed to have a pole at the cusp $\infty$, but must be holomorphic at all other cusps.
We assume also that all  $c(m) \in \Z$.

Our Shimura variety $\mathrm{Sh}(G,\mathcal{D})$ admits a natural map to an orthogonal Shimura variety.  Indeed, the $\kk$-vector space
\[
V=\Hom_\kk(W_0,W)
\]
admits a natural hermitian form  $\langle \cdot,\cdot\rangle$ of signature $(n-1,1)$, induced by the hermitian forms on $W_0$ and $W$.  The natural action of $G$ on $V$ determines an exact sequence 
\begin{equation}\label{G to U}
1 \to \mathrm{Res}_{\kk/\Q} \mathbb{G}_m \to G \to \mathrm{U}(V) \to 1
\end{equation}
of reductive groups over $\Q$.

We may also view $V$ as a $\Q$-vector space endowed with the  quadratic form $Q(x)=\langle x,x\rangle$ of signature $(2n-2, 2)$, and so obtain a homomorphism $G\to \SO(V)$.  This induces a map from $\mathrm{Sh}(G,\mathcal{D})$ to the Shimura variety associated with the group $\mathrm{SO}(V)$.

After possibly replacing $f$ by a nonzero integer multiple, Borcherds constructs   a meromorphic modular form on the orthogonal Shimura variety, which can be pulled back to a meromorphic modular form on $\mathrm{Sh}(G,\mathcal{D})(\C)$.  
The result is a meromorphic section $\bm{\psi}(f)$ of $\bm{\omega}^k$, where the weight 
\begin{equation}\label{k intro}
k = \sum_{r\mid D} \gamma_r  \cdot c_r(0) \in \Z
\end{equation}
 is the integer defined in  \S \ref{ss:unitary borcherds}.  
 The constant  $\gamma_r =\prod_{p\mid r}\gamma_p$ is a $4^\mathrm{th}$ root of unity (with $\gamma_1=1$) and $c_r(0)$ is the constant term of $f$ at the cusp 
 \[
 \infty_r = \frac{r}{D} \in \Gamma_0(D)\backslash  \mathbb{P}^1(\Q),
 \] 
 in the sense of Definition \ref{def:constant at cusp}.

Initially, $\bm{\psi}(f)$ is characterized by specifying $-\log\| \bm{\psi}(f) \|$, where $\| \cdot \|$ is the Petersson norm on $\bm{\omega}^k$.  
In particular, $\bm{\psi}(f)$ is only defined up to rescaling by a complex number of absolute value $1$ on each connected component of $\mathrm{Sh}(G,\mathcal{D})(\C)$.    
We prove that, after a suitable rescaling, $\bm{\psi}(f)$ is the analytification of a rational section of the line bundle $\bm{\omega}^k$ on $\mathrm{Sh}(G,\mathcal{D})$.  
In other words, the Borcherds product is algebraic and defined over the reflex field $\kk$.   
This allows us to  view $\bm{\psi}(f)$ as a rational section of $\bm{\omega}^k$ both on the integral model $\mathcal{S}_\Kra$, and on its toroidal compactification.

We compute the divisor of $\bm{\psi}(f)$ on the generic fiber of the toroidal compactification $\mathcal{S}^*_{\Kra/\kk}$, and find
\begin{equation}\label{intro borcherds generic}
\mathrm{div}(  \bm{\psi}(f) )_{/ \kk} =  \sum_{m> 0} c(-m) \cdot \mathcal{Z}_\Kra^\tot(m)_{/\kk}.
\end{equation}
The calculation of the divisor on the interior $\mathcal{S}_{\Kra/\kk}$ follows immediately from the corresponding calculations of Borcherds on the orthogonal Shimura variety.  The multiplicities of the boundary components are computed using the results of \cite{Ku:ABP}, which describe the structure of the Fourier-Jacobi expansions of $\bm{\psi}(f)$ along the various boundary components.

The equality of divisors (\ref{intro borcherds generic}) implies the relation
\[
k\cdot \bm{\omega}  =  \sum_{m> 0} c(-m) \cdot \mathcal{Z}_\Kra^\tot(m)_{/\kk}
\]
in the Chow group $\Ch^1_\Q( \mathcal{S}_{\Kra/\kk}^*)$.
The cusp $\infty_1=1/D$ is $\Gamma_0(D)$-equivalent to the usual cusp at $\infty$, and so $c_1(0)=c(0)$. 
Substituting the expression (\ref{k intro}) for $k$ into the left hand side and using (\ref{intro naked constant}) therefore yields the relation
\begin{equation}\label{intro borcherds generic 2}
  \sum_{  \substack{ r\mid D \\ r> 1 }} \gamma_r c_r(0) \cdot   \bm{\omega} =  \sum_{m\ge  0} c(-m) \cdot \mathcal{Z}_\Kra^\tot(m)_{/\kk}
\end{equation}
in $\Ch^1_\Q( \mathcal{S}_{\Kra/\kk}^*)$.   In \S \ref{ss:eisenstein} we construct for each $r\mid D$ an Eisenstein series 
\[
E_r(\tau)  = \sum_{ m \ge 0 } e_r(m) \cdot q^m  \in M_n(D,\chi_\kk^n),
\]
which, by a simple residue calculation, satisfies
\[
 c_r(0) =  -\sum_{ m>0 } c(-m) e_r(m).
\]
Substituting this expression  into (\ref{intro borcherds generic 2}) yields
\begin{equation}\label{intro borcherds generic 3}
0=  \sum_{m\ge  0} c(-m) \cdot  \Big( \mathcal{Z}_\Kra^\tot(m)_{/\kk}
+   \sum_{  \substack{ r\mid D \\ r> 1 }}   \gamma_r e_r(m) \cdot   \bm{\omega} \Big),
\end{equation}
where we have also used the relation  $e_r(0)=0$ for $r>1$.

We now invoke a variant of the modularity criterion of \cite{Bo2}, which is our Theorem \ref{thm:modularity criterion}: if a formal $q$-expansion
\[
\sum_{m\ge 0} d(m) q^m \in \C[[q]]
\]
satisfies 
$
0=\sum_{m\ge 0} c(-m) d(m)
$ 
for every input form (\ref{intro input form}), then  it must be the $q$-expansion of a modular form of weight $n$, level $\Gamma_0(D)$, and character $\chi_\kk^n$.    It follows immediately from this and  (\ref{intro borcherds generic 3})  that the formal $q$-expansion
\[
\sum_{m\ge 0} \Big(  \mathcal{Z}_\Kra^\tot(m)_{/\kk}
+  \sum_{  \substack{ r\mid D \\ r> 1 }}   \gamma_r e_r(m) \cdot   \bm{\omega} \Big) \cdot  q^m
\]
is modular in the sense of Theorem \ref{ThmA}.  Rewriting this as
\[
\sum_{m\ge 0}  \mathcal{Z}_\Kra^\tot(m)_{/\kk} \cdot q^m + \sum_{  \substack{ r\mid D \\ r> 1 }}   \gamma_r E_r(\tau) \cdot   \bm{\omega} 
\]
and using the modularity of each Eisenstein series $E_r(\tau)$, we deduce that (\ref{intro generic series}) is modular.


\subsection{Sketch of the proof, part II: vertical components}


In order to extend the arguments of \S \ref{ss:intro sketch I} to prove Theorem \ref{ThmA}, it is clear that one should attempt to 
compute the divisor of the Borcherds product $\bm{\psi}(f)$ on the integral model $\mathcal{S}^*_\Kra$ and hope for an expression similar to (\ref{intro borcherds generic}).
Indeed, the bulk of this paper is devoted to precisely this problem.

The  subtlety is that both $\mathrm{div}(\bm{\psi}(f))$ and $\mathcal{Z}_\Kra^\tot(m)$ will turn out to have vertical components supported in characteristics dividing $D$.  Even worse, in these bad characteristics the components of the exceptional divisor  $\mathrm{Exc}\subset \mathcal{S}^*_\Kra$  do not intersect the boundary, and so the multiplicities of these components in the divisor of $\bm{\psi}(f)$ cannot be detected by examining its Fourier-Jacobi expansion.

This is where the second integral model $\mathcal{S}^*_\Pap$  plays an essential role.  
The morphism (\ref{intro models}) sits in a cartesian diagram
\[
\xymatrix{
{  \mathrm{Exc}  } \ar[r]\ar[d]   &  {   \mathcal{S}^*_\Kra  }  \ar[d] \\
{  \mathrm{Sing} } \ar[r]  &   { \mathcal{S}^*_\Pap ,} 
}
\]
where the \emph{singular locus} $\mathrm{Sing} \subset \mathcal{S}^*_\Pap$ is the reduced closed substack of points at which the structure morphism $\mathcal{S}^*_\Pap \to \Spec(\co_\kk)$ is not smooth.  
It is $0$-dimensional and supported in characteristics dividing $D$.  The right vertical arrow restricts to an isomorphism
\begin{equation}\label{intro nonsing}
\mathcal{S}^*_\Kra \smallsetminus  \mathrm{Exc}  \iso   \mathcal{S}^*_\Pap\smallsetminus  \mathrm{Sing}.
\end{equation}
 For each connected component $s\in \pi_0(\mathrm{Sing})$ the fiber 
\[
\mathrm{Exc}_s= \mathrm{Exc}\times_{\mathcal{S}^*_\Pap} s 
\]
is a smooth, irreducible, vertical Cartier divisor on $\mathcal{S}^*_\Kra$, and $\mathrm{Exc}=\bigsqcup_s \mathrm{Exc}_s$.

As the $\co_\kk$-stack  $\mathcal{S}_\Pap^*$ is proper and normal with  normal fibers,
every irreducible vertical divisor on it is the reduction, modulo some prime of $\co_\kk$, of an entire connected (=irreducible) component.
From this it follows that every vertical divisor meets the boundary.  
Thus one could hope to use (\ref{intro nonsing}) to view $\bm{\psi}(f)$ as a rational section on $\mathcal{S}_\Pap^*$, compute its divisor there by examining Fourier-Jacobi expansions, and then pull that calculation back to $\mathcal{S}_\Kra^*$.

This is essentially what we do, but there is an added complication.  The line bundle $\bm{\omega}$ on (\ref{intro nonsing}) does not extend to $\mathcal{S}_\Pap^*$, and similarly the divisor $\mathcal{Z}_\Kra^*(m)$ on (\ref{intro nonsing}) cannot be extended across the singular locus to a Cartier divisor on $\mathcal{S}_\Pap^*$.   However, if you square the line bundle and the divisors, they have much better behavior.
This is the content of the following result, which is an amalgamation of Theorems \ref{thm:weight two nonsingular}, \ref{thm:pure divisor},  \ref{thm:cartier error}, and \ref{thm:toroidal} of the text.

\begin{bigtheorem}\label{ThmC}
There is a unique line bundle $\pure_\Pap$  on $\mathcal{S}_\Pap^*$ whose restriction to (\ref{intro nonsing}) is isomorphic to $ \bm{\omega}^2$. Denoting by $\pure_\Kra$ its pullback to $\mathcal{S}_\Kra^*$, there is an isomorphism
\[
\bm{\omega}^2  \iso  \pure_\Kra \otimes \co(\mathrm{Exc}).
\]
Similarly, there is a unique  Cartier divisor $\mathcal{Y}_\Pap^\tot(m)$ on $\mathcal{S}_\Pap^*$  whose restriction to (\ref{intro nonsing}) is equal  to $2\mathcal{Z}_\Kra^\tot(m)$. 
 Its pullback $\mathcal{Y}^\tot_\Kra(m)$ to $\mathcal{S}^*_\Kra$ satisfies
\[
2 \mathcal{Z}^\tot_\Kra(m)
 = \mathcal{Y}^\tot_\Kra(m) + \sum_{ s\in \pi_0(\mathrm{Sing}) }  \# \{ x\in L_s : \langle x,x\rangle =m \}  \cdot \mathrm{Exc}_s.
\]
Here $L_s$ is a positive definite self-dual hermitian lattice of rank $n$ associated to the singular point $s$, and $\langle \cdot\, ,\cdot\rangle$ is its hermitian form.
\end{bigtheorem}

Setting $\mathcal{Y}_\Pap^\tot(0) = \pure_\Pap^{-1}$, we obtain a  formal generating series 
\[
\sum_{m\ge 0} \mathcal{Y}_\Pap^\tot(m) \cdot q^m \in \Ch^1_\Q(\mathcal{S}_\Pap^*) [[q]]
\]
whose pullback via $\mathcal{S}_\Kra^* \to \mathcal{S}_\Pap^*$ is  twice the generating series of Theorem \ref{ThmA}, up to an error term coming from the exceptional divisors.  More precisely,  Theorem \ref{ThmC} shows that the pullback is
\[
2 \sum_{m\ge 0} \mathcal{Z}_\Kra^\tot(m) \cdot q^m - \sum_{s\in \pi_0(\mathrm{Sing}) } \vartheta_s(\tau) \cdot \mathrm{Exc}_s \in \Ch^1_\Q(\mathcal{S}_\Kra^*) [[q]],
\]
where each $\vartheta_s(\tau)$ is the classical theta function whose coefficients count points in the positive definite hermitian lattice $L_s$.

Over  (\ref{intro nonsing}) we have 
$
\bm{\omega}^{2k} \iso \pure_\Pap^k,
$
which allows us to view  $\bm{\psi}(f)^2$ as a rational section of the line bundle  $\pure_\Pap^k$  on $\mathcal{S}_\Pap^*$.   
We  examine its Fourier-Jacobi expansions along the boundary components and are able to compute its divisor completely (it happens to include nontrivial vertical components).  We then pull this calculation back to $\mathcal{S}_\Kra^*$, and find that $\bm{\psi}(f)$, when viewed as a rational section of $\bm{\omega}^k$, has divisor
\begin{align*}
\mathrm{div}( \bm{\psi}  (f) )   
& =   \sum_{m>0} c(-m) \cdot \mathcal{Z}_\Kra^\tot(m)     + 
   \sum_{ r \mid D} \gamma_r c_r(0)  \cdot  \Big(   \frac{ \mathrm{Exc}}{2}+    \sum_{p\mid r} \mathcal{S}^*_{\Kra /\F_\mathfrak{p}}  \Big) \\
 & \quad -  \sum_{m>0}   \frac{c(-m)}{2}    \sum_{ s\in \pi_0(\mathrm{Sing}) }   \# \{ x\in L_s : \langle x,x\rangle =m \} \cdot  \mathrm{Exc}_s \\
 & \quad -k\cdot  \mathrm{div}(\delta)
\end{align*}
where $\delta\in \co_\kk$ is a square root of  $-D$, $\mathfrak{p} \subset \co_\kk$ is  the unique prime above $p\mid D$, and $\mathcal{S}^*_{\Kra /\F_\mathfrak{p}}$ is the mod $\mathfrak{p}$ fiber of $\mathcal{S}^*_\Kra$, viewed as a divisor. 
This is stated in the text as Theorem \ref{thm:unitary borcherds II}.
Passing to the generic fiber recovers (\ref{intro borcherds generic}), as it must.

 As in the argument leading to (\ref{intro borcherds generic 3}), this implies the relation
\begin{align*}
0  
& =   \sum_{m\ge 0} c(-m) \cdot 
 \Bigg( \mathcal{Z}_\Kra^\tot(m)    -  \frac{1}{2}  \sum_{ s\in \pi_0(\mathrm{Sing}) }   \# \{ x\in L_s : \langle x,x\rangle =m \} \cdot  \mathrm{Exc}_s   \Bigg)\\
& \quad 
+  \sum_{m\ge 0} c(-m) \cdot 
  \sum_{  \substack{ r \mid D \\ r>1} } \gamma_r e_r(m)  \Bigg(  \bm{\omega}  -  \frac{\mathrm{Exc}}{2}   -  \sum_{p\mid r} \mathcal{S}^*_{\Kra /\F_\mathfrak{p}}  \Bigg) 
\end{align*}
in the Chow group of $\mathcal{S}_\Kra^*$,
and the modularity criterion implies that
\begin{align*}
\sum_{m\ge 0} \mathcal{Z}_\Kra^\tot(m) \cdot q^m & - \frac{1}{2}  \sum_{s\in \pi_0(\mathrm{Sing})} \vartheta_s(\tau) \cdot \mathrm{Exc}_s \\
& \quad +  \sum_{\substack{  r\mid D \\ r>1 } } \gamma_r E_r (\tau) \cdot  \Bigg( \bm{\omega}   -   \frac{\mathrm{Exc}}{2}   -  \sum_{p\mid r} \mathcal{S}^*_{\Kra /\F_\mathfrak{p}}  \Bigg)
\end{align*}
is a modular form.
As each theta series $\vartheta_s(\tau)$ and Eisenstein series $E_r(\tau)$  is modular, so is $\sum  \mathcal{Z}_\Kra^\tot(m)  \cdot  q^m$. 
 This completes the outline of the proof of Theorem \ref{ThmA}.


\subsection{The structure of the paper}


We now briefly describe the contents of the various sections of the paper. 

In \S \ref{s:unitary} we introduce the unitary Shimura variety associated to the group $G$ of (\ref{G def}),  and explain its realization 
as a moduli space of pairs $(A_0,A)$ of abelian varieties with extra structure. 
We then review the integral models constructed by Pappas and Kr\"amer, and the singular and exceptional loci of these models.
These are related by a cartesian diagram
\[
\xymatrix{
{  \mathrm{Exc}  } \ar[r]\ar[d]   &  {   \mathcal{S}_\Kra  }  \ar[d] \\
{  \mathrm{Sing} } \ar[r]  &   { \mathcal{S}_\Pap ,} 
}
\]
where the vertical arrow on the right is an isomorphism outside of the $0$-dimensional singular locus $\mathrm{Sing}$. 
We also define   the line bundle of modular forms $\bm{\omega}$ on $\mathcal{S}_\Kra$.

The first main result of \S \ref{s:unitary} is Theorem \ref{thm:weight two nonsingular}, which  asserts the existence of a line bundle $\pure_\Pap$ on $\mathcal{S}_\Pap$ restricting to $\bm{\omega}^2$ over
\[
\mathcal{S}_\Kra \smallsetminus  \mathrm{Exc}  \iso   \mathcal{S}_\Pap\smallsetminus  \mathrm{Sing}.
\]
We then define the Cartier divisor  $\mathcal{Z}_\Kra(m)$  on $\mathcal{S}_\Kra$ and prove  Theorem \ref{thm:pure divisor}, 
which asserts the existence of a  Cartier divisor $\mathcal{Y}_\Pap(m)$ on $\mathcal{S}_\Pap$
whose restriction to  $\mathcal{S}_\Pap\smallsetminus  \mathrm{Sing}$ coincides with $2\mathcal{Z}_\Kra(m)$.
Up to error terms supported on the exceptional locus $\mathrm{Exc}$, the pullbacks of $\Omega_\Pap$ and $\mathcal{Y}_\Pap(m)$ to $\mathcal{S}_\Kra$ are therefore equal to $\bm{\omega}^2$ and $2 \mathcal{Z}_\Kra(m)$, respectively.   The  error terms are  computed in  Theorem \ref{thm:cartier error}, which is the analogue of 
 Theorem \ref{ThmC} for the noncompactified Shimura varieties.

In \S  \ref{s:unitary compactification} we describe the canonical  toroidal compactifications $\mathcal{S}_\Kra^* \to \mathcal{S}_\Pap^*$, 
and the structure of their  formal completions along the boundary.  
In \S \ref{ss:cusp notation} and \S \ref{ss:mixed data}  we introduce the cusp labels $\Phi$ that  index the boundary components, and their associated mixed Shimura varieties.
In \S \ref{ss:one-motives} we construct  smooth integral models $\mathcal{C}_\Phi$ of these mixed Shimura varieties, following the general recipes of the 
theory of arithmetic toroidal compactification, as moduli spaces of $1$-motives.
In \S \ref{ss:second moduli} we give a second moduli interpretation of these integral models. 
This is one of the key technical steps in our work, and  allows us to compare 
Fourier-Jacobi expansions  on our unitary Shimura varieties to 
Fourier-Jacobi expansions on orthogonal Shimura varieties.  
See the remarks at the beginning of \S  \ref{s:unitary compactification}  for further discussion. 
In \S \ref{ss:mfbundle} and \S \ref{ss:mixed special divisors} we construct the line bundle of modular forms and the special divisors on the mixed Shimura varieties $\mathcal{C}_\Phi$.  Theorem~\ref{thm:toroidal} describes the canonical toroidal compactifications  $\mathcal{S}^*_\Kra$ and $\mathcal{S}_\Pap^*$ and their properties. 
In \S \ref{ss:abstract FJ} we describe the Fourier-Jacobi expansions of sections of $\bm{\omega}^k$ on $\mathcal{S}_\Kra^*$ in algebraic language, 
and in \S \ref{ss:explicit boundary}  we explain how to express  these Fourier-Jacobi coefficients in classical complex analytic coordinates.

In the short \S \ref{s:modular forms} we introduce the weakly holomorphic modular forms that will be used as inputs for the construction of  Borcherds products.
 We also state in Theorem~\ref{thm:modularity criterion} a variant of the modularity criterion of Borcherds.

In \S \ref{s:divisor calc}  we consider the unitary Borcherds products associated to weakly holomorphic forms 
\begin{equation}\label{intro input}
f\in M^{!,\infty}_{2-n}(D,\chi_\kk^{n-2} ).
\end{equation}
Ultimately, the integrality properties of the unitary Borcherds products  will be deduced from an analysis of their Fourier-Jacobi expansions. 
These expansions involve certain products of Jacobi theta  functions, and so, in \S \ref{s:divisor calc}
we review facts about the arithmetic theory of Jacobi forms.
For us, Jacobi forms will be sections of a suitable line bundle $\mathcal{J}_{k,m}$ on the 
universal elliptic curve living over the moduli stack (over $\Z$) of all elliptic curves. 
The key point is to have a precise description of the divisor of the canonical section 
\[
\Theta^{24}\in H^0(\mathcal{E},\mathcal{J}_{0,12})
\]
 of  Proposition~\ref{prop:integral jacobi}. 
In \S \ref{ss:borcherdsquad} we prove Borcherds quadratic identity, allowing  us to relate $\mathcal{J}_{0,1}$ to  a certain line bundle (determined by a Borcherds product) on the boundary component $\mathcal{B}_\Phi$ associated to a cusp label $\Phi$.

 After these technical preliminaries, we come to the statements of our main results about unitary Borcherds products.
Theorem~\ref{thm:unitary borcherds I} asserts that, for each weakly holomorphic form (\ref{intro input}) satisfying integrality conditions on the Fourier coefficients $c(m)$ with $m\le 0$,  there is a rational section  $\bm{\psi}(f)$ of the line bundle $\bm{\omega}^k$ on
$\mathcal{S}^*_\Kra$ with explicit divisor on the generic fiber and prescribed zeros and poles along each boundary component. 
Moreover, for each cusp label $\Phi$,  the leading Fourier-Jacobi coefficient  of $\bm{\psi}(f)$  has an expression as a product of three factors, two of which,  $P_\Phi^{vert}$ and $P_{\Phi}^{hor}$,   are constructed in terms of $\Theta^{24}$.   
Theorem~\ref{thm:unitary borcherds II} gives the precise divisor of $\bm{\psi}(f)$ on 
$\mathcal{S}_\Kra^*$, and Theorem~\ref{thm:unitary borcherds III} gives an analogous formula on 
$\mathcal{S}_\Pap^*$.   An essential ingredient in the calculation of these divisors is the calculation of the divisors of the factors  $P_\Phi^{vert}$ and $P_{\Phi}^{hor}$, which is done  in \S \ref{ss:div calc bd}.

In \S \ref{s:analytic borcherds} we prove the main results stated in \S \ref{ss:unitary borcherds}. 
In \S \ref{ss:vector-valued} we construct a vector valued form $\tilde{f}$  from  (\ref{intro input}),  
and give expressions for its Fourier  coefficients in terms of those of $f$.  
The vector valued form $\tilde{f}$ defines a Borcherds product $\tilde{\bm{\psi}}(f)$ on the symmetric space $\tilde{\mathcal{D}}$ for  
the orthogonal group of the quadratic space  $(V,Q)$ and, in \S \ref{ss:borcherds define}, we define the unitary Borcherds product $\bm{\psi}(f)$ as its pullback to $\mathcal{D}$.  
In \S \ref{ss:analytic-FJ} we determine the analytic Fourier-Jacobi expansion of $\bm{\psi}(f)$ at the cusp $\Phi$ by pulling back the  product formula for $\tilde{\bm{\psi}}(f)$ computed in \cite{Ku:ABP} along a one-dimensional boundary component of $\tilde{\mathcal{D}}$.
In \S \ref{ss:alganddescent}  we show that the unitary Borcherds product constructed analytically arises from a rational section of $\bm{\omega}^k$ and that, after rescaling by a constant of absolute value $1$, this section is defined over $\kk$.  This is Proposition~\ref{prop:B descent}.  
In \S \ref{ss:conclude pf} we complete the proofs of Theorems~\ref{thm:unitary borcherds I}, ~\ref{thm:unitary borcherds II}, and ~\ref{thm:unitary borcherds III}.  

In \S \ref{s:modularity} we use the calculation of the divisors of Borcherds products to prove the modularity results discussed in detail earlier in the introduction. 

In  \S \ref{s:appendix}  we provide some supplementary technical calculations.


\subsection{The case $n=2$}


Throughout the introduction we have assumed that $n\ge 3$, but one could ask if similar results hold for $n=2$.
This seems to be a delicate question.

The assumption that $n\ge 3$ guarantees that $W$ contains an isotropic $\kk$-line, which implies that $\mathrm{Sh}(G,\mathcal{D})$ has no compact (meaning proper over $\kk$) components.
When $n=2$ the Shimura variety $\mathrm{Sh}(G,\mathcal{D})$ is essentially a union of classical modular curves (if $W$ contains an isotropic $\kk$-line) or of compact quaternionic Shimura curves (if $W$ contains no isotropic $\kk$-line).    
 
When $n=2$ one could still construct Borcherds products on $\mathrm{Sh}(G,\mathcal{D})$ as pullbacks from orthogonal Shimura varies, and use the results of \cite{HMP} to prove that they are defined over the reflex field $\kk$.
Analyzing their divisors on the integral models $\mathcal{S}_\Kra \to \mathcal{S}_\Pap$ seems quite difficult.
The compact  case  falls well outside the reach of our arguments, which rely in an essential way on the anaysis of Fourier-Jacobi expansions near the boundary of a toroidal compactification.  

However, even in the noncompact $n=2$ case there are some technical issues that we do not know how to resolve. 
 Foremost among these is that  when $n=2$  the reduction of  $\mathcal{S}_\Pap$ at a prime of $\co_\kk$ above $D$ is not normal, and  so (as in the familiar case of modular curves)   the reduction of an irreducible component need not remain irreducible.
 This causes the proof of Proposition \ref{prop:interior divisor} to break down in a serious way.
In essence, we do not know how to exclude the possibility that  constants $\kappa_\Phi$ appearing in Proposition \ref{prop:algebraic BFJ}  contribute some nontrivial error term to the divisor of the Borcherds product.

In \S \ref{s:unitary} and \S \ref{s:unitary compactification} we  assume $n \ge 2$, but  from  \S \ref{s:divisor calc} onwards we restrict to $n \ge 3$ (the integer $n$ plays no role in the short \S \ref{s:modular forms}).


\subsection{Thanks} The results of this paper are the outcome of a long term project, begun initially in Bonn in June of 2013, and supported in a crucial way 
by three weeklong meetings at AIM, in Palo Alto (May of 2014) and San Jose (November of 2015 and 2016), as part of their AIM SQuaRE's program. The opportunity to spend these periods of intensely focused 
efforts on the problems involved was essential. We would like to thank the University of Bonn and AIM for their support. 



\subsection{Notation}
\label{ss:notation}


Throughout the paper,  $\kk\subset \C$ is a quadratic imaginary field of  odd  discriminant $\mathrm{disc}(\kk)=-D$.  
Denote by $\delta=\sqrt{-D} \in \kk$ the unique choice of square root with $\mathrm{Im}(\delta)>0$, and  by  $\mathfrak{d}=\delta\co_\kk$  the different of $\co_\kk$.

Fix a $\pi\in \co_\kk$ satisfying $\co_\kk = \Z + \Z\pi$.   If $S$ is any $\co_\kk$-scheme, define
\begin{align*}
\epsilon_S   & = \pi \otimes 1 - 1 \otimes i_S(\overline{\pi}) \in  \co_\kk \otimes_\Z \co_S\\
\overline{\epsilon}_S   & =  \overline{\pi} \otimes 1 - 1\otimes i_S( \overline{\pi} ) \in  \co_\kk \otimes_\Z \co_S ,
\end{align*}
where $i_S : \co_\kk \to \co_S$ is the structure map.  
The ideal sheaves generated by these sections are independent of the choice of $\pi$, and sit in exact sequences of free $\co_S$-modules
\[
0 \to (\overline{\epsilon}_S) \to \co_\kk \otimes_\Z \co_S \map{ \alpha\otimes x\mapsto i_S(\alpha) x} \co_S \to 0
\]
and
\[
0 \to (\epsilon_S) \to \co_\kk \otimes_\Z \co_S \map{ \alpha\otimes x\mapsto i_S( \overline{\alpha} ) x  } \co_S \to 0.
\]

It is easy to see that $\epsilon_S \cdot  \overline{\epsilon}_S=0$, and that the images of $(\epsilon_S)$ and $(\overline{\epsilon}_S)$ under
\begin{align*}
\co_\kk \otimes_\Z \co_S & \map{ \alpha\otimes x\mapsto i_S(\alpha) x} \co_S \\
\co_\kk \otimes_\Z \co_S & \map{ \alpha\otimes x\mapsto i_S( \overline{\alpha} ) x  } \co_S ,
\end{align*}
respectively, are both equal to the sub-sheaf $\mathfrak{d}\co_S$.   This defines isomorphisms of $\co_S$-modules
\begin{equation}\label{different bundle}
(\epsilon_S) \iso \mathfrak{d} \co_S \iso (\overline{\epsilon}_S).
\end{equation}

 If $N$ is an $\co_\kk \otimes_\Z \co_S$-module then $N/ \overline{\epsilon}_S N$ is the maximal quotient of $N$ on which $\co_\kk$ acts through the structure morphism $i_S : \co_\kk \to \co_S$, and  $N/\epsilon_S N$ is the maximal quotient on which $\co_\kk$ acts through the complex conjugate of the structure morphism. 
  If $D \in \co_S^\times$ then more is true: there is a decomposition
 \begin{equation}\label{idempotent decomp}
N = \epsilon_S N \oplus \overline{\epsilon}_S N,
\end{equation}
and the summands are the maximal submodules on which $\co_\kk$ acts through the structure morphism and its conjugate, respectively.  
From this discussion it is clear that one should regard $\epsilon_S$ and $\overline{\epsilon}_S$ as integral substitutes for the orthogonal idempotents in $\kk\otimes_\Q \C \iso \C \times \C$.  The $\co_\kk$-scheme $S$ will usually be clear from context, and we abbreviate $\epsilon_S$ and $\overline{\epsilon}_S$ to $\epsilon$ and $\overline{\epsilon}$.

 Let $\kk^\mathrm{ab} \subset \C$ be the maximal abelian extension of $\kk$ in $\C$, and let
\[
\art: \kk^\times \backslash \widehat{\kk}^\times \to \Gal(\kk^\mathrm{ab} /\kk)
\]
be the Artin map of class field theory, normalized as in \cite[\S 11]{Milne}.
As usual, $\mathbb{S}=\mathrm{Res}_{\C/\R} \mathbb{G}_m$ is Deligne's torus.

For a prime $p\leq \infty$ we write $(a,b)_p$ for the Hilbert symbol of $a,b\in \Q_p^\times$. 
Recall that the \emph{invariant} of a hermitian space $V$ over $\kk_p=\kk\otimes_\Q \Q_p$ is defined by 
\begin{align}
\label{eq:locinv}
\operatorname{inv}_p(V) =  (\det V, -D)_p,
\end{align}
where $\det V$ is the determinant of the matrix of the hermitian form with respect to a $\kk_p$-basis.
 If $p<\infty$ then $V$ is determined up to isomorphism by its $\kk_p$-rank and  invariant.
 If $p=\infty$  then $V$ is determined up to isomorphism by its signature $(r,s)$, and its invariant is  $\operatorname{inv}_\infty(V)=(-1)^s$.

The term \emph{stack} always means \emph{Deligne-Mumford stack}.


\section{Unitary Shimura varieties}
\label{s:unitary}


In this section we define a unitary Shimura variety $\mathrm{Sh}(G,\mathcal{D})$ over our quadratic imaginary field $\kk\subset \C$ and describe its moduli interpretation.
We then recall the work of Pappas and Kr\"amer, which provides us with two integral models 
related by a surjection $\mathcal{S}_\Kra \to \mathcal{S}_\Pap$.   This surjection becomes an isomorphism after restriction to $\co_\kk[1/D]$.
We  define a line bundle of weight one modular forms $\bm{\omega}$ and a family of Cartier divisors $\mathcal{Z}_\Kra(m)$,  $m>0$,    
on $\mathcal{S}_\Kra$, 

The line bundle $\bm{\omega}$ and the divisors $\mathcal{Z}_\Kra(m)$ do not descend to $\mathcal{S}_\Pap$, and the main original material in \S \ref{s:unitary} is the construction of suitable substitutes on $\mathcal{S}_\Pap$.  These substitutes consist of a line bundle $\pure_\Pap$ that 
agrees with $\bm{\omega}^2$ after restricting to $\co_\kk[1/D]$, and Cartier divisors $\mathcal{Y}_\Pap(m)$ that agree with  $2 \mathcal{Z}_\Kra(m)$ after   restricting to $\co_\kk[1/D]$.


\subsection{The Shimura variety}
\label{ss:shimura data}


 Let  $W_0$ and $W$ be $\kk$-vector spaces  endowed with hermitian forms $H_0$ and $H$ of signatures  $(1,0)$ and $(n-1,1)$, respectively.   We always assume that $n\ge 2$. Abbreviate 
\[
W(\R)=W\otimes_\Q \R,  \quad W(\C)=W\otimes_\Q\C, \quad W(\A_f) = W\otimes_\Q \A_f,
\]
 and similarly for $W_0$.      In particular, $W_0(\R)$ and $W(\R)$ are hermitian spaces over $\C=\kk\otimes_\Q \R$.
 
 We  assume the existence of $\co_\kk$-lattices $\mathfrak{a}_0 \subset W_0$ and  $\mathfrak{a} \subset W$,  self-dual with respect to  the hermitian  forms $H_0$ and $H$.  
As  the inverse of $\delta=  \sqrt{-D} \in \kk$ generates the inverse different of $\kk/\Q$, this is equivalent to self-duality with respect to the symplectic forms 
 \begin{equation}\label{symplectic construction}
 \psi_0 (w ,w') = \mathrm{Tr}_{\kk/\Q} H_0( \delta^{-1} w,w') ,\quad
  \psi (w ,w') = \mathrm{Tr}_{\kk/\Q}   H (\delta^{-1} w,w') .
  \end{equation}
 This data will remain fixed throughout the paper.

 As in (\ref{G def}), let $G  \subset    \GU( W_0 )  \times \GU( W )$ be the subgroup of pairs  for which the similitude factors are equal. We denote by   $\nu : G\to \mathbb{G}_m$ the common similitude character, and note that 
 $\nu (G(\R)) \subset \R^{>0}$.

 Let $\mathcal{D}(W_0)=\{ y_0\}$  be a one-point set,  and define 
\begin{equation}\label{GU hermitian}
\mathcal{D}(W) = 
\{   \mbox{negative definite  $\C$-planes  $y\subset W(\R)$}   \} ,
\end{equation}
so that $G(\R)$ acts on the connected hermitian domain
\[
\mathcal{D}  =\mathcal{D}( W_0 )  \times \mathcal{D}( W).
\]
   The lattices  $\mathfrak{a}_0$ and $\mathfrak{a}$ determine a  maximal compact open subgroup
\begin{equation}\label{K choice}
K = \big\{ g \in G(\A_f) :
g \widehat{\mathfrak{a}}_0 =  \widehat{\mathfrak{a}}_0 \mbox{ and } g \widehat{\mathfrak{a}}  =  \widehat{\mathfrak{a}}   \big\} \subset G(\A_f),
\end{equation}
and the orbifold quotient
\[
\mathrm{Sh}(G,\mathcal{D}) (\C) = G(\Q) \backslash \mathcal{D} \times G(\A_f) / K
\]
is the space of complex points of a smooth  $\kk$-stack of dimension $n-1$,  denoted $\mathrm{Sh}(G,\mathcal{D})$.

The  symplectic forms (\ref{symplectic construction})  determine a  $\kk$-conjugate-linear isomorphism
\begin{equation}\label{flippy map}
\Hom_\kk( W_0, W ) \map{  x\mapsto x^\vee} \Hom_\kk( W, W_0 ),
\end{equation}
characterized by $\psi ( x w_0 ,w ) = \psi_0 (  w_0 , x^\vee w )$.  The $\kk$-vector space
\[
V = \Hom_\kk (W_0,W)
\]  
carries a hermitian form of signature $(n-1,1)$ defined by 
\begin{equation}\label{hom hermitian}
 \langle x_1 , x_2 \rangle =  x_2^\vee \circ x_1 \in \End_\kk ( W_0) \iso \kk.
\end{equation}
The group   $G$  acts on  $V$ in a natural way, defining an exact sequence (\ref{G to U}).

The hermitian form  on $V$ induces a quadratic form $Q(x)=\langle x,x\rangle$, with associated $\Q$-bilinear form 
\begin{equation}\label{hom quadratic}
[x,y]= \mathrm{Tr}_{\kk/\Q} \langle x,y\rangle .
\end{equation}
In particular, we obtain a  representation $G\to \SO(V)$.

\begin{proposition}\label{prop:component count}
The stack  $\mathrm{Sh}(G,\mathcal{D})_{/\C}$ has $2^{1-o(D)} h^2$ connected components, where $h$ is the class number of $\kk$ and $o(D)$ is the number of prime divisors of $D$.  
\end{proposition}

\begin{proof}
Each  $g\in G(\A_f)$ determines $\co_\kk$-lattices
\[
g \mathfrak{a}_0 = W_0  \cap g\widehat{\mathfrak{a}}_0 ,\quad
g \mathfrak{a} = W \cap g\widehat{\mathfrak{a}} .
\]
The hermitian forms $H_0$ and $H$ need  not be $\co_\kk$-valued on these lattices. However,  if  $\mathrm{rat}(\nu(g))$ denotes the  unique positive rational number  such that 
\[
 \frac{ \nu(g) }  {   \mathrm{rat}(\nu(g)) }   \in \widehat{\Z}^\times
\]
then the rescaled hermitian forms $\mathrm{rat}(\nu(g))^{-1} H_0$ and $\mathrm{rat}(\nu(g))^{-1} H$ make $g\mathfrak{a}_0$ and $g\mathfrak{a}$  into self-dual hermitian lattices.

As $\mathcal{D}$ is connected, the components of $\mathrm{Sh}(G,\mathcal{D})_{/\C}$ are in bijection with the set $G(\Q) \backslash G(\A_f)/K$.  
The function  $g\mapsto (g \mathfrak{a}_0, g  \mathfrak{a})$  establishes a bijection from  $G(\Q) \backslash G(\A_f)/K$  to the set of isometry classes of pairs of self-dual hermitian $\co_\kk$-lattices $(\mathfrak{a}_0',\mathfrak{a}')$ of signatures $(1,0)$ and $(n-1,1)$,  respectively, for which the self-dual hermitian lattice $\Hom_{\co_\kk}(\mathfrak{a}_0' , \mathfrak{a}')$ lies in the same genus as $\Hom_{\co_\kk}(\mathfrak{a}_0 , \mathfrak{a} ) \subset V$.

Using the fact that $\mathrm{SU}(V)$ satisfies strong approximation, 
one can show that there are exactly $2^{1-o(D)} h$ isometry classes in the genus of $\Hom_{\co_\kk}(\mathfrak{a}_0 , \mathfrak{a})$, and each  isometry class arises from exactly $h$ isometry classes of pairs $(\mathfrak{a}_0',\mathfrak{a}')$.
\end{proof}

It will be useful at   times to have other  interpretations of the hermitian domain $\mathcal{D}$. 
The following remarks provide  alternate points of view.  
Recalling the idempotents $\epsilon,\overline{\epsilon} \in \kk\otimes_\Q \C$ of \S \ref{ss:notation}, define isomorphisms of real vector spaces
\begin{equation}\label{idem proj}
\mathrm{pr}_\epsilon : W(\R) \iso \epsilon W(\C) ,\quad \mathrm{pr}_{\overline{\epsilon}} : W(\R) \iso  \overline{\epsilon} W(\C)
\end{equation}
as, respectively,  the compositions
 \begin{align*}
W(\R) \hookrightarrow  W(\C)  & = \epsilon W(\C) \oplus \overline{\epsilon} W(\C) \map{\mathrm{proj.} }  \epsilon W(\C) \\ 
W(\R)  \hookrightarrow W(\C) &= \epsilon W(\C) \oplus \overline{\epsilon} W(\C) \map{\mathrm{proj.} } \overline{\epsilon} W(\C).
\end{align*}

\begin{remark}\label{rem:to so}
  Each pair $z=(y_0,y) \in \mathcal{D}$ determines a  line   $\mathrm{pr}_\epsilon (y) \subset W(\C)$, and hence a line 
\[
z=\Hom_\C ( W_0(\C) /  \overline{\epsilon} W_0(\C) ,  \mathrm{pr}_\epsilon ( y ) )  \subset  \epsilon V(\C).
\]
This construction identifies  
\[
 \mathcal{D}   \iso \big\{ z \in   \epsilon V(\C) :   [z,\overline{z}] <0    \big\} / \C^\times  \subset \mathbb{P}( \epsilon V(\C) )
 \]
 as an open subset of projective space.
\end{remark}

\begin{remark}\label{rem:hodge}
Define a  Hodge structure
\begin{equation*}
\Fil^1  W_0(\C) = 0  ,\quad \Fil^0  W_0(\C) =   \overline{\epsilon} W_0(\C)  ,\quad \Fil^{-1} W_0(\C) =W_0(\C)
\end{equation*}
on $W_0(\C)$, and identify the unique point  $y_0 \in \mathcal{D}(W_0)$ with the  corresponding morphism  $\mathbb{S} \to \GU(W_0)_\R$.   Every   $y\in\mathcal{D}(W)$ defines a Hodge structure
\begin{equation*}
\Fil^1 W(\C) = 0  ,\quad
\Fil^0  W(\C) =  \mathrm{pr}_{\epsilon} ( y)  \oplus \mathrm{pr}_{\overline{\epsilon} } ( y ^\perp)  , \quad 
\Fil^{-1} W (\C) =W(\C)
\end{equation*}
 on $W(\C)$.  If we identify $y\in \mathcal{D}(W)$ with the corresponding morphism $\mathbb{S} \to \GU(W)_\R$, then for any point $z=(y_0,y)\in\mathcal{D}$ the product morphism
\[
  y_0 \times y : \mathbb{S} \to  \GU( W_0)_{\R}  \times \GU( W )_{\R}
\]
takes values in  $G_\R$.  This realizes  $\mathcal{D}  \subset \Hom(\mathbb{S} , G_\R)$  as a $G(\R)$-conjugacy class.
\end{remark}

\begin{remark}
In fact, the discussion above shows that $\mathrm{Sh}(G,\mathcal{D})$ admits a map to the Shimura variety defined the group 
$\mathrm{U}(V)$ together with the homomorphism
$$h_{\text{Gross}}: \mathbb{S} \to \mathrm{U}(V)(\R), \qquad z\mapsto \text{diag}(1, \dots, 1,  \bar{z}/z).$$
Here we have chosen a basis for $V(\R)$ for which the hermitian form has matrix $\mathrm{diag}(1_{n-1},-1)$. Note that, for analogous choices of 
bases for $W_0(\R)$ and $W(\R)$, the corresponding map is
$$h:\mathbb{S} \to  G(\R), \qquad z\mapsto (z)\times \mathrm{diag}(z, \dots,z,\bar z),$$
which, under composition with the homomorphism $G(\R) \rightarrow \mathrm{U}(V)(\R)$,  gives $h_{\mathrm{Gross}}$. 
The existence of this map provides an answer to a question 
posed by Gross:  how can one explicitly relate the Shimura variety defined by the 
unitary group $\mathrm{U}(V)$, as opposed to the Shimura variety defined by the 
similitude group $\mathrm{GU}(V)$,   to a moduli space of abelian varieties?  
Our answer is that  Gross's unitary Shimura variety is a quotient of our $\mathrm{Sh}(G,\mathcal{D})$, whose interpretation as a 
moduli space is explained in the next section.
\end{remark}


\subsection{Moduli interpretation}
\label{ss:generic moduli}


We wish to interpret  $\mathrm{Sh}(G,\mathcal{D}) $  as a moduli space of pairs of abelian varieties with additional structure.  First, we recall some generalities on abelian schemes.

For an abelian scheme $\pi:A\to S$ over an arbitrary base $S$,  define the \emph{first relative de Rham cohomology sheaf}
$
H^1_\dR(A) = \mathbb{R}^1\pi_* \Omega^\bullet_{A/S}
$
as the relative hypercohomology of the de Rham complex $\Omega^\bullet_{A/S}$.  The \emph{relative de Rham homology}
 \[
 H_1^\dR(A) = \underline{\Hom} ( H^1_\dR(A) , \co_S)
 \]
 is a locally free $\co_S$-module of rank $2 \cdot \mathrm{dim}(A)$, sitting  in an exact sequence
\[
0 \to \Fil^0 H_1^\dR(A) \to H_1^\dR(A) \to \Lie(A) \to 0.
\]
Any  polarization of $A$ induces an $\co_S$-valued alternating pairing  on $H_1^\dR(A)$,  which in turn induces a pairing
\begin{equation}\label{fil-lie dual}
\Fil^0H_1^\dR(A) \otimes  \Lie(A) \to \co_S.
\end{equation}
If the polarization is principal then both pairings are perfect.  
When $S=\Spec(\C)$,   Betti homology  satisfies
$
H_1(A(\C),\C) \iso H_1^\dR(A),
$
and
\[
A(\C) \iso H_1(A(\C),\Z) \backslash H_1^\dR(A) /  \Fil^0 H_1^\dR(A).
\]

For any pair of nonnegative integers $(s,t)$, define an algebraic stack $M_{(s,t)}$ over $\kk$ as follows: for any   $\kk$-scheme $S$ let $M_{(s,t)}(S)$ be the groupoid of triples $(A,\iota,\psi)$ in which
\begin{itemize}
\item
$A\to S$ is an abelian scheme of relative dimension $s+t$,
\item
$\iota : \co_\kk \to \End(A)$ is an action  such that the locally free summands 
\[
 \Lie(A)  = \epsilon  \Lie(A) \oplus \overline{\epsilon} \Lie(A)
\]
of (\ref{idempotent decomp}) have $\co_S$-ranks $s$ and $t$, respectively,
\item
$\psi : A\to A^\vee$ is a principal polarization, such that the induced  Rosati involution $\dagger$
on $\End^0(A)$ satisfies $\iota (\alpha)^\dagger = \iota(\overline{\alpha})$ for all $\alpha\in \co_\kk$.
\end{itemize}
We usually omit $\iota$ and $\psi$ from the notation, and just write $A\in M_{(s,t)}(S)$.

\begin{proposition}\label{prop:shimura moduli}
The Shimura variety $\mathrm{Sh}(G,\mathcal{D})$ is isomorphic to an open and closed substack
\begin{equation}\label{moduli inclusion}
\mathrm{Sh}(G,\mathcal{D}) \subset M_{(1,0)}  \times_\kk M_{(n-1,1)}.
\end{equation}
More precisely,     $\mathrm{Sh}(G,\mathcal{D})(S) $ classifies,  for any $\kk$-scheme $S$, pairs
\begin{equation}\label{moduli pair}
(A_0,A) \in M_{(1,0)} (S)  \times M_{(n-1,1)}  (S)
\end{equation}
for which  there exists, at every geometric point  $s\to S$, an isomorphism of hermitian  $\co_{\kk,\ell}$-modules
\begin{equation}\label{tate genus}
\Hom_{\co_\kk}( T_\ell A_{0,s}   , T_\ell A_s  ) \iso \Hom_{\co_\kk} (  \mathfrak{a}_0 ,  \mathfrak{a} ) \otimes \Z_\ell
\end{equation}
for every  prime $\ell$.    
Here the hermitian form on the right hand side of (\ref{tate genus}) is the restriction of the hermitian form  (\ref{hom hermitian})  on $\Hom_\kk(   W_0 , W) \otimes \Q_\ell$. The hermitian form on the left hand side is defined similarly, replacing the symplectic forms (\ref{symplectic construction}) on $W_0$ and $W$ with the Weil pairings on the Tate modules $T_\ell A_{0,s}$ and  $T_\ell A_s$.
\end{proposition}

\begin{proof}
As this is routine, we only describe the open and closed immersion  on complex points.  Fix a point 
\[
( z , g ) \in \mathrm{Sh}( G , \mathcal{D} )(\C) .
\]
The component $g$ determines $\co_\kk$-lattices $g\mathfrak{a}_0\subset W_0$ and $g\mathfrak{a}\subset W$, which are self-dual with respect to the  symplectic forms 
\[
\mathrm{rat}(\nu(g))^{-1} \psi_0 \quad \mbox{ and } \quad  \mathrm{rat}(\nu(g))^{-1} \psi
\]
 of  (\ref{symplectic construction}),  rescaled as in the proof of Proposition \ref{prop:component count}.

By Remark \ref{rem:hodge} the point $z \in \mathcal{D}$ determines Hodge structures on $W_0$ and $W$,  and in this way  $(z  , g)$ determines principally polarized complex abelian varieties
\begin{align*}
A_0(\C)  &= g \mathfrak{a}_0  \backslash W_0(\C) / \Fil^0(W_0) \\
A(\C)  & = g \mathfrak{a} \backslash W(\C) / \Fil^0(W)
\end{align*}
with  actions of $\co_\kk$.  One can easily check that the pair  $(A_0,A)$ determines a complex point of  $M_{(1,0)} \times_\kk M_{(n-1,1)}$, and this construction defines  (\ref{moduli inclusion}) on complex points.
\end{proof}

The following lemma will be needed in \S \ref{ss:unitary integral models} for the construction of integral models for $\mathrm{Sh}(G,\mathcal{D})$.

\begin{lemma}\label{lem:jac herm}
Fix a $\kk$-scheme $S$, a geometric point $s\to S$, a prime $p$,   and a point (\ref{moduli pair}).  If the relation (\ref{tate genus}) holds for all $\ell\neq p$, then it also holds for $\ell=p$.
\end{lemma}

\begin{proof}
As the stack $\mathrm{Sh}(G,\mathcal{D})$ is of finite type over $\kk$, we may assume that $s=\Spec(\C)$.  The polarizations on $A_0$ and $A$ induce symplectic forms on the first homology groups $H_1(A_{0,s}(\C) ,\Z)$ and  $H_1(A_s(\C) ,\Z)$,  and the construction (\ref{hom hermitian}) makes 
\[
L_{\mathrm{Be}}(A_{0,s},A_s) = \Hom_{\co_\kk}\big( H_1(A_{0,s} (\C) ,\Z) , H_1(A_s (\C) ,\Z) \big)
\]
into a self-dual hermitian $\co_\kk$-lattice of signature $(n-1,1)$, satisfying
\[
L_{\mathrm{Be}}(A_{0,s},A_s) \otimes_\Z \Z_\ell \iso \Hom_{\co_\kk}( T_\ell A_{0,s}   , T_\ell A_s  )
\]
for all primes $\ell$.

 If the relation (\ref{tate genus}) holds for all primes $\ell\neq p$, then $L_{\mathrm{Be}}(A_{0,s},A_s)\otimes \Q$ and $\Hom_\kk(W_0,W)$ are isomorphic as $\kk$-hermitian spaces everywhere locally except at $p$, and so they are isomorphic  at $p$ as well.  In particular,  for every $\ell$ (including $\ell=p$) both sides of (\ref{tate genus}) are isomorphic to self-dual lattices in the hermitian space $\Hom_\kk(W_0,W)\otimes_\Q\Q_\ell$.   By the  results of Jacobowitz \cite{Jac}  all self-dual lattices in this local hermitian space are isomorphic\footnote{This uses our standing hypothesis that $D$ is odd.}, and so (\ref{tate genus}) holds for all $\ell$.
\end{proof}

\begin{remark}
For any positive integer $m$ define
\[
K(m) = \mathrm{ker}\big( K  \to  \Aut_{\co_\kk} (  \widehat{\mathfrak{a}}_0 / m   \widehat{\mathfrak{a}}_0 )
  \times   \Aut_{\co_\kk} (  \widehat{\mathfrak{a}} / m  \widehat{\mathfrak{a}} ) \big).
\]
 For a    $\kk$-scheme $S$, a \emph{$K(m)$-structure} on  
$
 (A_0,A)\in \mathrm{Sh}(G,\mathcal{D})(S)
 $
is a triple $(\alpha_0, \alpha, \zeta)$ in which
$
\zeta:   \underline{\mu_m }  \iso \underline{ \Z/m\Z }  
$
is an isomorphism of $S$-group schemes, and
\[
\alpha_0  :   A_0[ m ]    \iso   \underline{ \widehat{\mathfrak{a}}_0/ m \widehat{\mathfrak{a}}_0 } ,\quad
\alpha  : A[m] \iso \underline{ \widehat{\mathfrak{a}} / m \widehat{\mathfrak{a}} }
\]
are  $\co_\kk$-linear  isomorphisms  identifying  the  Weil  pairings  on $A_0[m]$ and $A[m]$ with  the $\Z/m\Z$-valued symplectic forms on $ \widehat{\mathfrak{a}}_0/ m \widehat{\mathfrak{a}}_0$ and  $\widehat{\mathfrak{a}}/ m \widehat{\mathfrak{a}}$ deduced from the pairings (\ref{symplectic construction}).   The Shimura variety $G(\Q) \backslash \mathcal{D} \times G(\A_f) / K(m)$  admits a canonical model over $\kk$,  parametrizing   $K(m)$-structures on points of  $ \mathrm{Sh}(G,\mathcal{D})$. 
\end{remark}


\subsection{Integral models}
\label{ss:unitary integral models}


In this subsection we describe two integral models of $\mathrm{Sh}(G,\mathcal{D})$ over $\co_\kk$, related by a morphism  
$\mathcal{S}_\Kra \to \mathcal{S}_\Pap .$

The  first step is to construct an integral model of the moduli space $M_{(1,0)}$.  More generally, we will construct an  integral model of $M_{(s,0)}$ for any $s>0$.  Define an $\co_\kk$-stack $\mathcal{M}_{(s,0)}$ as the moduli space of triples $(A, \iota,\psi )$ over $\co_\kk$-schemes $S$ such that
\begin{itemize}
\item $A\to S$ is an abelian scheme of relative dimension $s$,
\item $\iota:\co_\kk \to \End(A)$ is an action such $\overline{\epsilon} \Lie(A) =0$, or, equivalently,
such that  the induced action of $\co_\kk$ on the $\co_S$-module
$\Lie(A)$ is through the structure map $i_S : \co_\kk \to \co_S$,
\item $\psi :A\to A^\vee$ is a principal polarization whose Rosati involution satisfies $\iota(\alpha)^\dagger = \iota(\overline{\alpha})$
for all $\alpha\in \co_\kk$.
\end{itemize}
The stack $\mathcal{M}_{(s,0)}$ is smooth of relative dimension $0$ over $\co_\kk$ by \cite[Proposition 2.1.2]{Ho2}, and its generic fiber  is the stack $M_{(s,0)}$ defined earlier.

\begin{remark}
The stack $\mathcal{M}_{(n-2,0)}$ will play an important role in  \S \ref{s:unitary compactification}.
In the degenerate case $n=2$, we interpret this as  $\mathcal{M}_{(0,0)} = \Spec(\co_\kk)$.
 The universal abelian scheme over it should be understood as the $0$ group scheme.
\end{remark}

The question of integral models for $M_{(n-1,1)}$ is more subtle, but well-understood after work of Pappas and Kr\"amer.   
The first integral model was defined by Pappas \cite{Pa}. Let  
\[
\mathcal{M}_{(n-1,1)}^\Pap \to \Spec(\co_\kk)
\]
be the stack  whose functor of points assigns to an $\co_\kk$-scheme $S$ the groupoid of triples $(A,\iota,\psi)$ in which
\begin{itemize}
\item $A\to S$ is an abelian scheme of relative dimension $n$,
\item $\iota:\co_\kk \to \End(A)$ is an action satisfying  the determinant condition 
\[
\det( T- \iota(\alpha)  \mid \Lie(A) ) = (T- \alpha)^{n-1} (T- \overline{\alpha} ) \in \co_S[T]
\]
for all $\alpha\in \co_\kk$,  
\item 
$\psi :A\to A^\vee$ is a principal polarization whose Rosati involution satisfies $\iota(\alpha)^\dagger = \iota(\overline{\alpha})$
for all $\alpha\in \co_\kk$,
\item
viewing the elements $\epsilon_S$ and $\overline{\epsilon}_S$ of \S \ref{ss:notation} as endomorphisms of $\Lie(A)$, the induced endomorphisms
\begin{align*}
\bigwedge\nolimits^n  \epsilon_S : \bigwedge\nolimits^n \Lie(A) \to \bigwedge\nolimits^n \Lie(A) \\
\bigwedge\nolimits^2 \overline{\epsilon}_S : \bigwedge\nolimits^2 \Lie(A) \to \bigwedge\nolimits^2 \Lie(A)
\end{align*}
are trivial (\emph{Pappas's wedge condition}).
\end{itemize}
It is clear  that the generic fiber of $\mathcal{M}_{(n-1,1)}^\Pap$ is isomorphic to the moduli space $M_{(n-1,1)}$ defined earlier.  Denote by 
\[
\mathrm{Sing}_{(n-1,1)}  \subset \mathcal{M}^\Pap_{(n-1,1)}
\]
the singular locus: the reduced substack of points at which the structure morphism to $\co_\kk$ is not smooth.

\begin{theorem}[Pappas]\label{thm:pappas model}
The stack $\mathcal{M}_{(n-1,1)}^\Pap$  is  flat over $\co_\kk$ of relative dimension $n-1$, and is Cohen-Macaulay and normal.   Moreover:
\begin{enumerate}
\item
 For any prime $\mathfrak{p}\subset \co_\kk$, the reduction $\mathcal{M}_{(n-1,1) /\F_\mathfrak{p}} ^\Pap$ is Cohen-Macaulay.  If $n>2$ the reduction is  geometrically normal.
\item
The  singular locus is a $0$-dimensional stack, finite over $\co_\kk$ and  supported  in characteristics dividing $D$.  It is the reduced substack underlying the closed substack defined by $\delta \cdot \Lie(A) =0$.
\end{enumerate}
\end{theorem}

\begin{proof}
When $n>2$  all of this is proved in \cite{Pa} using the theory of local models, and it is straightforward to check that the  arguments carry over\footnote{
When $n=2$,  the $\co_\kk$-stack  $\mathcal{M}_{(n-1,1)}^\Pap$ admits a canonical descent to $\Z$, 
and Pappas  analyzes the structure of this descent.  The  descent is regular, but the regularity is destroyed by base change to $\co_\kk$. }
  to the case $n=2$.  
The only change is that if $\mathfrak{p} \subset \co_\kk$ lies above $p\mid D$, 
the stack  $\mathcal{M}^\Pap_{(1,1) /\co_{\kk,\mathfrak{p} } }$ 
is \'etale locally isomorphic to   
\[
\Spec( \co_{\kk,\mathfrak{p} } [ x,y] / ( xy-p)),
\]
 whose special fiber is not normal.
\end{proof}

The stack $\mathcal{M}^\Pap_{(n-1,1)}$ is not regular, but has a natural resolution of singularities.
This leads us to our second integral model of $M_{(n-1,1)}$.  As in the work of Kr\"amer \cite{Kr},  define 
\[
\mathcal{M}^\Kra_{(n-1,1)} \to \Spec(\co_\kk)
\]
to be the stack whose functor of points assigns to an $\co_\kk$-scheme $S$ the groupoid of quadruples  $(A,\iota,\psi,\mathcal{F}_A)$ in which
\begin{itemize}
\item $A\to S$ is an abelian scheme of relative dimension $n$,
\item $\iota:\co_\kk \to \End(A)$ is an action of $\co_\kk$,
\item $\psi :A\to A^\vee$ is a principal polarization satisfying $\iota(\alpha)^\dagger = \iota(\overline{\alpha})$
for all $\alpha\in \co_\kk$,
\item
$\mathcal{F}_A\subset \Lie(A)$ is an $\co_\kk$-stable  $\co_S$-module local direct summand of rank $n-1$ satisfying \emph{Kr\"amer's condition}:  $\co_\kk$ acts on $\mathcal{F}_A$ via the structure map $\co_\kk \to \co_S$, and acts on the line bundle $\Lie(A)/\mathcal{F}_A$ via the complex conjugate of the structure map.  
\end{itemize}

There is a proper morphism
\begin{equation}\label{blowup}
\mathcal{M}^\Kra_{(n-1,1)} \to \mathcal{M}^\Pap_{(n-1,1)}.
\end{equation}
defined by forgetting the subsheaf $\mathcal{F}_A$, and we define the \emph{exceptional locus}
\begin{equation}\label{kramer exceptional}
\mathrm{Exc}_{(n-1,1)} \subset \mathcal{M}^\Kra_{(n-1,1)}
\end{equation}
by the Cartesian diagram
\[
\xymatrix{
{  \mathrm{Exc}_{(n-1,1)}  } \ar[r]\ar[d]   &  {   \mathcal{M}^\Kra_{(n-1,1)}  }  \ar[d] \\
{  \mathrm{Sing}_{(n-1,1)} } \ar[r] &   { \mathcal{M}^\Pap_{(n-1,1)} .}  
}
\]

\begin{theorem}[Kr\"amer]\label{thm:kramer model}
The $\co_\kk$-stack $\mathcal{M}^\Kra_{(n-1,1)} $ is regular and flat with reduced fibers, and satisfies the following properties:
\begin{enumerate}
\item
The exceptional locus (\ref{kramer exceptional})  is a disjoint union of smooth Cartier divisors.
Its   fiber  over a   geometric point $s \to \mathrm{Sing}_{(n-1,1)}$  is isomorphic to the projective space $\mathbb{P}^{n-1}$ over $k(s)$.
 \item
The morphism  (\ref{blowup}) is proper and surjective, and restricts to an isomorphism
\[
\mathcal{M}^\Kra_{(n-1,1)}  \smallsetminus \mathrm{Exc}_{(n-1,1)} \iso\mathcal{M}^\Pap_{(n-1,1)}  \smallsetminus \mathrm{Sing}_{(n-1,1)}.
\]
For an $\co_\kk$-scheme $S$, the inverse of this isomorphism endows 
\[
A \in  \big(\mathcal{M}^\Pap_{(n-1,1)}  \smallsetminus \mathrm{Sing}_{(n-1,1)} \big) (S)
\]
with the subsheaf
$
\mathcal{F}_A  = \mathrm{ker} \big( \overline{\epsilon} : \Lie(A) \to \Lie(A) \big).
$
\end{enumerate}
 \end{theorem}

\begin{proof}
When  $n>2$ all of this is proved in  \cite{Kr} using the theory of local models, and it is straightforward to check that nearly everything\footnote{When $n>2$,  the statement of \cite[Theorem 4.4]{Kr} asserts that the special fiber of the local model of $\mathcal{M}^\Kra_{(n-1,1)}$ is the union of two smooth and geometrically irreducible varieties of dimension $n-1$, whose intersection is smooth and geometrically irreducible  of dimension $n-2$.
When $n=2$, the structure of the local model is slightly different: its geometric special fiber is a union $X_1\cup X_2\cup X_3$ of three irreducible varieties, each  isomorphic to  $\mathbb{P}^1$,  intersecting in such a way that 
$X_1\cap X_2$ and  $X_2\cap X_3$ are distinct reduced points. 
The difference between the two cases occurs because  the scheme $\mathcal{Q}$ defined in the proof of \cite[Theorem 4.4]{Kr}, which parametrizes isotropic lines in a quadratic space of dimension $n$ over a finite field, is geometrically irreducible only when $n>2$.}  
carries over to the case $n=2$.   In particular, if $n=2$ and  $\mathfrak{p} \subset \co_\kk$ lies above $p\mid D$, 
the same arguments used in [\emph{loc.~cit.}] show that   $\mathcal{M}^\Kra_{(1,1) /\co_{\kk,\mathfrak{p} } }$  is \'etale locally isomorphic to  the regular scheme
\[
\Spec( \co_{\kk,\mathfrak{p} } [ x,y] / ( xy-\pi)),
\]
for any uniformizer $\pi \in \co_{\kk,\mathfrak{p} }$.  
\end{proof}

Recalling  (\ref{moduli inclusion}), we define our first integral model 
\[
\mathcal{S}_\Pap  \subset  \mathcal{M}_{(1,0)} \times \mathcal{M}_{(n-1,1)}^\Pap
\]
  as the Zariski closure of $\mathrm{Sh}(G,\mathcal{D}) $ in the fiber product on the right, which, like all fiber products below,  is taken over  over $\Spec(\co_\kk)$. 
   Using Lemma \ref{lem:jac herm}, one can show that  it is characterized as the  open and closed substack  whose functor of points assigns to any $\co_\kk$-scheme $S$ the groupoid of  pairs 
 \[
(A_0, A) \in \mathcal{M}_{(1,0)}( S ) \times \mathcal{M}^\Pap_{(n-1,1)} (S)
\]
such that, at any geometric point $s \to S$, the relation (\ref{tate genus}) holds for all primes $\ell \neq \mathrm{char}(k(s))$.

 Our second integral model of $\mathrm{Sh}(G,\mathcal{D})$ is defined as   the cartesian product
\[
\xymatrix{
{ \mathcal{S}_\Kra } \ar[r]\ar[d]   &  {    \mathcal{M}_{(1,0)} \times \mathcal{M}^\Kra_{(n-1,1)} }  \ar[d] \\
{  \mathcal{S}_\Pap   } \ar[r]  & {    \mathcal{M}_{(1,0)} \times \mathcal{M}_{(n-1,1)}^\Pap   .}
}
\]
The \emph{singular locus} $\mathrm{Sing}   \subset \mathcal{S}_\Pap$ and \emph{exceptional locus}
$\mathrm{Exc}  \subset \mathcal{S}_\Kra$  are defined by the cartesian squares
\[
\xymatrix{
{  \mathrm{Exc}  } \ar[r]\ar[d]   &  {   \mathcal{S}_\Kra  }  \ar[d] \\
{  \mathrm{Sing} } \ar[r]\ar[d]  &   { \mathcal{S}_\Pap } \ar[d]  \\
{  \mathcal{M}_{(1,0)} \times \mathrm{Sing}_{(n-1,1)} }  \ar[r] & { \mathcal{M}_{(1,0)} \times \mathcal{M}_{(n-1,1)}^\Pap  }.
}
\]
Both loci are proper over $\co_\kk$, and supported in characteristics dividing $D$.

\begin{theorem}[Pappas, Kr\"amer]\label{thm:integral comparison}
The $\co_\kk$-stack $\mathcal{S}_\Kra $ is regular and flat with reduced fibers.  
The $\co_\kk$-stack $\mathcal{S}_\Pap$ is Cohen-Macaulay and normal, with Cohen-Macaulay  fibers.  Furthermore:
\begin{enumerate}
\item
If $n>2$, the geometric fibers of $\mathcal{S}_\Pap$  are  normal.
\item
The exceptional locus $\mathrm{Exc} \subset \mathcal{S}_\Kra$  is a disjoint union of smooth Cartier divisors.   The singular locus $\mathrm{Sing}\subset \mathcal{S}_\Pap$ is a reduced closed stack of dimension $0$, supported in characteristics dividing $D$.
\item
 The fiber of $\mathrm{Exc}$ over a   geometric point $s \to \mathrm{Sing}$  is isomorphic to the projective space $\mathbb{P}^{n-1}$ over $k(s)$.
 \item
The morphism  $\mathcal{S}_\Kra \to \mathcal{S}_\Pap$ is surjective, and restricts to an isomorphism
\begin{equation}\label{nonsingular locus}
\mathcal{S}_\Kra  \smallsetminus \mathrm{Exc} \iso \mathcal{S}_\Pap  \smallsetminus \mathrm{Sing}.
\end{equation}
For an $\co_\kk$-scheme $S$, the inverse of this isomorphism endows 
\[
(A_0,A) \in  \big(  \mathcal{S}_\Pap  \smallsetminus \mathrm{Sing} \big) (S)
\]
with the subsheaf
$
\mathcal{F}_A  = \mathrm{ker} \big( \overline{\epsilon} : \Lie(A) \to \Lie(A) \big).
$
\end{enumerate}
 \end{theorem}

\begin{proof}
All of this follows from Theorems \ref{thm:pappas model} and \ref{thm:kramer model}, along with the fact that 
 $\mathcal{M}_{(1,0)} \to \Spec(\co_\kk)$ is finite \'etale.  
\end{proof}

\begin{remark}\label{rem:simple hodge}
Let $(A_0,A)$ be the universal pair over  $\mathcal{S}_\Pap$.
The vector bundle $H^\dR_1(A_0)$ is locally free of rank one over $\co_\kk \otimes_\Z \co_{\mathcal{S}_\Pap}$ and, by definition of the moduli problem defining $\mathcal{S}_\Pap$, its quotient $\Lie(A_0)$ is annihilated by $\overline{\epsilon}$. From this it is not hard to see that 
\[
 \Fil^0 H_1^\dR(A_0) = \overline{\epsilon} H^\dR_1(A_0).
\]
\end{remark}

%


\subsection{The line bundle of modular forms}
\label{ss:unitary bundle}


We now construct a line bundle of modular forms $\bm{\omega}$ on $\mathcal{S}_\Kra$, and consider the subtle question of whether or not it descends to $\mathcal{S}_\Pap$.  
The short answer is that it doesn't, but a more complete answer can be found in Theorems \ref{thm:weight two nonsingular} and \ref{thm:cartier error}.

By Remark \ref{rem:hodge},  every point $z \in \mathcal{D}$ determines  Hodge structures on $W_0$ and $W$ of weight $-1$, and hence a  Hodge structure of weight $0$  on $V=\Hom_\kk( W_0 ,  W )$.
Consider the holomorphic  line bundle $\bm{\omega}^{an}$ on  $\mathcal{D}$ whose fiber at $z$  is the complex line
$
\bm{\omega}^{an}_z =   \Fil^1 V(\C)  
$
determined by this Hodge structure.

\begin{remark}
It is useful to interpret $\bm{\omega}^{an}$ in the notation of Remark \ref{rem:to so}.  
The fiber of $\bm{\omega}^{an}$ at $z=(y_0,y)$ is the line 
\begin{equation}\label{hermitian bundle}
\bm{\omega}^{an}_z = \Hom_\C ( W_0(\C) / \overline{\epsilon} W_0(\C) ,  \mathrm{pr}_\epsilon  ( y )    )   \subset \epsilon V(\C),
\end{equation}
and hence $\bm{\omega}^{an}$ is  simply the restriction of the tautological bundle via the inclusion
\[
 \mathcal{D}   \iso \big\{ w \in   \epsilon V(\C) :   [w,\overline{w}] <0    \big\} / \C^\times \subset \mathbb{P}( \epsilon V(\C)).
\]
\end{remark}

There is  a natural action of $G(\R)$ on the total space of $\bm{\omega}^{an}$,  lifting the natural action on $\mathcal{D}$, and so $\bm{\omega}^{an}$ descends to  a line bundle on the complex orbifold $\mathrm{Sh}(G,\mathcal{D})(\C)$.
This descent  is algebraic,  has a canonical model over the reflex field, and extends in a natural way to the integral model $\mathcal{S}_\Kra$, as we now explain.

Let  $(A_0,A)$ be the universal object over  $\mathcal{S}_\Kra$, 
let $\mathcal{F}_A \subset \Lie(A)$ be the universal subsheaf of Kr\"amer's moduli problem, and let 
\[
\mathcal{F}_A^\perp \subset \Fil^0 H_1^\dR(A) 
\]
be the  orthogonal to $\mathcal{F}_A$ under the pairing (\ref{fil-lie dual}).
It is a rank one $\co_{\mathcal{S}_\Kra}$-module local direct summand on which $\co_\kk$ acts through
 the structure morphism $\co_\kk \to \co_{\mathcal{S}_\Kra}$.  
Define the  \emph{line bundle of weight one modular forms} on $\mathcal{S}_\Kra$  by 
\[
\bm{\omega} = \underline{\Hom}( \Lie(A_0) , \mathcal{F}_A^\perp ),
\]
or, equivalently, 
$
\bm{\omega}^{-1} =  \Lie(A_0) \otimes   \Lie(A) / \mathcal{F}_A .
$

\begin{proposition}
The line bundle $\bm{\omega}$ on $\mathcal{S}_\Kra$ just defined restricts to the already defined $\bm{\omega}^{an}$ in the complex fiber.
Moreover,  on the complement of the exceptional locus $\mathrm{Exc}\subset \mathcal{S}_\Kra$ we have
\[
\bm{\omega} = \underline{\Hom}( \Lie(A_0) ,  \epsilon \Fil^0 H_1^\dR(A)  ).
\]
\end{proposition}

\begin{proof}
The equality $\mathcal{F}_A^\perp = \epsilon \Fil^0 H_1^\dR(A)$ on the complement of $\mathrm{Exc}$ follows from the characterization 
\[
\mathcal{F}_A = \mathrm{ker}( \overline{\epsilon} : \Lie(A) \to \Lie(A) ) 
\]
of Theorem \ref{thm:integral comparison}, and all of the claims follow easily from this and examination of the proof of Proposition \ref{prop:shimura moduli}.
\end{proof}

The line bundle $\bm{\omega}$ does not descend to $\mathcal{S}_\Pap$, but it is closely related to another line bundle that does.  This is the content of the following theorem, whose proof will occupy the remainder of \S \ref{ss:unitary bundle}.  The result will be strengthened in Theorem \ref{thm:cartier error}.

\begin{theorem}\label{thm:weight two nonsingular}
There is a unique line bundle $\pure_\Pap$ on $\mathcal{S}_\Pap$ whose restriction to the nonsingular locus (\ref{nonsingular locus})
 is isomorphic to $\bm{\omega}^2$.  We denote by $\pure_\Kra$ its pullback via $\mathcal{S}_\Kra \to \mathcal{S}_\Pap$.
\end{theorem}

\begin{proof}
Let $(A_0,A)$ be the universal object over $\mathcal{S}_\Pap$, and recall the short  exact sequence 
\[
0 \to \Fil^0 H_1^\dR(A) \to H^\dR_1(A) \map{q} \Lie(A) \to 0
\]
of vector bundles on $\mathcal{S}_\Pap$.    As  $H^\dR_1(A)$ is a locally free  $\co_\kk \otimes_\Z \co_{\mathcal{S}_\Pap}$-module of rank $n$,
the quotient  $H^\dR_1(A)/\overline{\epsilon}H^\dR_1(A)$ is a rank $n$  vector bundle.

Define a line bundle  
\[
\mathcal{P}_\Pap =  \underline{\Hom} \Big(   
\bigwedge\nolimits^n H^\dR_1(A) / \overline{\epsilon} H^\dR_1(A) , \bigwedge\nolimits^n \Lie(A) \Big)
\]
on $\mathcal{S}_\Pap$, and denote by $\mathcal{P}_\Kra$ its pullback  via $\mathcal{S}_\Kra \to \mathcal{S}_\Pap$.
Let  
\[
\psi : H^\dR_1(A) \otimes  H^\dR_1(A) \to \co_{\mathcal{S}_\Pap}
\]
be the alternating pairing induced by the principal polarization on $A$.
If $a$ and $b$ are local sections of $H^\dR_1(A)$, define a local section $P_{a\otimes b}$ of $\mathcal{P}_\Pap$ by
\[
P_{a\otimes b} ( e_1 \wedge \cdots \wedge e_n ) = 
\sum_{k=1}^n (-1)^{k+1} \cdot \psi( \overline{\epsilon} a , e_k ) \cdot q( \overline{\epsilon} b) \wedge
 \underbrace{ q(e_1) \wedge \cdots \wedge q(e_n) }_{ \mathrm{omit}\, q(e_k) } .
\]

\begin{remark}
To see that $P_{a\otimes b}$ is well-defined, one must check that modifying any $e_k$ by a section of $\overline{\epsilon} H^\dR_1(A)$ leaves  the right hand side  unchanged.
This is an easy  consequence of the vanishing of
\[
\bigwedge\nolimits^2 \overline{\epsilon} : \bigwedge\nolimits^2 \Lie(A) \to \bigwedge\nolimits^2 \Lie(A)
\]
imposed in the moduli problem defining $\mathcal{S}_\Pap$.  
\end{remark}

\begin{lemma}\label{lem:bundle swindle}
The morphism 
\begin{equation}\label{det map}
P : H^\dR_1(A) \otimes H^\dR_1(A) \to \mathcal{P}_\Pap
\end{equation}
defined by $a\otimes b \mapsto P_{a\otimes b}$  factors through a morphism
\[
P : \Lie(A) \otimes \Lie(A) \to \mathcal{P}_\Pap.
\]
 After pullback to $\mathcal{S}_\Kra$  there is a further factorization 
\begin{equation}\label{factored swindle}
P :   \Lie(A)/\mathcal{F}_A  \otimes \Lie(A)/\mathcal{F}_A  \to \mathcal{P}_\Kra,
\end{equation}
and this map  becomes  an isomorphism after restriction to $\mathcal{S}_\Kra \smallsetminus \mathrm{Exc}$ .
\end{lemma}

\begin{proof}
Let $a$ and $b$ be local sections of $H^\dR_1(A)$.

Assume first that  $a$ is contained in $\Fil^0 H_1^\dR(A)$. As $\Fil^0 H_1^\dR(A)$ is isotropic under the pairing $\psi$, 
  $P_{a\otimes b}$ factors through a map
\[
\bigwedge\nolimits^n \Lie(A) / \overline{\epsilon} \Lie(A) \to \bigwedge\nolimits^n \Lie(A).
\]
In the generic fiber of $\mathcal{S}_\Pap$, the sheaf $\Lie(A) / \overline{\epsilon} \Lie(A)$ is a vector bundle of rank $n-1$.  
This proves that $P_{ a\otimes b}$ is trivial over the generic fiber. As $P_{a\otimes b}$ is a morphism of vector bundles on a  flat $\co_\kk$-stack, we deduce that $P_{a\otimes b}=0$ identically on $\mathcal{S}_\Pap$.

If instead $b$ is contained in  $\Fil^0 H_1^\dR(A)$ then  $q(\overline{\epsilon} b)=0$, and  again   $P_{a\otimes b}=0$.  
These calculations prove that $P$ factors through $\Lie(A) \otimes \Lie(A)$.

Now pullback to $\mathcal{S}_\Kra$.  We need to check that $P_{a\otimes b}$ vanishes if either of $a$ or $b$ lies in $\mathcal{F}_A$.  Once again it suffices to check this in the generic fiber, where it is clear from 
\begin{equation}\label{nonsingular subsheaf}
\mathcal{F}_A= \mathrm{ker}(\overline{\epsilon} : \Lie(A) \to \Lie(A)).
\end{equation}

Over $\mathcal{S}_\Kra$ we now have a factorization  (\ref{factored swindle}), and it only remains to check that its restriction  to  (\ref{nonsingular locus}) is an isomorphism.  For this, it suffices to verify that  (\ref{factored swindle}) is surjective on the fiber at any geometric point 
\[
s=  \Spec(\F) \to \mathcal{S}_\Kra \smallsetminus \mathrm{Exc}.
\]

First suppose that $\mathrm{char}(\F)$ is prime to $D$.  In this case $\epsilon, \overline{\epsilon} \in \co_\kk\otimes_\Z \F$ are (up to scaling by $\F^\times$) orthogonal idempotents, $\mathcal{F}_{A_s} = \epsilon\Lie(A_s)$, and we may choose an $\co_\kk \otimes_\Z \F$-basis $e_1,\ldots, e_n\in H^\dR_1(A_s)$ in such a way that 
\[
\epsilon e_1 ,\overline{\epsilon} e_2,\ldots, \overline{\epsilon} e_n \in \Fil^0 H_1^\dR (A_s)
\] 
and
\[
q(\overline{\epsilon} e_1) , q(\epsilon e_2) , \ldots, q(\epsilon e_n) \in \Lie (A_s)
\]
are $\F$-bases.  This implies that 
\[
P_{e_1 \otimes e_1}(e_1\wedge \cdots \wedge e_n) = 
\psi(\overline{\epsilon} e_1, \epsilon e_1) \cdot q(\overline{\epsilon} e_1)\wedge q(\epsilon e_2) \wedge \cdots \wedge q(\epsilon e_n) \neq 0,
\]
and so 
\[
P_{e_1\otimes e_1} \in \Hom\big( \bigwedge\nolimits^n H^\dR_1(A_s) /\overline{\epsilon} H^\dR_1(A_s) , \bigwedge\nolimits^n \Lie(A_s) \big)
\]
is a generator.  Thus $P$ is surjective in the fiber at $z$.

Now suppose that $\mathrm{char}(\F)$ divides $D$.   In this case  there is an isomorphism
\[
\F[x]/(x^2) \map{x\mapsto \epsilon = \overline{\epsilon}} \co_\kk \otimes_\Z \F.
\]
By Theorem \ref{thm:integral comparison} the relation (\ref{nonsingular subsheaf}) holds in an \'etale neighborhood of $s$, and it follows that 
we may choose an $\co_\kk\otimes_\Z \F$-basis  $e_1,\ldots, e_n\in H^\dR_1(A_s)$ in such a way that 
\[
 e_2 , \epsilon e_2 , \epsilon e_3,\ldots, \epsilon e_n \in \Fil^0 H_1^\dR (A_s)
\] 
and
\[
q( e_1) , q(\epsilon e_1) , q(e_3)  \ldots, q( e_n) \in \Lie (A_s)
\]
are $\F$-bases.  This implies that 
\[
P_{e_1 \otimes e_1}(e_1\wedge \cdots \wedge e_n) = 
\psi( \epsilon e_1,  e_2) \cdot q( \epsilon e_1)\wedge q(e_1) \wedge q( e_3) \wedge \cdots \wedge q( e_n) \neq 0,
\]
and so, as above,  $P$ is surjective in the fiber at $z$.
\end{proof}

We now complete the proof of Theorem \ref{thm:weight two nonsingular}.
To prove the existence part of the claim,  we define  $\pure_\Pap$ by
\[
\pure_\Pap^{-1} =  \Lie(A_0)^{\otimes 2} \otimes \mathcal{P}_\Pap,
\]
and let $\pure_\Kra$ be its pullback via $\mathcal{S}_\Kra\to \mathcal{S}_\Pap$.
 Tensoring both sides of (\ref{factored swindle}) with $\Lie(A_0)^{\otimes 2}$ defines a morphism
\[
\bm{\omega}^{-2}  \to  \pure_\Kra^{-1},
\]
whose restriction to $\mathcal{S}_\Kra \smallsetminus \mathrm{Exc}$ is an isomorphism.
In particular  $\bm{\omega}^2$  and  $\pure_\Pap$  are isomorphic over (\ref{nonsingular locus}).

The uniqueness of $\pure_\Pap$ is clear: as $\mathrm{Sing} \subset \mathcal{S}_\Pap$ is a codimension $\ge 2$ closed substack of a normal stack, any line bundle on the complement of $\mathrm{Sing}$ admits at most one extension to all of $\mathcal{S}_\Pap$. 

 \end{proof}





\subsection{Special divisors}
\label{ss:special divisors}


Suppose $S$ is a connected $\co_\kk$-scheme, and
\[
(A_0,A) \in \mathcal{S}_\Pap (S).
\]
Imitating the construction of (\ref{hom hermitian}), there is a positive definite hermitian form  on $\Hom_{\co_\kk}(A_0,A)$  defined by
\begin{equation}\label{moduli hom hermitian}
\langle x_1,x_2\rangle = x_2^\vee \circ x_1 \in \End_{\co_\kk}(A_0) \iso \co_\kk,
\end{equation}
where 
\[
\Hom_{\co_\kk}(A_0,A) \map{x\mapsto x^\vee} \Hom_{\co_\kk}(A,A_0)
\]
is the $\co_\kk$-conjugate-linear isomorphism induced by the principal polarizations on $A_0$ and $A$.

For any positive $m\in \Z$, define the $\co_\kk$-stack $\mathcal{Z}_\Pap(m)$ as the moduli stack assigning to a connected $\co_\kk$-scheme $S$ the groupoid of triples $(A_0,A,x)$, where
\begin{itemize}
\item $(A_0,A) \in \mathcal{S}_\Pap(S)$,
\item $x\in \Hom_{\co_\kk} (A_0 , A )$ satisfies $\langle x,x\rangle = m$.
\end{itemize}
Define a stack  $\mathcal{Z}_\Kra(m)$ in exactly the same way, but replacing $\mathcal{S}_\Pap$ by $\mathcal{S}_\Kra$. Thus we obtain a cartesian diagram
\[
\xymatrix{
{  \mathcal{Z}_\Kra(m)  }  \ar[r] \ar[d]  & { \mathcal{S}_\Kra} \ar[d]   \\
{  \mathcal{Z}_\Pap(m)  }  \ar[r]  & { \mathcal{S}_\Pap ,}
}
\]
in which the horizontal arrows are relatively representable, finite, and unramified.

Each $\mathcal{Z}_\Kra(m)$ is, \'etale locally on $\mathcal{S}_\Kra$, a disjoint union of Cartier divisors.  More precisely, around any geometric point of $\mathcal{S}_\Kra$ one can find an \'etale neighborhood $U$ with the property that the morphism $\mathcal{Z}_\Kra(m)_U \to U$ restricts to a closed immersion on every connected component $Z \subset \mathcal{Z}_\Kra(m)_U$, and $Z \subset U$ is  defined locally by one equation; this is \cite[Proposition 3.2.3]{Ho2}, but a cleaner argument (working on the Rapoport-Zink space corresponding to $\mathcal{S}_\Kra$) can be found in \cite[Proposition 4.3]{Ho3}.
Summing over all connected components $Z$  allows us to view  $\mathcal{Z}_\Kra(m)_U$ as a Cartier divisor on $U$, and glueing as $U$ varies over an \'etale cover defines a Cartier divisor on $\mathcal{S}_\Kra$, which we again denote by $\mathcal{Z}_\Kra(m)$.

\begin{remark}\label{rem:not cartier}
It follows from (\ref{nonsingular locus}) and the paragraph above  that   $\mathcal{Z}_\Pap(m)$ is locally defined by one equation away from the singular locus, and so defines a Cartier divisor on $\mathcal{S}_\Pap \smallsetminus \mathrm{Sing}$.
This  Cartier divisor does not extend to all of  $\mathcal{S}_\Pap$.  
\end{remark}

\begin{remark}\label{rem:divisor uniformization}
We can make the specal divisors more explicit in the complex fiber, as in \cite[Proposition 3.5]{KR2} or \cite[\S 3.8]{Ho1}.
Recall  from \S \ref{ss:shimura data}  that the $\Q$-vector space  $V=\Hom_\kk (W_0, W)$ carries a quadratic form. Using the description  
\[
 \mathcal{D}   \iso \big\{ z \in   \epsilon V(\C) :   [z,\overline{z}] <0    \big\} / \C^\times  \subset \mathbb{P}( \epsilon V(\C) )
 \]
of Remark \ref{rem:to so}, every $x\in V$ with $Q(x) >0$ determines an analytic divisor
  \[
 \mathcal{D}(x) = \{ z\in \mathcal{D} :  [ z,x] =0 \}.
 \]
A choice of $g\in G(\A_f)$ determines a connected component
\[
( G(\Q) \cap gKg^{-1} ) \backslash \mathcal{D} \map{ z\mapsto (z,g) }     G(\Q) \backslash \mathcal{D} \times G(\A_f) / K \iso   \mathcal{S}_\Kra(\C) ,
\]
and if we set \[ L = \Hom_{\co_\kk}(g\mathfrak{a}_0 , g\mathfrak{a} ) \subset V\]  
the restriction of $\mathcal{Z}_\Kra(m)(\C) \to \mathcal{S}_\Kra(\C)$ to this component is 
\[
( G(\Q) \cap gKg^{-1} ) \backslash  \bigsqcup_{ \substack{   x \in L   \\  Q(x) =m } } \mathcal{D} (x) 
\to
( G(\Q) \cap gKg^{-1} ) \backslash \mathcal{D} .
\]
\end{remark}

The following theorem, whose proof will occupy the remainder of \S \ref{ss:special divisors}, shows that $\mathcal{Z}_\Kra(m)$ is closely related to another Cartier divisor on $\mathcal{S}_\Kra$ that descends to $\mathcal{S}_\Pap$.    This result will be strengthened in Theorem \ref{thm:cartier error}.

\begin{theorem}\label{thm:pure divisor}
For every $m>0$ there is a unique Cartier divisor $\mathcal{Y}_\Pap(m)$ on $\mathcal{S}_\Pap$ whose restriction to
$\mathcal{S}_\Pap \smallsetminus \mathrm{Sing}$ agrees with  $2 \mathcal{Z}_\Pap(m)$.   
In particular its pullback $\mathcal{Y}_\Kra(m)$  via $\mathcal{S}_\Kra\to \mathcal{S}_\Pap$ agrees with $2 \mathcal{Z}_\Kra(m)$ over $\mathcal{S}_\Kra \smallsetminus \mathrm{Exc}$.
\end{theorem}

\begin{proof}
The map  $\mathcal{Z}_\Pap(m) \to \mathcal{S}_\Pap$ is finite, unramified, and relatively representable.  It follows that
every geometric point of  $\mathcal{S}_\Pap$ admits an \'etale neighborhood  $U\to \mathcal{S}_\Pap$ such that  $U$ is a scheme,  and the  morphism
\[
\mathcal{Z}_\Pap(m)_U \to U
\]
restricts to a closed immersion  on every connected component \[Z\subset \mathcal{Z}_\Pap(m)_U.\]   
We will construct a Cartier divisor on any such $U$, and then glue them together
as $U$ varies over an \'etale cover to obtain the divisor  $\mathcal{Y}_\Pap(m)$.

Fix $Z$ as above, let $\mathcal{I} \subset \co_U$ be its ideal sheaf,  and  let $Z'$ be the closed subscheme of $U$ defined by the ideal sheaf $\mathcal{I}^2$.  Thus we have closed immersions
\[
Z\subset Z' \subset U,
\] 
the first of which is a square-zero thickening.

By the very definition of $\mathcal{Z}_\Pap(m)$, along $Z$ there is a universal $\co_\kk$-linear map $x: A_{0 Z} \to A_Z$.  This map does not extend to 
a map $A_{0 Z'} \to A_{Z'}$, however, by deformation theory \cite[Chapter 2.1.6]{Lan} the induced  $\co_\kk$-linear  morphism of vector bundles
\[
x : H^\dR_1(A_{0Z} ) \to H^\dR_1(A_Z)
\]
admits a canonical extension  to 
\begin{equation}\label{crys realization}
x' : H^\dR_1( A_{0Z'}) \to H^\dR_1(A_{Z'}).
\end{equation}

Recalling the morphism (\ref{det map}),  define  $Y\subset Z'$ as the largest closed subscheme over which the composition
\begin{equation}\label{new obstruction}
H^\dR_1(A_{0Z'}) \otimes H^\dR_1(A_{0Z'}) \map{ x' \otimes x' }
H^\dR_1(A_{Z'} ) \otimes H^\dR_1(A_{Z'}) \map{P}  \mathcal{P}_\Pap|_{Z'}
\end{equation}
vanishes.

\begin{lemma}\label{lem:purified nonsingular}
If $U\to \mathcal{S}_\Pap$ factors through $\mathcal{S}_\Pap \smallsetminus \mathrm{Sing}$, then $Y=Z'$.
\end{lemma}

\begin{proof}
Lemma \ref{lem:bundle swindle} provides us with a commutative diagram
 \[
\xymatrix{
{  H^\dR_1(A_{0 Z'})^{ \otimes  2} }  \ar[rr]^{  x' \otimes x' } \ar[drrrr]_{ (\ref{new obstruction}) }   &  & {  H^\dR_1(A_{ Z'})^{ \otimes  2} } \ar[rr]^{q\otimes q}  & &
   {   \big( \Lie(A_{Z'}) / \mathcal{F}_{A_{Z'} } \big)^{\otimes 2}   }  \ar[d]^{\iso}  \\
& &   & &   {    \mathcal{P}_\Pap|_{Z'}  , }  
}
\]
where   \[\mathcal{F}_{ A_{Z'} } = \mathrm{ker} ( \overline{\epsilon} : \Lie(A_{Z'}) \to \Lie(A_{Z'}))\] as in Theorem \ref{thm:integral comparison}.

By deformation theory, $Z\subset Z'$ is characterized as the largest closed subscheme over which (\ref{crys realization}) respects the Hodge filtrations.
Using Remark \ref{rem:simple hodge}, it is easily seen that $Z\subset Z'$ can also be characterized as the largest closed subscheme over which  
\[
H_1(A_{0 Z'}) \map{ q\circ x' }  \Lie(A_{Z'})/ \mathcal{F}_{A_{Z'}} 
\]
vanishes identically.   As $Z\subset Z'$ is a square zero thickening, it  follows first  that the horizontal composition in the above diagram vanishes identically, and then that (\ref{new obstruction}) vanishes identically.  In other words $Y=Z'$.
\end{proof}

\begin{lemma}\label{lem:pure cartier}
The closed subscheme $Y\subset U$ is defined locally by one equation.  
\end{lemma}

\begin{proof}
Fix a closed point $y\in Y$ of characteristic $p$, let $\co_{U,y}$ be the local ring of $U$ at $y$, and let $\mathfrak{m} \subset \co_{U,y}$ be the maximal ideal.  For a fixed $k>0$,  let 
\[
\bm{U} =\Spec( \co_{U,y} / \mathfrak{m}^k ) \subset U
\]
be the $k^\mathrm{th}$-order infinitesimal neighborhood of $y$ in $U$.  The point of passing to the infinitesimal neighborhood is that $p$ is nilpotent in $\co_{\bm{U}}$, and so we may apply Grothendieck-Messing deformation theory.

By construction we have closed immersions
\[
\xymatrix{
& { Y }  \ar[d] \\
{  Z  }   \ar[r]  &   {  Z' } \ar[r]  & { U }.  
}
\]
Applying the fiber product $\times_U\bm{U}$ throughout the diagram, we obtain closed immersions
\[
\xymatrix{
& { \bm{Y} }  \ar[d] \\
{  \bm{Z}  }   \ar[r]  &   {  \bm{Z}' } \ar[r]  & { \bm{U} }
}
\]
of Artinian schemes.  As $k$ is arbitrary, it suffices to prove that $\bm{Y}\subset \bm{U}$ is defined by one equation.

First suppose that $p\nmid D$.  In this case $\bm{U} \to U \to \mathcal{S}_\Pap$  factors through the nonsingular locus (\ref{nonsingular locus}).  It follows from Remark \ref{rem:not cartier} that  $\bm{Z}\subset \bm{U}$ is defined by one equation, and $\bm{Z}'$ is defined by the square of that equation.  By Lemma \ref{lem:purified nonsingular}, $\bm{Y}\subset \bm{U}$ is also defined by one equation.

For the remainder of the proof we assume that  $p\mid D$.  In particular $p>2$. 
Consider the closed subscheme $Z''\hookrightarrow U$ with ideal sheaf $\mathcal{I}^3$, so that we have closed immersions
$
Z\subset Z' \subset Z'' \subset U.
$
Taking the fiber product with $\bm{U}$, the above diagram extends to 
\[
\xymatrix{
& { \bm{Y} }  \ar[d] \\
{  \bm{Z}  }   \ar[r]  &   {  \bm{Z}' } \ar[r]  &  {  \bm{Z}'' } \ar[r]  & { \bm{U} }.
}
\]

As $p>2$, the cube zero thickening $\bm{Z} \subset \bm{Z}''$ admits  divided powers extending the trivial divided powers on $\bm{Z} \subset \bm{Z}'$.  
Therefore, by Grothendieck-Messing theory,   the restriction of (\ref{crys realization})  to 
\[
x' : H^\dR_1( A_{0\bm{Z}'}) \to H^\dR_1(A_{\bm{Z}'}).
\]
 admits a canonical extension to
\[
x'' : H^\dR_1( A_{0\bm{Z}''}) \to H^\dR_1(A_{\bm{Z}''}).
\]
Define $\bm{Y}' \subset  \bm{Z}''$ as the largest closed subscheme over which 
\begin{equation}\label{over obstructed}
H^\dR_1(A_{0\bm{Z}''}) \otimes H^\dR_1(A_{0\bm{Z}''}) \map{ x'' \otimes x'' }
H^\dR_1(A_{\bm{Z}''} ) \otimes H^\dR_1(A_{\bm{Z}''}) \map{P}  \mathcal{P}_\Pap|_{\bm{Z}''}
\end{equation}
vanishes identically, so that there are closed immersions
\[
\xymatrix{
& { \bm{Y} }  \ar[d] \ar[r]& { \bm{Y}' }  \ar[d] \\
{  \bm{Z}  }   \ar[r]  &   {  \bm{Z}' } \ar[r]  &  {  \bm{Z}'' } \ar[r]  & { \bm{U} }.
}
\]

We pause the proof of Lemma \ref{lem:pure cartier} for a sub-lemma.

\begin{lemma}
We have $\bm{Y}=\bm{Y}'$.  
\end{lemma}

\begin{proof}
As in the proof of Lemma \ref{lem:purified nonsingular}, we may characterize $\bm{Z} \subset \bm{Z}''$ as the largest closed subscheme along which $x''$ respects the Hodge filtrations.  Equivalently, by  Remark \ref{rem:simple hodge}, $\bm{Z} \subset \bm{Z}''$ is the largest closed subscheme over which the composition 
\[
H^\dR_1(A_{0 \bm{Z}''}) \map{ x''  \circ \overline{\epsilon} } H^\dR_1(A_{ \bm{Z}''})  \map{q}   \Lie(A_{\bm{Z}''})
\]
vanishes identically.   This implies that $\bm{Z}'\subset \bm{Z}''$ is the  largest closed subscheme over  which 
\begin{equation}\label{over obstructed 2}
 H^\dR_1(A_{0 \bm{Z}''}) ^{\otimes 2}  \map{ ( x'' \circ \overline{\epsilon})^{\otimes 2 } } 
 H^\dR_1(A_{ \bm{Z}''}) ^{\otimes 2}  \map{q^{\otimes 2}}  
 \Lie(A_{\bm{Z}''})^{\otimes 2}
\end{equation}
vanishes identically.

It follows directly from the definitions that  $\bm{Y}= \bm{Y}' \cap  \bm{Z}'$,  and hence it suffices to show that $\bm{Y}' \subset \bm{Z}'$.  
In other words,  it suffices to show that the vanishing of (\ref{over obstructed}) implies the vanishing of (\ref{over obstructed 2}).

 For local sections $a$ and $b$ of $H_1(A_{\bm{Z}''})$, define 
\[
Q_{a\otimes b} :  \Fil^0 H_1^\dR(A_{\bm{Z}''})  \otimes \bigwedge\nolimits^{n-1} \Lie(A_{\bm{Z}''}) \to  \bigwedge\nolimits ^n \Lie(A_{\bm{Z}''})
\]
by
\[
Q_{a\otimes b} ( e_1\otimes q(e_2)\wedge \cdots \wedge q(e_n) ) = \psi(a,e_1)\cdot q(b) \wedge q(e_2) \wedge \cdots \wedge q(e_n).
\]
It is clear that $Q_{a\otimes b}$  depends only on the images of $a$ and $b$ in $\Lie(A_{\bm{Z}''})$, and that this construction defines an isomorphism
\begin{equation}\label{Q iso}
 \Lie(A_{\bm{Z}''})^{\otimes 2} \map{Q} \underline{\Hom} \Big(
 \Fil^0 H_1^\dR(A_{\bm{Z}''})  \otimes \bigwedge\nolimits^{n-1} \Lie(A_{\bm{Z}''}) ,   \bigwedge\nolimits ^n \Lie(A_{\bm{Z}''})
\Big).
\end{equation}
It is related to the  map
\[
 \Lie(A_{\bm{Z}''})^{\otimes 2} \map{P} \underline{\Hom}\Big(
\bigwedge\nolimits^n H^\dR_1(A_{\bm{Z}''}) / \overline{\epsilon} H^\dR_1(A_{\bm{Z}''}) , \bigwedge\nolimits^n \Lie(A_{\bm{Z}''})
\Big)
\]
of Lemma \ref{lem:bundle swindle} by
\[
P_{a\otimes b}(e_1\wedge \cdots \wedge e_n) = Q_{\overline{\epsilon}a \otimes \overline{\epsilon} b } ( e_1 \otimes q(e_2)\wedge\cdots \wedge q(e_n) )
\]
for any local section $e_1\otimes e_2\otimes \cdots \otimes e_n$ of 
\[
\Fil^0 H_1^\dR(A_{\bm{Z}''}) \otimes H^\dR_1(A_{\bm{Z}''}) \otimes \cdots \otimes H^\dR_1(A_{\bm{Z}''}).
\]

Putting everything together, if (\ref{over obstructed}) vanishes,  then 
$
P_{ x''(a_0) \otimes x''(b_0) } =0
$
for all local sections $a_0$ and $b_0$ of $H^\dR_1(A_{0\bm{Z}''})$.  Therefore 
\[
Q_{ x''( \overline{\epsilon} a_0) \otimes x''( \overline{\epsilon} b_0) } =0
\]
for all local sections $a_0$ and  $b_0$, which implies, as (\ref{Q iso}) is an isomorphism, that (\ref{over obstructed 2}) vanishes.
This proves that  $\bm{Y}' \subset \bm{Z}'$, and hence $\bm{Y}=\bm{Y}'$.
\end{proof}

Returning to the proof of Lemma \ref{lem:pure cartier}, the map (\ref{over obstructed}), whose vanishing defines $\bm{Y}' \subset \bm{Z}''$, factors through a morphism of line bundles
\[
H^\dR_1(A_{0\bm{Z}''}) / \epsilon H^\dR_1(A_{0\bm{Z}''}) \otimes H^\dR_1(A_{0\bm{Z}''}) / \epsilon H^\dR_1(A_{0\bm{Z}''}) \to \mathcal{P}_\Pap|_{\bm{Z}''},
\]
 and hence $\bm{Y}=\bm{Y}'$ is defined inside of $\bm{Z}''$ locally by one equation.
In other words,  if we denote by $\bm{\mathcal{I}} \subset \co_{\bm{U}}$ and $\bm{\mathcal{J}} \subset \co_{\bm{U}}$
the ideal sheaves of $\bm{Z}\subset \bm{U}$  and $\bm{Y} \subset \bm{U}$, respectively,  
then  $ \bm{\mathcal{I}}^3$ is the ideal sheaf of $\bm{Z}''\subset \bm{U}$, and 
\[
\bm{\mathcal{J}}=(f)+ \bm{\mathcal{I}}^3
\]
 for some $f\in \co_{\bm{U}}$.   But  $\bm{Y}\subset  \bm{Z}'$ implies that $\bm{\mathcal{I}}^2 \subset \bm{\mathcal{J}}$, 
 and hence $\bm{\mathcal{I}}^3 \subset \bm{\mathcal{I}\mathcal{J}}$.   It follows that the image of $f$ under the composition
 \[
 \bm{\mathcal{J}} /  \bm{\mathcal{I}}^3 \to \bm{\mathcal{J}} / \bm{\mathcal{I}\mathcal{J}} \to \bm{\mathcal{J}} / \mathfrak{m} \bm{ \mathcal{J}}
 \]
is an $\co_{\bm{U}}$-module generator, and   $\bm{\mathcal{J}}$ is principal by  Nakayama's lemma.
\end{proof}

At last we can complete the proof of Theorem \ref{thm:pure divisor}.
For each connected component $Z\subset \mathcal{Z}_\Pap(m)_U$ we have now defined a closed subscheme $Y\subset Z'$.  
By Lemma \ref{lem:pure cartier} it is an effective  Cartier divisor,   and summing these Cartier divisors as $Z$ varies over all connected components yields an effective Cartier divisor $\mathcal{Y}_\Pap(m)_U$ on $U$.
Letting $U$ vary over an \'etale cover and applying \'etale descent  defines an effective Cartier divisor $\mathcal{Y}_\Pap(m)$ on  $\mathcal{S}_\Pap$.

The Cartier divisor $\mathcal{Y}_\Pap(m)$ just defined agrees with $2\mathcal{Z}_\Pap(m)$ on  $\mathcal{S}_\Pap \smallsetminus \mathrm{Sing}$.   This is clear from  Lemma \ref{lem:purified nonsingular} and the definition of $\mathcal{Y}_\Pap(m)$. 
The uniqueness claim follows from the normality of $\mathcal{S}_\Pap$, exactly as  in the proof of Theorem \ref{thm:weight two nonsingular}.
\end{proof}


\subsection{Pullbacks of Cartier divisors}
\label{ss:pure pullback}


After Theorem \ref{thm:weight two nonsingular} we have  two line bundles   $\pure_\Kra$ and $\bm{\omega}^2$ on $\mathcal{S}_\Kra$, 
which  agree over the complement of  the exceptional locus $\mathrm{Exc}$.  We wish to pin down more precisely the relation between them.

Similarly, after Theorem \ref{thm:pure divisor} we have Cartier divisors $\mathcal{Y}_\Kra(m)$ and $2  \mathcal{Z}_\Kra(m)$.  These agree on the complement of $\mathrm{Exc}$, and again we wish to pin down more precisely the relation between them.

Denote by $\pi_0(\mathrm{Sing})$ the set of connected components of the singular locus $\mathrm{Sing} \subset \mathcal{S}_\Pap$.
  For each  $s\in \pi_0(\mathrm{Sing})$  there is a corresponding irreducible effective Cartier divisor
\[
\mathrm{Exc}_s = \mathrm{Exc} \times_{ \mathcal{S}_\Pap } s \hookrightarrow \mathcal{S}_\Kra
\]
supported in a single characteristic dividing $D$.  These satisfy  
\[
\mathrm{Exc} = \bigsqcup_{ s\in \pi_0(\mathrm{Sing}) }\mathrm{Exc}_s.
\]

\begin{remark}
As  $\mathrm{Sing}$ is a reduced $0$-dimensional stack of finite type over  $\co_\kk / \mathfrak{d}$, 
each $s\in \pi_0(\mathrm{Sing})$ can be realized as the stack quotient   
\[
s\iso G_s\backslash \Spec(\F_s)
\]
 for a finite field $\F_s$ of characteristic  $p\mid D$ acted on by  a finite group $G_s$.  
\end{remark}

Fix a geometric point $\Spec(\F) \to s$, and set $p=\mathrm{char}(\F)$.    
By mild abuse of notation this geometric point will again be denoted simply by $s$.   
It determines a pair 
\begin{equation}\label{base deformation}
(A_{0,s} , A_s) \in \mathcal{S}_\Pap(\F),
\end{equation}
 and hence a positive definite hermitian $\co_\kk$-module
\[
L_s = \Hom_{\co_\kk}(A_{0,s} , A_s)
\]
as in (\ref{moduli hom hermitian}).  This hermitian lattice depends only on $s\in \pi_0(\mathrm{Sing})$, not on the choice of geometric point above it.

\begin{proposition}\label{prop:supersingular}
For each $s\in \pi_0(\mathrm{Sing})$ the abelian varieties $A_{0s}$ and $A_s$ are supersingular, and  there is an $\co_\kk$-linear isomorphism of $p$-divisible groups
\begin{equation}\label{p-div splitting}
A_s[p^\infty] \iso   \underbrace{ A_{0s}[p^\infty]   \times   \cdots \times   A_{0s}[p^\infty] }_{n\, \mathrm{ times}}
\end{equation}
identifying the polarization on the left with the product polarization on the right.  
Moreover, the hermitian $\co_\kk$-module $L_s$ is self-dual of rank $n$.
\end{proposition}

\begin{proof}
Certainly $A_{0s}$ is supersingular, as $p$ is ramified in $\co_\kk \subset \End(A_{0s})$.

Denote by  $\mathfrak{p}\subset \co_\kk$ be the unique prime above $p$.
Let $W=W(\F)$ be the Witt ring of $\F$, and let $\mathrm{Fr}\in \Aut(W)$ be the unique continuous lift of the $p$-power Frobenius on $\F$.
Let $\mathbb{D}(W)$ denote the covariant Dieudonn\'e module of $A_s$, endowed with its operators $F$ and $V$ satisfying $FV=p=VF$.   
The Dieudonn\'e module is  free of rank $n$ over $\co_\kk\otimes_\Z W$, and the short exact sequence
\[
0 \to \Fil^0 H_1^\dR(A_s) \to H_1^\dR(A_s) \to \Lie(A_s) \to 0
\]
of $\F$-modules is identified with
\[
0 \to V\mathbb{D}(W)/p\mathbb{D}(W) \to \mathbb{D}(W)/p\mathbb{D}(W) \to \mathbb{D}(W)/V\mathbb{D}(W) \to 0.
\]

As $D$ is odd, the element $\delta\in \co_\kk$ fixed in \S \ref{ss:notation} satisfies $\ord_\mathfrak{p}(\delta)=1$.
This implies that 
\[
\delta  \cdot  \mathbb{D}(W) = V\mathbb{D}(W).
\]
Indeed, by Theorem \ref{thm:pappas model} the Lie algebra $\Lie(A_s)$ is annihilated by $\delta$, and hence
$
\delta \cdot   \mathbb{D}(W) \subset V\mathbb{D}(W).
$
Equality holds as   
\[
\mathrm{dim}_\F \big( \mathbb{D}(W)/\delta \cdot  \mathbb{D}(W) \big) =n
 = \mathrm{dim}_\F \big( \mathbb{D}(W)/ V\mathbb{D}(W) \big) .
\]

Denote by $N \subset \mathbb{D}(W)$ the set of fixed points of the  $\mathrm{Fr}$-semilinear bijection  
\[
 V^{-1}  \circ \delta : \mathbb{D}(W)\to \mathbb{D}(W).
\]
It is a  free $\co_{\kk,\mathfrak{p}}$-module  of rank $n$ endowed with an isomorphism
\[
\mathbb{D}(W)\iso N\otimes_{\Z_p}W
\]
identifying   $V=\delta \otimes \mathrm{Fr}^{-1}$.  Moreover, the alternating form $\psi$ on $\mathbb{D}(W)$ induced by the polarization on $A_s$ has  the form
\[
\psi (n_1 \otimes w_1, n_2\otimes w_2) = w_1 w_2 \cdot \mathrm{Tr}_{\kk/\Q}  \left( \frac{h( n_1,n_2)}{\delta} \right)
\]
for a perfect hermitian pairing $h : N \times N \to \co_{\kk,\mathfrak{p}}$.  
By diagonalizing this hermitian form, we obtain an orthogonal decomposition of $N$ into rank one hermitian $\co_{\kk,\mathfrak{p}}$-modules,
and tensoring this decomposition with $W$ yields a decomoposition of $\mathbb{D}(W)$ as a direct sum of principally polarized Dieudonn\'e modules, each of height $2$ and slope $1/2$.  This corresponds to a decomposition  (\ref{p-div splitting}) on the level of $p$-divisible groups.  

In particular, $A_s$ is supersingular, and hence is isogenous to $n$ copies of $A_{0s}$.  
Using the Noether-Skolem theorem, this  isogeny may be chosen to be $\co_\kk$-linear.  
It follows first that  $L_s$ has $\co_\kk$-rank $n$, and then that the natural map
\[
L_s \otimes_\Z \Z_q  \iso \Hom_{\co_\kk}( A_{0s}[q^\infty] , A_s[q^\infty]) 
\]
is an isomorphism of hermitian $\co_{\kk,q}$-modules for every rational prime $q$. 
 It is easy to see, using (\ref{p-div splitting}) when $q=p$, that the hermitian module on the right is self-dual, and hence the same is true for $L_s\otimes_\Z\Z_q$.
\end{proof}

The remainder of \S \ref{ss:pure pullback} is devoted to proving the following result.

\begin{theorem}\label{thm:cartier error}
There is an isomorphism
\[
\bm{\omega}^2  \iso  \pure_\Kra \otimes \co(\mathrm{Exc})
\]
of line bundles on $\mathcal{S}_\Kra$, as well as an equality 
\[
2 \mathcal{Z}_\Kra(m) = \mathcal{Y}_\Kra(m) + \sum_{ s\in \pi_0(\mathrm{Sing}) }  \# \{ x\in L_s : \langle x,x\rangle =m \}  \cdot \mathrm{Exc}_s
\]
of Cartier divisors. 
\end{theorem}

\begin{proof}
Recall from the proof of Theorem \ref{thm:weight two nonsingular} the morphism 
\[
\xymatrix{
{  \bm{\omega}^{-2}  } \ar@{=}[d]   &  & &   {  \pure_\Kra^{-1} }  \ar@{=}[d]     \\
 {    \Lie(A_0)^{\otimes 2} \otimes (  \Lie(A) / \mathcal{F}_A) ) ^{\otimes 2}  }  \ar[rrr]^{ (\ref{factored swindle})  }   &    & &   {      \Lie(A_0)^{\otimes 2} \otimes \mathcal{P}_\Kra   , }
}
\]
whose restriction to $\mathcal{S}_\Kra \smallsetminus \mathrm{Exc}$ is an isomorphism.  If we view this morphism as a global section 
\begin{equation}\label{sigma}
\sigma \in H^0( \mathcal{S}_\Kra , \bm{\omega}^{2} \otimes \pure_\Kra^{-1} ),
\end{equation}
then 
\begin{equation}\label{sigma divisor}
\mathrm{div}(\sigma) = \sum_{ s\in \pi_0 (\mathrm{Sing}) }  \ell_s(0)  \cdot \mathrm{Exc}_s
\end{equation}
for some integers $\ell_s(0)\ge 0$, and hence
\begin{equation}\label{bundle up to exc}
\bm{\omega}^2 \otimes \pure_\Kra^{-1} \iso \bigotimes_{  s\in \pi_0 (\mathrm{Sing})  }  \co( \mathrm{Exc}_s)^{\otimes \ell_s(0)} .
\end{equation}
  We must show that each $\ell_s(0)=1$.

Similarly, suppose $m>0$.  It follows from Theorem \ref{thm:pure divisor} that 
\begin{equation}\label{pure up to exc}
2  \mathcal{Z}_\Kra(m) = \mathcal{Y}_\Kra(m) + \sum_{ s\in \pi_0(\mathrm{Sing}) } \ell_s(m)  \cdot \mathrm{Exc}_s
\end{equation}
for some integers $\ell_s(m)$.  Moreover, it is clear from the construction of $\mathcal{Y}_\Kra(m)$ that $2 \mathcal{Z}_\Kra(m) - \mathcal{Y}_\Kra(m)$ is effective, and so $\ell_s(m)\ge 0$.  We must show that 
\[
\ell_s(m) =  \# \{ x\in L_s : \langle x,x\rangle =m \}  .
\]

Fix  $s\in \pi_0(\mathrm{Sing})$,  and let $\Spec(\F) \to s$, $p=\mathrm{char}(\F)$, and   $(A_{0 s} , A_s) \in \mathcal{S}_\Pap(\F)$ be as in (\ref{base deformation}).   Let $W=W(\F)$ be the Witt ring of $\F$, and set $\mathcal{W} = \co_\kk \otimes_\Z W$.  It is a complete discrete valuation ring of absolute ramification degree $2$.  Fix a uniformizer
$
\varpi\in \mathcal{W}.
$
As $p$ is odd, the quotient map 
\[
\mathcal{W} \to \mathcal{W}/\varpi \mathcal{W} = \F
\] 
admits  canonical  divided powers. 

Denote by $\mathbb{D}_0$ and $\mathbb{D}$ the Grothendieck-Messing crystals of $A_{0s}$ and $A_s$, respectively.
  Evaluation of the crystals\footnote{If $p=3$, the divided powers on $\mathcal{W}\to \F$ are not nilpotent, and so we cannot evaluate the usual Grothendieck-Messing crystals on this thickening.  However, Proposition \ref{prop:supersingular} implies that the $p$-divisible groups of $A_{0s}$ and $A_s$ are formal, and Zink's theory of displays \cite{Zink} can be used as a substitute.}  along the divided power thickening $\mathcal{W} \to \F$ yields free $\co_\kk \otimes_\Z \mathcal{W}$-modules $\mathbb{D}_0(\mathcal{W})$ and $\mathbb{D} (\mathcal{W})$  endowed with alternating $\mathcal{W}$-bilinear forms $\psi_0$ and $\psi$, and $\co_\kk$-linear isomorphisms
\[
\mathbb{D}_0(\mathcal{W}) / \varpi  \mathbb{D}_0(\mathcal{W})   \iso  \mathbb{D}_0(\F) \iso H_1^\dR(A_{0s} )
\]
and
\[
\mathbb{D} (\mathcal{W}) / \varpi  \mathbb{D}(\mathcal{W})   \iso \mathbb{D} (\F) \iso H_1^\dR(A_s ).
\]

The $W$-modules $\mathbb{D}_0(W)$ and $\mathbb{D}(W)$ are canonically identified with the covariant Dieudonn\'e modules of $A_{0s}$ and $A_s$, respectively.  The operators $F$ and $V$ on these Dieudonn\'e modules induce operators, denoted the same way, on 
\[
\mathbb{D}_0(\mathcal{W})  \iso \mathbb{D}_0(W) \otimes_W \mathcal{W},\quad 
\mathbb{D} (\mathcal{W})  \iso \mathbb{D}(W) \otimes_W \mathcal{W}.
\]

For any elements $y_1,\ldots, y_k$ in an $\co_\kk\otimes_\Z\mathcal{W}$-module, let  $\langle y_1,\ldots, y_k \rangle$ be the  $\co_\kk \otimes_\Z  \mathcal{W}$-submodule generated by them.     Recall from \S \ref{ss:notation} the elements
\[
\epsilon, \overline{\epsilon} \in\co_\kk\otimes_\Z\mathcal{W}.
\]

\begin{lemma}\label{lem:filtration lift}
There is an  $\co_\kk\otimes_\Z \mathcal{W}$-basis $e_0\in \mathbb{D}_0(\mathcal{W})$ such that
\[
\Fil\, \mathbb{D}_0(\mathcal{W}) \define  \langle \overline{\epsilon} e_0 \rangle \subset \mathbb{D}_0(\mathcal{W})
\]
is a totally isotropic $\mathcal{W}$-module direct summand lifting the Hodge filtration on $\mathbb{D}_0(\F)$, and such that $V e_0= \delta e_0$.

Similarly, there is an $\co_\kk\otimes_\Z \mathcal{W}$-basis $e_1,\ldots, e_n\in \mathbb{D}(\mathcal{W})$ such that 
\[
\Fil\,  \mathbb{D} (\mathcal{W}) \define  \langle \epsilon e_1 ,  \overline{\epsilon} e_2,  \ldots , \overline{\epsilon} e_{n} \rangle   \subset \mathbb{D} (\mathcal{W}) 
\]
is a totally isotropic $\mathcal{W}$-module direct summand lifting the Hodge filtration on $\mathbb{D}(\F)$.  
This basis may be chosen so that   $V e_{k+1} = \delta e_k$, where the indices are understood in $\Z/n\Z$, and also so that 
\[
\psi \big( \langle e_i \rangle  ,  \langle e_j \rangle \big) =
\begin{cases}
\mathcal{W} & \mbox{if }i=j \\
0 & \mbox{otherwise.}
\end{cases}
\]
\end{lemma}

\begin{proof}
As in the proof of Proposition \ref{prop:supersingular}, we may identify 
\[
\mathbb{D}_0(W) \iso N_0 \otimes_{\Z_p} W
\]
for some free $\co_{\kk,\mathfrak{p}}$-module $N_0$ of rank $1$, in such a way that $V=\delta \otimes \mathrm{Fr}^{-1}$, and 
the alternating form on $\mathbb{D}_0(W)$ arises as the $W$-bilinear extension of an alternating form $\psi_0$ on $N_0$.
Any $\co_{\kk,\mathfrak{p}}$-generator $e_0 \in N_0$ determines  a generator of the $\co_{\kk,\mathfrak{p}} \otimes_{\Z_p} \mathcal{W}$-module
\[
\mathbb{D}_0(\mathcal{W}) \iso N_0 \otimes_{\Z_p} \mathcal{W},
\]
which, using Remark \ref{rem:simple hodge} has the desired properties.

Now set $N=N_0 \oplus \cdots \oplus N_0$ ($n$ copies), so that, by Proposition \ref{prop:supersingular}, there is an  isomorphism
\[
\mathbb{D}(W) \iso N \otimes_{\Z_p} W
\]
identifying $V=\delta \otimes \mathrm{Fr}^{-1}$, 
and the alternating bilinear form on $\mathbb{D}(W)$ arises from an alternating form $\psi$ on $N$.
Let $\Z_{p^n}\subset W$ be the ring of integers in the unique unramified degree $n$ extension of $\Q_p$, and fix an action 
\[
 \iota: \Z_{p^n} \to \End_{\co_{\kk,\mathfrak{p}}}(N) 
 \]  
 in such a way that $\psi( \iota(\alpha) x,y) = \psi( x, \iota(\alpha) y)$ for all $\alpha\in \Z_{p^n}$.

 There is an induced decomposition
\[
\mathbb{D}(W) \iso \bigoplus_{  k \in \Z/n\Z } \mathbb{D}(W)_k,
\]
where
\[
\mathbb{D}(W)_k = \{ e\in \mathbb{D}(W) :  \forall \alpha\in \Z_{p^n},\, \iota(\alpha) \cdot e =  \mathrm{Fr}^k (\alpha)  \cdot e \}
\]
is free of rank one over $\co_\kk   \otimes_\Z W$.    Now pick any $\Z_{p^n}$-module generator $e \in N$, view it as an element of $\mathbb{D}(W)$, and let $e_k \in \mathbb{D}(W)_k$ be its projection to the $k^\mathrm{th}$ summand.  This gives an $\co_\kk\otimes_\Z W$-basis  $e_1,\ldots, e_n\in \mathbb{D}(W)$, which determines an $\co_\kk\otimes_\Z \mathcal{W}$-basis of $\mathbb{D}(\mathcal{W})$ with the required properties.
\end{proof}

By the Serre-Tate theorem and Grothendieck-Messing theory, the lifts of the Hodge filtrations specified in Lemma \ref{lem:filtration lift} determine a lift  
\begin{equation}\label{pappas lift}
( \tilde{A}_{0s} ,  \tilde{A}_s)  \in \mathcal{S}_\Pap(\mathcal{W})
\end{equation}
of the pair $(A_{0s}, A_s)$.   These come with canonical identifications
\[
H_1^\dR ( \tilde{A}_{0s})  \iso \mathbb{D}_0(\mathcal{W}) ,\quad H_1^\dR ( \tilde{A}_s) \iso \mathbb{D}(\mathcal{W})
\]  
under which the Hodge filtrations correspond to the filtrations chosen in Lemma \ref{lem:filtration lift}.
In particular, the Lie algebra  of  $\tilde{A}_s$ is 
\[
\Lie(\tilde{A}_s) \iso \mathbb{D} (\mathcal{W}) / \Fil\, \mathbb{D} (\mathcal{W}) =  
\langle e_1, e_2,\ldots,e_n  \rangle/  \langle   \epsilon e_1 ,  \overline{\epsilon} e_2,\ldots, \overline{\epsilon} e_n \rangle.
\]
The $\mathcal{W}$-module direct summand
\[
\mathcal{F}_{\tilde{A}_s} = \langle e_2,\ldots,e_n  \rangle/  \langle  \overline{\epsilon} e_2,\ldots, \overline{\epsilon} e_n \rangle
\]
satisfies  Kr\"amer's condition (\S \ref{ss:unitary integral models}), and so  determines a lift of (\ref{pappas lift}) to
\[
( \tilde{A}_{0s} ,  \tilde{A}_s)  \in \mathcal{S}_\Kra(\mathcal{W}).
\]

To summarize:  starting from a geometric point $\Spec(\F) \to s$,  we have used Lemma \ref{lem:filtration lift} to construct a commutative diagram
\begin{equation}\label{big cm lift}
\xymatrix{
{   \Spec(\F)    }  \ar[d] \ar[r]    & { \mathrm{Exc}_s }  \ar[d] \ar[r] &    {     s    } \ar[d] \\
{  \Spec(\mathcal{W})  }  \ar[r]     &  {   \mathcal{S}_\Kra   }  \ar[r] &     {   \mathcal{S}_\Pap   } .
}
\end{equation}

\begin{lemma}\label{lem:bundle plumb}
The pullback of the map (\ref{factored swindle}) via $\Spec(\mathcal{W}) \to \mathcal{S}_\Kra$ vanishes identically along the closed subscheme
$\Spec(\mathcal{W} / \varpi \mathcal{W})$, but not along $\Spec(\mathcal{W} / \varpi^2 \mathcal{W})$.
\end{lemma}

\begin{proof}
The $\mathcal{W}$-submodule of
\begin{equation}\label{D lie}
\Lie(\tilde{A}_s) \iso \mathbb{D}(\mathcal{W}) / \langle   \epsilon e_1 ,  \overline{\epsilon} e_2,  \ldots , \overline{\epsilon} e_{n} \rangle 
\end{equation}
generated by $e_1$ is $\co_\kk$-stable. The action of $\co_\kk\otimes_\Z\mathcal{W}$ on this $\mathcal{W}$-line is via 
\[
\co_\kk \otimes_\Z \mathcal{W} \map{\alpha \otimes x \mapsto i_\mathcal{W}(\overline{\alpha}) x } \mathcal{W}
\]
(where $i_\mathcal{W} : \co_\kk \to \mathcal{W}$ is the inclusion), and this map sends  $\overline{\epsilon}$ to a uniformizer of $\mathcal{W}$; see \S \ref{ss:notation}.  Thus  the quotient map  $q: \mathbb{D}(\mathcal{W})  \to  \Lie(\tilde{A}_s)$ satisfies
$
q( \overline{\epsilon} e_1 )   =   \varpi q( e_1 ) 
$
up to multiplication by an element of $\mathcal{W}^\times$.     It follows that 
\[
P_{e_1\otimes e_1} ( e_1 \wedge \cdots \wedge e_n )  = 
\varpi \cdot \psi( \overline{\epsilon} e_1 , e_1 ) \cdot q( e_1) \wedge q(e_2) \wedge \cdots \wedge q(e_n)
\]
up to scaling by $\mathcal{W}^\times$.

We claim that    $\psi( \overline{\epsilon} e_1 , e_1 ) \in \mathcal{W}^\times$.  Indeed, 
as $q(e_1)$ generates a  $\mathcal{W}$-module direct summand of (\ref{D lie}), there is some
\[
x\in \Fil\,  \mathbb{D} (\mathcal{W}) = \langle \epsilon e_1 ,  \overline{\epsilon} e_2,  \ldots , \overline{\epsilon} e_{n} \rangle   \subset \mathbb{D} (\mathcal{W}) 
\]
such that $\psi(x,e_1) \in \mathcal{W}^\times$.  We chose our basis in Lemma \ref{lem:filtration lift} in such a way that $\psi( \overline{\epsilon} e_i,e_1)=0$ for $i>1$. 
 It follows  that $\psi( \epsilon e_1,e_1)$ is a unit, and hence the same is true for
$
\psi( \overline{\epsilon} e_1,e_1) = \psi( e_1,  \epsilon e_1 )= - \psi( \epsilon e_1,e_1).
$

We have now proved that 
\[
P_{e_1\otimes e_1} ( e_1 \wedge \cdots \wedge e_n )  = 
\varpi \cdot q( e_1) \wedge q(e_2) \wedge \cdots \wedge q(e_n)
\]
up to scaling by $\mathcal{W}^\times$, from which it follows  that
\[
P_{e_1\otimes e_1} ( e_1 \wedge \cdots \wedge e_n )  \in   \bigwedge\nolimits^n \Lie( \tilde{A}_s ) 
\]
is divisible by $\varpi$, but not by $\varpi^2$.

The quotient 
\[
H_1^\dR(\tilde{A}_s) / \overline{\epsilon} H_1^\dR(\tilde{A}_s) \iso  
\mathbb{D}(\mathcal{W}) / \langle  \overline{ \epsilon }  e_1 ,\ldots, \overline{ \epsilon }  e_n  \rangle
\]
is generated as a $\mathcal{W}$-module by $e_1,\ldots, e_n$.   From the calculation of the previous paragraph, it now follows that
$
P_{e_1\otimes e_1} \in  \mathcal{P}_\Kra  | _{\Spec(\mathcal{W}) }
$
is divisible by $\varpi$ but not by $\varpi^2$.   The quotient
\[
\Lie (\tilde{A}_s) / \mathcal{F}_{\tilde{A}_s} \iso  \mathbb{D}(\mathcal{W}) / \langle \epsilon e_1 , e_2,\ldots, e_n  \rangle 
\]
is generated as a $\mathcal{W}$-module by the image of $e_1$,  and we at last deduce that 
\[
P \in   \underline{\Hom} \big( (\Lie(A)/\mathcal{F}_A)^{ \otimes 2}  , \mathcal{P}_\Kra \big) | _{\Spec(\mathcal{W}) }
\]
is divisible by $\varpi$ but not by  $\varpi^2$.
\end{proof}

Recall the global section $\sigma$ of (\ref{sigma}).  It follows immediately from Lemma \ref{lem:bundle plumb} that its pullback via $\Spec(\mathcal{W}) \to \mathcal{S}_\Kra$ has divisor $\Spec(\mathcal{W}/\varpi\mathcal{W})$, and hence   
\[
\Spec(\mathcal{W}) \times_{ \mathcal{S}_\Kra } \mathrm{div}(\sigma)  =  \Spec(\mathcal{W}/\varpi \mathcal{W}),
\]
  Comparison with (\ref{sigma divisor}) proves  both that $\ell_s(0)=1$, and that 
\begin{equation}\label{special fiber intersection}
\Spec(\mathcal{W}) \times_{ \mathcal{S}_\Kra } \mathrm{Exc}_s =\Spec(\mathcal{W}/\varpi \mathcal{W}).
\end{equation}
Recalling (\ref{bundle up to exc}), this completes the proof that
\[
\bm{\omega}^2  \iso    \pure_\Kra   \otimes  \co(\mathrm{Exc}).
\]

It remains to prove the second claim of Theorem \ref{thm:cartier error}.
Given any 
$
x\in L_s = \Hom_{\co_\kk} ( A_{0s},  A_s ),
$
denote by $k(x)$  the largest integer such that $x$ lifts to a morphism 
\[
 \tilde{A}_{0s} \otimes_{\mathcal{W}}  \mathcal{W} /  ( \varpi^{k(x)} ) \to  
 \tilde{A}_s  \otimes_{ \mathcal{W} }  \mathcal{W} /  (\varpi^{k(x)})  .
\]

\begin{lemma}\label{lem:Z intersect}
As Cartier divisors on $\Spec(\mathcal{W})$, we have
\[
  \mathcal{Z}_\Kra(m)   \times_{\mathcal{S}_\Kra}   \Spec(\mathcal{W})   
= \sum_{  \substack{ x\in L_s \\  \langle x,x\rangle =m }  } \Spec( \mathcal{W} / \varpi^{k(x)} \mathcal{W}).
\]
\end{lemma}

\begin{proof}
Each $x\in L_s$ with  $\langle x,x\rangle =m$ determines a geometric point  
\begin{equation}\label{def hom point}
( A_{0z},  A_z ,x)  \in \mathcal{Z}_\Kra(m)(\F).
\end{equation}
and surjective morphisms 
\[
\xymatrix{
&   {  \co_{\mathcal{S}_\Kra ,x}  } \ar[dl]  \ar[dr]  \\
{  \co_{\mathcal{Z}_\Kra(m) , x  } } &    &   {   \mathcal{W}  , } 
}
\]
where $\co_{\mathcal{Z}_\Kra(m) , x  }$ is the \'etale local ring  at  (\ref{def hom point}), 
$\co_{\mathcal{S}_\Kra ,x}$ is the \'etale local ring  at the point below it, and the arrow on the right is induced by the map $\Spec(\mathcal{W}) \to \mathcal{S}_\Kra$ of  (\ref{big cm lift}).   There is an induced isomorphism of $\mathcal{W}$-schemes
\[
 \co_{\mathcal{Z}_\Kra(m) , x  } \otimes_{ \co_{\mathcal{S}_\Kra ,x}  }  \mathcal{W} \iso \mathcal{W}/ ( \varpi^{k(x)}  ),
\]
and the claim follows by summing over $x$. 
\end{proof}

\begin{lemma}\label{lem:Y intersect}
As Cartier divisors on $\Spec(\mathcal{W})$, we have
\[
\mathcal{Y}_\Kra(m)   \times_{\mathcal{S}_\Kra}   \Spec(\mathcal{W})   
= \sum_{  \substack{ x\in L_s \\  \langle x,x\rangle  = m}  } \Spec( \mathcal{W} / \varpi^{2 k(x)-1} \mathcal{W}).
\]
\end{lemma}

\begin{proof}
Each   $x\in L_s = \Hom_{\co_\kk}( A_{0s} , A_s )$  with $\langle x,x\rangle=m$ 
induces a morphism of crystals $\mathbb{D}_0\to \mathbb{D}$, and hence a map
\[
 \mathbb{D}_0(\mathcal{W}) \map{ x }  \mathbb{D}(\mathcal{W})  
\]
respecting the $F$ and $V$ operators.  By Grothendieck-Messing deformation theory, the integer $k(x)$ is characterized as the largest integer such that the composition
 \[
 \xymatrix{
{  \Fil^0 H^\dR_1( \tilde{A}_{0s})  } \ar[r]^{\subset}  \ar@{=}[d] &  { H^\dR_1( \tilde{A}_{0s})  }  \ar[r]^{ x }\ar@{=}[d]
&   {   H^\dR_1( \tilde{A}_s ) }  \ar[r]^{q} \ar@{=}[d]& {   \Lie( \tilde{A}_s )  }  \ar@{=}[d] \\
{   \overline{\epsilon}  \mathbb{D}_0(\mathcal{W})  } \ar[r]^{\subset}  &  {  \mathbb{D}_0(\mathcal{W})  }  \ar[r]^{ x }  &    {   \mathbb{D}(\mathcal{W})  }  \ar[r] & 
  { \frac{  \mathbb{D} (\mathcal{W}) }{  \langle \epsilon e_1, \overline{\epsilon} e_2, \ldots, \overline{\epsilon} e_n \rangle} . }
}
\]
vanishes modulo $\varpi^{k(x)}$.  In other words the composition 
\[
H^\dR_1( \tilde{A}_{0s})   \map{  x  \circ \overline{\epsilon}   } 
   H^\dR_1( \tilde{A}_{s}) 
   \map{q }       \Lie( \tilde{A}_s) 
\]
vanishes modulo $\varpi^{k(x)}$, but not modulo $\varpi^{k(x)+1}$.

Using the bases of Lemma \ref{lem:filtration lift}, we expand
\[
x(e_0) = a_1 e_1 + \cdots + a_n e_n
\]
with $a_1,\ldots, a_n \in \co_\kk\otimes_\Z \mathcal{W}$.    
The condition that $x$ respects $V$ implies that $a_1=\cdots=a_n$.  
Let us call this common value $a$,  so that 
\[
q(x (\overline{\epsilon} e_0 ) )  = \overline{\epsilon} \cdot q(  a e_1 + \cdots + a e_n  ) = a   \overline{\epsilon} \cdot q(  e_1 ) 
\]
in  $ \Lie( \tilde{A}_s )$.   By the previous paragraph, this element is divisible by $\varpi^{k(x)}$  but not by $\varpi^{k(x)+1}$, and so
\begin{equation}\label{lifting ord 1}
q( a   \overline{\epsilon}   e_1 ) = \varpi^{k(x)} q(e_1)
\end{equation}
up to scaling by $\mathcal{W}^\times$.   

On the other hand, the submodule of $ \Lie( \tilde{A}_s )$ generated by $q(e_1)$ is isomorphic to 
$(\co_\kk \otimes_\Z \mathcal{W}) / \langle \epsilon\rangle \iso \mathcal{W}$, and $\overline{\epsilon}$ acts on this quotient by a uniformizer in $\mathcal{W}$.  Thus 
\begin{equation}\label{lifting ord 2}
\overline{\epsilon} q(e_1) = \varpi q(e_1)
\end{equation}
up to scaling by $\mathcal{W}^\times$.

Combining (\ref{lifting ord 1}) and (\ref{lifting ord 2}) shows that, up to scaling by $\mathcal{W}^\times$, 
\[
a   \overline{\epsilon}  = \varpi^{ k(x) -1 }\overline{\epsilon}
\]
in the quotient $(\co_\kk \otimes_\Z \mathcal{W}) / \langle \epsilon\rangle$.  By the injectivity of the quotient map
$\langle \overline{\epsilon} \rangle  \to (\co_\kk \otimes_\Z \mathcal{W}) / \langle \epsilon\rangle$,
this same equality holds in $\langle \overline{\epsilon} \rangle \subset \co_\kk\otimes_\Z\mathcal{W}$.
Using this and (\ref{lifting ord 1}), we  compute
\begin{eqnarray*}\lefteqn{
P_{ x(e_0)  \otimes  x(e_0) }  (e_1\wedge \cdots \wedge e_n)   } \\
&= & \psi( a \overline{\epsilon}   e_1 , e_1 ) \cdot q( a \overline{\epsilon} e_1  ) \wedge
q(e_2) \wedge \cdots \wedge q(e_n)  \\
&= & \varpi^{2k(x)-1} \cdot  \psi(  \overline{\epsilon}   e_1 , e_1 ) \cdot q( e_1  ) \wedge
q(e_2) \wedge \cdots \wedge q(e_n)  \\
&= & \varpi^{2k(x)-1} \cdot   q( e_1  ) \wedge
q(e_2) \wedge \cdots \wedge q(e_n) 
\end{eqnarray*}
up to scaling by $\mathcal{W}^\times$.   
Here, as in the proof of Lemma \ref{lem:bundle plumb},  we have  used $\psi(  \overline{\epsilon}   e_1 , e_1 )\in \mathcal{W}^\times$.

This calculation shows that the composition 
\[
\xymatrix{
{ H^\dR_1( \tilde{A}_{0s})^{\otimes 2} }     \ar[r]^{  x \otimes x }  & 
{   H^\dR_1( \tilde{A}_s ) ^{\otimes 2}  }  \ar[r]^P  &  {   \mathcal{P}|_{ \Spec(\mathcal{W}) }   }
}
\]
vanishes modulo $\varpi^{ 2 k(x) -1 }$, but not modulo $\varpi^{ 2 k(x) }$, and  
the remainder of the proof is the same as that of Lemma \ref{lem:Z intersect}:
comparing with the definition of $\mathcal{Y}_\Kra(m)$, see especially (\ref{new obstruction}), shows that
\[
\co_{\mathcal{Y}_\Kra(m),x} \otimes_{ \co_{\mathcal{S}_\Kra,x} } \mathcal{W} \iso \mathcal{W} / ( \varpi^{2k(x)-1} ),
\]
and summing over all $x$ proves the claim.
\end{proof}

Combining Lemmas \ref{lem:Z intersect} and \ref{lem:Y intersect} shows that 
\[
\Spec(\mathcal{W}) \times_{\mathcal{S}_\Kra} \big( 2 \mathcal{Z}_\Kra(m) - \mathcal{Y}_\Kra(m) \big)
=  \sum_{  \substack{ x\in L_s \\  \langle x,x\rangle =m }  } \Spec( \mathcal{W} / \varpi  \mathcal{W})
\]
as  Cartier divisors on $\Spec(\mathcal{W})$. We know from  (\ref{special fiber intersection}) that
\[
\Spec(\mathcal{W}) \times_{\mathcal{S}_\Kra} \mathrm{Exc}_t
= \begin{cases}
\Spec(\mathcal{W}/\varpi \mathcal{W}) & \mbox{if } t=s \\
0 & \mbox{if }t\neq s,
\end{cases}
\]
and comparison with (\ref{pure up to exc}) shows that 
\[
\ell_s(m)= \# \{ x\in L_s : \langle x,x\rangle =m \},
\]
completing the proof of Theorem \ref{thm:cartier error}.
\end{proof}


\section{Toroidal compactification}
\label{s:unitary compactification}


In this section we describe canonical toroidal compactifications
\[
\xymatrix{
{  \mathcal{S}_\Kra  }  \ar[r]  \ar[d] &  {  \mathcal{S}^*_\Kra }  \ar[d] \\
{  \mathcal{S}_\Pap  }  \ar[r]  &  {  \mathcal{S}^*_\Pap ,}
}
\]
and the structure of their formal completions along the boundary. Using this description,  we define Fourier-Jacobi expansions of modular forms.

The existence of  toroidal compactifications with reasonable properties is not a new result.  In fact the proof of Theorem \ref{thm:toroidal}, which asserts the existence of good compactifications of $\mathcal{S}_\Pap$ and $\mathcal{S}_\Kra$, simply refers to \cite{Ho2}.  Of course [\emph{loc.~cit.}] is itself a very modest addition to the established literature   \cite{FC, Lan, Lar, Ra}.  Because of this, the reader is perhaps owed a few words of explanation as to why \S \ref{s:unitary compactification} is so long.

It is well-known that the boundary charts used to construct  toroidal compactifications of PEL-type Shimura varieties are themselves moduli spaces of $1$-motives (or, what is nearly the same thing,  degeneration data in the sense of \cite{FC}).  This moduli interpretation is explained in \S \ref{ss:one-motives}.

It is a special feature of our  particular Shimura variety $\mathrm{Sh}(G,\mathcal{D})$ that the boundary charts have a second, very different, moduli interpretation.  This second moduli interpretation is explained in \S \ref{ss:second moduli}.  In some sense, the main result of \S \ref{s:unitary compactification} is not Theorem \ref{thm:toroidal} at all, but rather Proposition \ref{prop:second moduli iso}, which proves the equivalence of the two moduli problems.

The point  is that our goal is to eventually study the integrality and rationality properties of Fourier-Jacobi expansions of Borcherds products on the integral models of $\mathrm{Sh}(G,\mathcal{D})$.  A complex analytic description of these Fourier-Jacobi expansions can be deduced from   \cite{Ku:ABP}, but it is not a priori clear how to deduce  integrality and rationality properties from  these purely complex analytic formulas.

To do so, we will exploit the fact that the formulas of \cite{Ku:ABP} express the Fourier-Jacobi coefficients in terms of the classical Jacobi theta function.  The Jacobi theta function can be viewed as a section of a line bundle on the universal elliptic curve fibered over the modular curve, and when interpreted in this way it has known integrality and rationality properties (this is explained in \S \ref{ss:jacobi}).  

By converting the  moduli interpretation  of the boundary charts from $1$-motives to an  interpretation that makes explicit reference to the universal elliptic curve and the line bundles that live over it, the integrality and rationality properties of the Fourier-Jacobi coefficients can be deduced, ultimately, from those of the classical Jacobi theta function.


\subsection{Cusp label representatives}
\label{ss:cusp notation}


Recall that $W_0$ and $W$ are $\kk$-hermitian spaces of signatures $(1,0)$ and $(n-1,1)$, respectively, with $n\ge 2$.
Tautologically, the subgroup 
\[
G \subset \GU(W_0) \times \GU(W)
\] 
acts on both $W_0$ and $W$.    If $J\subset W$ is an isotropic $\kk$-line, its stabilizer   $P=\Stab_G (    J     )$ in $G$  is a parabolic subgroup.  
This establishes a bijection between  isotropic $\kk$-lines in $W$  and proper parabolic subgroups of $G$.
If  $n>2$ then such isotropic $\kk$-lines always exist.

\begin{definition}\label{def:clr}
A   \emph{cusp label representative} for $(G,\mathcal{D})$ is a pair 
$
\Phi = (P , g)
$
in which $g\in G(\A_f)$  and $P\subset G$ is a parabolic subgroup.    If  $P=\Stab_G(J)$ for an isotropic $\kk$-line $J\subset W$, we call $\Phi$ a  \emph{proper cusp label representative}.  If $P=G$ we call $\Phi$ an \emph{improper cusp label representative}. 
\end{definition}

For each  cusp label representative   $\Phi=(P,g)$ there is a distinguished normal subgroup $Q_\Phi \normal P$.  If $P=G$ we simply take $Q_\Phi = G$.    If $P= \Stab_G (    J     )$ for an isotropic $\kk$-line $J\subset W$ then,  following the recipe of \cite[\S 4.7]{Pink},  we define  $Q_\Phi$ as the fiber product
\begin{equation}\label{Q fiber}
\xymatrix{
{  Q_\Phi  } \ar[r]^{\nu_\Phi}\ar[d] & {   \mathrm{Res}_{\kk/\Q} \mathbb{G}_m  } \ar[d]^{ a \mapsto  (a,  \mathrm{Nm}(a)  , a ,\mathrm{id}  ) } \\
{ P }  \ar[r] & { \GU(W_0) \times  \GL(J ) \times  \GU(J^\perp/J ) \times \GL(W / J^\perp) .}
}
\end{equation}
The  morphism $G\to \GU(W)$ restricts to an injection
$
Q_\Phi \hookrightarrow \GU(W),
$
as  the action of $Q_\Phi$ on  $J^\perp/J$ determines its action on $W_0$.

Let $K\subset G(\A_f)$ be the compact open subgroup  (\ref{K choice}).  Any cusp label representative $\Phi=(P,g)$  determines compact open subgroups
\[
K_\Phi   = gKg^{-1}  \cap Q_\Phi(\A_f),\quad 
\tilde{K}_\Phi   =  gKg^{-1} \cap P(\A_f) ,
\]
and a finite group
\begin{equation}\label{delta group}
\Delta_\Phi   =  \big( P  (\Q)  \cap Q_\Phi  (\A_f) \tilde{K}_\Phi \big) / Q_\Phi  (\Q).
\end{equation}

\begin{definition}
 Two  cusp label representatives 
$\Phi = (P,g)$ and $\Phi'=( P',g')$ are \emph{$K$-equivalent} if   there exist $\gamma \in G(\Q)$, $h  \in Q_\Phi(\A_f)$, and  $k \in K$ such that  
\[
( P'  , g') =   (  \gamma P \gamma^{-1}  , \gamma  h g k).
\]
One may easily verify that this is an equivalence relation.  Obviously, there is a unique $K$-equivalence class of improper cusp label representatives.   
\end{definition}

From now through \S \ref{ss:mixed special divisors}, we fix a proper cusp label representative $\Phi=(P,g)$, with $P\subset G$ the stabilizer of an isotropic $\kk$-line $J\subset W$.  There is an induced weight filtration  $\mathrm{wt}_iW\subset W$ defined by
\[
\xymatrix{
{ 0  } \ar@{=}[d]  \ar@{}[r] | \subset& { J  } \ar@{=}[d]  \ar@{}[r] | \subset &  {  J^\perp  } \ar@{=}[d]  \ar@{}[r] | \subset&  {   W } \ar@{=}[d]  \\
{ \mathrm{wt}_{-3}W}  \ar@{}[r] | \subset & {   \mathrm{wt}_{-2} W  }   \ar@{}[r] | \subset&  {   \mathrm{wt}_{-1} W  }  \ar@{}[r] | \subset&   {  \mathrm{wt}_0 W  },
}
\]
and  an induced weight filtration on $V=\Hom_\kk(W_0,W)$ defined by
\[
\xymatrix{
{  \Hom_{\kk} (W_0 , 0  )   } \ar@{=}[d]  \ar@{}[r] | \subset& {  \Hom_{\kk} (W_0 , J  )  } \ar@{=}[d]  \ar@{}[r] | \subset &  {   \Hom_{\kk} (W_0 , J^\perp )  } \ar@{=}[d]  \ar@{}[r] | \subset&  {   \Hom_\kk(W_0,W)  } \ar@{=}[d]  \\
{ \mathrm{wt}_{-2}V}  \ar@{}[r] | \subset & {   \mathrm{wt}_{-1} V  }   \ar@{}[r] | \subset&  {   \mathrm{wt}_{0} V  }  \ar@{}[r] | \subset&   {  \mathrm{wt}_1 V  },
}
\]
It is easy to see that $\mathrm{wt}_{-1}V$ is an isotropic  $\kk$-line, whose orthogonal   with respect to (\ref{hom hermitian}) is $\mathrm{wt}_0 V$.
Denote by 
$
\mathrm{gr}_i W =\mathrm{wt}_i W  / \mathrm{wt}_{i-1} W 
$
the graded pieces, and similarly for $V$.  

The  $\co_\kk$-lattice $g \mathfrak{a}\subset W$  determines an $\co_\kk$-lattice 
\[
\mathrm{gr}_i (g\mathfrak{a}  ) =    \big(  g\mathfrak{a} \cap \mathrm{wt}_i W  \big)  / \big(  g \mathfrak{a} \cap \mathrm{wt}_{i-1} W \big)     \subset \mathrm{gr}_i W . 
\]
The middle graded piece $\mathrm{gr}_{-1}(g\mathfrak{a})$ is endowed with a positive definite self-dual hermitian form, inherited from the self-dual hermitian form on $g\mathfrak{a}$ appearing in the proof of Proposition \ref{prop:component count}.
 The outer graded pieces
\begin{equation}\label{easy graded}
\mathfrak{m} =  \mathrm{gr}_{-2} ( g \mathfrak{a}  ) ,\quad   \mathfrak{n} =    \mathrm{gr}_{0} ( g \mathfrak{a}  ) 
\end{equation}
are   projective rank one $\co_\kk$-modules\footnote{In fact $\mathfrak{m}\iso \mathfrak{n}$ as $\co_\kk$-modules, but identifying them can only lead to confusion.}, endowed with a perfect $\Z$-bilinear pairing
$
\mathfrak{m} \otimes_\Z \mathfrak{n} \to \Z
$
inherited from the perfect symplectic form on $g \mathfrak{a}$ appearing in the proof of Proposition \ref{prop:shimura moduli}.

\begin{remark}\label{rem:determined by grade}
The isometry class of  $g\mathfrak{a}$ as a hermitian lattice is determined by the isomorphism classes of $ \mathfrak{m} $ and $ \mathfrak{n}  $ as $\co_\kk$-modules and the isometry class of $\mathrm{gr}_{-1}(g\mathfrak{a})$ as a hermitian lattice.  
This follows from the proof of \cite[Proposition 2.6.3]{Ho2}, which shows that one can find a splitting\footnote{This  uses our standing assumption that $\kk$ has odd discriminant.}  
\[
g\mathfrak{a} \iso \mathrm{gr}_{-2} ( g \mathfrak{a}  )  \oplus \mathrm{gr}_{-1} ( g \mathfrak{a}  ) \oplus\mathrm{gr}_{0} ( g \mathfrak{a}  ) ,
\]
in such a way that the outer summands are totally isotropic, and each is orthogonal to the middle summand.
\end{remark}

Exactly as in (\ref{flippy map}), there is a  $\kk$-conjugate linear isomorphism
\[
\Hom_\kk( W_0, \mathrm{gr}_{-1}W ) \map{x\mapsto x^\vee} \Hom_\kk(  \mathrm{gr}_{-1}W ,W_0).
\]
If we define 
\begin{align}\label{boundary herm}
L_0 &= \Hom_{\co_\kk}( g\mathfrak{a}_0 , \mathrm{gr}_{-1}(g \mathfrak{a}) ) \\
\Lambda_0 &= \Hom_{\co_\kk}  ( \mathrm{gr}_{-1} ( g \mathfrak{a}  ) , g \mathfrak{a}_0  ), \nonumber
\end{align}
then $x\mapsto x^\vee$ restricts to an $\co_\kk$-conjugate linear isomorphism $L_0 \iso \Lambda_0$.
These are, in a natural way, positive definite self-dual hermitian lattices.  
For $x_1,x_2\in L_0$ the hermitian form on $L_0$ is defined, as in (\ref{hom hermitian}), by
\[
\langle  x_1 , x_2 \rangle = x_1^\vee\circ x_2 \in \End_{\co_\kk}(g \mathfrak{a}_0) \iso \co_\kk,
\]
while the hermitian form on $\Lambda_0$ is defined by 
\[\langle x^\vee_2, x^\vee_1 \rangle  = \langle x_1,x_2\rangle.\]

\begin{lemma}\label{lem:unitary delta moduli}
Two proper cusp label representatives  $\Phi$ and $\Phi'$ are $K$-equivalent if and only if  $\Lambda_0 \iso \Lambda_0'$ as hermitian $\co_\kk$-modules and $\mathfrak{n} \iso \mathfrak{n}'$ as $\co_\kk$-modules.  Moreover, the finite group (\ref{delta group}) satisfies
\begin{equation}\label{delta iso}
\Delta_\Phi \iso \mathrm{U}(\Lambda_0  ) \times \GL_{\co_\kk}(\mathfrak{n}  ).
\end{equation}
\end{lemma}

\begin{proof}
The first claim is an elementary exercise, left to the reader.  
For the second claim we only define the isomorphism (\ref{delta iso}),  and again leave the   details to the reader. The group $P(\Q)$ acts on both $W_0$ and $W$, preserving their weight filtrations, and so acts on both the hermitian space $\Hom_\kk(\mathrm{gr}_{-1}W , W_0)$ and the $\kk$-vector space $\mathrm{gr}_0 W$.  The subgroup $P(\Q) \cap Q_\Phi(\A_f) \tilde{K}_\Phi$ preserves the lattices 
\[
\Lambda_0 \subset \Hom_\kk(\mathrm{gr}_{-1}W , W_0)
\]
and $\mathfrak{n} \subset \mathrm{gr}_0 W$, inducing  (\ref{delta iso}).
\end{proof}


\subsection{Mixed Shimura varieties}
\label{ss:mixed data}


The subgroup $Q_\Phi(\R) \subset G(\R)$ acts on 
\[
\mathcal{D}_\Phi(W) = \{ \mbox{$\kk$-stable $\R$-planes $y\subset W(\R)$} :  W(\R)= J^\perp(\R) \oplus y \},
\]
and so also acts on 
\[
\mathcal{D}_\Phi = \mathcal{D}(W_0) \times \mathcal{D}_\Phi(W).
\]
The hermitian domain of (\ref{GU hermitian}) satisifies $\mathcal{D}(W) \subset \mathcal{D}_\Phi(W)$, and hence there is a canonical  $Q_\Phi(\R)$-equivariant inclusion 
$
\mathcal{D}\subset \mathcal{D}_\Phi.
$

The mixed Shimura variety 
\begin{equation}
\label{mixed shimura}
\mathrm{Sh} (Q_\Phi , \mathcal{D}_\Phi)(\C)  = Q_\Phi(\Q) \backslash \mathcal{D}_\Phi \times Q_\Phi(\A_f) / K_\Phi 
\end{equation}
admits a canonical model  $\mathrm{Sh}( Q_\Phi, \mathcal{D}_\Phi)$ over $\kk$ by the general results of \cite{Pink}.  By rewriting the double quotient as
\[
\mathrm{Sh} (Q_\Phi , \mathcal{D}_\Phi)(\C)  \iso   Q_\Phi(\Q) \backslash   \mathcal{D}_\Phi \times  Q_\Phi (\A_f) \tilde{K}_\Phi / \tilde{K}_\Phi ,
\]
we see that (\ref{mixed shimura}) admits an action of the finite group $\Delta_\Phi$ of (\ref{delta group}), induced by the  action of 
  $P  (\Q)  \cap Q_\Phi  (\A_f) \tilde{K}_\Phi$ on both factors of $\mathcal{D}_\Phi \times  Q_\Phi (\A_f) \tilde{K}_\Phi$.  This action descends to an action on the canonical model.

\begin{proposition}\label{prop:mixed connected}
The morphism $\nu_\Phi$ of (\ref{Q fiber}) induces a surjection
\[
\mathrm{Sh} (Q_\Phi , \mathcal{D}_\Phi)(\C)  \map{ ( z,h) \mapsto \nu_\Phi(h) }   \kk^\times \backslash \widehat{\kk}^\times / \widehat{\co}_{\kk}^\times
\]
with connected fibers.  This map is $\Delta_\Phi$-equivariant, where $\Delta_\Phi$ acts trivially on the target.  In particular, the number of connected components of  (\ref{mixed shimura}) is equal to the class number of $\kk$, and the same is true of its orbifold quotient by the action of $\Delta_\Phi$.
\end{proposition}

\begin{proof}
The space $\mathcal{D}_\Phi$ is connected, and the kernel of $\nu_\Phi : Q_\Phi \to \mathrm{Res}_{\kk/\Q} \mathbb{G}_m$ is unipotent (so satisfies strong approximation).  Therefore 
\[
\pi_0\big( \mathrm{Sh} (Q_\Phi , \mathcal{D}_\Phi)(\C)  \big) 
\iso Q_\Phi(\Q) \backslash Q_\Phi(\A_f) / K_\Phi \iso  \kk^\times \backslash \widehat{\kk}^\times / \nu_\Phi(K_\Phi),
\]
and an easy calculation shows that $\nu_\Phi(K_\Phi) = \widehat{\co}_\kk^\times$.
\end{proof}

It will be useful to have other interpretations of  $\mathcal{D}_\Phi$.

\begin{remark}\label{rem:mixed hodge}
Any point $y\in \mathcal{D}_\Phi(W)$  determines a  mixed Hodge structure on $W$ whose  weight filtration $\mathrm{wt}_i W\subset W$ was defined above,  and whose Hodge filtration is defined exactly as in Remark \ref{rem:hodge}.  As in \cite[p.~64]{PS} or \cite[Proposition 1.2]{Pink} there is an induced bigrading  $W(\C) = \bigoplus W^{(p,q)}$,
and this bigrading is induced by a morphism $\mathbb{S}_\C \to  \GU(W) _\C$ taking values in the stabilizer of $J(\C)$.  
The product  of  this morphism with the morphism $\mathbb{S}_\C \to \GU(W_0)_{\C} $ of Remark \ref{rem:hodge} defines a map  $z  :  \mathbb{S}_\C \to Q_{\Phi \C}$, and this realizes   $ \mathcal{D}_\Phi   \subset \Hom(\mathbb{S}_\C, Q_{\Phi \C}).$
\end{remark}

\begin{remark}\label{rem:mixed to so}
Imitating the construction of Remark \ref{rem:to so} identifies
\[
\mathcal{D}_\Phi  \iso  \big\{ w \in \epsilon V(\C) :    V(\C) =  \mathrm{wt}_0V(\C)  \oplus \C w  \oplus \C \overline{w}   \big\} / \C^\times 
\subset \mathbb{P}( \epsilon V(\C)) 
\]
as an open subset of projective space.
\end{remark}


\subsection{The first moduli interpretation}
\label{ss:one-motives}


Using the pair  $( \Lambda_0 , \mathfrak{n})$ defined in \S \ref{ss:cusp notation}, we now construct a smooth integral model of the mixed Shimura variety (\ref{mixed shimura}).  
Following the general recipes of the theory of arithmetic toroidal compactifications, as in \cite{FC,Ho2, MP, Lan}, this integral model will be defined as the top layer of a tower of morphisms
\[
\mathcal{C}_\Phi \to \mathcal{B}_\Phi \to \mathcal{A}_\Phi \to \Spec(\co_\kk),
\]
smooth of relative dimensions $1$, $n-2$, and $0$, respectively.  

Recall from  \S \ref{ss:unitary integral models} the  smooth $\co_\kk$-stack 
\[
\mathcal{M}_{(1,0)} \times_{\co_\kk} \mathcal{M}_{(n-2,0)} \to \Spec(\co_\kk)
\] 
of relative dimension $0$ parametrizing  certain pairs $(A_0, B)$ of    polarized abelian schemes over $S$ with $\co_\kk$-actions.     
The \'etale sheaf $\underline{\Hom}_{\co_\kk} ( B  , A_0  )$  on $S$ is locally constant; this is a consequence of \cite[Theorem 5.1]{BHY}.

Define $\mathcal{A}_\Phi$ as  the  moduli space of triples $(A_0, B, \varrho )$ over  $\co_\kk$-schemes $S$, in which  $(A_0,B) $  is an $S$-point of $\mathcal{M}_{(1,0)} \times_{\co_\kk} \mathcal{M}_{(n-2,0)}$, and 
\[
\varrho :    \underline{\Lambda}_0  \iso \underline{\Hom}_{\co_\kk} (  B , A_0  )
\]  
is an isomorphism of \'etale sheaves of hermitian $\co_\kk$-modules.  

Define  $\mathcal{B}_\Phi$  as the  moduli space of quadruples  $(A_0, B, \varrho , c )$ over  $\co_\kk$-schemes $S$, in which $(A_0,B,\varrho ) $  is an $S$-point of $\mathcal{A}_\Phi$, and $c: \mathfrak{n} \to B$  is an $\co_\kk$-linear homomorphism of group schemes over $S$.    In other words, if $(A_0,B,\varrho)$ is the universal object over $\mathcal{A}_\Phi$, then
\[
\mathcal{B}_\Phi =  \underline{\Hom}_{\co_\kk}( \mathfrak{n} , B).
\]

Suppose we fix  $\mu,\nu \in \mathfrak{n}$.  For any scheme $U$ and any morphism $U\to \mathcal{B}_\Phi$, there is a corresponding quadruple $(A_0,B,\varrho, c)$ over $U$.   Evaluating the morphism of $U$-group schemes $c : \mathfrak{n} \to B$  at $\mu$ and $\nu$ determines $U$-points $c(\mu), c(\nu) \in B(U)$, and hence determines a morphism of $U$-schemes
\[
U  \map{c(\mu) \times  c (\nu) }   B \times B  \iso B  \times B^\vee.
\]
Denote by $\mathcal{L}(\mu,\nu)_U$ the pullback of the Poincar\'e bundle via this morphism.  
As $U$ varies, these line bundles are obtained as the pullback of a single line bundle     $\mathcal{L}(\mu,\nu)$ on $\mathcal{B}_\Phi$.

It follows from standard bilinearity properties of the Poincar\'e bundle that $\mathcal{L}(\mu,\nu)$
 depends, up to canonical isomorphism, only on the image of $\mu\otimes \nu$ in 
\[
\mathrm{Sym}_\Phi  = \mathrm{Sym}^2_\Z (\mathfrak{n} )/  \big\langle (x\mu)\otimes \nu- \mu\otimes(\overline{x}\nu)
: x\in \co_\kk,\, \mu,\nu \in \mathfrak{n}  \big\rangle.
\]
Thus we may  associate to every $\chi\in \mathrm{Sym}_\Phi$ a line bundle $\mathcal{L}(\chi)$ on $\mathcal{B}_\Phi$, and there   are canonical isomorphisms
\[
\mathcal{L}( \chi )  \otimes \mathcal{L}( \chi') \iso \mathcal{L}( \chi+\chi').
\]
Our assumption that $D$ is odd implies that $\mathrm{Sym}_\Phi$ is a free $\Z$-module of rank one.   Moreover, there is positive cone in  $\mathrm{Sym}_\Phi \otimes_\Z\R$ uniquely determined by the condition $\mu\otimes \mu \ge 0$ for all $\mu \in \mathfrak{n}$.     Thus  all of the line bundles $\mathcal{L}(\chi)$ are  powers of the distinguished line bundle  
 \begin{equation}\label{boundary bundle}
 \mathcal{L}_\Phi=\mathcal{L}(\chi_0)
 \end{equation}
determined by the unique positive generator  $\chi_0 \in \mathrm{Sym}_\Phi$.

At last,  define   $\mathcal{B}_\Phi$-stacks
\[
\mathcal{C}_\Phi  = \underline{\mathrm{Iso}} (  \mathcal{L}_\Phi , \co_{\mathcal{B}_\Phi} ) ,\quad 
\mathcal{C}_\Phi^* = \underline{\Hom} (  \mathcal{L}_\Phi ,  \co_{\mathcal{B}_\Phi} ) .
\]
In other words, $\mathcal{C}_\Phi^*$ is the total space of the line bundle $\mathcal{L}_\Phi^{-1}$, and  $\mathcal{C}_\Phi$ is the complement of the zero section
$\mathcal{B}_\Phi \hookrightarrow  \mathcal{C}_\Phi^*  .
$
In slightly fancier language,
\[
\mathcal{C}_\Phi
= \underline{\Spec}_{\mathcal{B}_\Phi }  \Big( \bigoplus_{ \ell \in \Z } \mathcal{L}_\Phi^\ell  \Big),
\quad 
\mathcal{C}_\Phi^*
= \underline{\Spec}_{  \mathcal{B}_\Phi  }  \Big(  \bigoplus_{ \ell \ge 0  } \mathcal{L}_\Phi^\ell \Big),
\]
and the zero section  $\mathcal{B}_\Phi \hookrightarrow  \mathcal{C}_\Phi^*$ is defined by the ideal sheaf  $\bigoplus_{ \ell > 0  } \mathcal{L}_\Phi^\ell$.

\begin{remark}
When $n=2$ the situation is a bit degenerate.  In this case 
\[
\mathcal{B}_\Phi = \mathcal{A}_\Phi = \mathcal{M}_{(1,0)},
\]
$\mathcal{L}_\Phi$ is the trivial bundle, and  $\mathcal{C}_\Phi \to \mathcal{B}_\Phi$ is  the trivial $\mathbb{G}_m$-torsor.
\end{remark}

\begin{remark}
Using the isomorphism of Lemma \ref{lem:unitary delta moduli},  the group $\Delta_\Phi$ acts on  $\mathcal{B}_\Phi$  via
\[
(u, t ) \action (A_0,B,\varrho, c ) = (A_0, B, \varrho \circ u^{-1}, c \circ t^{-1}),
\]
for  $(u,t) \in \mathrm{U}(\Lambda_0  ) \times \GL_{\co_\kk}(\mathfrak{n}  )$.   The  line bundle $\mathcal{L}_\Phi$ is invariant under  $\Delta_\Phi$, and hence the action of $\Delta_\Phi$  lifts to  both  $ \mathcal{C}_\Phi $ and $\mathcal{C}_\Phi^*$.
\end{remark}

\begin{proposition}\label{prop:boundary uniformization}
There is a $\Delta_\Phi$-equivariant isomorphism
\[
 \mathrm{Sh}  (Q_\Phi , \mathcal{D}_\Phi) \iso \mathcal{C}_{\Phi/\kk}  .
\]
\end{proposition}

\begin{proof}
This is a special case of the general fact that mixed Shimura varieties appearing at the boundary of PEL Shimura varieties are themselves moduli spaces of  $1$-motives endowed with polarizations, endomorphisms, and level structure.  The core of this is Deligne's theorem \cite[\S 10]{De} that the category of $1$-motives over $\C$ is equivalent to the category of integral mixed Hodge structures of types $(-1,-1)$, $(-1,0)$, $(0,-1)$, $(0,0)$.
See  \cite{MP}, where this is explained for Siegel modular varieties, and also \cite{Bry}.  A good introduction to $1$-motives is \cite{ABV}.

To make this a bit more explicit in our case, denote by $\mathcal{X}_\Phi$ the $\co_\kk$-stack whose functor of points assigns to an $\co_\kk$-scheme $S$ the groupoid $\mathcal{X}_\Phi(S)$ of  principally polarized $1$-motives $A$ consisting of diagrams
\[
\xymatrix{
 & & { \mathfrak{n}  }  \ar[d] \\
0 \ar[r]  &  \mathfrak{m} \otimes_\Z \mathbb{G}_m \ar[r] & \mathbb{B} \ar[r]  & B  \ar[r] & 0
}
\]
in which $B\in \mathcal{M}_{(n-2,0)}(S)$,   $\mathbb{B}$ is an extension of $B$ by the rank two torus $\mathfrak{m} \otimes_\Z \mathbb{G}_m$ in the category of group schemes with $\co_\kk$-action, and the arrows are morphisms of fppf sheaves of $\co_\kk$-modules.

To explain what it means to have a principal polarization of such a $1$-motive $A$, 
set $\mathfrak{m}^\vee = \Hom(\mathfrak{m} ,\Z)$ and $\mathfrak{n}^\vee =\Hom(\mathfrak{n} ,\Z)$, and recall from \cite[\S 10]{De} that $A$ has  a dual $1$-motive $A^\vee$ consisting of a diagram
\[
\xymatrix{
 & & { \mathfrak{m}^\vee  }  \ar[d] \\
0 \ar[r]  &  \mathfrak{n}^\vee \otimes_\Z \mathbb{G}_m \ar[r] & \mathbb{B}^\vee \ar[r]  & B^\vee  \ar[r] & 0.
}
\]
A principal polarization is an $\co_\kk$-linear isomorphism $\mathbb{B}\iso \mathbb{B}^\vee$ compatible with the given polarization $B\iso B^\vee$, and with the isomorphisms $\mathfrak{m}\iso \mathfrak{n}^\vee$ and $\mathfrak{n}\iso \mathfrak{m}^\vee$ determined by the perfect pairing $\mathfrak{m}\otimes_\Z \mathfrak{n} \to \Z$ defined after (\ref{easy graded}).

Using the ``description plus sym\'etrique''  of $1$-motives \cite[(10.2.12)]{De},
the $\co_\kk$-stack $\mathcal{C}_\Phi$ defined above can  be identified with the moduli space whose $S$-points are triples $(A_0,A,\varrho)$ in which 
\begin{itemize}
\item
$(A_0, A ) \in \mathcal{M}_{(1,0)}(S) \times \mathcal{X}_\Phi (S)$,
\item
$\varrho  :  \underline{\Lambda}_0 \iso \underline{\Hom}_{\co_\kk}( B ,A_0)$ is an  isomorphism of \'etale sheaves of hermitian $\co_\kk$-modules, where $B\in \mathcal{M}_{(n-2,0)}(S)$ is the abelian scheme part of $A$.
\end{itemize}

To verify that  $\mathrm{Sh}(Q_\Phi , \mathcal{D}_\Phi)$ has the same functor of points, one uses Remark \ref{rem:mixed hodge} to interpret $\mathrm{Sh}(Q_\Phi , \mathcal{D}_\Phi)(\C)$ as a moduli space of mixed Hodge structures on $W_0$ and $W$, and uses the theorem of Deligne cited above to interpret these mixed Hodge structures as $1$-motives.  This defines an  isomorphism $\mathrm{Sh}(Q_\Phi , \mathcal{D}_\Phi)(\C) \iso \mathcal{C}_\Phi(\C)$.   The proof that it descends to the reflex field is identical to the proof for Siegel mixed Shimura varieties   \cite{MP}.

We remark in passing that  any triple $(A_0,A,\varrho)$ as above automatically satisfies  (\ref{tate genus}) for every prime $\ell$.  Indeed,  both sides of (\ref{tate genus}) are now endowed with weight filtrations, analogous to the weight filtration on $\Hom_\kk(W_0,W)$ defined in \S \ref{ss:cusp notation}.  The isomorphism $\varrho$ induces an isomorphism (as hermitian $\co_{\kk,\ell}$-lattices) between the $\mathrm{gr}_0$ pieces on either side.  The $\mathrm{gr}_{-1}$ and $\mathrm{gr}_1$ pieces have no structure other then projective $\co_{\kk,\ell}$-modules of rank $1$, so are isomorphic.  These isomorphisms of graded pieces imply the existence of an isomorphism (\ref{tate genus}), exactly as in Remark \ref{rem:determined by grade}.
\end{proof}


\subsection{The second moduli interpretation}
\label{ss:second moduli}


In order to make explicit calculations,  it will be useful to interpret the moduli spaces 
\[
\mathcal{C}_\Phi \to \mathcal{B}_\Phi \to \mathcal{A}_\Phi \to \Spec(\co_\kk)
\]  
in a different way.

Suppose $E\to S$ is an elliptic curve over any base scheme, and denote by $\mathcal{P}_E$ the Poincar\'e bundle on 
\[
E\times_S E \iso E \times_S E^\vee.
\]
  If $U$ is any $S$-scheme and $a,b \in  E(U)$, we obtain an $\co_U$-module 
$\mathcal{P}_E(a,b)$  by pulling back the Poincare bundle via  
\[
U \map{(a,b)} E\times_S E \iso E \times_S E^\vee.
\]
 The notation is intended to remind the reader of the bilinearity properties of the Poincar\'e bundle, as expressed by  canonical $\co_U$-module isomorphisms
\begin{align}
\mathcal{P}_E(a+b,c)  &  \iso \mathcal{P}_E(a,c) \otimes \mathcal{P}_E(b,c) \label{poincare bilinear}\\
\mathcal{P}_E(a,b+c)  &  \iso \mathcal{P}_E(a,b) \otimes \mathcal{P}_E(a,c) \nonumber \\
\mathcal{P}_E(a,b) &  \iso \mathcal{P}_E(b,a), \nonumber
\end{align}
along with
$
\mathcal{P}_E(e,b) \iso \co_U \iso \mathcal{P}_E(a,e).
$  
Here $e\in E(U)$ is the zero section.


Let $E\to \mathcal{M}_{(1,0)}$ be the universal elliptic curve with complex multiplication by $\co_\kk$.  Its Poincar\'e bundle satisfies, for all $\alpha\in \co_\kk$,  the additional relation
$
\mathcal{P}_E( \alpha   a,b) \iso \mathcal{P}_E(  a, \overline{ \alpha}   b).
$

Recall the positive definite self-dual hermitian lattice $L_0$ of (\ref{boundary herm}).
Using Serre's tensor construction, we define an abelian scheme 
\begin{equation}\label{serre twist}
E\otimes L_0 = E\otimes_{\co_\kk}L_0
\end{equation}
over $\mathcal{M}_{(1,0)}$.  
As explained in detail in \cite{zavosh}, the principal polarization on $E$ and the hermitian form on $L_0$ can be combined to define a principal polarization on $E \otimes L_0$, and we denote by  $\mathcal{P}_{E\otimes L_0}$  the Poincar\'e bundle on
\[
(E\otimes L_0)\times_{\mathcal{M}_{(1,0)}} (E\otimes L_0)  \iso (E\otimes L_0)\times_{\mathcal{M}_{(1,0)}} (E\otimes L_0)^\vee. 
\]
The Poincar\'e bundle $\mathcal{P}_{E \otimes L_0}$ can be expressed in terms of $\mathcal{P}_E$.   
  If $U$ is a scheme, a morphism
\[
U \to (E\otimes L_0)\times_{\mathcal{M}_{(1,0)}} (E\otimes L_0)
\]
is given by a pair of $U$-valued points 
\[
c = \sum s_i \otimes x_i  \in E(U) \otimes L_0,\quad c' = \sum s_j' \otimes x_j'  \in E(U) \otimes L_0,
\]
and the pullback of $\mathcal{P}_{E\otimes L_0}$ to $U$ is 
\[
\mathcal{P}_{E\otimes L_0}(c,c') = \bigotimes_{i,j} \mathcal{P}_E ( \langle x_i,x_j' \rangle s_i, s_j').
\]

Define $\mathcal{Q}_{E\otimes L_0}$ to be the line bundle on $E\otimes L_0$ whose 
restriction to the $U$-valued point  $c= \sum s_i \otimes x_i  $  is 
\begin{equation}\label{Q def}
  \mathcal{Q}_{E\otimes L_0}(c)  =    \bigotimes_{ i < j } \mathcal{P}_E (    \langle x_i , x_j \rangle    s_i  ,   s_j  ) \otimes  \bigotimes_{ i } \mathcal{P}_{E } (   \gamma   \langle x_i , x_i \rangle s_i,      s_i   )  ,
\end{equation}
where
\[
\gamma=\frac{ 1+\delta }{2} \in \co_\kk.
\]
  It is related to $\mathcal{P}_{E\otimes L_0}$ by  canonical isomorphisms
\begin{align}\label{mumford sheaf}
\mathcal{P}_{E\otimes L_0} (a,b) &  \iso  \mathcal{Q}_{E\otimes L_0}(a+b)  \otimes  \mathcal{Q}_{E\otimes L_0}(a)^{-1} \otimes  \mathcal{Q}_{E\otimes L_0}(b)^{-1} \\
\mathcal{P}_{E\otimes L_0} (a,a) & \iso  \mathcal{Q}_{E\otimes L_0}(a)^{\otimes 2}. \nonumber
\end{align}
for all  $U$-valued points $a,b\in E(U)\otimes L_0$.  

\begin{remark}
As in the constructions of \cite[\S 1.3.2]{Lan} or \cite[\S 6.2]{MFK},  the line bundle $\mathcal{Q}_{E\otimes L_0}$ 
determines a morphism $E\otimes L_0 \to (E \otimes L_0)^\vee$. The relations (\ref{mumford sheaf}) amount to saying that this morphism is the principal polarization constructed in \cite{zavosh}.
\end{remark}

\begin{remark}
The line bundle $\mathcal{P}_{E\otimes L_0}( \delta a ,a)$ is canonically trivial.  This follows by comparing
\[
\mathcal{P}_{E\otimes L_0}( \gamma a ,a)^{\otimes 2}  \iso \mathcal{P}_{E\otimes L_0}(  a ,a) \otimes \mathcal{P}_{E\otimes L_0}( \delta a ,a)
\]
with
\[
\mathcal{P}_{E\otimes L_0}( \gamma a ,a)^{\otimes 2} 
\iso \mathcal{P}_{E\otimes L_0}( \gamma a ,a) \otimes \mathcal{P}_{E\otimes L_0}(  \overline{\gamma} a ,a)
\iso \mathcal{P}_{E\otimes L_0}(a,a) .
\]
\end{remark}

\begin{remark}
In the slightly degenerate case of $n=2$,
$E\otimes L_0$ is the trivial group scheme over $\mathcal{M}_{(1,0)}$, and 
$\mathcal{P}_{ E\otimes L_0}$ is the trivial bundle on $\mathcal{M}_{(1,0)}$.
\end{remark}

\begin{proposition}\label{prop:second moduli iso}
As above, let $E\to \mathcal{M}_{(1,0)}$ be the universal object.
 There are canonical isomorphisms
\[
\xymatrix{
{  \mathcal{C}_\Phi   }  \ar[r] \ar[d]^{\iso}    & {  \mathcal{B}_\Phi   }  \ar[r] \ar[d]^{\iso} &   {   \mathcal{A}_{\Phi}    }  \ar[d]^{\iso}  \\
{   \underline{\mathrm{Iso}}( \mathcal{Q}_{ E\otimes L_0 } , \co_{E\otimes L_0} )  }   \ar[r]  & {   E\otimes L_0   }   \ar[r]  &  {    \mathcal{M}_{(1,0)}   ,}
}
\]
and the middle vertical arrow identifies $\mathcal{L}_\Phi \iso \mathcal{Q}_{E\otimes L_0}$.
\end{proposition}

\begin{proof}
Define a morphism $\mathcal{A}_\Phi \to \mathcal{M}_{(1,0)}$ by sending a triple $(A_0,B,\varrho)$ to the CM elliptic curve
\begin{equation}\label{elliptic twist}
E =  \underline{\Hom}_{\co_\kk} ( \mathfrak{n} , A_0 ).
\end{equation}
To show that this map is an isomorphism we will construct the inverse.

If  $S$ is any $\co_\kk$-scheme and  $E\in \mathcal{M}_{(1,0)}(S)$, we may define  $(A_0,B,\varrho)\in \mathcal{A}_\Phi(S)$  by setting
\[
A_0= E\otimes_{\co_\kk}\mathfrak{n} ,\quad
B =  \underline{\Hom}_{\co_\kk} (\Lambda_0 , A_0),
\]
and taking for $\varrho:  \underline{\Lambda}_0 \iso \underline{\Hom}_{\co_\kk} (B,A_0)$ the tautological isomorphism.  The principal polarization on $B$ is defined  using the $\co_\kk$-linear isomorphism 
\[
A_0 \otimes_{\co_\kk} L_0 \map{ a\otimes x \mapsto \langle\, \cdot\,  , x^\vee \rangle a} \underline{\Hom}_{\co_\kk}(\Lambda_0 , A_0)
\]
and the principal polarization on $A_0 \otimes_{\co_\kk} L_0$  constructed in \cite{zavosh}, exactly as in the discussion following (\ref{serre twist}).   The construction $E\mapsto (A_0,B,\varrho)$ is inverse to the above morphism  $\mathcal{A}_\Phi \to \mathcal{M}_{(1,0)}$.

Now identify $\mathcal{A}_\Phi \iso \mathcal{M}_{(1,0)}$ using the above isomorphism, and denote by $(A_0,B,\varrho)$ and $E$ the universal objects on the source and target.  They are related by canonical isomorphisms
\begin{equation}\label{boundary nonsense}
\xymatrix{
&  { \mathcal{B}_\Phi=  \underline{\Hom}_{\co_\kk} ( \mathfrak{n} , B  ) }    \\
 {   \underline{\Hom}_{\co_\kk} ( \mathfrak{n} \otimes_{\co_\kk} \Lambda_0 , A_0 ) }  \ar[ur]^{\iso} \ar[dr]_{ \iso } \\
     & {  \underline{\Hom}_{\co_\kk}(\Lambda_0 , E  ) .}
}
\end{equation}
Combining this with the $\co_\kk$-linear isomorphism
\[
E \otimes L_0 \map{ a\otimes x \mapsto \langle\, \cdot\,  , x^\vee \rangle a} \underline{\Hom}_{\co_\kk}(\Lambda_0 , E)
\]
defines  $\mathcal{B}_\Phi \iso E \otimes L_0$.  All that remains is to prove that this isomorphism identifies $\mathcal{L}_\Phi$ with $\mathcal{Q}_{E\otimes L_0}$, which amounts to carefully keeping track of the relations between the three Poincar\'e bundles $\mathcal{P}_B$, $\mathcal{P}_E$, and $\mathcal{P}_{A_0}$.

Any fractional ideal $\mathfrak{b} \subset \kk$ admits a unique positive definite self-dual hermitian form, given explicitly by $\langle b_1,b_2\rangle = b_1\overline{b}_2 / \mathrm{N}(\mathfrak{b})$.  It follows that any rank one projective $\co_\kk$-module admits a unique positive definite self-dual hermitian form.    For the $\co_\kk$-module $\Hom_{\co_\kk}(\mathfrak{n} , \co_\kk)$, this hermitian form is 
\[
\langle \ell_1 , \ell_2 \rangle = \ell_1(\mu ) \overline{\ell_2(\nu )} + \ell_1(\nu) \overline{\ell_2(\mu)},
\]
where $\mu\otimes \nu= \chi_0 \in \mathrm{Sym}_\Phi$ is the positive generator appearing in (\ref{boundary bundle}).

The relation (\ref{elliptic twist})  implies a relation between the line bundles  $\mathcal{P}_E$ and  $\mathcal{P}_{A_0}$.  If $U$ is any $\mathcal{A}_\Phi$-scheme and   we are given points 
\[
s , s' \in E(U)  = \Hom_{\co_\kk}( \mathfrak{n} , A_{ 0 U} ) 
\]
of the form  $s = \ell  (\cdot) a$ and   $s' =\ell'  (\cdot) a'$  with $\ell ,\ell'  \in \Hom_{\co_\kk}(\mathfrak{n} ,\co_\kk)$ and  $a , a' \in A_0(U)$, then 
\begin{align*}
\mathcal{P}_E(s , s' )   & \iso   \mathcal{P}_{A_0} \big(   \langle \ell  , \ell' \rangle   a ,   a' \big) \\
\mathcal{P}_E( \gamma s, s) &  \iso  \mathcal{P}_{A_0} \big(   \ell (\mu )    a ,  \ell  (\nu )  a  \big) .
\end{align*}

Similarly, the isomorphism $B \iso   \underline{\Hom}_{\co_\kk} (\Lambda_0 , A_0)$
implies a relation between $\mathcal{P}_B$ and $\mathcal{P}_{A_0}$.
  If $U$ is an $S$-scheme, a morphism
$
U \to B\times_{\mathcal{A}_\Phi}  B
$
is given by a pair of points
\[
b, b' \in B(U) = \Hom_{\co_\kk}( \Lambda_0  , A_{ 0 U } )  
\]
of the form  $b = \langle \cdot, \lambda \rangle a$  and  $b' =\langle \cdot, \lambda' \rangle a'$  with $ \lambda ,\lambda' \in \Lambda_0$ and  $a, a' \in A_0(U)$.   
The pullback of $\mathcal{P}_B$ to $U$ is the line bundle
\[
\mathcal{P}_B(b,b')   =  \mathcal{P}_{A_0} ( a , \langle \lambda  ,\lambda'  \rangle a' ) .
\]

 Using the isomorphisms (\ref{boundary nonsense}), a point $c\in \mathcal{B}_\Phi(U)$  admits three different interpretations.  In one of them, $c$ has the form
\[
 c= \sum  \ell_i (\cdot ) \langle \cdot , \lambda_i \rangle a_i   \in   \Hom_{\co_\kk} ( \mathfrak{n} \otimes_{\co_\kk} \Lambda , A_{ 0 U } )  .
\]
By setting
\begin{align*}
b_i  = \langle \cdot , \lambda_i \rangle a_i  & \in \Hom_{\co_\kk}( \Lambda_0  , A_{0 U }) = B(U)\\
s_i = \ell_i ( \cdot) a_i & \in  \Hom_{\co_\kk}( \mathfrak{n} , A_{0 U }) = E(U) , 
\end{align*}
we find the other two interpretations
\begin{align*}
c  & = \sum  \ell_i(\cdot) b_i    \in \Hom_{\co_\kk} (\mathfrak{n} , B_U)  \\ 
 c & = \sum  \langle \cdot  , \lambda_i\rangle s_i    \in \Hom_{\co_\kk}(\Lambda_0, E_U) .
\end{align*}
The above relations between $\mathcal{P}_B$, $\mathcal{P}_E$, and $\mathcal{P}_{A_0}$  imply
\begin{eqnarray*}\lefteqn{
 \mathcal{P}_B (    c(\mu )   ,   c(\nu )  )    } \\
 & \iso & \bigotimes_{i,j}   \mathcal{P}_B (    \ell_i (\mu  ) b_i    ,    \ell_j (\nu  )b_j   )    \\
&  \iso & \bigotimes_{ i ,j }  \mathcal{P}_{A_0}  (    \ell_i (\mu )     a_i , \langle \lambda_i,\lambda_j\rangle  \ell_j(\nu ) a_j  ) \\
&  \iso & \bigotimes_{ i < j }  \mathcal{P}_{A_0}  (   \langle \ell_i , \ell_j \rangle   a_i , \langle \lambda_i,\lambda_j\rangle  a_j  )  
  \otimes  \bigotimes_{ i }  \mathcal{P}_{A_0}  (    \ell_i (\mu )    a_i   ,  \ell_i(\nu )    \langle \lambda_i,\lambda_i\rangle   a_i )   \\
& \iso & \bigotimes_{ i < j } \mathcal{P}_E  (    s_i ,   \langle \lambda_i , \lambda_j \rangle  s_j )    \otimes  \bigotimes_{ i } \mathcal{P}_{E} (    \gamma   s_i   ,      \langle \lambda_i,\lambda_i\rangle    s_i  )   .
\end{eqnarray*}
Now write  $\lambda_i=x_i^\vee$ with  $x_i\in L_0$, and use the relation 
\[
\mathcal{P}_E (    s_i ,   \langle \lambda_i , \lambda_j \rangle  s_j  ) 
= \mathcal{P}_E (      \langle \lambda_j , \lambda_i \rangle  s_i ,   s_j  ) 
=\mathcal{P}_E (     \langle x_i , x_j  \rangle  s_i ,   s_j  ) 
\]
to obtain an isomorphism 
$
\mathcal{P}_B (    c(\mu )   ,   c(\nu )  )  \iso \mathcal{Q}_{E\otimes L_0}(c).
$
The line bundle on the left is precisely  the pullback of $\mathcal{L}_\Phi$ via $c$, and letting $c$ vary  we obtain an isomorphism $\mathcal{L}_\Phi \iso  \mathcal{Q}_{E\otimes L_0}$.
\end{proof}


\subsection{The line bundle of modular forms}
\label{ss:mfbundle}


We now define a line bundle of weight one modular forms on our mixed Shimura variety, analogous to the one on the pure Shimura variety defined in \S \ref{ss:unitary bundle}.

The holomorphic line bundle $\bm{\omega}^{an}$ on $\mathcal{D}$  defined in \S \ref{ss:unitary bundle} admits a canonical extension to 
\[
\mathcal{D}_\Phi= \mathcal{D}(W_0) \times \mathcal{D}_\Phi(W),
\]  
which we denote by $\bm{\omega}^{an}_\Phi$.
Indeed, recalling that $\mathcal{D}(W_0) = \{y_0\}$ is a one-point set, an element  $z\in \mathcal{D}_\Phi$ is represented by a pair $(y_0,y)$ in which $y$ is a $\kk$-stable $\R$-plane in $W(\R)$ such that 
$W(\R)  = J^\perp(\R) \oplus y.$   The fiber of $\bm{\omega}^{an}_\Phi$  at $z$ is the line
\[
\Hom_\C ( W_0(\C) / \overline{\epsilon} W_0(\C) ,  \mathrm{pr}_\epsilon  ( y )    )   \subset \epsilon V(\C) ,
\]
exactly as in Remark \ref{rem:to so} and  (\ref{hermitian bundle}).  

If we embed $\mathcal{D}_\Phi$ into projective space over $\epsilon V(\C)$ as in Remark \ref{rem:mixed to so},  then $\bm{\omega}^{an}_\Phi$ is simply the restriction of the tautological bundle.   
 There is  an obvious action of $Q_\Phi(\R)$ on the total space of $\bm{\omega}^{an}_\Phi$,  lifting the natural action on $\mathcal{D}_\Phi$, and so $\bm{\omega}^{an}_\Phi$ determines    a  holomorphic line bundle on the complex orbifold $\mathrm{Sh}(Q_\Phi , \mathcal{D}_\Phi)(\C)$. 

As in \S \ref{ss:unitary bundle}, the holomorphic line bundle $\bm{\omega}^{an}_\Phi$ is algebraic and descends to the canonical model $\mathrm{Sh}(Q_\Phi , \mathcal{D}_\Phi)$.  In fact, it admits a canonical extension to the integral model $\mathcal{C}_\Phi$, as we  now explain.   

Recalling the $\co_\kk$-modules $\mathfrak{m}$ and $\mathfrak{n}$ of (\ref{easy graded}), 
define  rank two vector bundles on $\mathcal{A}_\Phi$ by
\[
\mathfrak{M} = \mathfrak{m}\otimes_\Z \co_{\mathcal{A}_\Phi},\quad \mathfrak{N} = \mathfrak{n}\otimes_\Z \co_{\mathcal{A}_\Phi}.
\]
Each is locally free of rank one over $\co_\kk \otimes_\Z \co_{\mathcal{A}_\Phi}$, and the perfect pairing between $\mathfrak{m}$ and $\mathfrak{n}$ defined after (\ref{easy graded}) induces a perfect bilinear pairing  $\mathfrak{M} \otimes  \mathfrak{N} \to \co_{\mathcal{A}_\Phi}$.  Using the almost idempotents $\epsilon,\overline{\epsilon} \in \co_\kk \otimes_\Z \co_{\mathcal{A}_\Phi}$ of \S \ref{ss:notation}, there is an induced 
isomorphism  of line bundles
\[
(  \mathfrak{M}/\epsilon \mathfrak{M} )   \otimes ( \epsilon \mathfrak{N})   \iso  \co_{\mathcal{A}_\Phi}.
\]

Recalling that $\mathcal{A}_\Phi$ carries over it a universal triple $(A_0,B,\varrho)$, in which $A_0$ is an elliptic curve with $\co_\kk$-action, we now define a line bundle on $\mathcal{A}_\Phi$ by 
\[
\bm{\omega}_\Phi = \underline{\Hom}(\Lie(A_0) ,  \epsilon \mathfrak{N}),
\]  
or, equivalently,
\[
\bm{\omega}_\Phi^{-1} = \Lie(A_0) \otimes_{\co_{\mathcal{A}_\Phi}} \mathfrak{M}/  \epsilon \mathfrak{M}.
\]
Denote in the same way its pullback  to any step in the tower
\[
\mathcal{C}_\Phi^* \to \mathcal{B}_\Phi\to \mathcal{A}_\Phi.
\]

The above definition of $\bm{\omega}_\Phi$ is a bit unmotivated, and so  we explain why $\bm{\omega}_\Phi$ is  analogous  to the line bundle  $\bm{\omega}$  on $\mathcal{S}_\Kra$ defined in \S \ref{ss:unitary bundle}.  Recall from the proof of Proposition \ref{prop:boundary uniformization} that $\mathcal{C}_\Phi$ carries over it a universal $1$-motive $A$.  This $1$-motive has a de Rham realization $H_1^\dR(A)$, defined as the Lie algebra of the universal vector extension of $A$, as in \cite[(10.1.7)]{De}.  It is a rank $2n$-vector bundle on $\mathcal{C}_\Phi$, locally free of rank $n$ over $\co_\kk\otimes_\Z \co_{\mathcal{C}_\Phi}$, and sits in a diagram of vector bundles
\[
\xymatrix{
 & {  0  }  \ar[d]   &    &  { 0  }  \ar[d]   \\
 & {  \Fil^0 H_1^\dR(B)  }  \ar[d]   &    &  {  \mathfrak{M}  }  \ar[d]   \\
0  \ar[r]  & {  \Fil^0 H_1^\dR(A)  } \ar[r]  \ar[d]  &  {   H_1^\dR(A) } \ar[r]  &  {  \Lie(A)  } \ar[r]   \ar[d] &  {  0  } \\
  & {  \mathfrak{N}  }  \ar[d]  &   &  {  \Lie(B)  } \ar[d]  \\
  & {  0  }     &    &  { 0  }  
}
\]
with exact rows and columns.  
The polarization on $A$ induces a perfect symplectic form on $H_1^\dR(A)$.  This induces  a perfect pairing 
\begin{equation}\label{fake fil-lie}
\Fil^0H_1^\dR(A) \otimes  \Lie(A) \to \co_{\mathcal{C}_\Phi}
\end{equation}
as in (\ref{fil-lie dual}), which is compatible (in the obvious sense) with the pairings 
\[
\Fil^0H_1^\dR(B) \otimes  \Lie(B) \to \co_{\mathcal{C}_\Phi}
\]
and $\mathfrak{N}\otimes \mathfrak{M} \to \co_{\mathcal{C}_\Phi}$ that we already have.

The signature condition on $B$ implies that $\epsilon \Fil^0H_1^\dR(B) =0$ and   $\overline{\epsilon} \Lie(B)=0$.
Using this, and arguing as in \cite[Lemma 2.3.6]{Ho2}, it is not difficult to see that 
\[
\mathcal{F}_A = \mathrm{ker}( \overline{\epsilon} : \Lie(A) \to \Lie(A) ) 
\]
is the unique codimension one local direct summand of $\Lie(A)$ satisfying Kramer's condition as in \S \ref{ss:unitary integral models},
and that its orthogonal under the pairing (\ref{fake fil-lie}) is
$
\mathcal{F}_A^\perp  = \epsilon \Fil^0 H_1^\dR(A).
$
Moreover, the natural maps
\[
\mathfrak{M}/\epsilon\mathfrak{M} \to \Lie(A) / \mathcal{F}_A,\quad 
\mathcal{F}_A^\perp \to \epsilon \mathfrak{N}
\]
are isomorphisms. 
These latter isomorphisms allow us to identify
\[
\bm{\omega}_\Phi = \underline{\Hom}(\Lie(A_0) ,  \mathcal{F}_A^\perp), \quad \bm{\omega}_\Phi^{-1} = \Lie(A_0) \otimes \Lie(A) / \mathcal{F}_A
\]
in perfect analogy with \S \ref{ss:unitary bundle}.

\begin{proposition}\label{prop:mixed modular forms}
The isomorphism 
\[ 
\mathcal{C}_\Phi(\C) \iso \mathrm{Sh}(Q_\Phi , \mathcal{D}_\Phi)(\C)
\] 
of Proposition \ref{prop:boundary uniformization} identifies the analytification of $\bm{\omega}_\Phi$ with the already defined $\bm{\omega}^{an}_\Phi$.  Moreover, the isomorphism $\mathcal{A}_\Phi \iso \mathcal{M}_{(1,0)}$ of Proposition \ref{prop:second moduli iso}  identifies
\[
\bm{\omega}_\Phi \iso  \mathfrak{d}\cdot \Lie( E)^{-1} \subset \Lie( E)^{-1}
\]
where $\mathfrak{d}=\delta\co_\kk$ is the different of $\co_\kk$, and $E \to \mathcal{M}_{(1,0)}$ is the universal elliptic curve with CM by $\co_\kk$.
\end{proposition}

\begin{proof}
Any point $z=(y_0,y)\in  \mathcal{D}_\Phi$ determines, by Remarks \ref{rem:hodge} and \ref{rem:mixed hodge},   a pure Hodge structure on $W_0$ and a mixed Hodge structure on  $W$, these induce a mixed Hodge structure on $V=\Hom_\kk(W_0,W)$, and the fiber of $\bm{\omega}_\Phi^{an}$ at $z$ is 
\[
\bm{\omega}_{\Phi,z}^{an} = \Fil^1 V(\C)= \Hom_\C( W_0(\C) / \overline{\epsilon} W_0(\C) , \epsilon \Fil^0 W(\C) ).
\]
On the other hand, we have just seen that 
\[
\bm{\omega}_\Phi = \underline{\Hom}(\Lie(A_0) ,  \mathcal{F}_A^\perp)
=\underline{\Hom}(\Lie(A_0) ,  \epsilon \Fil^0 H_1^\dR(A) ).
\]
With these identifications, the proof of the first claim amounts to carefully tracing through the construction of the isomorphism of Proposition \ref{prop:boundary uniformization}.

For the second claim, the isomorphism $A_0 \iso E\otimes_{\co_\kk} \mathfrak{n}$ induces a canonical isomorphism
\[
\Lie(A_0) \iso \Lie(E) \otimes_{\co_\kk} \mathfrak{n} \iso \Lie(E) \otimes \mathfrak{N}/\overline{\epsilon}\mathfrak{N},
\]
where we have used the fact that 
$
  \mathfrak{n} \otimes_{\co_\kk} \co_{\mathcal{A}_\Phi} = \mathfrak{N}/\overline{\epsilon}\mathfrak{N}
$ 
is the largest quotient of $\mathfrak{N}$ on which $\co_\kk$ acts via the structure morphism $\co_\kk \to \co_{\mathcal{A}_\Phi}$.
Thus
\begin{align*}
\bm{\omega}_\Phi &= \underline{\Hom}(\Lie(A) , \epsilon \mathfrak{N})  \\ 
& \iso  \underline{\Hom}(\Lie(E) \otimes \mathfrak{N}/\overline{\epsilon}\mathfrak{N} , \epsilon \mathfrak{N})  \\
& \iso \Lie(E)^{-1} \otimes_{ \co_{\mathcal{A}_\Phi } }  \underline{\Hom}( \mathfrak{N}/\overline{\epsilon}\mathfrak{N} , \epsilon \mathfrak{N}).
\end{align*}

Now recall the ideal sheaf  $(\epsilon) \subset \co_\kk \otimes_\Z \co_{\mathcal{A}_\Phi}$
of \S \ref{ss:notation}. There are canonical  isomorphisms of line bundles
\[
\mathfrak{d} \co_{\mathcal{A}_\Phi} \iso (\epsilon) \iso \underline{\Hom}( \mathfrak{N}/\overline{\epsilon}\mathfrak{N} , \epsilon \mathfrak{N}),
\]
where the first is (\ref{different bundle}) and the second is the tautological isomorphism
sending $\epsilon$ to the multiplication-by-$\epsilon$ map $\mathfrak{N}/\overline{\epsilon}\mathfrak{N} \to  \epsilon \mathfrak{N}$.
These constructions determine the desired isomorphism
\[
\bm{\omega}_\Phi \iso \Lie(E)^{-1} \otimes_{ \co_{\mathcal{A}_\Phi } }  \mathfrak{d} \co_{\mathcal{A}_\Phi }. \qedhere
\]
\end{proof}


\subsection{Special divisors}
\label{ss:mixed special divisors}


Let $\mathcal{Y}_0(D)$ be the moduli stack  over $\co_\kk$ parametrizing cyclic $D$-isogenies of elliptic curves over $\co_\kk$-schemes, and let $\mathcal{E} \to \mathcal{E}'$ be the universal object.   
See \cite[Chapter 3]{KM} for the definitions.

Let $(A_0,B,\varrho,c)$ be the universal object over $\mathcal{B}_\Phi$.  Recalling the $\co_\kk$-conjugate linear isomorphism $L_0\iso \Lambda_0$ defined after (\ref{boundary herm}), each $x \in L_0$ defines a morphism
\[
 \mathfrak{n} \map{c} B\map{ \varrho(x^\vee) } A_0
\]
of sheaves of $\co_\kk$-modules on $\mathcal{B}_\Phi$. Define  $\mathcal{Z}_\Phi(x) \subset \mathcal{B}_\Phi$
as the largest closed substack over which this morphism is trivial.   We will see in a moment that this closed substack is defined locally by one equation.   For any $m>0$  define a stack over $\mathcal{B}_\Phi$ by 
\begin{equation}\label{boundary special}
\mathcal{Z}_\Phi(m) = \bigsqcup_{ \substack{   x \in L_0 \\ \langle x,x\rangle =m    }  }   \mathcal{Z}_\Phi (x) .
\end{equation}
We  also view $\mathcal{Z}_\Phi(m)$ as a divisor on $\mathcal{B}_\Phi$,   and denote in the same way the pullback of this divisor via $ \mathcal{C}_\Phi^* \to \mathcal{B}_\Phi$.

\begin{remark}
In the slightly degenerate case $n=2$ we have   $L_0=0$, and every special divisor $\mathcal{Z}_\Phi(m)$ is empty.
\end{remark}

We will now reformulate the definition of $\mathcal{Z}_\Phi(x)$ in terms of the moduli problem of \S \ref{ss:second moduli}.
Recalling the isomorphisms of  Proposition \ref{prop:second moduli iso}, every $x\in L_0$ determines a  commutative diagram
\[
\xymatrix{
{  \mathcal{B}_\Phi  } \ar[r]^\iso  \ar[d]  &  {  E\otimes L_0  }  \ar[r]^{  \langle \cdot, x\rangle  } \ar[d] & {  E }  \ar[r] \ar[d]&   {  \mathcal{E} } \ar[d]  \\
{   \mathcal{A}_\Phi    }  \ar[r]^{\iso} & {  \mathcal{M}_{(1,0)}     }  \ar@{=}[r]  &  {\mathcal{M}_{(1,0)}   }  \ar[r]  &   { \mathcal{Y}_0(D) ,}
}
\]
where $\mathcal{M}_{(1,0)} \to \mathcal{Y}_0(D)$ sends $E$ to the cyclic $D$-isogeny 
\[
E\to E\otimes_{\co_\kk} \mathfrak{d}^{-1},
\] 
and the rightmost square is cartesian.  The upper and lower horizontal compositions are denoted $j_x$ and $j$, giving the diagram
\begin{equation}\label{boundary collapse}
\xymatrix{
{  \mathcal{B}_\Phi  } \ar[r]^{j_x}\ar[d]  &  {  \mathcal{E} } \ar[d]  \\
{   \mathcal{A}_\Phi    }  \ar[r]^{j} &   { \mathcal{Y}_0(D) .}
}
\end{equation}

\begin{proposition}\label{prop:boundary divisors}
For any nonzero $x \in L_0$, the closed substack $\mathcal{Z}_\Phi(x) \subset \mathcal{B}_\Phi$ is equal to the pullback of the zero section along $j_x$.  It is   an effective Cartier divisor,   flat over $\mathcal{A}_\Phi$.  In particular, as  $\mathcal{A}_\Phi$ is flat over $\co_\kk$, so is each divisor $\mathcal{Z}_\Phi(x)$.
\end{proposition}

\begin{proof}
Recall the isomorphisms 
\[
E \iso \underline{\Hom}_{\co_\kk}   (\mathfrak{n} , A_0) ,\quad B\iso \underline{\Hom}_{\co_\kk}(\Lambda_0, A_0)
\]
from the proof of Proposition \ref{prop:second moduli iso}.  If we identify $A_0\otimes_{\co_\kk}L_0  \iso B$ using
\[
A_0\otimes_{\co_\kk}L_0 \map{ a \otimes x \mapsto \langle \cdot , x^\vee \rangle a } \underline{\Hom}_{\co_\kk}(\Lambda_0, A_0) \iso B,
\]
we obtain a commutative diagram of $\mathcal{A}_\Phi$-stacks
\[
\xymatrix{
{  E\otimes_{\co_\kk} L_0 } \ar[r] \ar[d]_{ \langle \cdot,x \rangle} &  {  \underline{ \Hom}_{\co_\kk}(\mathfrak{n}, A_0 \otimes_{\co_\kk}L_0 )   } \ar[r]  & {   \underline{\Hom}_{\co_\kk}(\mathfrak{n},B)  = \mathcal{B}_\Phi}   \ar[d]^{ \varrho(x^\vee )} \\
{  E  } \ar[rr]    &  & { \underline{\Hom}_{\co_\kk}(\mathfrak{n}  , A_0) ,}
}
\]
in which all horizontal arrows are isomorphisms.  The first claim follows immediately.

The remaining claims now follow from the cartesian diagram
\[
\xymatrix{
{  \mathcal{Z}_\Phi(x)  } \ar[rr]\ar[d]  &{  }  &   { \mathcal{M}_{(1,0)} }  \ar[d]^{e} \\
{\mathcal{B}_\Phi} \ar[r]^{\iso}  & { E\otimes L_0  }  \ar[r]^{\langle \cdot,x \rangle }   &  { E. }  
}
\]
  The zero section $e: \mathcal{M}_{(1,0)} \hookrightarrow E$  is locally defined by a single  nonzero equation \cite[Lemma 1.2.2]{KM}, and so the same is true of its pullback $\mathcal{Z}_\Phi(x) \hookrightarrow \mathcal{B}_\Phi$.  Composition along the bottom row  is flat by  \cite[Lemma 6.12]{MFK}, and hence so is the top horizontal arrow.  
\end{proof}

\begin{remark}
For those who prefer the language of $1$-motives:  As in the proof of Proposition \ref{prop:boundary uniformization}, there is a universal triple $(A_0,A,\varrho)$ over $\mathcal{C}_\Phi$ in which $A_0$ is an elliptic curve with $\co_\kk$-action and $A$ is a principally polarized $1$-motive with $\co_\kk$-action.   The functor of points of $\mathcal{Z}_\Phi(m)$ assigns to any scheme $S\to \mathcal{C}_\Phi$ the  set 
\[
\mathcal{Z}_\Phi(m)(S) =\{ x\in \Hom_{\co_\kk}(A_{0,S} , A_S ) : \langle x,x\rangle =m\},
\]
where  the positive definite hermitian form $\langle\cdot,\cdot\rangle$ is defined  as in  (\ref{moduli hom hermitian}).  Thus  our special divisors are the exact  analogues  of the special divisors on $\mathcal{S}_\Kra$ defined in \S \ref{ss:special divisors}.
\end{remark}


\subsection{The toroidal compactification}
\label{ss:main toroidal}


We describe the canonical toroidal compactification of the integral models $\mathcal{S}_\Kra \to \mathcal{S}_\Pap$  of   \S \ref{ss:unitary integral models}.

\begin{theorem}\label{thm:toroidal}
Let $\mathcal{S}_\boxempty$ denote either $\mathcal{S}_\Kra$ or $\mathcal{S}_\Pap$.
 There is a  canonical toroidal compactification
$
\mathcal{S}_\boxempty\hookrightarrow \mathcal{S}_\boxempty^*,
$
 flat over $\co_\kk$ of relative dimension $n-1$.  It admits  a stratification
 \[
 \mathcal{S}^*_\boxempty=  \bigsqcup_{  \Phi  }  \mathcal{S}^*_\boxempty(\Phi)
 \]
as a disjoint union of locally closed substacks, indexed by the $K$-equivalence classes of cusp label representatives 
(defined in \S \ref{ss:cusp notation}).

 \begin{enumerate}
 \item
The $\co_\kk$-stack $\mathcal{S}^*_\Kra$ is regular.

\item
The $\co_\kk$-stack $\mathcal{S}^*_\Pap$ is Cohen-Macaulay and normal, with Cohen-Macaulay fibers.
If $n>2$ its fibers are  geometrically  normal.

\item
The open dense substack $\mathcal{S}_\boxempty \subset \mathcal{S}_\boxempty^*$ is the stratum indexed by the unique equivalence class of improper cusp label representatives.  Its complement 
 \[
\partial  \mathcal{S}^*_\boxempty =  \bigsqcup_{  \Phi \mathrm{\ proper} }  \mathcal{S}_\boxempty^*(\Phi)
 \]
 is a smooth divisor, flat over $\co_\kk$.   
 \item
 For each proper  $\Phi$ the stratum $\mathcal{S}^*_\boxempty(\Phi)$ is closed.  
All  components  of  $\mathcal{S}^*_\boxempty(\Phi)_{/\C}$ are  defined over the Hilbert class field $\kk^\mathrm{Hilb}$, and they are permuted simply transitively by $\mathrm{Gal}(\kk^\mathrm{Hilb}/\kk)$.    Moreover,  there is  a canonical identification of $\co_\kk$-stacks
\[
\xymatrix{
{   \Delta_\Phi  \backslash  \mathcal{B}_\Phi  }  \ar@{=}[rr] \ar[d] & &  {\mathcal{S}^*_\boxempty(\Phi)  } \ar[d] \\
{ \Delta_\Phi  \backslash \mathcal{C}_\Phi^*  }
& & { \mathcal{S}^*_\boxempty }
}
\]
such that the two stacks in the bottom row become isomorphic after completion along their common closed substack in the top row.  In other words, there is a canonical isomorphism of formal stacks
\begin{equation}\label{formal boundary iso}
\Delta_\Phi  \backslash (\mathcal{C}_\Phi^* )^\wedge_{\mathcal{B}_\Phi} 
\iso  ( \mathcal{S}^*_\boxempty)^\wedge_{   \mathcal{S}_\boxempty^*(\Phi) }.
 \end{equation}

 \item
 The morphism  $\mathcal{S}_\Kra\to \mathcal{S}_\Pap$ extends uniquely to a stratum preserving morphism of toroidal compactifications.  This extension restricts to an isomorphism
 \begin{equation}\label{compact nonsingular}
\mathcal{S}^*_\Kra  \smallsetminus  \mathrm{Exc} \iso  \mathcal{S}^*_\Pap \smallsetminus \mathrm{Sing},
\end{equation}
compatible with (\ref{formal boundary iso}) for any proper $\Phi$.

 \item
The line bundle  $\bm{\omega}$  on $\mathcal{S}_\Kra$ defined in \S \ref{ss:unitary bundle} admits a unique extension (denoted the same way) to the toroidal compactification in such a way that    (\ref{formal boundary iso}) identifies it with the line bundle $\bm{\omega}_\Phi$ on $\mathcal{C}_\Phi^*$.
A similar statement holds for $\bm{\Omega}_\Kra$, and these two extensions are related by 
\[
\bm{\omega}^2  \iso   \bm{\Omega}_\Kra  \otimes \co(\mathrm{Exc}).
\]

 \item
The line bundle  $\pure_\Pap$  on $\mathcal{S}_\Pap$ defined in \S \ref{ss:unitary bundle} admits a unique extension (denoted the same way) to the toroidal compactification, in such a way that  (\ref{formal boundary iso}) identifies it with  $\bm{\omega}^2_\Phi$.

\item
For any $m> 0$, define  $\mathcal{Z}_\Kra^*(m)$  as the Zariski closure of  $\mathcal{Z}_\Kra(m)$ in  $\mathcal{S}_\Kra^*$.   The  isomorphism (\ref{formal boundary iso}) identifies it   with the Cartier divisor $\mathcal{Z}_\Phi(m)$ on $\mathcal{C}_\Phi^*$.

\item
For any $m> 0$, define   $\mathcal{Y}_\Pap^*(m)$  as the Zariski closure of  $\mathcal{Y}_\Pap(m)$ in $\mathcal{S}_\Pap^*$.   The  isomorphism (\ref{formal boundary iso}) identifies it   with  $2 \mathcal{Z}_\Phi(m)$.  Moreover, the pullback of $\mathcal{Y}_\Pap^*(m)$ to  $\mathcal{S}_\Kra^*$, denoted $ \mathcal{Y}_\Kra^*(m)$, satisfies 
\[
 2 \mathcal{Z}_\Kra^*(m) = \mathcal{Y}_\Kra^*(m) +  \sum_{s\in \pi_0(\mathrm{Sing}) } \# \{ x\in L_s : \langle x,x\rangle=m\} \cdot \mathrm{Exc}_s.
\]

\end{enumerate}
\end{theorem}

\begin{proof}
Briefly, in \cite[\S 2]{Ho2} one finds the construction of a canonical toroidal compactification 
\[
\mathcal{M}^\boxempty_{(n-1,1)} \hookrightarrow \mathcal{M}^{\boxempty,*}_{(n-1,1)}.
\]
Using the open and closed immersion 
\[
\mathcal{S}_\boxempty \hookrightarrow \mathcal{M}_{(1,0)} \times \mathcal{M}^\boxempty_{(n-1,1)},
\] 
the toroidal compactification $\mathcal{S}_\boxempty^*$ is defined  as the Zariski closure of $\mathcal{S}_\boxempty$ in 
$\mathcal{M}_{(1,0)} \times \mathcal{M}^{\boxempty,*}_{(n-1,1)}$.    All of the  claims  follow by examination of the construction  of the compactification, along with Theorem \ref{thm:cartier error}.
\end{proof}

\begin{remark}
If $W$ is anisotropic, so that $(G,\mathcal{D})$ has no proper cusp label representatives, the only new information in the theorem is that $\mathcal{S}_\Pap$ and $\mathcal{S}_\Kra$ are already proper over $\co_\kk$, so that
\[
\mathcal{S}_\Pap=\mathcal{S}^*_\Pap ,\qquad \mathcal{S}_\Kra=\mathcal{S}^*_\Kra.
\]
\end{remark}

\begin{corollary} \label{cor:weak divisor flatness}  
Assume that $n> 2$.
The Cartier divisor $\mathcal{Y}^*_\Pap(m)$ on  $\mathcal{S}^*_\Pap$  is $\co_\kk$-flat, as is the restriction of  $\mathcal{Z}_\Kra^*(m)$ to $ \mathcal{S}^*_\Kra \smallsetminus \mathrm{Exc}$.
\end{corollary}

\begin{proof}
Fix a prime $\mathfrak{p}\subset \co_\kk$, and let $\F_\mathfrak{p}$ be its residue field.
To prove the first claim, it suffices to show that the support of the Cartier divisor  $\mathcal{Y}^*_\Pap(m)$  contains no irreducible components of the  reduction
$\mathcal{S}^*_{\Pap/ \F_\mathfrak{p}}$.

By way of contradiction, suppose $\mathcal{E}_\mathfrak{p} \subset \mathcal{S}^*_{\Pap/ \F_\mathfrak{p}}$ is an irreducible component contained in $\mathcal{Y}^*_\Pap(m)$, and let $\mathcal{E} \subset \mathcal{S}^*_\Pap$ be the connected component containing it.  
Properness of $ \mathcal{S}^*_\Pap$ over $\co_{ \kk,\mathfrak{p} }$ implies that the reduction $\mathcal{E}_{/\F_\mathfrak{p}}$ is connected \cite[Corollary 8.2.18]{FGA}.
The reduction $\mathcal{E}_{/\F_\mathfrak{p}}$  is  normal by Theorem \ref{thm:toroidal} and our assumption that $n>2$, and hence is irreducible.  Thus
\[
\mathcal{E}_\mathfrak{p} = \mathcal{E}_{/\F_\mathfrak{p}}.
\]

Our assumption that $n>2$ also guarantees that $W$ contains a nonzero isotropic vector, from which it follows that the  boundary
  \[
  \partial \mathcal{C} = \mathcal{C} \cap \partial \mathcal{S}^*_\Pap
  \]
is nonempty (one can check this in the complex fiber).

 Proposition \ref{prop:boundary divisors} implies that $\mathcal{Z}_\Phi(m)$ is $\co_\kk$-flat for  every proper cusp label representative $\Phi$, and so it follows from  Theorem \ref{thm:toroidal}  that $\mathcal{Y}^*_\Pap(m)$ is $\co_\kk$-flat when restricted to some \'etale neighborhood $U \to \mathcal{C}$ of $ \partial \mathcal{C}$.
On the other hand, the closed immersion 
\[
 U  _{/\F_\mathfrak{p} }  \iso \mathcal{C}_\mathfrak{p}  \times_{\mathcal{S}_\Pap^*} U     \to  \mathcal{Y}^*_\Pap(m) \times_{\mathcal{S}_\Pap^*} U
\] 
 shows that the divisor   $ \mathcal{Y}^*_\Pap(m) |_U \to U$  contains the special fiber  $U_{/\F_\mathfrak{p}}$, so is not $\co_\kk$-flat.  This  contradiction completes the proof that  $\mathcal{Y}^*_\Pap(m)$ is flat.

As the isomorphism (\ref{compact nonsingular}) identifies $\mathcal{Y}_\Pap^* (m)$ with  $2 \mathcal{Z}_\Kra^*(m)$, it follows that the restriction of 
$\mathcal{Z}_\Kra^*(m)$ to the complement of $\mathrm{Exc}$ is also flat.
\end{proof}


\subsection{Fourier-Jacobi expansions}
\label{ss:abstract FJ}


We now  define Fourier-Jacobi expansions of sections of the line bundle $\bm{\omega}^k$ of weight $k$ modular forms on $\mathcal{S}^*_\Kra$.

Fix a proper cusp label representative $\Phi=(P,g)$.   Suppose $\psi$ is a rational function on $\mathcal{S}^*_\Kra$, regular on an open neighborhood of the closed stratum $\mathcal{S}^*_\Kra(\Phi)$.   Using the isomorphism (\ref{formal boundary iso}) we obtain a  formal function, again denoted $\psi$, on the formal completion 
\[
( \mathcal{C}_\Phi^* )^\wedge_{\mathcal{B}_\Phi} = 
\underline{\mathrm{Spf}}_{  \mathcal{B}_\Phi  }  \Big(  \prod_{ \ell  \ge 0  } \mathcal{L}_\Phi^\ell \Big).
\]
  Tautologically, there is a formal  Fourier-Jacobi expansion
\begin{equation}\label{FJ expansion}
\psi = \sum_{\ell \ge 0} \mathrm{FJ}_{\ell} (\psi) \cdot q^\ell
\end{equation}
with coefficients $\mathrm{FJ}_\ell (\psi) \in H^0( \mathcal{B}_\Phi , \mathcal{L}_\Phi^\ell )$.
In the same way, any rational section $\psi$ of $\bm{\omega}^k$ on $\mathcal{S}^*_\Kra$, regular on an open neighborhood of $\mathcal{S}^*_\Kra(\Phi)$,  admits a Fourier-Jacobi expansion (\ref{FJ expansion}), but now with coefficients 
\[
\mathrm{FJ}_\ell (\psi) \in H^0( \mathcal{B}_\Phi ,  \bm{\omega}_\Phi^k \otimes \mathcal{L}_\Phi^{ \ell } ).
\]

\begin{remark}\label{rem:q}
Let $\pi :\mathcal{C}_\Phi^* \to \mathcal{B}_\Phi$ be the natural map.  The formal symbol $q$ can be understood as follows.  
As  $\mathcal{C}_\Phi^*$ is the total space of the line bundle $\mathcal{L}_\Phi^{-1}$, there is a tautological section
\[
q\in H^0( \mathcal{C}_\Phi^* , \pi^* \mathcal{L}_\Phi^{-1} )
\]
whose divisor is the zero section $\mathcal{B}_\Phi \hookrightarrow \mathcal{C}_\Phi^*$.   Any  $\mathrm{FJ}_\ell  \in H^0( \mathcal{B}_\Phi , \mathcal{L}_\Phi^\ell )$   pulls back to a section of $\pi^*\mathcal{L}_\Phi^{ \ell }$, and so defines a  function $\mathrm{FJ}_\ell  \cdot q^\ell$  on $\mathcal{C}_\Phi^*$.
\end{remark}


\subsection{Explicit coordinates}
\label{ss:explicit boundary}


Once again, let $\Phi=(P,g)$ be a proper cusp label representative.   
The algebraic  theory of  \S \ref{ss:abstract FJ}  realizes the Fourier-Jacobi coefficients of 
\begin{equation}\label{test form}
\psi \in H^0( \mathcal{S}_\Kra^*, \bm{\omega}^k)
\end{equation}
 as sections of line bundles on the stack 
\[
\mathcal{B}_\Phi \iso E\otimes L_0.
\]   
Here $E \to \mathcal{M}_{(1,0)}$ is the universal CM elliptic curve, the tensor product is over $\co_\kk$, and we are using the isomorphism of Proposition \ref{prop:second moduli iso}.  Our goal is to relate this to the classical analytic theory of Fourier-Jacobi expansions  by choosing explicit complex coordinates, so as to identify each coefficient $\mathrm{FJ}_\ell(\psi)$  with a holomorphic function on a complex vector space satisfying a particular transformation law.

The point of this discussion is to allow us, eventually, to show that the  Fourier-Jacobi coefficients of Borcherds products, expressed in the classical way as holomorphic functions satisfying certain transformation laws, have algebraic meaning.  More precisely, the following discussion  will be used to deduce  the algebraic statement of Proposition \ref{prop:algebraic BFJ} from the analytic statement of Proposition \ref{prop:ortho FJ formula}.

Consider the commutative diagram 
\[
\xymatrix{
{  \mathrm{Sh}(Q_\Phi , \mathcal{D}_\Phi)(\C)  } \ar[r]^{\qquad \iso}  \ar[d] & {\mathcal{C}_\Phi(\C) } \ar[r] & {\mathcal{B}_\Phi(\C) } \ar[r] & {\mathcal{A}_\Phi(\C) } \ar[d]^{\iso}     \\
{   \kk^\times \backslash \widehat{\kk}^\times / \widehat{\co}_\kk^\times } \ar[rrr]_{ a \mapsto E^{(a)}  } & & & \mathcal{M}_{(1,0)}(\C) .
}
\]
Here the isomorphisms are those of Propositions \ref{prop:boundary uniformization} and \ref{prop:second moduli iso}, and the vertical arrow on the left is the surjection of Proposition \ref{prop:mixed connected}.  The bottom horizontal arrow is defined as the unique function making the diagram commute.  It is a bijection, and is given explicitly by the following recipe: recalling the $\co_\kk$-module $\mathfrak{n}$ of (\ref{easy graded}), each $a\in \widehat{\kk}^\times$ determines a projective $\co_\kk$-module
\[
\mathfrak{b} =  a \cdot \Hom_{\co_\kk}(   \mathfrak{n} , g \mathfrak{a}_0 ) 
\]
of rank one, and  the elliptic curve $E^{(a)}$ has complex points
\begin{equation}\label{twisty CM curve}
E^{(a)} (\C) =   \mathfrak{b} \backslash  (\mathfrak{b}\otimes_{\co_\kk} \C).
\end{equation}
For each $a\in \widehat{\kk}^\times$ there is a cartesian diagram 
\[
\xymatrix{
{   E^{(a)}  \otimes L_0   }  \ar[r]\ar[d]  &    {  E  \otimes L_0   }  \ar[d]  \\
{  \Spec(\C)  }  \ar[r]^{ E^{(a)} }    &    {  \mathcal{M}_{(1,0)} }  .  
}
\]

Now suppose we have a section $\psi$ as in (\ref{test form}).  Using the isomorphisms $\mathcal{B}_\Phi\iso E\otimes L_0$ and $\bm{\omega_\Phi} \iso \mathfrak{d}\cdot \Lie(E)^{-1}$ of Propositions \ref{prop:second moduli iso} and \ref{prop:mixed modular forms}, we view its Fourier-Jacobi coefficients
\[
\mathrm{FJ}_\ell (\psi) \in H^0( \mathcal{B}_\Phi ,  \bm{\omega}_\Phi^k \otimes \mathcal{L}_\Phi^{ \ell } )
\]
as sections
\[
\mathrm{FJ}_\ell (\psi) \in  H^0\big( E\otimes L_0 ,  \mathfrak{d}^k\cdot \Lie(E)^{-k}\otimes \mathcal{Q}_{E\otimes L_0}^\ell\big),
\]
which we pull back  along the top map in the above diagram to obtain a  section 
\begin{equation}\label{algan preFJ}
\mathrm{FJ}^{(a)}_\ell (\psi) \in  H^0\big( E^{(a)} \otimes L_0 ,  \Lie(E^{(a)})^{-k}\otimes \mathcal{Q}_{E^{(a)}\otimes L_0}^\ell \big).
\end{equation}

\begin{remark}
Recalling that $\mathfrak{d}=\delta\co_\kk$ is the different of $\kk$,  we are using the inclusion $\mathfrak{d}^k \subset \kk \subset \C$ to identify 
\[
\mathfrak{d}^k\cdot\Lie(E^{(a)})^{-k} \iso \Lie(E^{(a)})^{-k}.
\] 
In particular, this isomorphism is \emph{not} multiplication by $\delta^{-k}$.
\end{remark}

The explicit coordinates we will use to express (\ref{algan preFJ}) as a holomorphic function arise from a choice of Witt decomposition of the hermitian space $V = \Hom_\kk(W_0,W)$.  The following lemma will allow us to choose this decomposition in a particularly nice way.

\begin{lemma}\label{lem:mixed section}
The homomorphism $\nu_\Phi$ of (\ref{Q fiber}) admits a section 
\[
\xymatrix{
{  Q_\Phi }   \ar[rr]_{\nu_\Phi}  & &  {  \mathrm{Res}_{\kk/\Q}\mathbb{G}_m   }  \ar@/_1pc/[ll]_s.
}
\]
This section may be chosen so that $s(\widehat{\co}_\kk^\times) \subset K_\Phi$, and such a choice determines a decomposition
\begin{equation}\label{mixed component}
\bigsqcup_{ a \in \kk^\times \backslash \widehat{\kk}^\times / \widehat{\co}_\kk^\times }  ( Q_\Phi(\Q) \cap s(a)  K_\Phi s(a)^{-1}) \backslash \mathcal{D}_\Phi \iso \mathrm{Sh}(Q_\Phi , \mathcal{D}_\Phi)(\C),
\end{equation}
where the isomorphism is $z\mapsto (z,s(a))$ on the copy of $\mathcal{D}_\Phi$ indexed by $a$.
\end{lemma}

\begin{proof}
Fix an isomorphism of hermitian $\co_\kk$-modules
\[
g\mathfrak{a}_0 \oplus g\mathfrak{a} \iso g\mathfrak{a}_0 \oplus \mathrm{gr}_{-2}(ga) \oplus \mathrm{gr}_{-1}(ga) \oplus\mathrm{gr}_0(ga) 
\]
as in Remark \ref{rem:determined by grade}.  After tensoring with $\Q$, we let $\kk^\times$ act on the right hand side by 
$a\mapsto ( a, \mathrm{Nm}(a) , a, 1)$.  This defines a morphism $\kk^\times \to G(\Q)$, which, using (\ref{Q fiber}), is easily seen to take values in the subgroup $Q_\Phi(\Q)$.  This defines the desired section $s$, and  the decomposition (\ref{mixed component}) is immediate from Proposition \ref{prop:mixed connected}.
\end{proof}

Fix a section $s$ as in Lemma \ref{lem:mixed section}.
Recall from \S \ref{ss:cusp notation}  the weight filtration $\mathrm{wt}_i V \subset V$ whose graded pieces
\begin{align*}
\mathrm{gr}_{-1}  V & =\Hom_\kk( W_0,\mathrm{gr}_{-2}W) \\
 \mathrm{gr}_{0} V& =\Hom_\kk( W_0,\mathrm{gr}_{-1}W) \\
 \mathrm{gr}_{1}V & =\Hom_\kk( W_0,\mathrm{gr}_{0}W)
\end{align*}
have $\kk$-dimensions $1$, $n-2$, and $1$, respectively.   Recalling  (\ref{Q fiber}), which describes the action of $Q_\Phi$ on the graded pieces of $V$,
the section $s$ determines a splitting  $V = V_{-1}\oplus V_0 \oplus V_1$ of the weight filtration by 
\begin{align*}
V_{-1} &= \{ v\in V :  \forall\,  a \in \kk^\times, s(a)  v = \overline{a} v  \} \\
V_0 &= \{ v\in V :  \forall\,  a \in \kk^\times, s(a)  v =  v  \} \\
V_1 &= \{ v\in V :  \forall\,  a \in \kk^\times, s(a)  v = a^{-1} v  \} .
\end{align*}
The summands  $V_{-1}$ and $V_1$ are  isotropic $\kk$-lines, and  $V_0$ is the  orthogonal complement of $V_{-1} + V_1$ with respect to the hermitian form on $V$.  In particular, the restriction of the hermitian form to $V_0 \subset V$ is positive definite.

Fix an $a\in \widehat{\kk}^\times$ and define  an $\co_\kk$-lattice
\[
L  = \Hom_{\co_\kk}( s(a) g\mathfrak{a}_0  ,  s(a) g\mathfrak{a} ) \subset V.
\]
Using the assumption $s(\widehat{\co}_\kk^\times) \subset K_\Phi$, we obtain a decomposition 
\[
L  =L_{-1}\oplus L_0 \oplus L_1 
\]
with $L_i = L \cap V_i$. 
The images of the lattices $L_i$ in the graded pieces $\mathrm{gr}_i V$ are given by
\begin{align*}
L_{-1} &=   \overline{a}  \cdot  \Hom_{\co_\kk} (  g \mathfrak{a}_0 , \mathrm{gr}_{-2}( g \mathfrak{a}) ) \\
L_0 & =   \Hom_{\co_\kk} (  g \mathfrak{a}_0 , \mathrm{gr}_{-1}(  g \mathfrak{a}) )   \\ 
L_1 & =  a^{-1} \cdot  \Hom_{\co_\kk} (  g \mathfrak{a}_0 , \mathrm{gr}_{0}(  g \mathfrak{a}) ). 
\end{align*}
In particular,  $L_0$ is independent of $a$ and agrees with  (\ref{boundary herm}).

Choose  a $\Z$-basis $\eee_{-1} , \fff_{-1} \in L_{-1}$, and let $\eee_1, \fff_1 \in \mathfrak{d}^{-1}L_1$ be the dual basis with respect to the (perfect) $\Z$-bilinear pairing
\[
 [\, \cdot\, , \, \cdot\, ] : L_{-1} \times \mathfrak{d}^{-1} L_1 \to \Z
\]
obtained by restricting (\ref{hom quadratic}).   This basis may be chosen so that 
\begin{equation}\label{def-Lpm}
\begin{array}{lcl}
L_{-1} = \Z \eee_{-1}+\Z \fff_{-1}   &  &     \mathfrak{d}^{-1} L_{-1} = \Z \eee_{-1}+D^{-1}\Z \fff_{-1} \\ \\
  L_1 = \Z \eee_1 + D\Z \fff_1    &  &   \mathfrak{d}^{-1}L_1=\Z \eee_1 + \Z \fff_1. 
  \end{array}
\end{equation}

As $\epsilon V_1(\C) \subset V_1(\C)$ is a line, there is a unique $\tau\in \C$ satisfying
\begin{equation}\label{the tau}
\tau\eee_1 + \fff_1 \in \epsilon V_1(\C).
\end{equation}
After  possibly replacing both $\eee_1$ and $\eee_{-1}$ by their negatives, we may assume that $\mathrm{Im}(\tau)>0$.

\begin{proposition}\label{prop:explicit FJ}
The $\Z$-lattice $\mathfrak{b}  = \Z\tau+\Z$ is contained in $\kk$, and is a fractional $\co_\kk$-ideal.  The  elliptic curve 
\begin{equation}\label{elliptic param}
E^{(a)}(\C) =  \mathfrak{b} \backslash \C
\end{equation}
is isomorphic to  (\ref{twisty CM curve}), and  there is an $\co_\kk$-linear isomorphism of complex abelian varieties
\begin{equation}\label{KS uniform}
E^{(a)}(\C) \otimes L_0 \iso  \mathfrak{b} L_0 \backslash V_0(\R)  .
\end{equation}
Under this isomorphism the inverse of  the line bundle (\ref{Q def}) has the form
\begin{equation}\label{bundle param}
 \mathcal{Q}^{-1}_{E^{(a)}(\C) \otimes L_0 }  \iso   \mathfrak{b} L_0\backslash ( V_0(\R) \times \C ) ,
\end{equation}
where the action of $y_0 \in \mathfrak{b} L_0$ on $V_0(\R) \times \C$  is  
\[
 y_0 \cdot (w_0, q )  
  = \big(   w_0 + \epsilon y_0, q \cdot 
  e^{ \pi i  \frac{ \langle y_0 , y_0 \rangle  }{  \mathrm{N}(\mathfrak{b})  } }
 e^{  -  \pi   \frac{  \langle w_0,  y_0\rangle  }{ \mathrm{Im}(\tau)  }    -  \pi    \frac{   \langle y_0 , y_0 \rangle }{  2 \mathrm{Im}(\tau)} } \big). 
\]
\end{proposition}

\begin{proof}
Consider the $\Q$-linear map
 \begin{equation}\label{prewee conj}
 V_{-1} \map{ \alpha \mathrm{e}_{-1} + \beta \mathrm{f}_{-1} \mapsto \alpha \tau +\beta}  \C.
 \end{equation}
 Its $\C$-linear extension $V_{-1}(\C) \to \C$ kills the vector
 $
 \mathrm{e}_{-1} -\tau \mathrm{f}_{-1}  \in \epsilon V_{-1}(\C),
 $
and hence   factors through an isomorphism
 $
V_{-1}(\C) / \epsilon V_{-1}(\C) \iso \C.
$
 This implies that  (\ref{prewee conj}) is $\kk$-conjugate linear.  As this map identifies 
 $L_{-1}\iso \mathfrak{b}$, we find that  the $\Z$-lattice $\mathfrak{b} \subset\C$  is $\co_\kk$-stable.  From   $1\in \mathfrak{b}$ we then deduce that $\mathfrak{b} \subset\kk$, and is   a fractional $\co_\kk$-ideal.   Moreover, we have just shown that
 \begin{equation}\label{wee conj}
 L_{-1} \map{ \alpha \mathrm{e}_{-1} + \beta \mathrm{f}_{-1} \mapsto \alpha \tau +\beta}  \mathfrak{b}.
 \end{equation}
is an $\co_\kk$-conjugate linear isomorphism.

Exactly as in (\ref{flippy map}), the self-dual hermitian forms on $g\mathfrak{a}_0$ and $g\mathfrak{a}$ induce an $\co_\kk$-conjugate-linear isomorphism
\[
\Hom_{\co_\kk} (  g \mathfrak{a}_0 , \mathrm{gr}_{-2}( g \mathfrak{a}) )  \iso \Hom_{\co_\kk} (  \mathrm{gr}_0( g \mathfrak{a}), g \mathfrak{a}_0  ),
\]
and hence determine an $\co_\kk$-conjugate-linear isomorphism
\begin{align*}
L_{-1}  & = \overline{a}  \cdot \Hom_{\co_\kk} (  g \mathfrak{a}_0 , \mathrm{gr}_{-2}( g \mathfrak{a}) )  \\
& \iso 
a\cdot  \Hom_{\co_\kk} (  \mathrm{gr}_0( g \mathfrak{a}), g \mathfrak{a}_0  )\\
& =
a\cdot \Hom_{\co_\kk} ( \mathfrak{n}  , g \mathfrak{a}_0  ) .
\end{align*}

The composition 
\[
 a\cdot \Hom_{\co_\kk} ( \mathfrak{n}  , g \mathfrak{a}_0  ) \iso L_{-1}  \map{(\ref{wee conj})}  \mathfrak{b}
\]
is  an $\co_\kk$-linear isomorphism, which identifies the fractional ideal $\mathfrak{b}$ with the projective $\co_\kk$-module used in the definition of (\ref{twisty CM curve}). In particular it   identifies the elliptic curves   (\ref{twisty CM curve}) and (\ref{elliptic param}), and   also identifies
\[
E^{(a)} (\C) \otimes  L_0 =  (  \mathfrak{b}  \backslash \C ) \otimes  L_0 \iso 
 (\mathfrak{b} \otimes L_0 ) \backslash (\C\otimes  L_0 ) .
\]
Here, and throughout the remainder of the proof, all tensor products are over $\co_\kk$.
Identifying  $\C\otimes  L_0  \iso V_0(\R)$ proves (\ref{KS uniform}).

It remains to explain the isomorphism (\ref{bundle param}).  First consider the Poincar\'e bundle on the product
\[
E^{(a)} (\C) \times E^{(a)} (\C) \iso (\mathfrak{b} \times \mathfrak{b}) \backslash ( \C \times \C).
\]
Using classical formulas, the space of this line bundle can be identified with the quotient 
\[
\mathcal{P}_{E^{(a)}(\C) } =  (\mathfrak{b} \times \mathfrak{b}) \backslash ( \C \times \C \times \C),
\]
where the action is given by 
\[
(b_1,b_2) \cdot (z_1,z_2, q ) =
\left(
 z_1+b_1, z_2+b_2 ,   q\cdot e^{   \pi H_\tau(z_1,b_2) + \pi H_\tau(z_2,b_1) + \pi  H_\tau(b_1,b_2)  } 
 \right),
\]
and we have set $H_\tau( w,z) = w\overline{z}/\mathrm{Im}(\tau)$ for complex numbers $w$ and $z$.

  Directly from the definition, the line bundle  (\ref{Q def}) on
\[ 
E^{(a)} (\C) \otimes L_0 \iso  ( \mathfrak{b} \otimes  L_0 ) \backslash ( \C \otimes  L_0 )
\]
is given by 
\[
\mathcal{Q}_{E^{(a)} (\C) \otimes L_0} \iso   ( \mathfrak{b} \otimes L_0 ) \backslash \big(  ( \C \otimes  L_0 ) \times \C  \big) ,
\]
where the action of $\mathfrak{b} \otimes  L_0$  on $( \C \otimes  L_0 ) \times \C$ is given as follows:
Choose any set $x_1,\ldots, x_n \in L_0$ of $\co_\kk$-module generators, and extend the $\co_\kk$-hermitian form on $L_0$ to a $\C$-hermitian form on $\C\otimes L_0$.   If
\[
y_0 = \sum_i b_i \otimes x_i \in \mathfrak{b} \otimes  L_0
\]
and 
\[
w_0 =  \sum_i  z_i \otimes x_i  \in \C \otimes   L_0
\]
then 
\[
y_0  \cdot (w_0   ,  q )  = ( w_0 + y_0 , q \cdot  e^{ \pi X +\pi Y} ),
\]
where the factors $X$ and $Y$ are  
\begin{align*}
X& =    \sum_{i<j}  \Big( H_\tau ( \langle x_i,x_j\rangle z_i , b_j )   
+ H_\tau (  z_j ,  \langle x_i , x_j\rangle b_i)   
+ H_\tau(\langle x_i,x_j\rangle b_i , b_j )  \Big) \\
& =  \frac{1}{\mathrm{Im}(\tau) } \sum_{i\neq j}  \langle  z_i  \otimes x_i, b_j  \otimes x_j\rangle   
+  \frac{1}{\mathrm{Im}(\tau) } \sum_{i < j}    \langle b_i   \otimes x_i, b_j  \otimes x_j\rangle 
\end{align*}
and, recalling $\gamma = (1+\delta)/2$, 
\begin{align*}
Y&= \sum_i  \Big( H_\tau (  \gamma \langle x_i , x_i \rangle z_i , b_i)   
+ H_\tau (  z_i ,  \gamma \langle x_i , x_i \rangle b_i)   
+ H_\tau( \gamma \langle x_i,x_i \rangle b_i , b_i )  \Big) \\
& =   \frac{1}{\mathrm{Im}(\tau) } \sum_i   \langle z_i \otimes x_i, b_i \otimes x_i \rangle  
  +  \frac{1}{\mathrm{Im}(\tau) } \sum_i  \gamma \langle b_i \otimes x_i, b_i  \otimes x_i \rangle    .
\end{align*}

For elements $y_1,y_2\in \mathfrak{b} \otimes  L_0$, we abbreviate
\[
\alpha ( y_1, y_1 ) = \frac{\langle y_1,y_2\rangle}{\delta \mathrm{N}(\mathfrak{b}) } - \frac{\langle y_2,y_1\rangle}{\delta \mathrm{N}(\mathfrak{b}) }  \in \Z.
\]
Using $2 i \mathrm{Im}(\tau) = \delta \mathrm{N}(\mathfrak{b})$, some elementary calculations show that 
\begin{eqnarray*}\lefteqn{
\pi X+\pi Y  -   \frac{ \pi \langle w_0,y_0\rangle }{\mathrm{Im}(\tau) }  } \\
&= &
 \frac{2 \pi i }{\delta \mathrm{N}(\mathfrak{b} )} \sum_{i < j}    \langle b_i  \otimes x_i, b_j \otimes x_j\rangle
+ \frac{2 \pi  i }{\delta \mathrm{N}(\mathfrak{b} ) } \sum_i   \langle  \gamma b_i \otimes x_i, b_i \otimes x_i \rangle \\
   &= &
  \frac{  \pi  }{ 2  \mathrm{Im}(\tau) } \sum_{i , j}    \langle  b_i   \otimes x_i, b_j   \otimes   x_j\rangle  
   - \frac{ \pi i }{ \mathrm{N}(\mathfrak{b} )} \sum_{i , j}    \langle   b_i   \otimes   x_i, b_j   \otimes   x_j\rangle  \\
 & &
 +  2 \pi i \sum_{i < j}  \alpha(  \gamma b_i  \otimes    x_i, b_j   \otimes   x_j ) 
 + \frac{  2\pi  i }{ \mathrm{N}(\mathfrak{b} ) } \sum_i   \langle b_i  \otimes   x_i, b_i   \otimes  x_i \rangle .
\end{eqnarray*}
All terms in the final line lie in $2\pi i \Z$, and so 
\[
e^{\pi X + \pi Y} = e^{    \frac{ \pi \langle w_0,y_0\rangle }{\mathrm{Im}(\tau) }  } 
  e^{   \frac{  \pi   \langle  y_0 ,  y_0 \rangle  }{ 2  \mathrm{Im}(\tau) }    }
   e^{  - \frac{ \pi i  \langle   y_0 ,  y_0 \rangle }{ \mathrm{N}(\mathfrak{b} )}     } .
\]
The relation  (\ref{bundle param}) follows immediately.
\end{proof}

Proposition  \ref{prop:explicit FJ} allows us to express Fourier-Jacobi coefficients explicitly as  functions on $V_0(\R)$ satisfying certain transformation laws.   Suppose we start with a global section 
\begin{equation}\label{general section}
\psi \in H^0\big( \mathcal{S}^*_{\Kra /\C}   ,   \bm{\omega}^k \big).
\end{equation}
 For each $a\in \widehat{\kk}^\times$ and $\ell\ge 0$ we have the algebraically defined Fourier-Jacobi coefficient  
\begin{equation}\label{algan FJ}
\mathrm{FJ}_\ell^{(a)}(\psi) \in H^0 \big(    E^{(a)}   \otimes L_0   ,   \mathcal{Q}_{E^{(a)} \otimes L_0}^\ell   \big)
\end{equation}
of (\ref{algan preFJ}), where we have trivialized $\Lie(E^{(a)})$ using (\ref{elliptic param}).   
The isomorphism (\ref{bundle param}) now identifies (\ref{algan FJ}) with  a function on $V_0(\R)$  satisfying the transformation law
\begin{equation}\label{FJ-trans}
\mathrm{FJ}^{(a)}_\ell(\psi)(w_0 + y_0) = 
\mathrm{FJ}^{(a)}_\ell(\psi)(w_0 ) \cdot 
e^{  i \pi \ell   \frac{ \langle y_0 ,y_0 \rangle  }{  \mathrm{N}(\mathfrak{b})  } }
 e^{    \pi \ell  \frac{    \langle w_0,  y_0\rangle  }{ \mathrm{Im}(\tau)  }    +  \pi  \ell  \frac{ \langle y_0 , y_0 \rangle }{ 2  \mathrm{Im}(\tau)} }
 \end{equation}
for all  $y_0\in  \mathfrak{b} L_0$.

\begin{remark}\label{rem:switch to unitary}
If we  use  the isomorphism $\mathrm{pr}_\epsilon: V_0(\R) \iso   \epsilon V_0(\C)$  of (\ref{idem proj}) to view (\ref{algan FJ}) as a function of $w_0\in \epsilon V_0(\C)$,   the transformation law can be expressed in terms of the $\C$-bilinear form  $[\cdot,\cdot]$  as
\[
\mathrm{FJ}^{(a)}_\ell(\psi)(w_0 +  \mathrm{pr}_\epsilon (y_0) ) 
= \mathrm{FJ}^{(a)}_\ell(\psi)(w_0 ) \cdot   
e^{ i \pi \ell   \frac{ Q( y_0) }{  \mathrm{N}(\mathfrak{b})  } }
 e^{   \pi \ell   \frac{   [w_0,  y_0]  }{ \mathrm{Im}(\tau)  }    +  \pi  \ell   \frac{ Q( y_0)}{  2 \mathrm{Im}(\tau)} }
\]
for all  $y_0\in  \mathfrak{b} L_0$.  This uses the (slightly confusing) commutativity of 
\[
\xymatrix{
  V_0(\R)   \ar[r]^{\mathrm{pr}_\epsilon} \ar[d]_{ \langle \cdot, y_0 \rangle } &   \epsilon V_0(\C) \ar[r]^{\subset}   &    V_0(\C) \ar[d]^{  [ \cdot, y_0 ]  }      \\ 
 \kk\otimes_\Q \R \ar@{=}[rr] & & \C .
}
\]
\end{remark}

In order to give another interpretation of our explicit coordinates,  let $N_\Phi \subset Q_\Phi$ be the unipotent radical, and  let $U_\Phi \subset N_\Phi$ be its center.     The unipotent radical may be characterized as the kernel of the morphism
$\nu_\Phi$  of (\ref{Q fiber}), or, equivalently, as the largest subgroup acting trivially on all graded pieces $\mathrm{gr}_i V$.

\begin{proposition}\label{prop:supersiegel}
There is a commutative diagram
\begin{equation}\label{coord diagram}
\xymatrix{
{ ( U_\Phi(\Q) \cap s(a) K_\Phi s(a)^{-1}) \backslash \mathcal{D}_\Phi    }  \ar[rr]^{  \qquad z \mapsto ( w_0 , q )   } \ar[d]   & &   { \epsilon V_0(\C) \times \C^\times       }   \ar[d]  \\
{    ( N_\Phi(\Q) \cap s(a) K_\Phi s(a)^{-1}) \backslash \mathcal{D}_\Phi     }  \ar[rr]   & &    {    \mathfrak{b} L_0\backslash ( \epsilon V_0(\C) \times \C^\times )     }    
}
\end{equation}
in which the  horizontal arrows are  holomorphic isomorphisms, and the action of $\mathfrak{b} L_0$ on  
\[
\epsilon V_0(\C) \times \C^\times    \iso V_0(\R) \times \C^\times    
\]
 is the same  as in Proposition \ref{prop:explicit FJ}.
\end{proposition}

\begin{proof}
 Recall from Remark \ref{rem:mixed to so} the isomorphism
\[
\mathcal{D}_\Phi  \iso  \big\{ w \in \epsilon V(\C) :   \epsilon V(\C) =  \epsilon V_{-1}(\C) \oplus \epsilon V_0(\C)  \oplus \C w    \big\} / \C^\times .
\]
As  $\epsilon V(\C)$ is totally isotropic with respect to  $[\cdot,\cdot]$, a simple calculation shows that every line $w \in \mathcal{D}_\Phi$ has a unique representative of the form 
\[
 - \xi ( \eee_{-1} -  \tau    \fff_{-1} )  + w_0 +  ( \tau \eee_1 + \fff_1 ) \in \epsilon V_{-1}(\C) \oplus \epsilon V_0(\C)  \oplus \epsilon V_1(\C)
\]
with  $\xi\in \C$ and $w_0 \in   \epsilon V_0(\C) = V_0(\R) $.   These coordinates   define  an isomorphism of complex manifolds
\begin{equation}\label{siegel coords}
\mathcal{D}_\Phi \map{ w \mapsto (  w_0  ,   \xi  ) }   \epsilon V_0(\C) \times \C.
\end{equation}

The action of $G$ on $V$ restricts to a faithful action of  $N_\Phi$, allowing us to express elements of $N_\Phi(\Q) $ as matrices
\[
n(\phi , \phi^* , u) = \left( \begin{matrix}
1 & \phi^* & u + \frac{1}{2} \phi^*\circ \phi \\
& 1 & \phi \\
& & 1
\end{matrix}\right)   \in N_\Phi(\Q)
\]
for maps
\[
\phi   \in \Hom_\kk ( V_1 , V_0 ) ,\quad
\phi^*  \in \Hom_\kk(V_0, V_{-1}) ,\quad
u    \in \Hom_\kk( V_1, V_{-1} ) 
\]
satisfying the relations
\begin{align*}
0 & = \langle \phi(x_1) , y_0 \rangle + \langle x_1 , \phi^*(y_0) \rangle  \\
0 & = \langle u(x_1), y_1 \rangle + \langle x_1,u(y_1) \rangle
\end{align*}
for  $x_i , y_i \in V_i$. The subgroup $U_\Phi(\Q)$ is defined by $\phi=0=\phi^*$.

The group $U_\Phi(\Q) \cap s(a) K_\Phi s(a)^{-1}$ is cyclic, and   generated by the element  $n(0,0,u)$ defined by 
\[
u(x_1) =  \frac{ \langle x_1 ,a  \rangle }{  [  L_{-1} : \co_\kk a ]  } \cdot \delta a
\]
for any  $a\in L_{-1}$.  In terms of the bilinear form, this can be rewritten as
\[
u(x_1) =  -[ x_1,  \mathrm{f}_{-1} ] \mathrm{e}_{-1} + [  x_1  ,  \mathrm{e}_{-1} ] \mathrm{f}_{-1} .
\]
  In the coordinates of (\ref{siegel coords}), the action of $n(0,0,u)$ on   $\mathcal{D}_\Phi$ becomes  
\[
(w_0,\xi) \mapsto  (w_0 , \xi+1),
\]
and  setting $q=e^{2\pi i \xi}$ defines the top horizontal isomorphism in (\ref{coord diagram}).

Let $\overline{V}_{-1}=V_{-1}$ with its conjugate action of $\kk$.  There are group isomorphisms
\begin{equation}\label{mixed domain action}
N_\Phi(\Q)/U_\Phi(\Q)  \iso   \overline{V}_{-1} \otimes_{\kk} V_0 \iso  V_0.
\end{equation}
 The first sends
\[
n(\phi , \phi^* ,u) \mapsto y_{-1}\otimes y_0,
\]
where $y_{-1}$ and $y_0$ are defined by the relation  $\phi(x_1) = \langle x_1 , y_{-1} \rangle  \cdot y_0$, and the second sends 
 \[
 (\alpha \mathrm{e}_{-1}+\beta \mathrm{f}_{-1}) \otimes y_0 \mapsto ( \alpha \tau +\beta) y_0.
 \]
 Compare with  (\ref{wee conj}).

A slightly tedious calculation shows that (\ref{mixed domain action}) identifies
 \[
  ( N_\Phi(\Q) \cap s(a) K_\Phi s(a)^{-1})  /   ( U_\Phi(\Q) \cap s(a) K_\Phi s(a)^{-1})   \iso \mathfrak{b} L_0,
 \]
 defining the bottom horizontal arrow in (\ref{coord diagram}), and that the resulting action of $\mathfrak{b} L_0$ on $\epsilon V_0(\C) \times \C^\times$ agrees with the one defined in Proposition \ref{prop:explicit FJ}.    We leave this to the reader.
\end{proof}

Any section (\ref{general section}) may now be   pulled  back via
\[
  ( N_\Phi(\Q) \cap s(a) K_\Phi s(a)^{-1}) \backslash \mathcal{D} \map{  z\mapsto ( z, s(a)g ) }   \mathrm{Sh}(G , \mathcal{D} )(\C) 
\]
to define a holomorphic  section  of  $(\bm{\omega}^{an})^k$, the $k^\mathrm{th}$ power of the tautological bundle  on  
\[
 \mathcal{D}  \iso  \big\{ w \in \epsilon V(\C) : [w,\overline{w}] <0   \big\} / \C^\times.
\]
The tautological bundle admits a natural $N_\Phi(\R)$-equivariant trivialization:  any  line $w\in \mathcal{D}$ must satisfy $[w, \mathrm{f}_{-1}]\neq 0$, yielding an isomorphism
\[
[\,\cdot\, ,\mathrm{f}_{-1}]  : \bm{\omega}^{an} \iso \co_\mathcal{D}.
\]
 This trivialization allows us to identify $\psi$ with a holomorphic function on $\mathcal{D} \subset \mathcal{D}_\Phi$, which then has an \emph{analytic} Fourier-Jacobi expansion 
\begin{equation}\label{analytic FJ}
\psi = \sum_{\ell} \mathrm{FJ}^{(a)}_\ell(\psi)(w_0) \cdot q^\ell
\end{equation}
defined using the coordinates of Proposition \ref{prop:supersiegel}.  The fact that the coefficients here agree with (\ref{algan FJ}) is a special case of the main results of   \cite{Lan12}, which compare algebraic and analytic Fourier-Jacobi coefficients on general PEL-type Shimura varieties.


\section{Classical modular forms}
\label{s:modular forms}


 Throughout \S \ref{s:modular forms} we let $D$ be any odd squarefree positive integer, and abbreviate $\Gamma=\Gamma_0(D)$.    Let $k$ be any positive integer.  


\subsection{Weakly holomorphic forms}


  The positive divisors  of $D$ are in bijection with the cusps of the complex modular curve $X_0(D)(\C)$,  by sending $r\mid D$ to 
\[
\infty_r = \frac{r}{D} \in  \Gamma  \backslash \mathbb{P}^1(\Q) .
\]
Note that $r=1$ corresponds to the usual cusp at infinity, and so we sometimes abbreviate $\infty=\infty_1$.

Fix a positive divisor $r\mid D$, set $s=D/r$ and choose
\[
R_r =\left( \begin{matrix} \alpha & \beta \\ s\gamma & r \delta \end{matrix} \right) \in  \Gamma_0(s)
\]
with $\alpha,\beta,\gamma,\delta \in\Z$. The corresponding Aktin-Lehner operator is defined by the matrix
\[
W_r =\begin{pmatrix}  r \alpha & \beta \\ D \gamma & r \delta \end{pmatrix} =
R_r \begin{pmatrix}  r   \\  & 1 \end{pmatrix}.
\]
The matrix $W_r$ normalizes  $\Gamma$, and so acts on the cusps of $X_0(D)(\C)$.  This action satisfies $W_r  \cdot \infty =\infty_r$.

Let $\chi$ be a quadratic Dirichlet character modulo $D$, and let 
\[
\chi=\chi_r\cdot \chi_{s}
\]   
be the unique factorization as a product of quadratic Dirichlet characters  $\chi_r$ and $\chi_{s}$  modulo $r$ and $s$, respectively.  Write 
\[
M_k(D,\chi) \subset M^!_k(D,\chi)
\]
for the spaces of holomorphic modular forms and  weakly holomorphic modular forms  of weight $k$, level $\Gamma$, and character $\chi$.
We assume that $\chi(-1) = (-1)^k$, since otherwise $M^!_k(D,\chi)=0$.  
 
Denote by  $\GL_2^+(\R)\subset \GL_2(\R)$ the subgroup of elements with positive determinant. It  acts on functions  on the upper half plane  by the usual weight $k$ slash operator
\[
(f\mid_k \gamma)(\tau) = \det(\gamma)^{k/2} (c\tau + d)^{-k} f(\gamma \tau),\quad \gamma=\abcd\in \GL_2^+(\R),
\]
  and  $f\mapsto f\mid_k W_r$ defines an endomorphism of  $M^!_k(D,\chi)$ satisfying
\[
f\mid_k W_r^2 = \chi_r(-1)\chi_{s}(r)  \cdot f.
\]
 In particular, $W_r$ is an involution when $\chi$ is trivial.

Any weakly holomorphic modular form
\[
f(\tau)= \sum_{m \gg -\infty} c(m) \cdot  q^m \in M^!_k(D,\chi)
\]
determines another weakly holomorphic modular form 
\[
 \chi_r(\beta)\chi_{s}(\alpha) \cdot ( f\mid_k W_r)\in M^!_k(D,\chi),
\]
which is easily seen to be independent of the choice of  parameters $\alpha,\beta,\gamma,\delta$ in the definition of $W_r$.  This second modular form has a $q$-expansion at $\infty$,  denoted
\begin{equation}\label{other cusps}
 \chi_r(\beta)\chi_{s}(\alpha) \cdot ( f\mid_k W_r) = \sum_{m \gg -\infty} c_r(m) \cdot q^m.
\end{equation}

\begin{definition}\label{def:constant at cusp}
We call (\ref{other cusps}) the \emph{$q$-expansion of $f$ at $\infty_r$}.  
Of special interest is $c_r(0)$, the \emph{constant term of $f$ at $\infty_r$.}
\end{definition}

\begin{remark}
If $\chi$ is nontrivial,  the coefficients of (\ref{other cusps}) need not lie in the subfield of $\C$ generated by the Fourier coefficients of $f$.
\end{remark}


\subsection{Eisenstein series and the modularity criterion}
\label{ss:eisenstein}


Fix an integer $k\ge 2$.   Denote by
\[
M^{!,\infty} _{2-k} (D,\chi) \subset M^!_{2-k} (D,\chi)
\]
the subspace of weakly holomorphic forms that are holomorphic outside the cusp  $\infty$, and by 
\[
M^\infty_{k}(D,\chi)\subset M_{k}(D,\chi)
\]
 the subspace of holomorphic modular forms that vanish at all cusps different from $\infty$.

 If $k>2$ there is a decomposition
\[
M^\infty_{k}(D,\chi) = \C E \oplus S_k(D,\chi),
\] 
where $E$ is the Eisenstein series
\[
E  =\sum_{\gamma\in \Gamma_\infty\backslash \Gamma} \chi(d) \cdot ( 1\mid_k \gamma ) \in M_k(D,\chi).
\]
Here   $\Gamma_\infty\subset\Gamma$ is the stabilizer of $\infty\in \mathbb{P}^1(\Q)$, and  
$
\gamma = \left(\begin{smallmatrix} a & b \\ c & d \end{smallmatrix}\right) \in \Gamma.
$

We also define the (normalized) Eisenstein series for the cusp $\infty_r$ by
\[
E_r = \chi_r(-\beta) \chi_{s}(\alpha r) \cdot ( E \mid_k W_r ) \in M_k(D,\chi).
\]
It is independent of the choice of the parameters in $W_r$, and we denote by 
\[
E_r(\tau) = \sum_{m \ge 0} e_r(m) \cdot q^m
\]
its $q$-expansion at $\infty$. 

\begin{remark}\label{rem:eisenstein constant}
Our notation for the $q$-expansion of $E_r$ is slightly at odds with (\ref{other cusps}), as the $q$-expansion of $E$ at $\infty_r$ is not $\sum e_r(m) q^m$.  Instead,   the $q$-expansion of $E$ at $\infty_r$ is  $\chi_r(-1)\chi_s(r) \sum e_r(m)q^m$, while the $q$-expansion of $E_r$ at $\infty_r$ is $\sum e_1(m) q^m$.   In any case, what matters most is that 
\[
\mbox{constant term of $E_r$ at $\infty_s$} = \begin{cases}
1& \mbox{if $s=r$} \\
0 &\mbox{otherwise.}
\end{cases}
\]
\end{remark}


 The constant terms of weakly holomorphic modular forms in $M_{2-k}^{!,\infty}(D,\chi)$ can be computed using the above Eisenstein series.  

\begin{proposition}
\label{prop:distribute cusps}
Assume $k>2$.  Suppose $r\mid D$ and 
\[
f(\tau) = \sum_{ m\gg -\infty} c (m) \cdot q^m    \in M_{2-k}^{!,\infty}(D,\chi).
\] 
The constant term of $f$ at the cusp $\infty_r$, in the sense of Definition \ref{def:constant at cusp}, satisfies
\[
c_r(0)  + \sum_{m > 0}  c(-m)  e_r(m) =0.
\]
\end{proposition}

\begin{proof}
The meromorphic differential form $f(\tau) E_r(\tau)\, d\tau$ on $X_0(D)(\C)$ is holomorphic away from the  cusps $\infty$ and $\infty_r$.  Summing its residues at these cusps gives the desired equality.
\end{proof}

\begin{theorem}[Modularity criterion]
\label{thm:modularity criterion}
Suppose $k\geq 2$.  For a formal power series
\begin{equation}\label{formal q}
 \sum_{m\ge 0} d (m)q^m\in \C[[q]],
\end{equation}
the following are equivalent.
\begin{enumerate}

\item
The relation $\sum_{m\ge 0} c(-m ) d(m)=0$ holds for every weakly holomorphic form 
\[
 \sum_{ m\gg -\infty} c (m) \cdot q^m   \in M^{!,\infty} _{2-k} (D,\chi).
\]

\item
The formal power series (\ref{formal q}) is the $q$-expansion of a modular form in $M^\infty_{k}(D,\chi)$.
\end{enumerate}
\end{theorem}

\begin{proof}
As we assume $k\ge 2$,  that the map sending a weakly holomorphic modular form $f$ to its principal part at $\infty$ identifies 
\[ 
M^{!,\infty} _{2-k} (D,\chi) \subset \C[q^{-1}]. 
\]  
On the other hand, the map sending a holomorphic modular form to its $q$-expansion identifies
\[ 
M^\infty_{k}(D,\chi) \subset \C[[q]]. 
\]  
A slight variant of the modularity criterion of \cite[Theorem 3.1]{Bo2} shows that  each  subspace is the exact  annihilator of the other under  the  bilinear pairing
$
\C[q^{-1}]\otimes \C[[q]]\longrightarrow  \C
$
sending $P\otimes g$ to the constant term of $P\cdot g$.  The claim follows.
\end{proof}


\section{Unitary Borcherds products}
\label{s:divisor calc}


The goal of \S \ref{s:divisor calc} is to state Theorems \ref{thm:unitary borcherds I}, \ref{thm:unitary borcherds II}, and  \ref{thm:unitary borcherds III}, which assert the existence of Borcherds products on $\mathcal{S}^*_\Kra$ and $\mathcal{S}^*_\Pap$ having prescribed divisors and  prescribed leading Fourier-Jacobi coefficients. 
 These theorems are the technical core of this work, and their proofs  will occupy all of  \S \ref{s:analytic borcherds}.

We assume $n\ge 3$ throughout \S \ref{s:divisor calc}.


\subsection{Jacobi forms}
\label{ss:jacobi}


In this section we recall some of the rudiments of the arithmetic theory of Jacobi forms.  A more systematic treatment can be found in the work of Kramer \cite{Kra1, Kra2}.

Let $\mathcal{Y}$ be the moduli stack over $\Z$ classifying elliptic curves, and let $\pi : \mathcal{E}\to \mathcal{Y}$ be the universal elliptic curve.   Abbreviate $\Gamma=\SL_2(\Z)$, and let $\mathfrak{H}$ be the complex upper half-plane.  The groups $\Gamma$ and $\Z^2$ each act on   $\mathfrak{H} \times \C$   by 
 \begin{align*}
 \left(\begin{matrix} a & b\\ c & d \end{matrix}\right)   \cdot  ( \tau, z)     = \left(  \frac{a\tau + b}{c \tau +d}  ,  \frac{z}{c\tau+d}   \right) , \quad 
 \left[ \begin{matrix}  \alpha \\ \beta  \end{matrix} \right]   \cdot   ( \tau, z)   = \left( \tau   ,  z+ \alpha \tau+ \beta    \right),
 \end{align*}
and this defines an action of the semi-direct product  $\Gamma^* = \Gamma \imes \Z^2$.
We identify the commutative diagrams (of complex orbifolds)
 \begin{equation}\label{modular parametrization}
 \xymatrix{
  {   \Gamma  \backslash (\mathfrak{H} \times \C)   }   \ar[d]  \ar[dr]    &  & { \Lie(\mathcal{E}(\C)) } \ar[d]_{\exp} \ar[dr] \\
{  \Gamma^* \backslash (\mathfrak{H} \times \C) } \ar[r] &  \Gamma \backslash \mathfrak{H} & { \mathcal{E}(\C)  } \ar[r]  &  { \mathcal{Y}(\C)  }
 }
 \end{equation}
 by sending $( \tau ,z)\in \mathfrak{H} \times \C$ to the vector $z$ in the Lie algebra of  $\C/ (\Z\tau + \Z)$.

 Define a line bundle  $\co(e)$ on $\mathcal{E}$ as  the inverse  ideal sheaf of the zero section  $ e : \mathcal{Y} \to \mathcal{E}$.   The Lie algebra  $\Lie(\mathcal{E})$ is (by definition) $e^* \co(e)$, and  $\bm{\omega}_\mathcal{Y}=\Lie(\mathcal{E})^{-1}$ is  the usual line bundle of weight one modular forms on $\mathcal{Y}$ (see Remark \ref{rem:q-expansion} below).     In particular, the line bundle 
\[
 \mathcal{Q} =  \co(e) \otimes \pi^*\bm{\omega}_\mathcal{Y}
\] 
on $\mathcal{E}$ is canonically trivialized along the zero section.     
By the constructions of \cite[\S 1.3.2]{Lan} and \cite[\S 6.2]{MFK}, this line bundle induces a homomorphism
\begin{equation}\label{elliptic polarization}
\mathcal{E} \to \mathcal{E}^\vee,
\end{equation}
which is none other than the unique principal polarization of $\mathcal{E}$ (one can verify this fiber-by-fiber over geometric points of $\mathcal{Y}$, reducing the claim to standard properties of elliptic curves over fields).
 Denote by $\mathcal{P}$ the pullback of the Poincar\'e bundle via 
\[
\mathcal{E}\times_\mathcal{Y} \mathcal{E}  \iso \mathcal{E}\times_\mathcal{Y} \mathcal{E}^\vee.
\]

For a scheme $U$ and   points $a,b \in  \mathcal{E}(U)$,  denote by $\mathcal{Q}(a)$ the pullback of $\mathcal{Q}$ via  $a:U\to \mathcal{E}$, and by $\mathcal{P}(a,b)$ the pullback of $\mathcal{P}$ via 
$
(a,b) : U \to \mathcal{E} \times_{\mathcal{Y}} \mathcal{E}.
$
There are canonical isomorphisms
\[
\mathcal{P}(a,b) \iso \mathcal{Q}(a+b) \otimes \mathcal{Q}(a)^{-1}\otimes \mathcal{Q}(b)^{-1}
\]
and 
\[
\mathcal{P}(a,a) \iso \mathcal{Q}(a) \otimes \mathcal{Q}(a).
\] 
Given the way that  (\ref{elliptic polarization}) is constructed from $\mathcal{Q}$,  the first isomorphism is essentially a tautology.  The second is a consequence of the isomorphisms
\[
 \mathcal{Q}(2a) \iso   \mathcal{Q}(a)^{\otimes 3}  \otimes  \mathcal{Q}(-a) \iso   \mathcal{Q}(a)^{\otimes 4},
\]
which follow from the theorem of the cube \cite[Theorem I.1.3]{FC} and  the invariance of $\mathcal{Q}$ under pullback by $[-1]: \mathcal{E} \to \mathcal{E}$, respectively.

\begin{definition}
The diagonal restriction
\[ 
\mathcal{J}_{0,1} = (\mathrm{diag})^* \mathcal{P} \iso \mathcal{Q}^{2}
\] 
 is the line bundle of \emph{Jacobi forms of weight $0$ and index $1$} on $\mathcal{E}$.   More generally, 
\[
\mathcal{J}_{k,m} = \mathcal{J}_{0,1}^{m} \otimes \pi^* \bm{\omega}_\mathcal{Y}^k
\]
 is the line bundle of \emph{Jacobi forms of weight $k$ and index $m$} on $\mathcal{E}$.
\end{definition}

The isomorphism of the following proposition is presumably well-known.  We include the proof in order to make  explicit the normalization of the isomorphism  (see Remark \ref{rem:q-expansion} below, for example). 

\begin{proposition}\label{prop:analytic jacobi}
 Let $p : \mathfrak{H} \times \C \to \mathcal{E}(\C)$ be the quotient map.   The holomorphic line bundle $\mathcal{J}^{an}_{k,m}$  on $\mathcal{E}(\C)$ is  isomorphic to the holomorphic line bundle whose sections over an open set $\mathscr{U} \subset \mathcal{E}(\C)$  are  holomorphic functions $F(\tau, z)$ on $p^{-1}(\mathscr{U})$ satisfying the transformation laws
\[
F \left( \frac{a \tau  + b}{ c \tau + d } , \frac{z}{ c\tau+ d}  \right)   =F( \tau , z)   \cdot (c\tau +d)^k \cdot e^{    2\pi i m c z^2 / (  c\tau +d )    } 
\]
and
\begin{equation}\label{JF second law}
F ( \tau, z+ \alpha \tau + \beta )   = F( \tau , z)  \cdot  e^{  - 2  \pi i  m ( \alpha^2    \tau     + 2 \alpha  z )   } .
\end{equation}
\end{proposition}

\begin{proof}
Let $J_{k,m}$ be the  holomorphic line bundle  on $\mathcal{E}(\C)$ defined by the above transformation laws.  

 By identifying the diagrams (\ref{modular parametrization}),  a function $f$,  defined on a $\Gamma$-invariant open subset of  $\mathfrak{H}$ and satisfying the transformation law 
\[
f \left( \frac{a \tau  + b}{ c \tau + d } \right)  =f( \tau )   \cdot (c\tau +d)^{-1}
\]
of a weight $-1$ modular form,   defines a section $\tau \mapsto ( \tau , f(\tau))$ of the line bundle
\[
 \Gamma\backslash (\mathfrak{H} \times \C) \iso \Lie(\mathcal{E}(\C) ) \iso ( \bm{\omega}_{\mathcal{Y}}^{an})^{-1}
\]
on $\Gamma\backslash\mathfrak{H}$.  This determines an isomorphism
$J_{1,0} \iso \mathcal{J}_{1,0}^{an}.$  It now suffices to construct an isomorphism $J_{0,1} \iso \mathcal{J}_{0,1}^{an}$, and then take tensor products.

Fix $\tau \in \mathfrak{H}$, set $E_\tau =\C/  (\Z \tau +\Z)$, and restrict both  $\mathcal{J}^{an}_{0,1}$ and $J_{0,1}$ to  line bundles on  $E_\tau\subset \mathcal{E}(\C)$.  The imaginary part of the hermitian form
\[
H_\tau( z_1, z_2) = \frac{z_1 \overline{z_2} }{ \mathrm{Im}(\tau) }
\]
on $\C$ restricts to a Riemann form on  $\Z\tau+\Z $.  
Using  classical formulas for the Poincar\'e bundle on complex abelian varieties, as found in the proof of \cite[Theorem 2.5.1]{BL}, the restriction of $\mathcal{J}_{0,1}^{an}$ to the fiber $E_\tau$  is isomorphic to the holomorphic line bundle determined by  the Appell-Humbert data  $2H_\tau$ and  the trivial character $\Z\tau+\Z\to \C^\times$.  The sections of this holomorphic line bundle  are, by definition,   holomorphic functions $g_\tau$ on $\C$ satisfying the transformation law
\[
g_\tau( z+\ell) = g_\tau(z) \cdot e^{2\pi H_\tau(z,\ell) + \pi H_\tau(\ell ,\ell )} 
\]
for all  $\ell \in \Z\tau+\Z$.  If we set
\[
F ( \tau , z) = g_\tau(z) \cdot e^{ - \pi H_\tau(  z , \overline{z} )  },
\]
this transformation law becomes (\ref{JF second law}).  

The above shows that $\mathcal{J}_{0,1}^{an}$ and $J_{0,1}$ are isomorphic when restricted to the fiber over any point of $\mathcal{Y}(\C)$, but such an isomorphism is only determined up to scaling by $\C^\times$.   To pin down the scalars, and to get an isomorphism over the total space, use the fact that both $\mathcal{J}^{an}_{0,1}$ and $J_{0,1}$ come (by construction) with     canonical  trivializations along the zero section. By the Seesaw Theorem \cite[Appendix A]{BL}, there is a unique isomorphism $\mathcal{J}^{an}_{0,1}\iso J_{0,1}$ compatible with these trivializations.
\end{proof}

\begin{remark}\label{rem:q-expansion}
The proof of Proposition \ref{prop:analytic jacobi} identifies a classical modular form $f(\tau)=\sum c(m) q^m$ of weight $k$ and level $\Gamma$ with a holomorphic section  of  $(\bm{\omega}_\mathcal{Y}^{an})^k$,  again denoted $f$, satisfying an additional growth condition at the cusp.   Under our identification, the $q$-expansion principle takes the following form: if $R\subset \C$ is any subring, then $f$ is the analytification of a global section 
$
f\in H^0(\mathcal{Y}_{/R} , \bm{\omega}_{\mathcal{Y}/R}^k )
$
if and only if $c(m) \in   (2\pi i )^k R$ for all $m$.
\end{remark}

For $\tau\in \mathfrak{H}$ and $z\in \C$, we denote by
\[
\vartheta_1(\tau,z)= \sum_{n\in \Z} e^{ \pi i \left(n+\frac{1}{2}\right)^2\tau+ 2\pi i \left(n+\frac{1}{2}\right)\left(z-\frac{1}{2}\right) }
\]
the classical Jacobi theta function, and by 
\[
\eta(\tau) = e^{ \pi i  \tau/ 12 }\prod_{n=1}^\infty (1-e^{2n \pi i  \tau})
\]
 Dedekind's eta function.   Set 
 \[
 \Theta(\tau , z) \define  i   \frac{ \vartheta_1 (\tau,z) }{  \eta(\tau)} = 
 q^{1/12} (\zeta^{1/2} - \zeta^{-1/2})  \prod_{n=1}^\infty (1-\zeta q^n)(1-\zeta^{-1}q^n)
 \]
where $q=e^{2\pi i \tau}$ and $\zeta=e^{2\pi i z}$.

\begin{proposition}\label{prop:integral jacobi}
The  Jacobi form $\Theta^{24}$  defines a global section
\[
\Theta^{24} \in  H^0( \mathcal{E} , \mathcal{J}_{0,12} )
\]
 with divisor $24 e$, while $(2\pi i\eta^2)^{12}$ determines a nowhere vanishing section
 \[
(2\pi i\eta^2)^{12} \in H^0( \mathcal{Y} , \bm{\omega}_\mathcal{Y}^{ 12} ).
\]
\end{proposition}

\begin{proof}
It is a classical fact that $(2\pi i\eta^2)^{12}$ is a nowhere vanishing modular form of weight $12$.  
Noting Remark \ref{rem:q-expansion}, the $q$-expansion principle shows that it descends to a section on $\mathcal{Y}_{ / \Q}$, and thus may be viewed as a rational section on $\mathcal{Y}$.  Another application of the $q$-expansion principle shows that its divisor has no vertical components.    Thus its divisor is trivial.

Classical formulas show that $\Theta^{24}$ defines a holomorphic section of $\mathcal{J}^{an}_{0,12}$ with divisor $24 e$, and so the problem is to show that $\Theta^{24}$ is defined over $\Q$, and extends to a section on the integral model with the stated divisor.  
One could presumably deduce this from the $q$-expansion principle for Jacobi forms as in  \cite{Kra1, Kra2}.  We instead borrow an argument from \cite[\S 1.2]{Scholl}, which  requires only the more elementary $q$-expansion principle for \emph{functions} on $\mathcal{E}$.

Let $d$ be any positive integer.    The bilinear relations (\ref{poincare bilinear})  imply that the line bundle   $\mathcal{J}_{0,1}^{d^2}\otimes [d]^*\mathcal{J}_{0,1}^{-1}$ on $\mathcal{E}$ is canonically trivial, and so 
 \[
 \theta_d^{24} =  \Theta^{24 d^2}  \otimes   [d]^* \Theta^{-24 }
 \]
 defines a meromorphic function on $\mathcal{E}(\C)$.  The crucial point is that $\theta_d^{24}$ is actually a rational function defined over $\Q$, and extends to a rational function on the integral model $\mathcal{E}$ with  divisor 
 \begin{equation}\label{theta shift divisor}
 \mathrm{div}(\theta_d^{24}) = 24 \big( d^2 \mathcal{E}[1] -   \mathcal{E}[d] \big).
 \end{equation}
 As in  \cite[p.~387]{Scholl},  this follows by computing the divisor first in the complex fiber, then using   the explicit  formula
\[
 \theta_d^{24} ( \tau,z) = q^{ 2(d^2-1)} \zeta^{ -12 d(d-1) } 
 \left(  \prod_{n\ge 0}\frac{(1-q^n\zeta)^{  d^2} }{  1-q^n \zeta^d } 
 \prod_{n>0}  \frac{(1-q^n\zeta^{-1})^{  d^2} }{ 1-q^n \zeta^{-d}  }  \right)^{24}
\]
and  the $q$-expansion principle on $\mathcal{E}$ to see that the divisor has no vertical components.

The line bundle  $\bm{\omega}_\mathcal{Y}^{12}$ is trivial, and hence there are  isomorphisms
\[
\mathcal{J}_{0,12} \iso  \mathcal{Q}^{24} \iso \co(e)^{24} \otimes \pi^*\bm{\omega}_\mathcal{Y}^{12} \iso  \co(e)^{24}.
\]
 Thus there is \emph{some}
$
\tilde{\Theta}^{24} \in H^0(\mathcal{E} , \mathcal{J}_{0,12} )
$
with divisor $24 e$, and   the rational function
 \[
\tilde{\theta}_d^{24} =  \tilde{\Theta}^{24 d^2}  \otimes   [d]^* \tilde{\Theta}^{-24 }  
 \]
 on $\mathcal{E}$ also has divisor (\ref{theta shift divisor}).

Consider the meromorphic function  
$
\rho = \Theta^{24}/  \tilde{\Theta}^{24}
$
on $\mathcal{E}(\C)$.  By computing the divisor in the complex fiber, we see that $\rho$ is a nowhere vanishing holomorphic function, and hence is constant.  But this implies that 
 \[
\rho^{ d^2 -1 }  =  \theta_d^{24} / \tilde{\theta}_d^{24}.
 \]
By what was said above, the right hand side is (the analytification  of) a nowhere vanishing function on $\mathcal{E}$.  This implies that $\rho^{d^2-1}= \pm 1$, and the only way this can hold for all $d>1$ is if $\rho= \pm 1$.
 \end{proof}

Now consider the tower of stacks 
\[
\mathcal{Y}_1(D) \to \mathcal{Y}_0(D) \to \mathcal{Y}
\] 
over $\Spec(\Z)$ parametrizing elliptic curves with Drinfeld $\Gamma_1(D)$-level structure, $\Gamma_0(D)$-level structure, and no level structure, respectively.  See  \cite[Chapter 3]{KM} or \cite{DR} for the definitions.
We denote by $\mathcal{E}$ the universal elliptic curve over any one of these bases, and  view the line bundle of Jacobi forms $\mathcal{J}_{0,12}$ as a line bundle on any one of the three universal elliptic curves.  Similarly, we view the Jacobi forms  $\Theta^{24}$ and $(2\pi i \eta^2)^{12}$ of Proposition \ref{prop:integral jacobi} as being defined over any one of these bases.

The following lemma will be needed in \S \ref{ss:unitary borcherds}.

\begin{lemma}\label{lem:torsion section}
Let $ Q :\mathcal{Y}_1(D) \to \mathcal{E}$ be the universal $D$-torsion point.  For any $r\mid D$ the line bundle 
\begin{equation}\label{torsion jacobi twist}
\bigotimes_{ \substack{  b \in \Z/ D \Z \\ b\neq 0 \\ rb=0   }} (b Q)^*\mathcal{J}_{0,12}
\end{equation}
on $\mathcal{Y}_1(D)$ is canonically trivial, and its section 
\[
F_r^{24} =  \bigotimes_{ \substack{  b \in \Z/D \Z \\ b\neq 0 \\ rb=0   }} (b Q )^*\Theta^{24} 
\]
 admits a canonical descent, denoted the same way, to a section of the trivial bundle on   $\mathcal{Y}_0(D)$.
\end{lemma}

\begin{proof}
If  $x_1,\ldots, x_r$ are integers representing the $r$-torsion subgroup of  $\Z/D\Z$, then
$ 6 \sum x_i^2 \equiv 0 \pmod{D}$. The bilinear relations (\ref{poincare bilinear}) therefore provide a canonical isomorphism 
\[
\bigotimes _{ \substack{  b \in \Z/D\Z \\ b\neq 0 \\ rb=0  }} \mathcal{P}( b Q ,b Q)^{\otimes 12} \iso
\bigotimes _{ \substack{  b \in \Z/D\Z \\ b\neq 0  \\rb=0  }} \mathcal{P}(  Q , 12 b^2 Q) \iso \mathcal{P}(Q, e) \iso \co_{\mathcal{Y}_1(D) }
\]
of line bundles on $\mathcal{Y}_1(D)$.  This is the desired trivialization of (\ref{torsion jacobi twist}).  The section $F^{24}_r$ is obviously invariant under the action of the diamond operators on $\mathcal{Y}_1(D)$, and so descends to $\mathcal{Y}_0(D)$. 
\end{proof}


\subsection{Borcherds' quadratic identity}
\label{ss:borcherdsquad}


For the remainder of \S \ref{s:divisor calc} we denote by $\chi_\kk : (\Z/D\Z)^\times \to \{\pm 1\}$  the Dirichlet character determined by the extension $\kk/\Q$,  abbreviate
\begin{equation}\label{quad character}
\chi = \chi_\kk^{n-2},
\end{equation}
and fix a weakly holomorphic form
\begin{equation}\label{input form}
f (\tau) = \sum_{  m\gg -\infty } c(m) q^m  \in M^{!,\infty}_{2-n}(D,\chi)
\end{equation}
with $c(m)\in \Z$ for all $m\le 0$.

For a proper cusp label representative $\Phi$ as in Definition \ref{def:clr},  recall the self-dual hermitian $\co_\kk$-lattice $L_0$ of signature $(n-2,0)$ defined by (\ref{boundary herm}).  The hermitian form  on $L_0$  determines  a quadratic form $Q(x) = \langle x, x \rangle$, with associated $\Z$-bilinear form
$
[ x_1 ,x_2 ] = \mathrm{Tr}_{\kk/\Q} \langle x_1 , x_2 \rangle
$
of signature $(2n-4,0)$.

The modularity criterion of Theorem \ref{thm:modularity criterion} implies the following identity of quadratic forms on $L_0\otimes \R$.

\begin{proposition}[Borcherds' quadratic identity] \label{prop:Bquad}
For all $u \in L_0\otimes \R$,
\[
\sum_{\substack{x \in L_0}} c(-Q(x)) \cdot [u,x]^2 =\frac{[u,u] }{2n-4} \sum_{\substack{x \in L_0}} c(-Q(x))\cdot [x,x].
\]

\end{proposition}

\begin{proof}
The homogeneous polynomial
\[
P(u,v)= [u,v]^2-\frac{[u,u]\cdot [v,v]}{2n-4}
\]
on $L_0\otimes \R$ is harmonic in both variables $u$ and $v$.
For any fixed $u\in L_0\otimes \R$ there is a corresponding theta series
\[
\theta(\tau,u,P)= \sum_{x\in L_0} P(u,x)\cdot q^{Q(x)} \in S_n(D,\chi) .
\]
The modularity criterion of Theorem \ref{thm:modularity criterion} therefore shows that
\[
\sum_{m>0} c(-m)\sum_{\substack{x \in L_0\\ Q(x)=m}} \left([u,x]^2-\frac{[u,u]\cdot [x,x]}{2n-4}\right) =0
\]
for all $u\in L_0\otimes \R$. This implies the assertion.
\end{proof}

 Recall from (\ref{boundary collapse}) that every  $x \in L_0$ determines a diagram 
\begin{equation}\label{late boundary collapse}
\xymatrix{
{  \mathcal{B}_\Phi  } \ar[r]^{j_x}  \ar[d] & {  \mathcal{E} } \ar[d]  \\
{   \mathcal{A}_\Phi        }  \ar[r]^j & { \mathcal{Y}_0(D) ,}
}
\end{equation}
where, changing notation slightly from \S \ref{ss:jacobi}, $\mathcal{Y}_0(D)$  is now  the open modular curve over $\co_\kk$. 
 Recall also that  $\mathcal{B}_\Phi$ carries a distinguished line bundle $\mathcal{L}_\Phi$ defined by (\ref{boundary bundle}), used to define the Fourier-Jacobi expansions of (\ref{FJ expansion}).  We will use Borcherds' quadratic identity to relate the line bundle $\mathcal{L}_\Phi$  to the line bundle $\mathcal{J}_{0,1}$ of Jacobi forms on $\mathcal{E}$.

\begin{proposition}\label{prop:quadratic bundles}
The rational number
\begin{equation}\label{f boundary mult}
\mathrm{mult}_\Phi(f) =  \sum_{ m >0 }   \frac{  m\cdot  c(-m)  }{n-2}    \cdot \#\{ x \in L_0 :  Q( x) =m \}
\end{equation}
lies in $\Z$, and there is a canonical isomorphism 
\[
\mathcal{L}_\Phi^{  2\cdot  \mathrm{mult}_\Phi(f)  } \iso \bigotimes_{m > 0} \bigotimes_{  \substack{ x \in L_0 \\ Q(x) =m }  } j_x^*\mathcal{J}_{0,1}^{ c(-m) }
\]
of line bundles on $\mathcal{B}_\Phi$.   
\end{proposition}

\begin{proof}
Proposition \ref{prop:Bquad} implies the equality of hermitian forms
\begin{align*}
\sum_{\substack{x \in L_0}} c(-Q(x)) \cdot \langle u,x\rangle \cdot \langle x,v\rangle
& =  \frac{\langle u,v\rangle }{2n-4} \sum_{\substack{x \in L_0}} c(-Q(x))\cdot [x,x] \\
& =  \langle u,v\rangle \cdot \mathrm{mult}_\Phi(f)
\end{align*}
for all $u,v\in L_0$.  As $L_0$ is self-dual, we may choose $u$ and $v$ so that $\langle u,v\rangle=1$, and the integrality of $\mathrm{mult}_\Phi(f)$ follows from the integrality of $c(-m)$.

Set $E= \mathcal{E}\times_{\mathcal{Y}_0(D) } \mathcal{A}_\Phi $, and use Proposition \ref{prop:second moduli iso} to identify $\mathcal{B}_\Phi \iso E\otimes L_0$.   The pullback of the line bundle
\[
\bigotimes_{m > 0} \bigotimes_{  \substack{ x \in L_0 \\ Q(x)  =m }  } j_x^*\mathcal{J}_{0,1}^{ \otimes c(-m) } 
\iso \bigotimes_{ x \in L_0  }  j_x^*\mathcal{J}_{0,1}^{ \otimes c(- Q(x) ) }
\]
via any $T$-valued point 
$
a=\sum t_i \otimes y_i \in E(T)\otimes L_0 
$
is, in the notation of \S \ref{ss:second moduli},
\begin{eqnarray*}\lefteqn{ 
 \bigotimes_{ x \in L_0  }  \mathcal{P}_E\Big(  \sum_i \langle y_i, x\rangle   t_i  , \sum_j \langle  y_j, x \rangle   t_j    \Big)^{ \otimes c(- Q(x)) }  } \\
 &  \iso &
 \bigotimes_{i,j}  \bigotimes_{ x \in L_0  }  \mathcal{P}_E \big(  c(- Q(x)) \cdot \langle  y_i , x \rangle \cdot   \langle x,  y_j \rangle    \cdot  t_i  ,  t_j  \big) \\
 & \iso & 
 \bigotimes_{i,j}   \mathcal{P}_E \big(    \langle y_i,y_j\rangle \cdot t_i , t_j      \big)^{ \otimes  \mathrm{mult}_\Phi(f)} \\
 &\iso &  \mathcal{P}_{E\otimes L_0}(a,a)^{ \otimes  \mathrm{mult}_\Phi(f)} \\
 &\iso &  \mathcal{Q}_{E\otimes L_0}(a)^{\otimes 2 \cdot  \mathrm{mult}_\Phi(f)  }.
\end{eqnarray*}
This, along with  the isomorphism $ \mathcal{Q}_{E\otimes L_0} \iso \mathcal{L}_\Phi$  of Proposition \ref{prop:second moduli iso}, proves that 
\[
 \mathcal{L}_\Phi ^{  2\cdot  \mathrm{mult}_\Phi(f)  } \iso  \mathcal{Q}_{E\otimes L_0}^{  2\cdot  \mathrm{mult}_\Phi(f)  } \iso \bigotimes_{m > 0} \bigotimes_{  \substack{ x \in L_0 \\ Q(x) =m }  } j_x^*\mathcal{J}_{0,1}^{ c(-m) }.
\]
\end{proof}


\subsection{The unitary Borcherds product}
\label{ss:unitary borcherds}


We now state our main results on Borcherds products.

For a prime $p$ dividing $D$  define 
\begin{equation}\label{def-gamma-p}
\gamma_p = 
 \epsilon_p^{-n} \cdot (D,p)_p^n \cdot \operatorname{inv}_p(V_p) \in \{ \pm 1, \pm i\} ,
\end{equation}
where $\operatorname{inv}_p(V_p)$ is the invariant  of $V_p=\Hom_\kk(W_0,W)\otimes_\Q\Q_p$ in the sense of  \eqref{eq:locinv}, and 
\[
\epsilon_p=\begin{cases}1 & \mbox{if }p\equiv 1\pmod{4}  \\
i & \mbox{if } p\equiv 3\pmod{4}.
\end{cases}
\]
It is equal to the local Weil index of the Weil representation of $\SL_2(\Z_p)$ on $S_{L_p}\subset S(V_p)$, where $V_p$ is viewed as a quadratic space as in (\ref{hom quadratic}).  
This is explained in more detail in \S \ref{ss:weil indices}. For any $r$ dividing $D$ we define 
\begin{equation}\label{gamma def}
\gamma_r  = \prod_{p\mid r} \gamma_p .
\end{equation}

Let  $c_r(0)$ denote the constant term of $f$ at the cusp $\infty_r$, as in  Definition \ref{def:constant at cusp}, and 
define
\[
k = \sum_{r\mid D} \gamma_r  \cdot c_r(0).
\]
We will see later in Corollary \ref{cor:vector rational} that all $\gamma_r\cdot c_r(0)\in \Q$.

For every $m>0$ define a divisor 
\begin{equation}\label{m boundary mult}
\mathcal{B}_\Kra(m) = \frac{ m }{n-2}   \sum_\Phi  \#\{ x \in L_0 :  \langle x , x \rangle  =m \} \cdot \mathcal{S}^*_\Kra(\Phi)
\end{equation}
with rational coefficients on $\mathcal{S}^*_\Kra$.  
Here the sum is over all $K$-equivalence classes of proper cusp label representatives $\Phi$ in the sense of \S \ref{ss:mixed data},  $L_0$ is the hermitian $\co_\kk$-module of signature $(n-2,0)$ defined by (\ref{boundary herm}), and $\mathcal{S}^*_\Kra(\Phi)$ is the boundary divisor of Theorem \ref{thm:toroidal}.   
 It follows immediately from the definition (\ref{f boundary mult}) that
\[
\sum_{m > 0 } c(-m) \cdot  \mathcal{B}_\Kra(m) = \sum_\Phi \mathrm{mult}_\Phi(f) \cdot \mathcal{S}^*_\Kra(\Phi).
\]
For $m>0$ define the \emph{total special divisor}
\[
\mathcal{Z}_\Kra^\tot(m) = \mathcal{Z}_\Kra^*(m)  +  \mathcal{B}_\Kra(m),
\]
where $\mathcal{Z}_\Kra^*(m)$ is the special divisor defined on the open Shimura variety in \S \ref{ss:special divisors}, and extended to the toroidal compactification in  Theorem \ref{thm:toroidal}.

The following theorems assert the existence of   Borcherds products on $\mathcal{S}^*_\Kra$ and $\mathcal{S}^*_\Pap$ having prescribed divisors and prescribed leading Fourier-Jacobi coefficients.   Their proofs  will occupy all of  \S \ref{s:analytic borcherds}.

\begin{theorem}\label{thm:unitary borcherds I}
After possibly replacing the form $f$ of (\ref{input form}) by a positive integer multiple, there is a  rational section $\bm{\psi}(f)$ of the line bundle $\bm{\omega}^k$ on $\mathcal{S}^*_\Kra$ with the following properties.

\begin{enumerate}
\item
In the generic fiber, the divisor of $\bm{\psi}(f)$ is
\[
\mathrm{div}( \bm{\psi}  (f) )_{/\kk}   
 =   \sum_{m>0} c(-m) \cdot \mathcal{Z}_\Kra^\tot(m)_{/\kk}.
\]

\item
For every proper cusp label representative $\Phi$,  the Fourier-Jacobi expansion of $\bm{\psi}(f)$, in the sense of (\ref{FJ expansion}),  along the boundary divisor 
\[
\Delta_\Phi \backslash \mathcal{B}_\Phi \iso \mathcal{S}^*_\Kra(\Phi)
\]
 has the form
\[
\bm{\psi}(f) = q^{ \mathrm{mult}_\Phi(f) } \sum_{ \ell\ge 0} \bm{\psi}_\ell \cdot q^\ell,
\]
where $\bm{\psi}_\ell$ is a rational section of  $\bm{\omega}_\Phi^k \otimes \mathcal{L}_\Phi^{ \mathrm{mult}_\Phi(f) + \ell }$ over $\mathcal{B}_\Phi$.

\item
For any $\Phi$ as above, the leading coefficient $\bm{\psi}_0$  admits a factorization
\[
\bm{\psi}_0 =    P_\Phi^\eta \otimes P^{hor}_\Phi \otimes P^{vert}_\Phi,
\]
where the three terms on the right are defined  as follows.

\begin{enumerate}
\item
Proposition \ref{prop:mixed modular forms} provides us with an isomorphism 
\[
\mathfrak{d}^{-1} \bm{\omega}_\Phi \iso   j^* \bm{\omega}_{\mathcal{Y}} 
\] 
of line bundles on $\mathcal{A}_\Phi$, where $j :\mathcal{A}_\Phi \to  \mathcal{Y}_0(D) $ is the morphism of (\ref{late boundary collapse}), 
and $\bm{\omega}_\mathcal{Y}=\Lie(\mathcal{E})^{-1}$ is the pullback via $\mathcal{Y}_0(D) \to \mathcal{Y}$ of the line bundle of weight one modular forms.
Pulling back the modular form $(2\pi i \eta^2)^{12}$ of  Proposition \ref{prop:integral jacobi}   defines a nowhere vanishing section 
\[
j^*(2\pi i \eta^2)^k \in H^0( \mathcal{A}_\Phi ,  \mathfrak{d}^{-k} \bm{\omega}_\Phi^k ).
\] 
Using the canonical  inclusion $\bm{\omega}_\Phi  \subset   \mathfrak{d}^{-1} \bm{\omega}_\Phi$, define
\[
P_\Phi^\eta = j^*(2\pi i \eta^2)^k,
\] 
but viewed as a rational section of $\bm{\omega}_\Phi^k$ over $\mathcal{A}_\Phi$.  Denote in the same way its pullback to $\mathcal{B}_\Phi$.
\item
Recalling the function
\[
F_r^{24} =  \bigotimes_{ \substack{  b \in \Z/D \Z \\ b\neq 0 \\ rb=0   }} (b Q )^*\Theta^{24} 
\]
on $\mathcal{Y}_0(D)$ of  Lemma \ref{lem:torsion section}, define  a rational function 
\[
P^{vert}_\Phi =  \bigotimes_{ \substack{ r \mid D \\ r>1 } } j^* F_r^{\gamma_r  c_r(0) }
\]
 on $\mathcal{A}_\Phi$, and again pull back to $\mathcal{B}_\Phi$.   
 \item
Using Proposition \ref{prop:quadratic bundles},  define a rational section
\[
P^{hor}_\Phi =  \bigotimes_{m>0} \bigotimes_{  \substack{ x \in L_0 \\ \langle x, x\rangle =m } }  j_x^* \Theta^{ c(- m )}
\]
 of the line bundle $\mathcal{L}_\Phi^{ \mathrm{mult}_\Phi(f)}$ on $\mathcal{B}_\Phi$. 
 \end{enumerate}
\end{enumerate}
These properties determine $\bm{\psi}(f)$ uniquely.
\end{theorem}

\begin{remark}\label{rem:very divisible}
In  replacing $f$ by a positive integer multiple, we are tacitly assuming that the constants $\gamma_r c_r(0)$ and $c(-m)$ are integer multiples of $24$ for all $r\mid D$ and all $m> 0$.  This is necessary in order to guarantee $k\in 12 \Z$, and to make sense of the three factors  $(2\pi i \eta_\Phi^2)^k$,  $P^{hor}_\Phi$, and  $P^{vert}_\Phi$.
\end{remark}

In fact, we can strengthen Theorem \ref{thm:unitary borcherds I} by computing precisely the divisor of $\bm{\psi}(f)$ on the integral model $\mathcal{S}_\Kra^*$.

\begin{theorem}\label{thm:unitary borcherds II}
The rational section $\bm{\psi}(f)$ of $\bm{\omega}^k$  has divisor
\begin{align*}
\mathrm{div}( \bm{\psi}  (f) )   
& =   \sum_{m>0} c(-m) \cdot \mathcal{Z}_\Kra^\tot(m)   \\
& \quad   + k\cdot \left(  \frac{\mathrm{Exc}}{2}   
  -   \mathrm{div}(\delta)\right)  +  \sum_{ r \mid D} \gamma_r c_r(0)   \sum_{p\mid r} \mathcal{S}^*_{\Kra /\F_\mathfrak{p}}  \\
 & \quad -  \sum_{m>0}   \frac{c(-m)}{2}    \sum_{ s\in \pi_0(\mathrm{Sing}) }   \# \{ x\in L_s : \langle x,x\rangle =m \} \cdot  \mathrm{Exc}_s ,
\end{align*}
where $\mathfrak{p} \subset \co_\kk$ is  the unique prime above $p$, $L_s$ is the self-dual Hermitian $\co_\kk$-lattice defined  in \S \ref{ss:pure pullback}, and  $\mathrm{Exc}_s\subset  \mathrm{Exc}$ is the fiber over the component $s\in \pi_0( \mathrm{Sing)}$.   Recall that $\delta = \sqrt{-D}\in \kk$.
\end{theorem}

It is possible to give a statement analogous to Theorem \ref{thm:unitary borcherds II} for the integral model $\mathcal{S}_\Pap^*$. 
To do this we first define, exactly as in (\ref{m boundary mult}), a Cartier divisor
\[
\mathcal{Y}_\Pap^\tot(m) = \mathcal{Y}_\Pap^*(m)  + 2 \mathcal{B}_\Pap(m) 
\]
with rational coefficients on $\mathcal{S}_\Pap^*$.  Here  $\mathcal{Y}^*_\Pap(m)$ is the Cartier  divisor  of Theorem  \S \ref{thm:toroidal}, and 
\[
\mathcal{B}_\Pap(m) = \frac{ m }{n-2}   \sum_\Phi  \#\{ x \in L_0 :  \langle x , x \rangle  =m \} \cdot \mathcal{S}^*_\Pap(\Phi).
\]
It is clear from Theorem \ref{thm:toroidal} that 
\begin{equation}\label{total error}
 2\cdot \mathcal{Z}_\Kra^\tot(m) = \mathcal{Y}_\Kra^\tot(m) +  \sum_{s\in \pi_0(\mathrm{Sing}) } \# \{ x\in L_s : \langle x,x\rangle=m\} \cdot \mathrm{Exc}_s,
\end{equation}
where  $\mathcal{Y}_\Kra^\tot(m)$ denotes the pullback of $\mathcal{Y}_\Pap^\tot(m)$ via
$
\mathcal{S}_\Kra^* \to \mathcal{S}_\Pap^*.
$

The isomorphism \[\bm{\omega}^2  \iso \bm{\Omega}_\Kra  \otimes \co(\mathrm{Exc})\]  of Theorem \ref{thm:toroidal} identifies  $\bm{\omega}^{2k} \iso \bm{\Omega}^k_\Kra$ in the generic fiber of $\mathcal{S}_\Kra^*$,  allowing us to view $\bm{\psi}(f)^2$ as a rational section of $\bm{\Omega}_\Kra^k$.  As $\mathcal{S}^*_\Kra \to \mathcal{S}^*_\Pap$ is an isomorphism in the generic fiber, this section descends to a rational section of the line bundle $\pure_\Pap^k$ on  $\mathcal{S}^*_\Pap$.

\begin{theorem}\label{thm:unitary borcherds III}
When viewed as a rational  section of $\pure_\Pap^k$, the Borcherds product $\bm{\psi}(f)^2$ has divisor
\begin{align*}
\mathrm{div}( \bm{\psi}  (f)^2 )   
  & =   \sum_{m>0} c(-m) \cdot \mathcal{Y}^\tot_\Pap(m)       \\
& \quad  -2k\cdot  \mathrm{div}(\delta) + 2 \sum_{ r \mid D} \gamma_r c_r(0)   \sum_{p\mid r} \mathcal{S}^*_{\Pap /\F_\mathfrak{p}} .
\end{align*}
\end{theorem}

These three theorems will be proved simultaneously in \S \ref{s:analytic borcherds}.  Briefly, we will map our unitary Shimura variety $\mathrm{Sh}(G,\mathcal{D})$ to an orthogonal Shimura variety, where a meromorphic Borcherds product  is already known to exist.   If we pull back this Borcherds product to 
$\mathrm{Sh}(G,\mathcal{D})(\C)$, the  leading  coefficient in its analytic Fourier-Jacobi expansion is known from \cite{Ku:ABP}, up to multiplication by some unknown constants of absolute value $1$.

By converting this analytic Fourier-Jacobi expansion into algebraic language, we will deduce the existence of a Borcherds product $\bm{\psi}(f)$ satisfying all of the properties stated in Theorem \ref{thm:unitary borcherds I}, up to some unknown constants  in the leading Fourier-Jacobi coefficient.  These unknown constants are the $\kappa_\Phi$'s appearing in Proposition \ref{prop:algebraic BFJ}.   We then rescale the Borcherds product to make many $\kappa_\Phi=1$ simultaneously.    

After such a rescaling, the divisor of $\bm{\psi}(f)^2$ on $\mathcal{S}_\Pap^*$ can be computed from the Fourier-Jacobi expansions, and agrees with the divisor written in  Theorem \ref{thm:unitary borcherds III}.
Pulling back  that divisor calculation via $\mathcal{S}_\Kra^*\to \mathcal{S}_\Pap^*$, and   using  Theorem \ref{thm:cartier error}, yields the divisor of Theorem \ref{thm:unitary borcherds II}.

Using the above divisor  calculations, we  prove that all $\kappa_\Phi$ are roots of unity. 
Thus, after replacing $f$ by a multiple, which replaces $\bm{\psi}(f)$ by a power, we can  force all $\kappa_\Phi=1$, completing the proofs.


\subsection{A divisor calculation at the boundary}
\label{ss:div calc bd}


Let $\Phi$ be a proper cusp label representative for $(G,\mathcal{D})$. 
 The following  proposition is a key ingredient  in the proofs of Theorems \ref{thm:unitary borcherds I}, \ref{thm:unitary borcherds II},  and  \ref{thm:unitary borcherds III}.

 \begin{proposition}\label{prop:leading divisor}
The rational sections $P_\Phi^\eta$, $P_\Phi^{hor}$, and $P_\Phi^{vert}$ of the line bundles $\bm{\omega}_\Phi^k$,  $\mathcal{L}_\Phi^{\mathrm{mult}_\Phi(f)}$, and $\co_{\mathcal{B}_\Phi}$, respectively,    have divisors
\begin{align*}
\mathrm{div}( P_\Phi^\eta )  & =  -k\cdot  \mathrm{div}(\delta) \\
  \mathrm{div} ( P^{hor}_\Phi )   &  = \sum_{m>0} c(-m)  \mathcal{Z}_\Phi (m) \\
   \mathrm{div}(  P^{vert}_\Phi )  & = \sum_{ r \mid D} \gamma_r c_r(0)  \sum_{p\mid r} \mathcal{B}_{\Phi/\F_\mathfrak{p}}.
\end{align*}
In particular, the divisor  of   $P^{hor}_\Phi$ is purely horizontal (Proposition \ref{prop:boundary divisors}), while the divisors of $P^\eta_\Phi$ and  $P^{vert}_\Phi$  are purely vertical.
\end{proposition}

\begin{proof}
By Proposition \ref{prop:integral jacobi}  the section
\[
j^*(2\pi i \eta^2)^k \in H^0( \mathcal{A}_\Phi ,  \mathfrak{d}^{-k} \bm{\omega}_\Phi^k ) \iso H^0( \mathcal{Y}_0(D) , \bm{\omega}_\mathcal{Y}^k )
\] 
has trivial divisor.  When we use the inclusion $\bm{\omega}_\Phi  \subset \mathfrak{d}^{-1} \bm{\omega}_\Phi$  to view it instead as a rational section $P_\Phi^\eta$ of $\bm{\omega}_\Phi^k$, its divisor becomes $\mathrm{div}(\delta^{-k})$.  This proves the first equality.

To prove the remaining two equalities, 
let $\mathcal{E} \to \mathcal{Y}_0(D)$ be the universal elliptic curve, and denote by $e:\mathcal{Y}_0(D) \to \mathcal{E}$ the $0$-section.
It is  an effective  Cartier divisor on $\mathcal{E}$.

Directly from the definition of $P^{hor}_\Phi$ we have the equality
\[
\mathrm{div} ( P^{hor}_\Phi )  = \sum_{m>0}   \frac{c(-m)}{24}  \sum_{  \substack{ x \in L_0 \\ \langle x, x\rangle =m } } \mathrm{div}( j_x^* \Theta^{24} ) .
\]
Combining Proposition \ref{prop:integral jacobi} with  (\ref{boundary special})  shows that
\[
  \sum_{  \substack{ x \in L_0 \\ \langle x, x\rangle =m } } \mathrm{div}( j_x^* \Theta^{24} ) 
=
   \sum_{  \substack{ x \in L_0 \\ \langle x, x\rangle =m } }   24    j_x^* ( e) 
=
 \sum_{  \substack{ x \in L_0 \\ \langle x, x\rangle =m } }    24   \mathcal{Z}_\Phi(x)
=  24  \mathcal{Z}_\Phi(m),
\]
and the first equality follows immediately.
 
Recall the morphism $j :\mathcal{A}_\Phi \to \mathcal{Y}_0(D)$ of \S \ref{ss:mixed special divisors}. 
For the second equality it  suffices to prove that the function $F^{24}_r$  on $\mathcal{Y}_0(D)$ defined in Lemma \ref{lem:torsion section} satisfies 
\begin{equation}\label{OT-main}
 \mathrm{div} ( j^* F_r^{24}  ) = 24 \sum_{ p\mid r } \mathcal{A}_{\Phi / \F_\mathfrak{p} }.
\end{equation}

Let  $C \subset \mathcal{E}$ be the universal cyclic subgroup scheme of order $D$.   
For each $s\mid D$ denote by  $C[s] \subset C$  the $s$-torsion subgroup,
and by $C[s]^\times\subset C[s]$ the closed \emph{subscheme of generators}. 
This is defined as follows. Noting that
\[
C[s]=\prod_{p\mid s}C[p],
\]
 we  define
 \[
 C[s]^\times=\prod_{p\mid s}C[p]^\times,
 \]
  where $C[p]^\times$ denotes the closed subscheme of generators of $C[p]$ as in  \cite[\S 3.3]{HR12}.
  Note that  $C[p]^\times$ coincides with the subscheme of points of exact order $p$ Z (see \cite[Remark 3.3.2]{HR12})  which allows the comparison with the formulation of the moduli problem in \cite[Chapter 3]{KM}. 
  Here and in the sequel, we are using  \cite[\S 3.3]{HR12} as a convenient
  summary of Oort-Tate theory (see also \cite{GT}) and of facts from \cite{KM} and \cite{DR}.

There is an equality of Cartier divisors 
\[
\frac{1}{24} \mathrm{div}( F_r^{24} ) =  \big( C[r] -e \big) \times_{ \mathcal{E} , e}  \mathcal{Y}_0(D)   = 
\sum_{  \substack{ s\mid r \\ s\neq1} }  \big( C[s]^\times  \times_{ \mathcal{E} , e}  \mathcal{Y}_0(D)  \big)
\]
on $\mathcal{Y}_0(D)$. Indeed, one can check this after pullback to $\mathcal{Y}_1(D)$, where it is clear from 
Proposition \ref{prop:integral jacobi}, which asserts that the  divisor of the section $\Theta^{24}$ appearing in the definition of $F_r^{24}$ is equal to $24 e$.  If $s$ is divisible by two distinct primes then 
\[
\big(C[s]^\times  \times_{ \mathcal{E} , e}  \mathcal{Y}_0(D)\big) = 0,
\]
and hence
\[
 \mathrm{div}( F_r^{24} ) = 24 \sum_{p\mid r}  \big( C[p]^\times  \times_{ \mathcal{E} , e}  \mathcal{Y}_0(D)  \big).
\]

Now pull back this equality of Cartier divisors by $j$.
Recall that $j$   is defined as the composition 
\[
\mathcal{A}_\Phi \iso \mathcal{M}_{(1,0)} \map{i} \mathcal{Y}_0(D),
\]
where the  isomorphism  is the one provided by  Proposition \ref{prop:second moduli iso}, 
and the arrow labeled $i$ endows the universal CM elliptic curve  $E\to  \mathcal{M}_{(1,0)}$ with its cyclic subgroup scheme $E[\delta]$.  Thus
\begin{equation}\label{OT-sub}
 i^*\mathrm{div}( F_r^{24} ) = 24 \sum_{p\mid r}  \big( E[\mathfrak{p}]^\times  \times_{ E , e}  \mathcal{M}_{(1,0)}  \big) ,
\end{equation}
where $\mathfrak p$ denotes the unique prime ideal in $\co_\kk$ over $p$. 

Fix a geometric point $z:\Spec(\F_\mathfrak{p}^\alg) \to \mathcal{M}_{(1,0)}$, and view $z$ also as a geometric point of $E$ or  $\mathcal{E}$ using
\[
\mathcal{M}_{(1,0)} \map{e}  E \map{i} \mathcal{E}.
\] 
Let $\co_{E,z}$ and $\co_{\mathcal{E},z}$ denote the completed \'etale local rings  of  $E$ and $\mathcal{E}$ at $z$.

 There is an isomorphism
\[
\co_{\mathcal{E},z} \iso W[[ X,Y,Z]] /( XY - w_p)
\]
for some uniformizer $w_p$ in the Witt ring $W=W(\F^\alg_\mathfrak{p})$.  Compare with  \cite[Theorem 3.3.1]{HR12}.
Under this isomorphism the $0$-section of $\mathcal{E}$ is defined by the equation $Z=0$, and the divisor $C[p]^\times$ is defined by $Z^{p-1}-X=0$.  
Moreover, noting that the completed \'etale local ring of $\mathcal{M}_{(1,0)}$ at $z$ can be identified with $\co_\kk\otimes W$,  the  natural map $\co_{\mathcal{E},z} \to \co_{E,z}$ is identified with the quotient map 
\[
W[[ X,Y,Z]] /( XY - w_p) \to  W[[ X,Y,Z]] /( XY - w_p , X-uY )
\]
for some $u\in W^\times$.

Under these identifications,  the closed immersion
\[
E[\mathfrak{p}]^\times  \times_{ E , e}  \mathcal{M}_{(1,0)}  \hookrightarrow    \mathcal{M}_{(1,0)}
\]  
corresponds, on the level of completed local rings,  to the quotient map
\[
\xymatrix{
{  \co_{\mathcal{M}_{(1,0)} , z }  }   \ar@{=}[r]  &   {   W[[ X,Y,Z  ]] /( XY - w_p , X-uY ,Z  )  }   \ar[d]  \\
{   \F^\alg_\mathfrak{p}  }   \ar@{=}[r] &    {    W[[ X,Y,Z ]] /( XY - w_p , X-uY , Z, Z^{p-1} -X )  . } 
}
\]
This implies that 
\[
E[\mathfrak{p}]^\times  \times_{ E , e}  \mathcal{M}_{(1,0)}  = \mathcal{M}_{(1,0)/\F_\mathfrak{p}^\alg} . 
\]
The equality (\ref{OT-main}) is clear from this and  (\ref{OT-sub}).
\end{proof}


\section{Calculation of the Borcherds product divisor}
\label{s:analytic borcherds}


In this section we prove Theorems \ref{thm:unitary borcherds I}, \ref{thm:unitary borcherds II}, and \ref{thm:unitary borcherds III}.   We assume throughout that  $n\ge 3$.

 Throughout \S \ref{s:analytic borcherds} we keep $f$ as in (\ref{input form}), and again assume that $c(-m)\in \Z$ for all $m\ge 0$.   
Recall  that $V=\Hom_\kk(W_0,W)$   is endowed with the hermitian form  $\langle x, y \rangle$ of (\ref{hom hermitian}), as well as  the $\Q$-bilinear form $[x,y]$ of (\ref{hom quadratic}).  The associated quadratic form is 
\[
Q(x) = \langle x,x\rangle= \frac{[x,x]} {2} . 
\]


\subsection{Vector-valued modular forms}
\label{ss:vector-valued}


Let $L\subset V$ be any  $\co_\kk$-lattice, self-dual with respect to the hermitian form.  The dual lattice of $L$ with respect to the bilinear form $[\cdot,\cdot]$ is $L'=\mathfrak{d}^{-1}L$.

Let $\omega$ be the restriction to $\SL_2(\Z)$ of the Weil representation of $\SL_2(\widehat \Q)$ (associated with the standard additive character of $\A/\Q$) on the Schwartz-Bruhat functions on $L\otimes_\Z\A_f$. The restriction of $\omega$ to $\SL_2(\Z)$
preserves the subspace $S_L=\C[L'/L]$ of Schwartz-Bruhat functions that are supported on $\widehat L'$ and invariant under translations by
$\widehat L$.  We obtain a representation
\[
\omega_L:\SL_2(\Z)\to \Aut(S_L).
\]
For $\mu\in L'/L$, we denote by $\phi_\mu\in S_L$ the characteristic function of $\mu$.    
 
\begin{remark}
The conjugate representation $\overline{\omega}_L$ on $S_L$, defined by
\[
\overline{\omega}_L(\gamma)(\phi)= \overline{\omega_L(\gamma)(\overline{\phi})}
\]
for $\phi\in S_L$,  is the representation denoted $\rho_L$ in \cite{Bo1, Br1, BF}.
\end{remark}

Recall the scalar valued modular form
\[
f (\tau) = \sum_{  m\gg -\infty } c(m)  \cdot q^m  \in M^{!,\infty}_{2-n}(D,\chi)
\]
 of (\ref{input form}), and continue to assume that $c(m)\in \Z$ for all $m\le 0$.  We will  convert $f$  into a $\C[L'/L]$-valued modular form $\tilde f$, to be used as  input for  Borcherds' construction of meromorphic modular forms on orthogonal Shimura varieties.  The restriction of $\omega_L$ to $\Gamma_0(D)$ acts on the line $\C \cdot \phi_0$ via the character $\chi=\chi_\kk^{n-2}$, and hence  the induced function
\begin{align}
\label{eq:vectorization}
\tilde f(\tau) = \sum_{\gamma \in \Gamma_0(D) \backslash \SL_2(\Z)} (f\mid_{2-n}\gamma)(\tau) \cdot \omega_L(\gamma)^{-1} \phi_0
\end{align}
is an $S_L$-valued weakly holomorphic modular form for $\SL_2(\Z)$ of weight $2-n$ with representation $\omega_L$.  Its Fourier expansion is denoted 
\begin{equation}\label{eq:tilde-f-Fourier}
\tilde f(\tau) =\sum_{m \gg -\infty}  \tilde c(m) \cdot q^m,
\end{equation}
and we denote by  $\tilde c(m,\mu)$ the value of $\tilde{c}(m)\in S_L$ at a coset $\mu \in L'/L$.

For any $r\mid D$ let $\gamma_r\in \{\pm 1,\pm i\}$ be as in (\ref{gamma def}), and let   $c_r(m)$ be the $m^\mathrm{th}$ Fourier coefficient of $f$ at the cusp $\infty_r$  as in (\ref{other cusps}).    For any $\mu\in L'/L$ define $r_\mu \mid D$ by 
\begin{equation}\label{eq:r-mu}
r_\mu= \prod_{\mu_p \neq 0} p,
\end{equation}
where $\mu_p\in L'_p/L_p$ is the  $p$-component of $\mu$.

\begin{proposition}\label{prop:promoted coefficients}
  For all $m\in \Q$ the coefficients $\tilde{c}(m)\in S_L$ satisfy
\[
\tilde c(m, \mu) =\begin{cases}
   \sum_{r_\mu \mid r \mid D}  \gamma_r   \cdot c_r(mr)  &\text{if $m \equiv -Q(\mu)  \pmod{ \Z}$,}
   \\ 
    0  &\text{otherwise}.
    \end{cases}
\]
Moreover, for $m<0$ we have
\[
\tilde c(m, \mu) = \begin{cases}
    c(m) &\text{if $\mu =0$,}
    \\
    0  &\text{if $\mu \neq 0$,}
    \end{cases}
\]
and the constant term of $\tilde f$ is given by
\[
\tilde c(0, \mu) = \sum_{r_\mu \mid  r \mid D}  \gamma_r \cdot c_r(0) .
\]
\end{proposition}

\begin{proof}
The first formula is a special case of  results of Scheithauer \cite[Section 5]{Scheithauer}.  
For the reader's benefit we provide a direct proof in \S \ref{ss:appendix promotion}.

The formula for the $m=0$ coefficient is immediate from the general formula.  So is  the formula for $m<0$, using the fact that the singularities of $f$ are supported at the cusp at $\infty$.
\end{proof}

\begin{remark}
The first formula of Proposition \ref{prop:promoted coefficients} actually also holds for $f$ in the larger space $M_{2-n}^!(D,\chi)$.
\end{remark}

\begin{corollary}\label{cor:vector rational}
The coefficients $c(m)$ and $\tilde{c}(m)$ satisfy the following:
\begin{enumerate}
\item
The   $c(m)$ are rational for all $m$.
\item
The  $\tilde c(m, \mu)$ are rational for all $m$ and $\mu$, and are integral if $m<0$.
\item
For all $r\mid D$ we have $\gamma_r \cdot c_r(0)  \in \Q$. In particular
\[
\tilde c(0, 0) = \sum_{r\mid D}   \gamma_r \cdot c_r(0) \in \Q.
\]
\end{enumerate}
\end{corollary}

\begin{proof}
For the first claim, fix any $\sigma \in \Aut(\C/\Q)$.  
The form $ f^\sigma - f \in M_{2-n}^{!, \infty}$ is  holomorphic at all cusps other than $\infty$, and vanishes at the cusp $\infty$ by the assumption that  as $c(m) \in \Z$ for $m \le 0$.
Hence $ f^\sigma - f$  is a holomorphic modular form of  weight  $2-n<0$, and therefore vanishes identically. 
 It follows that  $c(m)\in \Q$ for all $m$.

Now consider the second claim.   In view of the Proposition \ref{prop:promoted coefficients}  the coefficients
 $\tilde c(m,\mu)$ of $\tilde f$ with $m<0$ are integers. Hence,  for any $\sigma \in \Aut(\C/\Q)$,  the function $\tilde f^\sigma - \tilde f$
 is a holomorphic modular form of weight $2-n <0$, which is therefore identically $0$.  Therefore $\tilde f$ has rational Fourier coefficients.
%

The third claim follows from the second claim and the formula for the constant term of $\tilde f$ given in Proposition \ref{prop:promoted coefficients}.
\end{proof}


\subsection{Construction of the Borcherds product}
\label{ss:borcherds define}


We now construct the Borcherds product $\bm{\psi}(f)$ of Theorem \ref{thm:unitary borcherds I} as the pullback of a Borcherds product  on the orthogonal Shimura variety defined by the quadratic space $(V,Q)$.    Useful references here include \cite{Bo1,Br1,KuBorcherds, Ho14}.

After Corollary \ref{cor:vector rational} we may replace $f$ by a positive integer multiple in order to assume that $c(-m) \in 24\Z$ for all $m\ge 0$, and that $\gamma_r c_r(0) \in 24\Z$ for all $r\mid D$.  In particular the rational number
\[
k =  \tilde{c}(0,0)
\]
 of Corollary \ref{cor:vector rational} is an integer.  Compare with Remark \ref{rem:very divisible}.

Define a hermitian domain
\begin{equation}\label{first isotropy}
\tilde{\mathcal{D}} = \{ w\in V(\C) : [w,w]=0,\, [w, \overline{w}]<0 \} /\C^\times.
\end{equation}
Let $\tilde{\bm{\omega}}^{an}$ be the tautological bundle  on $\tilde{\mathcal{D}}$, whose fiber at $w$ is the line $\C w\subset V(\C)$.  The group of real points of $\SO(V)$ acts on (\ref{first isotropy}), and this action lifts to an action on $\tilde{\bm{\omega}}^{an}$.

As  in Remark \ref{rem:to so}, any point $z\in \mathcal{D}$  determines a line $\C w \subset \epsilon V(\C).$  This construction defines a closed immersion 
\begin{equation}\label{utoso}
\mathcal{D} \hookrightarrow \tilde{\mathcal{D}} ,
\end{equation}
under which $\tilde{\bm{\omega}}^{an}$ pulls back to the line bundle $\bm{\omega}^{an}$ of \S \ref{ss:unitary bundle}.     
The hermitian domain $\tilde{\mathcal{D}}$ has two connected components.
Let  $\tilde{\mathcal{D}}^+ \subset \tilde{\mathcal{D}}$ be the connected component containing  $\mathcal{D}$.

For a fixed $g\in G(\A_f)$, we apply the constructions of \S \ref{ss:vector-valued} to the input form $f$ and the  self-dual hermitian $\co_\kk$-lattice 
 \[
L= \Hom_{\co_\kk}( g\mathfrak{a}_0, g\mathfrak{a})  \subset V.
\]
 The result is a vector-valued modular form $\tilde{f}$  of weight $2-n$ and representation  $\omega_{L}:\SL_2(\Z) \to S_L$. 
The  form $\tilde{f}$ determines a Borcherds product $\Psi(\tilde{f})$ on $\tilde{\mathcal{D}}^+$; see  \cite[Theorem 13.3]{Bo1} and Theorem~\ref{prop:greenbp}. For us it is more convenient to use the rescaled Borcherds product
\begin{equation}\label{B renorm}
\tilde{\bm{\psi}}_g( f ) = (2\pi i)^{ \tilde{c}(0,0) } \Psi(2 \tilde{f})
\end{equation}
determined by $2 \tilde{f}$.   It is  a meromorphic section  of  $(\tilde{\bm{\omega}}^{an})^k$.

The subgroup $\SO(L)^+ \subset \SO(L)$  of elements preserving the component $\tilde{\mathcal{D}}^+$  acts  on $\tilde{\bm{\psi}}_g( f )$ through a finite order character \cite{Bo2:correction}.  Replacing $f$ by $m f$  has the effect of replacing $\tilde{\bm{\psi}}_g( f )$ by  $\tilde{\bm{\psi}}_g( f )^m$, and so after replacing $f$ by  a multiple we assume that $\tilde{\bm{\psi}}_g( f )$ is invariant under this action.

Denote by $\bm{\psi}_g(f)$ the pullback of  $\tilde{\bm{\psi}}_g( f )$ via the map
\[
 (  G(\Q) \cap g K g^{-1}  ) \backslash \mathcal{D}  \to \SO(L)^+ \backslash \tilde{\mathcal{D}}^+
\]
 induced by (\ref{utoso}).   It is a meromorphic section of  $(\bm{\omega}^{an})^k$ on the  connected component 
\[
(  G(\Q) \cap g K g^{-1}  ) \backslash \mathcal{D} \map{z\mapsto (z,g)} \mathrm{Sh}(G,\mathcal{D})(\C).
\]
By repeating the construction for all $g\in G(\Q) \backslash G(\A_f) /K$, we obtain a meromorphic section $\bm{\psi}(f)$ of the line bundle $(\bm{\omega}^{an})^k$ on  
\[
\mathrm{Sh}(G,\mathcal{D})(\C) \iso \mathcal{S}_\Kra(\C).
\]
After rescaling on every connected component by a complex constant of absolute value $1$,  this will be  the section whose existence is asserted in  Theorem \ref{thm:unitary borcherds I}.

\begin{proposition}\label{prop:complex borcherds divisor}
The divisor of   $\bm{\psi}(f)$ is
\[
\mathrm{div}( \bm{\psi}(f) ) = \sum_{m>0} c(-m)  \cdot \mathcal{Z}_\Kra(m)(\C).
\]
\end{proposition}

\begin{proof}
The divisor of $\tilde{\bm{\psi}}_g(f)$ on $\tilde{\mathcal{D}}^+$ was computed by Borcherds in terms of the Fourier coefficients $\tilde{c}(-m)$ of $\tilde{f}$, and from this it is easy to obtain a formula for the divisor of $\bm{\psi}_g(f)$ on $\mathcal{D}$.   See \cite[Theorem 3.22]{Br1} and \cite[Theorem 8.1]{Ho14} for the details.
The claim therefore follows by using Proposition \ref{prop:promoted coefficients} to rewrite this formula in terms of the $c(-m)$, and comparing with the explicit description of $\mathcal{Z}_\Kra(m)(\C)$ stated in  Remark \ref{rem:divisor uniformization}.
\end{proof}


\subsection{Analytic Fourier-Jacobi coefficients}
\label{ss:analytic-FJ}


We return to the notation of \S \ref{ss:explicit boundary}. 
 Thus $\Phi=(P,g)$ is a proper cusp label representative for $(G,\mathcal{D})$,  we have chosen 
\[
s:\mathrm{Res}_{\kk/\Q}\mathbb{G}_m \to Q_\Phi
\] 
as in Lemma \ref{lem:mixed section}, and have fixed  $a\in \widehat{\kk}^\times$.
This data determines a lattice
\[
L = \Hom_{\co_\kk}( s(a) g\mathfrak{a}_0  , s(a) g\mathfrak{a} ), 
\]
and  Witt decompositions  
\[
V=V_{-1}\oplus V_0\oplus V_1 ,\quad L=L_{-1}\oplus L_0 \oplus L_1.
\] 
Choose bases $\eee_{-1},\fff_{-1} \in L_{-1}$ and  $\eee_1 , \fff_1\in L_1$ as in   \S \ref{ss:explicit boundary}.

Imitating the construction of (\ref{siegel coords}) yields a commutative diagram
 \[
\xymatrix{
{  \mathcal{D} }\ar[rr]^{ (\ref{utoso})} \ar[d]_{ w\mapsto (w_0,\xi)}  & & { \tilde{\mathcal{D}}^+ } \ar[d]^{ w \mapsto (\tau, w_0, \xi) } \\
 {     \epsilon V_0(\C) \times \C  }  \ar[rr]  &  &   {  \mathfrak{H} \times  V_0(\C) \times \C   }
 }
\]
in which the vertical arrows are open immersions, and the horizontal arrows are closed immersions.   The vertical arrow on the right is defined as follows:   Any  $w\in \tilde{\mathcal{D}}$ pairs nontrivially with the isotropic vector $\mathrm{f}_{-1}$, and so may be scaled so that $[w,\mathrm{f}_{-1}]=1$.  This allows us to identify
\[
\tilde{\mathcal{D}} =  \{ w\in V(\C) : [w,w]=0,\, [w, \overline{w}]<0,  \, [w, \fff_{-1}]=1 \} .
\]
Using this model, any $w\in \tilde{\mathcal{D}}^+$ has the form
\[
w= -\xi  \eee_{-1} +(\tau \xi - Q(w_0)) \fff_{-1} + w_0 + \tau \eee_1 + \fff_1
\]
with  $\tau \in \mathfrak{H}$,  $w_0 \in V_0(\C)$, and $\xi\in \C$.  The bottom horizontal arrow is $(w_0,\xi) \mapsto (\tau,w_0,\xi)$, where $\tau$ is determined by the relation (\ref{the tau}).

  The construction above singles out a nowhere vanishing section of $ \tilde{\bm{\omega}}^{an}$,  whose value at an isotropic line $\C w$ is the unique vector in that line with $[w, \fff_{-1}]=1$.  As in the discussion leading to (\ref{analytic FJ}), we obtain a trivialization
  \[
[\,\cdot\, , \mathrm{f}_{-1} ] :   \tilde{\bm{\omega}}^{an} \iso \co_{\tilde{\mathcal{D}}^+ }.
  \]
  
 Now consider the Borcherds product $\tilde{\bm{\psi}}_{s(a)g}( f )$ on $\tilde{\mathcal{D}}^+$ determined by the lattice $L$ above (that is, replace $g$ by $s(a)g$  throughout \S \ref{ss:borcherds define}).   It is a meromorphic section of  $(\tilde{\bm{\omega}}^{an})^k$, and we use the trivialization above to identify it with a meromorphic function.   In a neighborhood of the rational boundary component  associated to the isotropic plane $V_{-1}\subset V$,   this meromorphic function   has a product expansion.

\begin{proposition}[{\cite{Ku:ABP}}] \label{prop:ortho FJ formula}
There are positive constants $A$ and $B$ with the following property:  For all points $w\in \tilde{\mathcal{D}}^+$  satisfying 
\[
\Im(\xi)- \frac{ Q(\Im(w_0))}{ \Im(\tau )} > A \Im(\tau ) + \frac{ B} {\Im(\tau )} ,
\]
 there is a factorization 
\[
\tilde{\bm{\psi}}_{s(a)g}( f )
= \kappa \cdot  ( 2\pi i )^k  \cdot \eta^{2k}(\tau) \cdot   e^{2\pi i I\xi  } \cdot    P_0(\tau)  \cdot P_1(\tau,w_0)  \cdot  P_2(\tau,w_0,\xi)
\]
in which  $\kappa \in \C^\times$ has  absolute value $1$, $\eta$ is the Dedekind $\eta$-function, and
\[
I  = \frac{1}{12}\sum_{\substack{b\in \Z/D\Z}} \tilde{c} \left(0, - \frac{b}{D}\fff_{-1} \right) -2\sum_{m>0} \sum_{x\in L_0} c(-m ) \cdot  \sigma_1(m-Q(x)).
\]
The factors $P_0$ and $P_1$ are defined by
\[
P_0 (\tau) =\prod_{\substack{b\in \Z/D\Z\\  b\ne 0}}   \Theta\left(  \tau,  \frac{b}{D}  \right)^{ \tilde{c} (0, \frac{b}{D}  \fff_{-1})   }
\]
and
\[
P_1(\tau,w_0) =  \prod_{m>0} \prod_{\substack{x\in L_0\\   Q(x) = m }}  \Theta\big(  \tau,   [  w _0,x]   \big)^{ c(-m)}.
\]
The remaining factor  is
\[
P_2(\tau,w_0,\xi) =  \prod_{\substack{ x\in \delta^{-1} L_0  \\ a\in \Z  \\  b \in \Z/D\Z  \\   c\in \Z_{>0}  }}
  \Big(1- e^{2\pi i c \xi}  e^{2\pi i a\tau}   e^{  2\pi i b / D }  e^{- 2\pi i [x,w _0]} \Big) ^{  2\cdot \tilde{c} ( ac-Q(x) , \mu)},
\]
where $\mu = -a \eee_{-1} - \frac{b}{D}  \fff_{-1} +x  + c \eee_1 \in \delta^{-1}L / L$.   
 \end{proposition}

 \begin{proof}
This is just a restatement of \cite[Corollary 2.3]{Ku:ABP},  with some simplifications arising from the fact that the vector-valued form $\tilde{f}$ used to define the Borcherds product is induced from a scalar-valued form via (\ref{eq:vectorization}).

A more detailed description of how these expressions arise from the general formulas in \cite{Ku:ABP} is given in the appendix. 
\end{proof}

If we  pull back the formula for the Borcherds product  $\tilde{\bm{\psi}}_{s(a)g}( f )$  found in Proposition  \ref{prop:ortho FJ formula} via the closed immersion (\ref{utoso}), we obtain a formula for the Borcherds product $\bm{\psi}_{s(a)g}(f)$ on the  connected component 
\[
(  G(\Q) \cap s(a) g K g^{-1}s(a)^{-1}  ) \backslash \mathcal{D} \map{z\mapsto (z,s(a) g)} \mathrm{Sh}(G,\mathcal{D})(\C),
\]
 from which we can read off the leading analytic Fourier-Jacobi coefficient.

\begin{corollary}\label{cor:analytic BFJ}
The  analytic Fourier-Jacobi expansion of $\bm{\psi}(f)$,  in the sense of (\ref{analytic FJ}),  has the form 
\[
\bm{\psi}_{s(a) g} (f) =  \sum_{ \ell\ge I } \mathrm{FJ}^{(a)}_\ell ( \bm{\psi}(f) ) (w_0) \cdot q^\ell ,
\]
where $I$ is the integer of Proposition \ref{prop:ortho FJ formula}.
The  leading coefficient $\mathrm{FJ}^{(a)}_I( \bm{\psi}(f))$, viewed as a function on  $V_0(\R)$ as in the discussion leading to (\ref{FJ-trans}),
 is given by
\begin{equation}\label{analytic leading}
\mathrm{FJ}^{(a)}_I ( \bm{\psi}(f) ) (w_0) = \kappa \cdot (2\pi i)^k \cdot  \eta(\tau)^{2k} \cdot P_0(\tau)  \cdot P_1( \tau, w_0 ),
\end{equation}
where  $\tau\in \mathfrak{H}$ is determined by  (\ref{the tau}), 
\[
P_0(\tau) = \prod_{r\mid D}  \prod_{\substack{b\in \Z/D\Z\\  b\ne 0 \\ rb=0 }}   \Theta\left(  \tau,  \frac{b}{D}  \right)^{ \gamma_r c_r (0)   }
\]
and
\[
P_1(\tau,w_0) =  \prod_{m>0} \prod_{\substack{x \in L_0\\   Q(x) = m }}  \Theta\big(  \tau,   \langle  w _0, x \rangle  \big)^{ c(-m)}.
\]
The constant $\kappa\in \C$, which depends on both $\Phi$ and $a$,  has absolute value $1$.
\end{corollary}

\begin{proof}
Using Proposition  \ref{prop:ortho FJ formula}, the pullback of $\tilde{\bm{\psi}}_{s(a)g}( f )$  via (\ref{utoso}) factors as a product
\[
\bm{\psi}_{s(a)g}(f) =  \kappa \cdot (2\pi i)^k \cdot  \eta^{2k}(\tau)\cdot  e^{2\pi i \xi I}  \cdot P_0(\tau) P_1(\tau,w_0) P_2(\tau,w_0,\xi),
\]
where  $\xi\in \C^\times$ and $w_0 \in  V(\R) \iso \epsilon V(\C)$.  The parameter  $\tau\in \mathfrak{H}$ is now fixed, determined by (\ref{the tau}).  
The equality 
\[
\prod_{\substack{b\in \Z/D\Z\\  b\ne 0}}   \Theta\left(  \tau,  \frac{b}{D}  \right)^{ \tilde{c} (0, \frac{b}{D}  \fff_{-1})   }
=
\prod_{r\mid D}  \prod_{\substack{b\in \Z/D\Z\\  b\ne 0 \\ rb=0 }}   \Theta\left(  \tau,  \frac{b}{D}  \right)^{ \gamma_r c_r (0)   }
\]
follows from Proposition \ref{prop:promoted coefficients}, and allows us to rewrite  $P_0$ in the stated form.  To rewrite the factor $P_1$ in terms of $\langle\cdot,\cdot\rangle$ instead of $[\cdot,\cdot]$, use the commutative diagram of Remark \ref{rem:switch to unitary}.  Finally, as $\mathrm{Im}(\xi)\to \infty$, so $q=e^{2\pi i \xi} \to 0$,  the factor $P_2$ converges to $1$.  This $P_2$  does not contribute to the leading Fourier-Jacobi coefficient.
\end{proof}

\begin{proposition}\label{prop:other mult}
The integer $I$ defined in Proposition \ref{prop:ortho FJ formula} is equal to the integer $\mathrm{mult}_\Phi(f)$ defined by (\ref{f boundary mult}), and the product (\ref{analytic leading}) satisfies the transformation law (\ref{FJ-trans}) with $\ell=\mathrm{mult}_\Phi(f)$.
\end{proposition}

\begin{proof}
The Fourier-Jacobi coefficient $\mathrm{FJ}^{(a)}_I ( \bm{\psi}(f) )$ appearing on the left hand side of (\ref{analytic leading}) is, by definition,  a section  of the line bundle $\mathcal{Q}_{E^{(a)} \otimes L }^I$ 
on $E^{(a)} \otimes L$.  When viewed as a function of the variable $w_0 \in V_0(\R)$ using our explicit coordinates, it therefore satisfies the transformation law (\ref{FJ-trans}) with $\ell=I$.

Now consider the right hand side of (\ref{analytic leading}), and recall that $\tau$ is fixed, determined by (\ref{the tau}).
In our explicit coordinates the function  $\Theta( \tau, \langle w_0,x \rangle )^{24}$  of $w_0 \in V_0(\R)$ is identified with a section of the line bundle $j_x^* \mathcal{J}_{0,12}$ on $E^{(a)} \otimes L$; this is clear from the definition of $j_x$ in (\ref{boundary collapse}), and Proposition \ref{prop:integral jacobi}.    Thus $P_1(\tau,w_0)$, and hence the entire  right hand side of  (\ref{analytic leading}), defines a section of the line bundle 
\[
\bigotimes_{m>0} \bigotimes_{\substack{ x\in L_0 \\ Q(x) = m } } j_x^* \mathcal{J}_{0,1}^{c(-m)/2}
 \iso \mathcal{L}_\Phi^{  2  \cdot \mathrm{mult}_\Phi(f/2) } \iso  \mathcal{Q}_{E^{(a)} \otimes L }^{\mathrm{mult}_\Phi(f) },
\]
where the isomorphisms are those of  Proposition \ref{prop:quadratic bundles} and  Proposition \ref{prop:second moduli iso}.   
This implies  that the right hand side of (\ref{analytic leading}) satisfies the transformation law (\ref{FJ-trans}) with $\ell=\mathrm{mult}_\Phi(f)$.

A function on $V_0(\R)$ cannot satisfy the transformation law (\ref{FJ-trans}) for two different values of $\ell$, and hence $I = \mathrm{mult}_\Phi(f)$.  Note that we are using here the standing hypothesis $n>2$; if $n=2$ then $V_0(\R)=0$, and the transformation law (\ref{FJ-trans}) is vacuous.

  For a more direct proof of the proposition, see \S \ref{ss:direct mult}.
\end{proof}


\subsection{Algebraization and descent}
\label{ss:alganddescent}


The following weak form of Theorem \ref{thm:unitary borcherds I} shows that  $\bm{\psi}(f)$  is algebraic, and provides an algebraic interpretation of its leading Fourier-Jacobi coefficients.

\begin{proposition}\label{prop:algebraic BFJ}
The meromorphic section $\bm{\psi}(f)$  is the analytification of a rational section of the line bundle $\bm{\omega}^k$ on  $\mathcal{S}_{\Kra /\C} $.   This rational section satisfies the following properties:
\begin{enumerate}
\item
When viewed as a rational section over the toroidal compactification, 
\[
\mathrm{div}( \bm{\psi}  (f) )  =   \sum_{m>0} c(-m) \cdot \mathcal{Z}^*_\Kra(m)_{/\C}   
+ \sum_\Phi \mathrm{mult}_\Phi(f) \cdot \mathcal{S}^*_\Kra(\Phi)_{ / \C}.
\]
\item
For every proper cusp label representative $\Phi$, the Fourier-Jacobi expansion of $\bm{\psi}(f)$ along $\mathcal{S}^*_\Kra(\Phi)_{/\C}$, in the sense of \S \ref{ss:abstract FJ},  has the form
\[
\bm{\psi}(f) = q^{ \mathrm{mult}_\Phi(f) } \sum_{ \ell\ge 0} \bm{\psi}_\ell \cdot q^\ell.
\]
\item
The  leading coefficient $\bm{\psi}_0$, a rational section of  $\bm{\omega}_\Phi^k \otimes \mathcal{L}_\Phi^{\mathrm{mult}_\Phi(f)}$ over $\mathcal{B}_{\Phi/\C}$,   factors as
\[
\bm{\psi}_0 =  \kappa_\Phi \otimes P_\Phi^\eta \otimes P^{hor}_\Phi \otimes P^{vert}_\Phi
\]
for a unique section 
\[
\kappa_\Phi  \in H^0( \mathcal{A}_{\Phi/\C} , \co_{\mathcal{A}_\Phi /\C}^\times).
\]
This section  satisfies $|\kappa_\Phi(z)| =1$ at every complex point $z\in  \mathcal{A}_\Phi(\C)$. 
(The other factors appearing on the right hand side were defined in Theorem \ref{thm:unitary borcherds I}.) 
\end{enumerate}
\end{proposition}

\begin{proof}
Using Corollary \ref{cor:analytic BFJ}  and Proposition \ref{prop:other mult}, one  sees that $\bm{\psi}(f)$ extends to a meromorphic section of $\bm{\omega}^k$ over the toroidal compactification $\mathcal{S}^*_\Kra(\C)$, vanishing to order $I=\mathrm{mult}_\Phi(f)$ along the closed stratum 
\[
\mathcal{S}_\Kra^*(\Phi)_{/\C} \subset \mathcal{S}^*_{\Kra/\C}
\]
 indexed by a proper cusp label representative $\Phi$.

The calculation of the divisor of $\bm{\psi}(f)$ over the open Shimura variety $\mathcal{S}_\Kra(\C)$ is Proposition \ref{prop:complex borcherds divisor}. 
The algebraicity claim now follows from GAGA (using the fact that the divisor is already known to be algebraic), proving all parts of  the first claim.
The second and third claims are just a  translation of  Corollary \ref{cor:analytic BFJ} into the algebraic language of Theorem \ref{thm:unitary borcherds I}, using the explicit coordinates of \S \ref{ss:explicit boundary} and the change of notation $(2\pi i \eta^2)^k = P_\Phi^\eta$, $P_0=P_\Phi^{vert}$ and $P_1=P_\Phi^{hor}$.
\end{proof}

We now prove that $\bm{\psi}(f)$, after minor rescaling, descends to $\kk$.
This can be deduced from the analogous statement about Borcherds products on orthogonal Shimura varieties proved in \cite{HMP}, but in the unitary case there is a much more elementary proof.  
This will require the following two lemmas.

\begin{lemma}
The geometric components of $\mathrm{Sh}(G,\mathcal{D})$ are defined over the Hilbert class field $\kk^\mathrm{Hilb}$ of $\kk$, and  each such component has trivial stabilizer in $\Gal( \kk^\mathrm{Hilb} /\kk)$. 
\end{lemma}

\begin{proof}
One could prove this using Deligne's reciprocity law for connected components of Shimura varieties \cite[\S 13]{Milne}, but it also follows easily from the theory of toroidal compactification.

Our assumption that $n>2$  guarantees that every connected  component of $\mathcal{S}^*_{\Kra/\C}$ contains some connected component of the boundary.
It is a part\footnote{This  particular part of  Theorem \ref{thm:toroidal}  follows from  the reciprocity law for the boundary components of $\mathcal{M}_{(n-1,1)}^\Pap$ proved in \cite[Proposition 2.6.2]{Ho2}.}  of Theorem \ref{thm:toroidal} that  all such boundary components  are defined over the Hilbert class field, and it follows that the same is true for components of  $\mathcal{S}^*_{\Kra/\C}$.
The same is therefore true for the components of the interior
\[
\mathcal{S}_{\Kra/\C} \iso \mathrm{Sh}(G,\mathcal{D})_{/\C}.
\]

The claim about  stabilizers follows from  the open and closed immersion
\[
\mathrm{Sh}(G,\mathcal{D}) \subset M_{(1,0)}  \times_\kk M_{(n-1,1)}
\]
of (\ref{moduli inclusion}), along with the classical fact (from the theory of complex multiplication of elliptic curves) that the geometric components of $M_{(1,0)}$ form a simply transitive $\Gal(\kk^\mathrm{Hilb}/\kk)$-set.
\end{proof}

The lemma allows us to choose  a set of connected components 
\[
 \{ X_i \}  \subset \pi_0 \big(  \mathrm{Sh}(G,\mathcal{D})_{/\kk^\mathrm{Hilb}} )
 \]
   in such a way that 
\[
\mathrm{Sh}(G,\mathcal{D})_{/\kk^\mathrm{Hilb}}  = \bigsqcup_i \bigsqcup_{\sigma \in \Gal(\kk^\mathrm{Hilb}/\kk) }  \sigma (X_i) .
\]
For each index $i$,  pick  $g_i\in G(\A_f)$ in such a way that $X_i(\C)$ is equal to the image of 
\[
 ( G(\Q) \cap g_i K g_i^{-1} ) \backslash \mathcal{D} \map{z\mapsto (z,g_i)} \mathrm{Sh}(G,\mathcal{D})(\C).
\]
Choose  an isotropic $\kk$-line $J\subset W$, let $P\subset G$ be its stabilizer, and define a proper cusp label representative $\Phi_i=(P,g_i)$.   The above choices pick out one boundary component on every component of the toroidal compactification, as the following lemma demonstrates.

\begin{lemma}\label{lem:enough cusps}
The natural maps
\[
\xymatrix{
 & &   {   \bigsqcup_i \mathcal{S}^*_\Kra(\Phi_i) } \ar[r]  \ar[dd]^{\iso} & {  \mathcal{S}^*_\Kra }   \ar[dd] \\ 
   {   \bigsqcup_i \mathcal{A}_{\Phi_i}  }   &  {   \bigsqcup_i \mathcal{B}_{\Phi_i} }   \ar[ur]\ar[dr]  \ar[l] \\
 & &  {   \bigsqcup_i \mathcal{S}^*_\Pap(\Phi_i) } \ar[r]    &  {  \mathcal{S}^*_\Pap }      
}
\]
induce  bijections on connected components.  The same is true after base change to $\kk$ or  $\C$.
\end{lemma}

\begin{proof}
Let $X_i^* \subset \mathcal{S}^*_\Pap(\C)$ be the closure of $X_i$.  By examining the complex analytic construction of the toroidal compactification \cite{Ho2,Lan12,Pink}, one sees that some component of the divisor $\mathcal{S}^*_\Pap(\Phi_i) (\C)$ lies on $X_i^*$.

Recall from Theorem \ref{thm:toroidal} that the  components of $\mathcal{S}^*_\Pap(\Phi_i) (\C)$ are defined over $\kk^\mathrm{Hilb}$, and that the action of  $\Gal(\kk^\mathrm{Hilb}/\kk)$ is simply transitive.  It follows immediately that 
\[
\mathcal{S}^*_\Pap(\Phi_i) (\C) \subset \bigsqcup_{\sigma \in \Gal(\kk^\mathrm{Hilb}/\kk) } \sigma(X_i^*),
\]
and the inclusion induces a bijection on  components.   By Proposition \ref{prop:mixed connected} and the isomorphism of Proposition \ref{prop:boundary uniformization}, the quotient map
\[
\mathcal{C}_\Phi(\C) \to \Delta_{\Phi_i} \backslash \mathcal{C}_{\Phi_i}(\C) 
\]
induces a bijection on connected components, and  both maps $\mathcal{C}_\Phi \to \mathcal{B}_\Phi\to \mathcal{A}_\Phi$ have geometrically connected fibers (the first is a $\mathbb{G}_m$-torsor, and the second is an abelian scheme).  We deduce that all maps in 
\[
 \mathcal{A}_{\Phi_i}(\C) \leftarrow \mathcal{B}_{\Phi_i}(\C) \to \Delta_{\Phi_i} \backslash \mathcal{B}_{\Phi_i}(\C)
  \iso \mathcal{S}^*_\Kra(\Phi_i) (\C) \iso \mathcal{S}^*_\Pap(\Phi_i) (\C)
\] 
induce  bijections on connected components.

The above proves the claim over $\C$, and  the claim over $\kk$ follows formally from this.  The claim for integral models follows from the claim in the generic fiber, using the fact that all integral models in question are normal  and flat over $\co_\kk$.
\end{proof}

\begin{proposition}\label{prop:B descent}
After possibly rescaling by a constant of absolute value $1$ on every connected component of $\mathcal{S}^*_{\Kra/\C}$, the Borcherds product $\bm{\psi}(f)$ is defined over $\kk$, and the sections of Proposition \ref{prop:algebraic BFJ} satisfy 
\[
\kappa_\Phi \in  H^0( \mathcal{A}_{\Phi/\kk} , \co_{\mathcal{A}_\Phi /\kk}^\times)
\]
for all proper cusp label representatives $\Phi$.  Furthermore,  we may arrange that $\kappa_{\Phi_i}=1$ for all $i$. 
\end{proposition}

\begin{proof}
Lemma \ref{lem:enough cusps} establishes a bijection between the connected components of $\mathcal{S}^*_\Kra(\C)$ and the finite set 
$\bigsqcup_i \mathcal{A}_{\Phi_i}(\C)$.  On the component indexed by $z\in \mathcal{A}_{\Phi_i}(\C)$,  rescale $\bm{\psi}(f)$ by $\kappa_{\Phi_i}(z)^{-1}$.    For this rescaled $\bm{\psi}(f)$ we have $\kappa_{\Phi_i}=1$ for all $i$.

Suppose $\sigma \in \Aut(\C/\kk)$. The first claim of Proposition \ref{prop:algebraic BFJ} implies that the divisor of $\bm{\psi}(f)$, when computed on the compactification $\mathcal{S}^*_{\Kra /\C}$, is defined over $\kk$.  Therefore $\sigma(\bm{\psi}(f)) / \bm{\psi}(f)$ has trivial divisor, and so is constant on every connected component.   

By the third claim of Proposition \ref{prop:algebraic BFJ}, the leading coefficient in the Fourier-Jacobi expansion of $\bm{\psi}(f)$ along the boundary stratum $\mathcal{S}^*_\Kra(\Phi_i)$  is 
\[
\bm{\psi}_0 = P^\eta_{\Phi_i} \otimes P^{hor}_{\Phi_i} \otimes P^{vert}_{\Phi_i},
\]
which is defined over $\kk$.  From this it follows that    $\sigma(\bm{\psi}(f)) / \bm{\psi}(f)$  is identically equal to  $1$ on every connected component of $\mathcal{S}^*_{\Kra/\C}$ meeting this boundary stratum.  Varying $i$  and using Lemma \ref{lem:enough cusps} shows that $\sigma(\bm{\psi}(f)) = \bm{\psi}(f)$.

This proves that $\bm{\psi}(f)$ is defined over $\kk$, hence so are all of its Fourier-Jacobi coefficients along \emph{all} boundary strata $\mathcal{S}^*_\Kra(\Phi)$.  Appealing again to the calculation of the leading Fourier-Jacobi coefficient of Proposition \ref{prop:algebraic BFJ}, we deduce finally that $\kappa_\Phi$ is defined over $\kk$ for all $\Phi$.
\end{proof}


\subsection{Calculation of the divisor, and completion of the proof}
\label{ss:conclude pf}


The Borcherds product $\bm{\psi}(f)$ on $\mathcal{S}^*_{\Kra/\kk}$ of  Proposition \ref{prop:B descent} may   be viewed as a rational section of $\bm{\omega}^k$ on the integral model $\mathcal{S}^*_\Kra$.

Let $\Phi$ be any  proper cusp label representative.  Combining Propositions \ref{prop:algebraic BFJ} and \ref{prop:B descent} shows that the leading Fourier-Jacobi coefficient of $\bm{\psi}(f)$ along the boundary divisor $\mathcal{S}^*_\Kra(\Phi)$ is
\begin{equation}\label{leading factor}
\bm{\psi}_0 = \kappa_\Phi \otimes P_\Phi^\eta \otimes P^{hor}_\Phi \otimes P^{vert}_\Phi.
\end{equation}
Recall that this is a rational section of $\bm{\omega}_\Phi^k \otimes \mathcal{L}_\Phi^{\mathrm{mult}_\Phi(f)}$ on $\mathcal{B}_\Phi$.  
Here, by mild abuse of notation, we are viewing $\kappa_\Phi$ as a rational function on $\mathcal{A}_\Phi$, and denoting in the same way its pullback to any step in the tower  
\[
\mathcal{C}_\Phi^* \map{\pi} \mathcal{B}_\Phi \to \mathcal{A}_\Phi.
\]

\begin{lemma}\label{lem:leading FJ suffices}
Recall that $\pi$ has a canonical section $\mathcal{B}_\Phi \hookrightarrow \mathcal{C}_\Phi^*$, realizing $\mathcal{B}_\Phi$ as a divisor on $\mathcal{C}_\Phi^*$.    If we use  the isomorphism  (\ref{formal boundary iso}) to view $\bm{\psi}(f)$ as a rational section of the line bundle $\bm{\omega}_\Phi^k$ on the formal completion $(\mathcal{C}_\Phi^* )^\wedge_{\mathcal{B}_\Phi}$,  its divisor satisfies
\begin{align*}
\mathrm{div}( \bm{\psi}(f) )  
& =    \mathrm{div}( \delta^{-k} \kappa_\Phi) +   \mathrm{mult}_\Phi(f)\cdot  \mathcal{B}_\Phi\\
& \quad   + \sum_{m>0} c(-m)  \mathcal{Z}_\Phi (m)  + \sum_{ r \mid D} \gamma_r c_r(0)  \sum_{p\mid r} \pi^* (\mathcal{B}_{\Phi/\F_\mathfrak{p}})  .
\end{align*}
\end{lemma}

\begin{proof}
The key step is to prove that the divisor of $\bm{\psi}(f)$ can be computed from the divisor of its leading Fourier-Jacobi coefficient $\bm{\psi}_0$ by the formula
\begin{equation}\label{divisor near boundary}
\mathrm{div}( \bm{\psi}(f) )  
 = \pi^*\mathrm{div}( \bm{\psi}_0 ) + \mathrm{mult}_\Phi(f)\cdot  \mathcal{B}_\Phi.
\end{equation}
Recalling  the tautological section $q$ with divisor $\mathcal{B}_\Phi$ from Remark \ref{rem:q}, consider  the rational section 
\[
R = q^{  -\mathrm{mult}_\Phi(f)  } \cdot \bm{\psi} (f) = \sum_{i\ge 0} \bm{\psi}_i \cdot q^i
\]
of $\bm{\omega}_\Phi^k\otimes \pi^*\mathcal{L}_\Phi^{ \mathrm{mult}_\Phi(f)  }$ on the formal completion $(\mathcal{C}^*_\Phi)^\wedge_{\mathcal{B}_\Phi}$.

We claim that $\mathrm{div}(  R) = \pi^* \Delta$ for \emph{some} divisor $\Delta$ on $\mathcal{B}_\Phi$.  Indeed, whatever $\mathrm{div}(  R)$ is, it may decomposed as a sum of horizontal and vertical components.  
We know from Theorem \ref{thm:toroidal} and Proposition \ref{prop:algebraic BFJ}  that the horizontal part is a linear combination of the divisors $\mathcal{Z}_\Phi(m)$ on $\mathcal{C}^*_\Phi$ defined by (\ref{boundary special});  these divisors are, by construction, pullbacks of divisors on $\mathcal{B}_\Phi$.  
On the other hand,  the morphism $\mathcal{C}_\Phi^* \to \mathcal{B}_\Phi$ is the total space of a line bundle, and hence is smooth with connected fibers.  Thus \emph{every} vertical divisor on $\mathcal{C}_\Phi^*$, and in particular the vertical part of $\mathrm{div}(R)$, is the pullback of some divisor on $\mathcal{B}_\Phi$.

Denoting by $i: \mathcal{B}_\Phi \hookrightarrow \mathcal{C}_\Phi^*$  the zero section, we compute
\[
\Delta = i^* \pi ^* \Delta = i^* \mathrm{div}(R) = \mathrm{div}(i^* R) = \mathrm{div}(\bm{\psi}_0).
\] 
Pulling  back by $\pi$ proves that  $\mathrm{div}(  R) =\pi^*\mathrm{div}( \bm{\psi}_0)$, and (\ref{divisor near boundary}) follows.

We now compute the divisor of $\bm{\psi}_0$ on $\mathcal{B}_\Phi$ using (\ref{leading factor}).  The divisors of $P_\Phi^\eta$, $P^{hor}_\Phi$, and  $P^{vert}_\Phi$ were computed in Proposition \ref{prop:leading divisor}, which shows that on  $\mathcal{B}_\Phi$ we have the equality
\begin{align*}
\mathrm{div}( \bm{\psi}_0 )   = \mathrm{div}( \delta^{-k} \kappa_\Phi)  +   \sum_{m>0} c(-m)  \mathcal{Z}_\Phi (m)  
+ \sum_{ r \mid D} \gamma_r c_r(0)  \sum_{p\mid r} \mathcal{B}_{\Phi/\F_\mathfrak{p}}.
\end{align*}
Combining this with (\ref{divisor near boundary})  completes the proof.
\end{proof}

\begin{proposition}\label{prop:interior divisor}
When viewed as a rational section of  $\bm{\omega}^k$ on $\mathcal{S}^*_\Kra$, the Borcherds product $\bm{\psi}(f)$ has divisor 
\begin{align}
\mathrm{div}( \bm{\psi}  (f) )   & =   \sum_{m>0} c(-m) \cdot \mathcal{Z}^*_\Kra(m)   + \sum_\Phi \mathrm{mult}_\Phi(f) \cdot \mathcal{S}^*_\Kra(\Phi) \nonumber \\
&\quad   + \mathrm{div}(\delta^{-k})    + \sum_{ r \mid D} \gamma_r c_r(0)   \sum_{p \mid r}   \mathcal{S}^*_{\Kra /\F_\mathfrak{p}}   \label{main div calc}
\end{align}
 up to a linear combination of irreducible components of the exceptional divisor $\mathrm{Exc}\subset \mathcal{S}^*_\Kra$.  Moreover, each section $\kappa_\Phi$ of Proposition \ref{prop:B descent} has finite multiplicative order, and extends  to a section
$
\kappa_\Phi \in  H^0( \mathcal{A}_{\Phi} , \co_{\mathcal{A}_\Phi}^\times).
$
\end{proposition}

\begin{proof}
Recall from Lemma \ref{lem:enough cusps} that the natural maps 
\[
\xymatrix{
 {   \bigsqcup_i \mathcal{B}_{\Phi_i } }   \ar[r]  \ar[d]   &   {   \bigsqcup_i \mathcal{S}_\Pap^* (\Phi_i)   } \ar[r]  &  {  \mathcal{S}^*_{ \Pap } }    \\
  {   \bigsqcup_i \mathcal{A}_{\Phi_i }  }  &
}
\]
 induce  bijections on connected components, as well as on connected  components of the generic  fibers.

 All  stacks in the diagram are proper over $\co_\kk$, and have normal fibers.  (For $\mathcal{S}^*_\Pap$ this follows from  Theorem \ref{thm:toroidal} and our assumption that $n>2$.  The other stacks appearing in the diagram are smooth over $\co_\kk$.)
 It follows from this and  \cite[Corollary 8.2.18]{FGA} that all arrows in the diagram  induce bijections between the irreducible  (=connected) components modulo any prime $\mathfrak{p} \subset \co_\kk$.

Deleting the ($0$-dimensional)  singular locus $\mathrm{Sing} \subset  \mathcal{S}^*_\Pap$ does not change the irreducible components of $\mathcal{S}^*_\Pap$ or its fibers, and so if we define
 \[
\mathcal{U} \define \mathcal{S}^*_\Pap \smallsetminus \mathrm{Sing} \iso \mathcal{S}^*_\Kra \smallsetminus \mathrm{Exc}
\]
then the natural maps 
 \[
\xymatrix{
 {   \bigsqcup_i \mathcal{B}_{\Phi_i } }   \ar[r]  \ar[d]   &   {   \bigsqcup_i \mathcal{S}_\Pap^* (\Phi_i)   } \ar[r]  &  {  \mathcal{U} }    \\
  {   \bigsqcup_i \mathcal{A}_{\Phi_i }  }  &
}
\]
 induce  bijections on irreducible components, as well as on irreducible components modulo any prime $\mathfrak{p}\subset \co_\kk$.

Suppose $\Phi$ is any proper cusp label representative, and let 
$
\mathcal{U}_\Phi \subset \mathcal{U}
$
 be the union of all irreducible components that meet  $\mathcal{S}_\Pap^* (\Phi)$.    
If  we interpret $\mathrm{div}(\kappa_\Phi)$ as a divisor on $\mathcal{U}$ using the bijection 
\[
\{ \mbox{vertical divisors on }\mathcal{A}_{\Phi } \} \iso \{ \mbox{vertical divisors on }\mathcal{U}_\Phi \},
\]
 then the equality of divisors  (\ref{main div calc})  holds after pullback to $\mathcal{U}_\Phi$, up to the error term $\mathrm{div}(\kappa_\Phi)$.  Indeed, this equality  holds in the generic fiber of $\mathcal{U}_\Phi$ by  Proposition \ref{prop:algebraic BFJ}, and it holds over an open neighborhood of $\mathcal{S}_\Pap^* (\Phi)$ by   Lemma \ref{lem:leading FJ suffices}  and the isomorphism of formal completions
 (\ref{formal boundary iso}).  As the union of the generic fiber with this open neighborhood is an open substack whose complement has codimension $\ge 2$,  the stated equality holds over all of $\mathcal{U}_\Phi$.

Letting $\Phi$ vary over the $\Phi_i$ and using  $\kappa_{\Phi_i}=1$, we see from the paragraph above that (\ref{main div calc}) holds over  $\bigsqcup_i \mathcal{U}_{\Phi_i} = \mathcal{U}$.  
With this in hand,  we may reverse the argument  to see that the error term $\mathrm{div}(\kappa_\Phi)$ vanishes  for every $\Phi$. It follows   that $\kappa_\Phi$ extends to a global section of $\co^\times_{\mathcal{A}_\Phi}$.

It only remains to show that each $\kappa_\Phi$ has finite order.   Choose a finite extension $L/\kk$ large enough that every elliptic curve over $\C$ with complex multiplication by $\co_\kk$  admits a model over $L$ with everywhere good reduction.     Choosing such models determines a faithfully flat morphism
\[
\bigsqcup \Spec(\co_L) \to \mathcal{M}_{(1,0)} \iso \mathcal{A}_\Phi,
\]
and the pullback of $\kappa_\Phi$ is represented by a tuple of units $(x_\ell)\in  \prod \co_L^\times$.   Each $x_\ell$ has  absolute value $1$ at every complex embedding of $L$ (this follows from the final claim of Proposition \ref{prop:algebraic BFJ}), and is therefore   a root of unity.  This implies that  $\kappa_\Phi$ has finite order.
\end{proof}

\begin{proof}[Proof of Theorem \ref{thm:unitary borcherds I}]
Start with a weakly holomorphic form (\ref{input form}).  As in \S \ref{ss:borcherds define}, after possibly replacing $f$ by a positive integer multiple, we obtain a Borcherds product $\bm{\psi}(f)$.  This is a meromorphic section of $(\bm{\omega}^{an})^k$.  By Proposition \ref{prop:algebraic BFJ} it is algebraic, and by Proposition \ref{prop:B descent} it may be rescaled by a constant of absolute value $1$ on each connected component in such a way that it descends to $\kk$.

Now view $\bm{\psi}(f)$ as a rational section of $\bm{\omega}^k$ over $\mathcal{S}^*_\Kra$.  By  Proposition \ref{prop:interior divisor} we may replace $f$ by a further positive integer multiple, and replace $\bm{\psi}(f)$ by a corresponding tensor power,  in order to make all $\kappa_\Phi=1$.  
Having trivialized the $\kappa_\Phi$, the existence part of Theorem \ref{thm:unitary borcherds I} now follows from Proposition \ref{prop:algebraic BFJ}. 
For uniqueness, suppose $\bm{\psi}'(f)$ also satisfies the conditions of that theorem.   The quotient of the two Borcherds products is a 
rational function with trivial divisor, which is therefore constant on every connected component of $\mathcal{S}^*_\Kra(\C)$.  As the leading Fourier-Jacobi coefficients of $\bm{\psi}'(f)$ and $\bm{\psi}(f)$ are equal along every boundary stratum, those constants are all equal to $1$. 
\end{proof}

\begin{proof}[Proof of Theorem  \ref{thm:unitary borcherds III}]
As in the statement of the theorem, we  now view  $\bm{\psi}(f)^2$ as a rational section of the line bundle $\pure_\Pap^k$ on $\mathcal{S}^*_\Pap$.    
Combining Proposition \ref{prop:interior divisor}  with the isomorphism 
\[
\mathcal{S}^*_\Kra  \smallsetminus  \mathrm{Exc} \iso  \mathcal{S}^*_\Pap \smallsetminus \mathrm{Sing},
\]
of  (\ref{compact nonsingular}), and recalling from Theorem \ref{thm:toroidal} that this isomorphism identifies 
\[
\bm{\omega}^{2k}\iso \bm{\Omega}_\Kra^k\iso  \bm{\Omega}_\Pap^k,
\]
 we deduce  the equality
\begin{align}
\mathrm{div}( \bm{\psi}  (f) ^2)  
 & =   \sum_{m>0} c(-m) \cdot \mathcal{Y}^*_\Pap(m)   + 2 \sum_\Phi \mathrm{mult}_\Phi(f) \cdot \mathcal{S}^*_\Pap(\Phi) \nonumber \\
& \quad  + \mathrm{div}(\delta^{-2k})    +  2 \sum_{ r \mid D} \gamma_r c_r(0)   \sum_{p \mid r}   \mathcal{S}^*_{\Pap /\F_\mathfrak{p}} 
  \label{pappas divisor}
\end{align}
 of Cartier divisors on $\mathcal{S}^*_\Pap\smallsetminus \mathrm{Sing}$.   As $\mathcal{S}^*_\Pap$ is normal and $\mathrm{Sing}$ lies in codimension $\ge 2$, this same equality must hold on the entirety of $\mathcal{S}^*_\Pap$.  
\end{proof}

 \begin{proof}[Proof of Theorem  \ref{thm:unitary borcherds II}]
If we pull back via $\mathcal{S}^*_\Kra \to  \mathcal{S}^*_\Pap$ and  view  $\bm{\psi}(f)^2$ as a rational section of the line bundle 
\[
\pure_\Kra^k \iso  \bm{\omega}^{2k} \otimes   \co(\mathrm{Exc})^{-k},
\]
the equality (\ref{pappas divisor})  on $\mathcal{S}_\Pap^*$ pulls back to
\begin{align*}
\mathrm{div}( \bm{\psi}  (f) ^2)   
& =   \sum_{m>0} c(-m) \cdot \mathcal{Y}^*_\Kra(m)     +  2 \sum_\Phi \mathrm{mult}_\Phi(f) \cdot \mathcal{S}^*_\Kra(\Phi)  \\
& \quad + \mathrm{div}(\delta^{-2k})  + 2 \sum_{ r \mid D} \gamma_r c_r(0)   \sum_{p \mid r}   \mathcal{S}^*_{\Kra /\F_\mathfrak{p}}  . 
\end{align*}
 Theorem \ref{thm:cartier error} allows us   to rewrite this as
\begin{align*}
\mathrm{div}( \bm{\psi}  (f) ^2)   
 & =  2  \sum_{m>0} c(-m) \cdot \mathcal{Z}^*_\Kra(m)     +  2 \sum_\Phi \mathrm{mult}_\Phi(f) \cdot \mathcal{S}^*_\Kra(\Phi) \\
&\quad + \mathrm{div}(\delta^{-2k})   + 2 \sum_{ r \mid D} \gamma_r c_r(0)   \sum_{p \mid r}   \mathcal{S}^*_{\Kra /\F_\mathfrak{p}} \\
 & \quad -  \sum_{m>0}   c(-m)    \sum_{ s\in \pi_0(\mathrm{Sing}) }   \# \{ x\in L_s : \langle x,x\rangle =m \} \cdot  \mathrm{Exc}_s .
\end{align*}

If we instead view $\bm{\psi}(f)^2$ as a rational section of $\bm{\omega}^{2k}$, this becomes
\begin{align*}
\mathrm{div}( \bm{\psi}  (f) ^2)   
 & =  2  \sum_{m>0} c(-m) \cdot \mathcal{Z}^*_\Kra(m)     +  2 \sum_\Phi \mathrm{mult}_\Phi(f) \cdot \mathcal{S}^*_\Kra(\Phi) \\
&\quad + \mathrm{div}(\delta^{-2k})   + 2 \sum_{ r \mid D} \gamma_r c_r(0)   \sum_{p \mid r}   \mathcal{S}^*_{\Kra /\F_\mathfrak{p}} \\
 & \quad -  \sum_{m>0}   c(-m)    \sum_{ s\in \pi_0(\mathrm{Sing}) }   \# \{ x\in L_s : \langle x,x\rangle =m \} \cdot  \mathrm{Exc}_s \\
 & \quad   + k\cdot \mathrm{Exc} 
\end{align*}
as desired.
\end{proof}


\section{Modularity of the generating series}
\label{s:modularity}


Now armed with the modularity criterion of Theorem \ref{thm:modularity criterion} and the arithmetic theory of Borcherds products   provided by 
Theorems \ref{thm:unitary borcherds I},  \ref{thm:unitary borcherds II},  and \ref{thm:unitary borcherds III}, we prove our main results: the modularity of generating series of divisors on the integral models $\mathcal{S}_\Kra^*$ and $\mathcal{S}_\Pap^*$ of the unitary Shimura variety $\mathrm{Sh}(G,\mathcal{D})$.  
The strategy follows that of \cite{Bo2}, which proves  modularity of the generating series of divisors on the complex fiber of an orthogonal Shimura variety.

Throughout \S \ref{s:modularity} we assume  $n\ge 3$.


\subsection{The modularity theorems}


Denote by 
\[
\Ch^1_\Q (\mathcal{S}^*_\Kra ) \iso \mathrm{Pic} (\mathcal{S}^*_\Kra )  \otimes_\Z\Q
\]
 the Chow group of rational equivalence classes of Cartier divisors on $\mathcal{S}^*_\Kra$ with $\Q$ coefficients, and similarly for $\mathcal{S}^*_\Pap$.  
 There is a natural pullback map
\[
 \Ch^1_\Q (\mathcal{S}^*_\Pap ) \to \Ch^1_\Q (\mathcal{S}^*_\Kra ).
\]

 Let $\chi=\chi_\kk^n$ be the quadratic Dirichlet character  (\ref{quad character}).

\begin{definition}
If $V$ is any $\Q$-vector space, a formal $q$-expansion 
\begin{equation}\label{potentially modular}
\sum_{m\ge 0} d(m) \cdot q^m\in V[[q]]
\end{equation}
is  \emph{modular of level $D$, weight $n$, and character $\chi$} if for any $\Q$-linear map $\alpha : V\to \C$ the $q$-expansion
\[
\sum_{m\ge 0}  \alpha( d(m)) \cdot q^m\in \C[[q]]
\]
is the $q$-expansion of an element of $M_n(D,\chi)$.
\end{definition}

\begin{remark}\label{rem:finite span}
If (\ref{potentially modular}) is modular then its coefficients $d(m)$ span a subspace of $V$ of dimension $\le \mathrm{dim}\, M_n(D,\chi)$. We leave the proof as an exercise for the reader.
\end{remark}

We also define the notion of the constant term of (\ref{potentially modular}) at a cusp $\infty_r$, generalizing Definition \ref{def:constant at cusp}.

\begin{definition}
Suppose a formal $q$-expansion $g \in V[[q]]$ is modular of level $D$, weight $n$, and character $\chi$.  For any $r\mid D$, a  vector  $v \in V(\C)$ is said to be the \emph{constant term of $g$ at the cusp $\infty_r$} if, for every linear functional $\alpha:V(\C) \to \C$, $\alpha(v)$ is the constant term of $\alpha(g)$ at the cusp $\infty_r$  in the sense of Definition \ref{def:constant at cusp}.
\end{definition}

For $m>0$ we have defined  in  \S \ref{ss:unitary borcherds} effective Cartier divisors 
\[
\mathcal{Y}^\tot_\Pap(m)\hookrightarrow \mathcal{S}^*_\Pap,\quad 
\mathcal{Z}^\tot_\Kra(m) \hookrightarrow \mathcal{S}^*_\Kra
\] 
related by (\ref{total error}).   We have defined in \S  \ref{ss:main toroidal}  line bundles 
 \[
 \pure_\Pap \in \mathrm{Pic}( \mathcal{S}^*_\Pap)  ,\quad \bm{\omega} \in \mathrm{Pic} ( \mathcal{S}^*_\Kra)
 \] 
 extending the line bundles on the open integral models defined in  \S \ref{ss:unitary bundle}.  For notational uniformity, we  define 
\[
\mathcal{Y}^\tot_\Pap(0)  = \pure_\Pap^{-1}  ,\quad \mathcal{Z}_\Kra^\tot(0) = \bm{\omega}^{-1} \otimes \co(\mathrm{Exc}).
\]

\begin{theorem}\label{thm:main modularity I}
The  formal $q$-expansion
\[
\sum_{m\ge 0}  \mathcal{Y}_\Pap^\tot(m)  \cdot q^m \in \Ch^1_\Q (\mathcal{S}^*_\Pap )  [[ q ]] ,
\]
is a modular form of level $D$, weight $n$, and character $\chi$.   For any $r\mid D$, its constant term at the cusp $\infty_r$ is 
\[
\gamma_r \cdot  \Big( \mathcal{Y}_\Pap^\tot(0)+2\sum_{p\mid r} \mathcal{S}^*_{\Pap / \F_\mathfrak{p}} \Big)  \in \Ch^1_\Q (\mathcal{S}^*_\Pap)  \otimes_\Q \C.
\]
Here  $\gamma_r\in \{\pm 1 , \pm i \}$ is defined by (\ref{gamma def}),  $\mathfrak{p}\subset \co_\kk$ is the unique prime above $p\mid r$, and $\F_\mathfrak{p}$ is its residue field. 
\end{theorem}

\begin{proof}
Let $f$ be a weakly holomorphic form as in (\ref{input form}), and assume  again that $c(m)\in \Z$ for all $m\le 0$.  The space $M^{!,\infty}_{2-n}(D,\chi)$ is spanned by such forms.   The Borcherds product  $\bm{\psi}(f)$  of Theorem \ref{thm:unitary borcherds I} is a rational section of the line bundle
\[
\bm{\omega}^k = \bigotimes_{r\mid D} \bm{\omega}^{ \gamma_r  c_r(0)},
\]
on $\mathcal{S}_\Kra^*$.  If we view $\bm{\psi}(f)^2$ as a rational section of the line bundle
\[
\pure_\Pap^k \iso \bigotimes_{r\mid D} \pure_\Pap^{  \gamma_r  c_r(0)}
\]
on $\mathcal{S}_\Pap^*$, exactly as in Theorem \ref{thm:unitary borcherds III}, then 
\[
 \mathrm{div}( \bm{\psi}  (f)^2 ) =  -\sum_{r\mid D} \gamma_r  c_r(0) \cdot  \mathcal{Y}_\Pap^\tot(0)
\]
holds in  the Chow group of $\mathcal{S}^*_\Pap$.
  Comparing this with the calculation of the divisor of $\bm{\psi}(f)^2$ found in Theorem \ref{thm:unitary borcherds III} shows that
 \begin{equation}\label{borcherds pappas relation}
0  =   \sum_{m\ge 0} c(-m) \cdot  \mathcal{Y}_\Pap^\tot(m)  
+ \sum_{  \substack{ r \mid D \\ r>1 }}  \gamma_r c_r(0) \cdot ( \mathcal{Y}_\Pap^\tot(0) + 2  \mathcal{V}_r ) ,
 \end{equation}
 where we abbreviate $\mathcal{V}_r = \sum_{p\mid r} \mathcal{S}^*_{\Pap / \F_\mathfrak{p}}$.

For each $r\mid D$ we have defined in \S \ref{ss:eisenstein} an Eisenstein series
\[
E_r(\tau) = \sum_{m \ge 0} e_r(m) \cdot q^m \in M_n(D,\chi),
\]
and Proposition \ref{prop:distribute cusps}  allows us to rewrite the above equality as 
 \begin{align*}
0  =  \sum_{m \ge 0} c(-m) \cdot \Big[ \mathcal{Y}_\Pap^\tot(m)
 -   \sum_{  \substack{ r \mid D \\ r>1 }} \gamma_r    e_r(m) \cdot ( \mathcal{Y}_\Pap^\tot(0) + 2 \mathcal{V}_r ) \Big] .
 \end{align*}
Note that we have used $e_r(0)=0$ for $r>1$, a consequence of Remark \ref{rem:eisenstein constant}.

The modularity criterion of Theorem \ref{thm:modularity criterion} now shows that 
\[
 \sum_{m \ge 0}   \mathcal{Y}_\Pap^\tot(m)  \cdot q^m   -  
   \sum_{  \substack{ r \mid D \\ r>1 }}  \gamma_r     E_r \cdot  ( \mathcal{Y}_\Pap^\tot(0) + 2 \mathcal{V}_r ) 
   \]
is a modular form of level $D$, weight $n$, and character $\chi$,  whose constant term vanishes at every cusp different from $\infty$. 

 The theorem now follows from the modularity of each $E_r$, together with the description of  their constant terms found in Remark \ref{rem:eisenstein constant}.
\end{proof}

\begin{theorem}\label{thm:kramer modularity}
The  formal $q$-expansion
\[
\sum_{m\ge 0}  \mathcal{Z}_\Kra^\tot(m)  \cdot q^m \in \Ch^1_\Q (\mathcal{S}^*_\Kra )  [[ q ]] ,
\]
is a modular form of level $D$, weight $n$, and character $\chi$.   
\end{theorem}

\begin{proof}
Recall from Theorems \ref{thm:cartier error} and \ref{thm:toroidal} that  pullback via $\mathcal{S}_\Kra^* \to \mathcal{S}_\Pap^*$ sends
\[
\mathcal{Y}_\Pap^\tot(m) \mapsto 2\cdot \mathcal{Z}_\Kra^\tot(m) - \sum_{ s\in \pi_0(\mathrm{Sing}) }   \# \{ x\in L_s : \langle x,x\rangle =m \}  \cdot  \mathrm{Exc}_s  
\]
for all $m>0$.  This relation also holds for $m=0$,  as those same theorems show that  
\[
\mathcal{Y}_\Pap^\tot(0) =  \bm{\Omega}_\Pap^{-1} \mapsto \bm{\omega}^{-2} \otimes \co( \mathrm{Exc} )  = 2\cdot \mathcal{Z}_\Kra^\tot(0) - \mathrm{Exc}.
\]

Pulling back the relation (\ref{borcherds pappas relation}) shows that 
\begin{align*}
0  &= \sum_{m\ge 0} c(-m) \cdot 
 \Big( \mathcal{Z}_\Kra^\tot(m)    -  \sum_{ s\in \pi_0(\mathrm{Sing}) }   \frac{\# \{ x\in L_s : \langle x,x\rangle =m \}}{2} \cdot  \mathrm{Exc}_s   \Big)\\
& \quad  +    \sum_{  \substack{ r \mid D \\ r>1} } \gamma_r c_r(0)  \cdot  \Big(    \mathcal{Z}^\tot_\Kra(0)  - \frac{1}{2} \cdot \mathrm{Exc}  + \mathcal{V}_r \Big) 
\end{align*}
in $ \Ch^1_\Q(\mathcal{S}_\Kra^*)$
for any input form (\ref{input form}), where we now abbreviate \[ \mathcal{V}_r = \sum_{p\mid r} \mathcal{S}^*_{\Kra / \F_\mathfrak{p}}.\]
Using Proposition \ref{prop:distribute cusps} we rewrite this as
\begin{align*}
0  
& =   \sum_{m\ge 0} c(-m) \cdot 
 \Big( \mathcal{Z}_\Kra^\tot(m)    -  \sum_{ s\in \pi_0(\mathrm{Sing}) }   \frac{\# \{ x\in L_s : \langle x,x\rangle =m \}}{2} \cdot  \mathrm{Exc}_s   \Big)\\
& \quad 
- \sum_{m \ge 0}   c(-m)  \sum_{  \substack{ r \mid D \\ r>1} } \gamma_r e_r(m)  
 \Big(    \mathcal{Z}^\tot_\Kra(0)  - \frac{1}{2} \cdot \mathrm{Exc}  + \mathcal{V}_r \Big) ,
\end{align*}
where we have again used the fact that $e_r(0)=0$ for $r>1$.

The modularity criterion of Theorem \ref{thm:modularity criterion} now implies the modularity of 
\begin{align*}
\sum_{m\ge 0} \mathcal{Z}_\Kra^\tot(m) \cdot q^m & - \frac{1}{2}  \sum_{s\in \pi_0(\mathrm{Sing})} \vartheta_s(\tau) \cdot \mathrm{Exc}_s \\
& \quad - \sum_{\substack{  r\mid D \\ r>1 } } \gamma_r E_r (\tau) \cdot  \Big(    \mathcal{Z}^\tot_\Kra(0)  - \frac{1}{2} \cdot \mathrm{Exc}  + \mathcal{V}_r \Big) .
\end{align*}
The theorem follows from the modularity of the Eisenstein series $E_r(\tau)$ and the theta series
\[
\vartheta_s(\tau) = \sum_{ x\in L_s} q^{ \langle x,x\rangle} \in M_n(D,\chi).
\]
\end{proof}

\subsection{Green functions}

Here we construct Green functions for special divisors on  
 $\calS^*_\Kra $ as regularized theta lifts of harmonic Maass forms.  

Recall from Section \ref{s:unitary} the isomorphism of complex orbifolds
\[
\calS_\Kra(\C)\cong  \mathrm{Sh}(G,\calD) (\C) = G(\Q) \backslash \mathcal{D} \times G(\A_f) / K.
\]
We use the uniformization on the right hand side and the regularized theta lift to construct Green functions
for the special divisors 
\[
\calZ_\Kra^\tot(m)=\calZ_\Kra^*(m) + \calB_\Kra(m)
\] 
on 
$\calS^*_\Kra $.
The construction is a variant of the ones in \cite{BF} and \cite{BHY},  adapted to our situation.

We now recall some of the basic notions of the theory of harmonic Maass forms, as in  \cite[Section 3]{BF}.
Let $H_{2-n}^\infty(D,\chi)$ denote the space of harmonic Maass forms  $f$ of weight $2-n$ for $\Gamma_0(D)$ with character $\chi$ such that 
\begin{itemize}
\item $f$ is bounded at all cusps of $\Gamma_0(D)$ different from the cusp $\infty$,
\item $f$ has polynomial growth at $\infty$, in sense that there is  a 
\[
P_f=\sum_{m<0} c^+(m)q^m\in \C[q^{-1}]
\]
such that $f-P_f$ is bounded as $q$ goes to $0$. 
\end{itemize}
A harmonic Maass form $f\in H_{2-n}^\infty(D,\chi)$ has a Fourier expansion of the form
\begin{align}
\label{eq:fourierf}
f(\tau)=\sum_{\substack{m\in \Z\\ m\gg -\infty }} c^+(m) q^m
+\sum_{\substack{m\in \Z \\ m<0 } } c^-(m) \cdot  \Gamma\big(n-1, 4\pi |m| \Im(\tau) \big)  \cdot q^m,
\end{align}
where 
\[
\Gamma(s,x)=\int_x^\infty e^{-t}t^{s-1}dt
\]
is the incomplete gamma function. The first summand on the right hand side of \eqref{eq:fourierf}
is denoted by $f^+$ and is called the \emph{holomorphic part} of $f$, the second summand is denoted by $f^-$ and is called the \emph{non-holomorphic part}.

If $f\in H_{2-n}^\infty(D,\chi)$ then \eqref{eq:vectorization} defines 
an $S_L$-valued  harmonic Maass form for $\SL_2(\Z)$ of weight $2-n$ with representation $\omega_L$. 
Proposition \ref{prop:promoted coefficients} extends to such lifts of harmonic Maass forms, giving the same formulas for the coefficients $\tilde c^+(m,\mu)$ of the holomorphic part $\tilde f^+$ of $\tilde f$. In particular, if $m<0$ we have 
\begin{align}
\label{eq:hmfcoeff}
\tilde c^+(m, \mu) = \begin{cases}
    c^+(m) &\text{if $\mu =0$,}
    \\
    0  &\text{if $\mu \neq 0$,}
    \end{cases}
\end{align}
and the constant term of $\tilde f$ is given by
\[
 \tilde c^+(0, \mu) = \sum_{r_\mu \mid  r \mid D}  \gamma_r \cdot c^+_r(0) .
\]
The formula of Proposition \ref{prop:distribute cusps} for the contant terms $c^+_r(0)$ of $f$ at the other cusps also extends.

As before, we consider the hermitian self-dual $\co_\kk$-lattice 
$L=\Hom_{\calO_\kk}(\mathfrak{a}_0,\mathfrak{a})$ in $V=\Hom_{\kk}(W_0,W)$.
The dual lattice of $L$ with respect to the bilinear form $[\cdot,\cdot]$ is $L'=\mathfrak{d}^{-1}L$.  
Let 
\[
S_L\subset S(V(\A_f))
\]
 be the space of Schwartz-Bruhat functions that are supported on $\widehat L'$ and invariant under translations by
$\widehat L$.

Recall from Remark \ref{rem:to so} that we may identify 
\[
\mathcal{D} \iso \{ w\in \epsilon V(\C) : [w,\overline{w} ]  <0\} /\C^\times,
\]
and also
\[
\mathcal{D} \iso \{ \mbox{negative definite $\kk$-stable $\R$-planes  $z\subset V(\R)$} \} .
\]
For any $x\in V$ and $z\in \mathcal{D}$, let $x_z$ be the orthogonal projection of $x$ to the plane $z\subset V(\R)$, and let $x_{z^\perp}$ be the orthogonal projection to $z^\perp$.

For  $(\tau,z,g) \in \mathfrak{H} \times  \calD \times G(\A_f)$ and  $\varphi\in S_L$,  we define a theta function 
\begin{align*}
\theta(\tau,z,g,\varphi) = \sum_{x\in V} \varphi(g^{-1}x)  \cdot \varphi_\infty(\tau,z,x),
\end{align*}
where the Schwartz function at $\infty$,
\[
\varphi_\infty(\tau,z,x) = v \cdot e^{  2\pi i   Q(x_{z^\perp})\tau + 2\pi i    Q(x_{z})\bar \tau  },
\]
is the usual Gaussian involving the majorant associated to $z$.  
We may view $\theta$ as a function $\mathfrak{H}\times\mathcal{D}\times G(\A_f) \to S_L^\vee$.  
As a function in $(z,g)$ it is invariant under the left action of $G(\Q)$.
Under the right action of $K$ it satisfies the transformation law
\[
\theta(\tau,z,gk,\varphi) = \theta(\tau,z,g,\omega_L(k)\varphi), \quad k\in K,
\]
where $\omega_L$ denotes the action of $K$ on $S_L$ by the Weil representation and $v=\Im(\tau)$. 
In the variable $\tau\in \mathfrak{H}$ it transforms as a $S_L^\vee$-valued modular form of weight $n-2$ for $\SL_2(\Z)$.
 
Fix an $f\in H_{2-n}^\infty(D,\chi)$ with Fourier expansion as in \eqref{eq:fourierf}, and assume that $c^+(m)\in \Z$ for  $m\leq 0$.
We associate to $f$  the divisors 
\begin{align*}
\calZ_\Kra(f)&= \sum_{m>0} c^+(-m) \cdot \calZ_\Kra(m)\\
\calZ_\Kra^\tot(f)&= \sum_{m>0} c^+(-m) \cdot  \calZ_\Kra^\tot(m)
\end{align*}
on $\calS_\Kra$ and $\calS_\Kra^*$, respectively.
As the actions of $\SL_2(\Z)$ and $K$ via the Weil representation commute, the associated $S_L$-valued harmonic Maass form $\tilde f$ is invariant under $K$. Hence   the natural pairing 
$S_L\times S_L^\vee\to \C$ gives rise to a scalar valued function 
$(\tilde f(\tau), \theta(\tau,z,g))$  in the variables $(\tau,z,g) \in \mathfrak{H} \times \mathcal{D} \times G(\A_f)$,
which is invariant under the right action of $K$ and  the left action of $G(\Q)$. Hence it descends to a function 
on $\SL_2(\Z)\bs \mathfrak{H} \times \mathrm{Sh}(G,\calD) (\C)$.

We define the \emph{regularized theta lift of $f$} as
\begin{align*}
\TR(z,g,f)= \int_{\SL_2(\Z)\bs\mathfrak{H}}^\reg \big(\tilde f(\tau), \theta(\tau,z,g)\big)
\, \frac{du\,dv}{v^2}.
\end{align*}
Here the regularization of the integral is defined as in \cite{Bo1, BF, BHY}.
We extend the incomplete Gamma function 
\begin{equation}\label{incomplete gamma}
\Gamma(0,t)=\int_{t}^{\infty}e^{-v}\frac{dv}{v}
\end{equation}
 to a function on $\R_{\geq 0}$ by setting
\[
\widetilde\Gamma(0,t)=\begin{cases}
\Gamma(0,t) &\text{if $t>0$,}\\
0 &\text{if $t=0$.}
\end{cases}
\]
 
\begin{theorem}
\label{thm:greenopen}
The regularized theta lift $\TR(z,g,f)$ defines a smooth function on $\calS_\Kra(\C)\smallsetminus \calZ_\Kra(f)(\C)$.
For $g\in G(\A_f)$ and $z_0\in \calD$, there exists a neighborhood $U\subset \calD$ of $z_0$ such that 
\[
\TR(z,g,f)-\sum_{\substack{x\in gL\\ x\perp z_0}}
c^+(-\langle x,x\rangle)\cdot 
\widetilde\Gamma\big(0,4\pi |\langle x_z,x_z\rangle |\big)
\]
is a smooth function on $U$.
\end{theorem}

\begin{proof}
Note that the sum over $x\in gL\cap z_0^\perp$ is finite.
Since $ \mathrm{Sh}(G,\calD) (\C)$ decomposes into a finite disjoint union of connected components of the form
\[
(G(\Q)\cap gKg^{-1})\bs \calD,
\] 
where $g\in G(\A_f)$, it suffices to consider the restriction of $\TR(f)$ to these components.

On such a component, $\TR(z,g,f)$ is the regularized theta lift considered in \cite[Section 4]{BHY} of the vector valued form $\tilde f$  for the lattice 
\[
gL=g\widehat L\cap V=\Hom_{\calO_\kk}(g\mathfrak{a}_0,g\mathfrak{a})\subset V,
\]
and hence the assertion follows from  \eqref{eq:hmfcoeff} and \cite[Theorem 4.1]{BHY}.
\end{proof}

\begin{remark}
Let $\Delta_\calD$ denote the $\mathrm{U}(V)(\R)$-invariant Laplacian on $\calD$. There exists a non-zero real constant $c$ (which only depends on the normalization of $\Delta_\calD$ and which is independent of $f$), such that
\[
\Delta_\calD \TR(z,g,f) = c\cdot \deg \calZ_\Kra(f)(\C)
\]
on the complement of the divisor $\calZ_\Kra(f)(\C)$.
\end{remark}

Using the fact that
\[
\Gamma(0,t)=-\log(t)+\Gamma'(1)+o(t)
\] 
as $t\to 0$, Theorem~\ref{thm:greenopen} implies that 
$\TR(f)$ is a (sub-harmonic) logarithmic Green function for the divisor $\calZ_\Kra(f)(\C)$ on the non-compactified Shimura variety $\calS_\Kra(\C)$. These properties, together with an integrability condition, characterize it uniquely up to addition of a locally constant function
\cite[Theorem 4.6]{BHY}.
The following result describes the behavior of $\TR(f)$ on the toroidal compactification.

\begin{theorem}
\label{thm:greencomp}
On $\calS_\Kra^*(\C)$, the function $\TR(f)$ is a logarithmic Green function for the divisor $\calZ_\Kra^\tot(f)(\C)$ with possible additional log-log singularities along the boundary in the sense of \cite{BKK}.
\end{theorem}

\begin{proof}
As in the proof of Theorem \ref{thm:greenopen} we reduce this to showing that $\TR(f)$ has the correct growth along the boundary of the connected components of $\calS_\Kra^*(\C)$. Then it is a direct consequence of \cite[Theorem~4.10]{BHY} and \cite[Corollary~4.12]{BHY}.
\end{proof}

Recall that $\bm{\omega}^{an}$ is the tautological bundle on 
\[
\mathcal{D}   \iso \big\{ w \in   \epsilon V(\C) :   [w,\overline{w}] <0    \big\} / \C^\times .
\]
We define the Petersson metric $\|\cdot \|$ on $\bm{\omega}^{an}$ by 
\[
\|w\| ^2=-\frac{ [  w, \overline{w}]  }{4\pi e^\gamma} ,
\]
where $\gamma=-\Gamma'(1)$ denotes Euler's constant. 
This choice of metric on $\bm{\omega}^{an}$ induces a metric on the line bundle $\bm{\omega}$ on $\mathcal{S}_\Kra(\C)$ defined in \S \ref{ss:unitary bundle}, which extends to a metric over $\mathcal{S}_\Kra^*(\C)$ with log-log singularities along the boundary \cite[Proposition 6.3]{BHY}.
We obtain a hermitian line bundle on $\calS_\Kra^*$,  denoted
\[
\widehat{\bm{\omega}}=(\bm{\omega}, \|\cdot \|).
\]

If $f$ is actually weakly holomorphic, that is, if it belongs to $M_{2-n}^{!,\infty}(D,\chi)$, then $\TR(f)$ is given by the logarithm of a Borcherds product. More precisely, we have the following theorem, which follows immediately from \cite[Theorem 13.3]{Bo1} and our construction of $\bm{\psi}(f)$ as the pullback of a Borcherds product, renormalized by (\ref{B renorm}), on an orthogonal Shimura variety.

\begin{theorem}
\label{prop:greenbp}
Let $f\in M_{2-n}^{!,\infty}(D,\chi)$ be as in (\ref{input form}).  The Borcherds product $\bm{\psi}(f)$ of Theorem \ref{thm:unitary borcherds I} satisfies
\[
\TR(f)= -\log\| \bm{\psi}(f)\|^2.
\]  
\end{theorem}


\subsection{Generating series of arithmetic special divisors}
\label{ss:arithmetic modular}


We can now define arithmetic special divisors on $\calS_{\Kra}^*$, and prove a modularity result for the corresponding generating series in the codimension one arithmetic Chow group.  This result extends  Theorem \ref{thm:kramer modularity}.

Recall our hypothesis that $n>2$, and let $m$ be a positive integer. As in \cite[Proposition 3.11]{BF}, or using Poincar\'e series, it can be shown that there exists a unique  $f_m\in H_{2-n}^\infty(D,\chi)$ whose Fourier expansion at the cusp $\infty$ has the form
\[
f_m=q^{-m}+O(1)
\]
as $q\to 0$.
According to Theorem \ref{thm:greencomp}, 
its regularized theta lift $\TR(f_m)$ is a logarithmic Green function for $\calZ_\Kra^\tot(m)$.

Denote by $\Cha^1_\Q (\calS^*_\Kra ) $ the arithmetic Chow group  \cite{GS} of rational equivalence classes of arithmetic divisors with $\Q$-coefficients.
 We allow the Green functions of our arithmetic divisors to have  possible additional log-log error terms along the boundary of $\calS^*_\Kra(\C)$, as in the theory of \cite{BKK}.  For $m>0$ define an arithmetic special divisor 
\[
\widehat{\calZ}_\Kra^\tot(m)= (\calZ_\Kra^\tot(m), \TR(f_m)) \in \Cha^1_\Q (\calS^*_\Kra ) 
\]
on $\calS_\Kra^*$, and for  $m=0$ set
\[
\widehat\calZ^\tot_\Kra(0)= \widehat{\bm{\omega}}^{-1} + ( \mathrm{Exc} , -\log(D) )  \in  \Cha^1_\Q (\calS^*_\Kra ) .
\]

In the theory of arithmetic Chow groups one usually works on a regular scheme such as $\mathcal{S}^*_\Kra$.  
However, the codimension one arithmetic Chow group of $\mathcal{S}^*_\Pap$ makes perfect sense: one only needs to specify that it consists of rational equivalence classes of \emph{Cartier} divisors on $\mathcal{S}^*_\Pap$ endowed with a Green function.

With this in mind one can use the equality \[ \mathcal{Y}_\Pap^\tot(m)(\C) =  2 \mathcal{Z}_\Kra^\tot(m)(\C)\]  in the complex fiber
$\mathcal{S}_\Pap^*(\C) = \mathcal{S}_\Kra^*(\C)$ to define arithmetic divisors
\[
\widehat{\mathcal{Y}}_\Pap^\tot(m)= (\mathcal{Y}_\Pap^\tot(m), 2  \TR(f_m)) \in \widehat{\mathrm{Ch}}^1_\Q ( \mathcal{S}_\Pap^* )
\]
for $m>0$.  For  $m=0$ we define
\[
\widehat{\mathcal{Y}}^\tot_\Pap(0)= \widehat{\bm{\Omega}}^{-1} + ( 0, -2 \log(D) )  \in \widehat{\mathrm{Ch}}^1_\Q ( \mathcal{S}_\Pap^* ),
\]
where the metric on $\bm{\Omega}$ is induced from that on $\bm{\omega}$, again  using  $\bm{\Omega}\iso \bm{\omega}^2$ in the complex  fiber.

\begin{theorem}
\label{thm:arithmeticmodularity}
The  formal $q$-expansions
\begin{align}
\label{eq:holgenser}
\widehat\phi(\tau)=
\sum_{m\ge 0} \widehat{\calZ}_\Kra^\tot(m) \cdot q^m \in \Cha^1_\Q (\calS^*_\Kra )  [[ q ]] 
\end{align}
and
\[
\sum_{m\ge 0} \widehat{\mathcal{Y}}_\Pap^\tot(m) \cdot q^m \in \Cha^1_\Q (\calS^*_\Pap )  [[ q ]] 
\]
are modular forms of level $D$, weight $n$, and character $\chi$.  
\end{theorem}

\begin{proof}
For any input form $f\in M_{2-n}^{!,\infty}(D,\chi)$ as in (\ref{input form}),  the relation in the Chow group given by the Borcherds product $\bm{\psi}(f)$ is compatible with the Green functions, in the sense that 
\[
-\log\|\bm{\psi}(f)\|^2 =\sum_{m>0} c(-m) \cdot \TR(f_m).
\]
Indeed, this directly follows from $f =\sum_{m>0} c(-m)f_m$ and Theorem~\ref{prop:greenbp}.

This observation allows us to simply repeat the argument of Theorems \ref{thm:main modularity I} and  \ref{thm:kramer modularity} on the level of arithmetic Chow groups.  Viewing $\bm{\psi}(f)^2$ as a rational section of the metrized line bundle $\bm{\Omega}_\Pap^k$, the arithmetic divisor
\[
\widehat{\mathrm{div}} ( \bm{\psi}  (f)^2 )   \define \big( \mathrm{div}(\bm{\psi}(f)^2) , - 2 \log\|\bm{\psi}(f)\|^2 \big) \in \Cha^1_\Q (\calS^*_\Pap )
\]
satisfies both
\begin{equation}\label{BP1}
\widehat{\mathrm{div}}( \bm{\psi}  (f)^2 ) =  \widehat{\bm{\Omega}}_\Pap^k =
-2k\cdot (0,\log(D) ) 
 -\sum_{r\mid D} \gamma_r c_r(0) \cdot \widehat{\mathcal{Y}}_\Pap^\tot(0)
\end{equation}
and, recalling $\delta =\sqrt{-D} \in \kk$, 
\begin{eqnarray}\lefteqn{ 
\widehat{\mathrm{div}}( \bm{\psi}  (f)^2 )   } \nonumber \\
& = &    \sum_{m>0} c(-m) \cdot \widehat{\mathcal{Y}}_\Pap^\tot(m)    - 2k \cdot ( \mathrm{div}(\delta) , 0)     +  2 \sum_{ r \mid D} \gamma_r c_r(0) \cdot   \widehat{ \mathcal{V}}_r  \nonumber   \\
& = &    \sum_{m>0} c(-m) \cdot \widehat{\mathcal{Y}}_\Pap^\tot(m)    -2k \cdot  ( 0 ,  \log (D)  ) 
  +  2 \sum_{ r \mid D} \gamma_r c_r(0)   \cdot   \widehat{\mathcal{V}}_r , \label{BP2}
\end{eqnarray}
where $\widehat{\mathcal{V}}_r$ is the the vertical divisor 
$
\mathcal{V}_r =  \sum_{p\mid r}  \mathcal{S}^*_{\Pap /\F_\mathfrak{p}}
$  
 endowed with the trivial Green function.  Note that in the second equality we have used the relation 
\[
0 = \widehat{\mathrm{div}}( \delta ) = ( \mathrm{div}(\delta) ,  - \log |\delta^2|  ) = ( \mathrm{div}(\delta) ,  0  ) - ( 0,   \log (D)  ) 
\]
in the arithmetic Chow group.
Combining (\ref{BP1}) and (\ref{BP2}),  we deduce that
\[
 0  =
  \sum_{m\ge 0} c(-m) \cdot \widehat{\mathcal{Y}}_\Pap^\tot(m)   
  +    \sum_{ \substack{ r \mid D \\ r>1 } } \gamma_r c_r(0)   \left(   \widehat{\mathcal{Y}}_\Pap^\tot(0) + 2\cdot  \widehat{\mathcal{V}}_r  \right).
\]
With this relation in hand,  both proofs go through verbatim.
\end{proof}


\subsection{Non-holomorphic generating series of special divisors}
\label{ss:nonhol}


In this subsection we discuss a non-holomorphic variant of the generating series \eqref{eq:holgenser}, which is obtained by endowing the special divisors with other Green functions, namely with those  constructed in \cite{Ho1,Ho2}  following the method of \cite{Ku97}. By combining Theorem \ref{thm:arithmeticmodularity} with a recent result of Ehlen and Sankaran \cite{ES}, we show that the non-holomorphic generating series is also  modular.

For every $m\in \Z$ and $v\in \R_{>0}$ define a divisor 
\begin{equation*}
\mathcal{B}_\Kra(m,v) = \frac{ 1 }{4\pi v}   \sum_\Phi  \#\{ x \in L_0 :  \;\langle x , x \rangle  =m \} \cdot \mathcal{S}^*_\Kra(\Phi)
\end{equation*}
with real coefficients on $\mathcal{S}^*_\Kra$.  
Here the sum is over all $K$-equivalence classes of proper cusp label representatives $\Phi$ in the sense of \S \ref{ss:mixed data},  $L_0$ is the hermitian $\co_\kk$-module of signature $(n-2,0)$ defined by (\ref{boundary herm}), and $\mathcal{S}^*_\Kra(\Phi)$ is the boundary divisor of Theorem \ref{thm:toroidal}. Note that $\mathcal{B}_\Kra(m,v)$ is trivial for all $m<0$.
We define classes in $\Ch^1_\R (\calS^*_\Kra ) $,  depending on the parameter $v$,  by 
\begin{align*}
\calZ_\Kra^\tot(m,v)= \begin{cases}
\calZ^*_\Kra(m)+\calB_\Kra(m,v) &\text{if $m\neq 0$}\\[1ex]
\bm{\omega}^{-1}  + \mathrm{Exc} +\calB_\Kra(0,v) &\text{if $m=0$.}
\end{cases}
\end{align*}

Following \cite{Ho1, Ho2, Ku97}, Green functions for these divisors can be constructed as follows.
For $x\in V(\R)$ and $z\in \calD$ we put 
\[
R(x,z)=-2Q(x_z).
\]
Recalling the incomplete Gamma function   (\ref{incomplete gamma}), for  $m\in \Z$ and 
\[
(v,z,g) \in \R_{>0} \times  \calD \times G(\A_f)
\]  
we define a Green function 
\begin{align}
\label{eq:defkudlagreen}
\Xi(m,v,z,g) = \sum_{\substack{x\in V\smallsetminus\{0\}\\ Q(x)=m}} \chi_{\widehat L}(g^{-1}x)  \cdot 
\Gamma(0, 2\pi v R(x,z)),
\end{align}
where $\chi_{\widehat L}\in S_L$ denotes the characteristic function of $\widehat L$.
As a function of the variable $(z,g)$,  \eqref{eq:defkudlagreen} is invariant under the left action of $G(\Q)$ and under the right action of $K$, and so descends to a function on $\R_{>0}\times \mathrm{Sh}(G,\calD) (\C)$.
It was proved in \cite[Theorem 3.4.7]{Ho2} that 
$\Xi(m,v)$ is a logarithmic Green function for $\calZ_\Kra^\tot(m,v)$ when $m\neq 0$.
When $m=0$ it is a logarithmic Green function for $\calB_\Kra(0,v)$.

Consequently, we obtain arithmetic special divisors in $\Cha^1_\R (\calS^*_\Kra ) $ depending on the parameter $v$ by putting
\begin{align*}
\widehat\calZ_\Kra^\tot(m,v)= \begin{cases}
\left(\calZ_\Kra^\tot(m,v), \Xi(m,v)\right) &\text{if $m\neq 0$ }\\[1ex]
\widehat{\bm{\omega}}^{-1}   +\left(\calB_\Kra(0,v),\Xi(0,v)\right) + (\mathrm{Exc}, - \log (Dv) ) &\text{if $m=0$.}
\end{cases}
\end{align*}
Note that for $m<0$ these divisors are  supported in the archimedian fiber.

\begin{theorem}
\label{thm:arithmeticmodularity2}
The  formal $q$-expansion
\begin{align*}
\widehat\phi_{\text{\rm non-hol}}(\tau)=
\sum_{m\in \Z} \widehat{\calZ}_\Kra^\tot(m,v) \cdot q^m \in \Cha^1_\R (\calS^*_\Kra )  [[ q ]] ,
\end{align*}
is a non-holomorphic modular form of level $D$, weight $n$, and character $\chi$.  
Here $q=e^{2\pi i\tau}$ and $v=\Im(\tau)$. 
\end{theorem}

\begin{proof}
Theorem~4.13 of \cite{ES} states that the difference 
\begin{equation}\label{theta diff}
\widehat\phi_{\text{\rm non-hol}}(\tau)-\widehat\phi(\tau)
\end{equation}
 is a non-holomorphic modular form of level $D$, weight $n$, and character $\chi$, valued  in $\Cha^1_\C(\calS^*_\Kra)$. Hence the assertion follows from 
Theorem~\ref{thm:arithmeticmodularity}.
\end{proof}

The meaning of modularity in Theorem \ref{thm:arithmeticmodularity2} is to be understood as in \cite[Definition 4.11]{ES}. 
In our situation it reduces to 
the statement that there is a smooth function $s(\tau,z,g)$ on  $\mathfrak{H}\times \mathrm{Sh}(G,\calD) (\C)$ with the following properties:
\begin{enumerate}
\item in $(z,g)$ the function $s(\tau,z,g)$ has at worst $\log$-$\log$-singularities at the boundary of $\mathrm{Sh}(G,\calD) (\C)$ (in particular it is a Green function for the trivial divisor);
\item $s(\tau,z,g)$ transforms in $\tau$ as a non-holomorphic modular form of level $D$, weight $n$, and character $\chi$;
\item
the difference 
$\widehat\phi_{\text{\rm non-hol}}(\tau)-s(\tau,z,g)$ belongs to the space
\[
M_n(D,\chi)\otimes_\C \Cha^1_\C (\calS^*_\Kra )\oplus (R_{n-2}M_{n-2}(D,\chi))\otimes_\C \Cha^1_\C (\calS^*_\Kra ),
\]
where $R_{n-2}$ denotes the Maass raising operator as in Section \ref{ss:direct mult}.
\end{enumerate}



\section{Appendix: some technical calculations}
\label{s:appendix}


We collect  some  technical arguments and calculations.  
Strictly speaking, none of these are essential to the proofs in the body of the text.  
We explain the connection between the fourth roots of unity $\gamma_p$ defined by  (\ref{def-gamma-p}) and the local Weil indices appearing in the theory of the Weil representation,  provide alternative proofs of  Propositions \ref{prop:promoted coefficients} and \ref{prop:other mult}, and explain in greater detail how Proposition~\ref{prop:ortho FJ formula} is deduced from the formulas of \cite{Ku:ABP}.


\subsection{Local Weil indices}
\label{ss:weil indices}


In this subsection, we explain how the quantity $\gamma_p$ defined in (\ref{def-gamma-p}) is related to the local Weil representation.

Let $L\subset V$ be as in \S \ref{ss:vector-valued}, and 
recall that  $S_L=\C[L'/L]$ is identified with a subspace of $S(V(\A_f))$ by sending  $\mu\in L'/L$ to   the characteristic function $\phi_\mu$  of   $\mu+\widehat{L} \subset V(\A_f)$.

As  $\dim_\Q V = 2n$ and  $D$ is odd, the representation $\omega_L$ of $\SL_2(\Z)$ on $S_L$ is the pullback via
\[
\SL_2(\Z) \longrightarrow  \prod_{p\mid D} \SL_2(\Z_p)
\]
of the  representation
\[
\omega_L = \bigotimes_{p\mid D} \omega_{p},
\]
where $\omega_p=\omega_{L_p}$ is the Weil representation of $\SL_2(\Z_p)$ on $S_{L_p} \subset S(V_p)$.
These Weil representations are defined using the standard global additive character $\psi= \otimes_p\psi_p$ which is trivial on $\widehat{\Z}$ and on $\Q$ and whose restriction to 
$\R\subset \A$ is given by $\psi(x) = \exp(2\pi i x)$. 
Recall that, for $a\in \Q_p^\times$ and $b\in \Q_p$,
\begin{align*}
\omega_p(n(b))\phi(x) &= \psi_p(b Q(x))\cdot\phi(x)\\
\noalign{\smallskip}
\omega_p(m(a))\phi(x) &= \chi_{\kk,p}^n(a)\cdot|a|_p^n\cdot\phi(ax)\\
\noalign{\smallskip}
\omega_p(w) \phi(x) &=\gamma_p\int_{V_p}  \psi_p(-[x,y])\cdot\phi(y)\, dy, \quad w = \begin{pmatrix}{}&-1\\1&{}\end{pmatrix},
\end{align*}
where $\gamma_p=\gamma_p(L)$ is the Weil index of the quadratic space $V_p$ with respect to $\psi_p$ and $\chi_{\kk,p}$ is the quadratic character of $\Q_p^\times$ corresponding to 
$\kk_p$.  Note that $dy$ is the self-dual measure with respect to the pairing $\psi_p([x,y])$.

\begin{lemma} The Weil representation $\omega_p$ satisfies the following properties.
\begin{enumerate}
\item
 When restricted to the subspace $S_{L_p}\subset S(V_p)$, the action of $\gamma\in \SL_2(\Z_p)$ depends only on the image of $\gamma$ in $\SL_2(\F_p)$.
\item
 The Weil index is given by 
\[
\gamma_p = \epsilon_p^{-n}\cdot (D,p)_p^n\cdot \mathrm{inv}_p(V_p)  
\]
where $(a,b)_p$ is the Hilbert symbol for $\Q_p$ and $\mathrm{inv}_p(V_p)$ is the invariant of  $V_p$ in the sense of (\ref{eq:locinv}).
\end{enumerate}
\end{lemma}

\begin{proof} (i)  It suffices to check this on the generators. We omit this.\hfill\break
(ii) We can choose an $O_{\kk,p}$-basis for $L_p$ such that the matrix for the hermitian form is $\text{diag}(a_1,\dots,a_n)$, with $a_j\in \Z_p^\times$. 
The matrix for the bilinear form $[x,y] = \mathrm{Tr}_{K_p/\Q_p}( \langle x,y \rangle)$ is then $\text{diag}(2a_1,\dots,2a_n,2Da_1,\dots,2Da_n)$. 
Then, according to the formula for $\beta_V$ in \cite[p.~379]{kudla.splitting}, we have
\[
\gamma_p^{-1} = \gamma_{\Q_p}(\frac12\cdot\psi_p\circ V) = \prod_{j=1}^n \gamma_{\Q_p}(a_j\psi_p)\cdot\gamma_{\Q_p}(D a_j \psi_p),
\]
where we note that, in the notation there,  $x(w) =1$, and $j=j(w)=1$. Next by Proposition~A.11 of  the Appendix to \cite{rangarao1993},   for any $\alpha\in \Z_p^\times$, we have 
$\gamma_{\Q_p}(\alpha \psi_p)=1$ and
\[
\gamma_{\Q_p}(\alpha p \psi_p) =  \left(\frac{-\alpha}{p}\right)\cdot\epsilon_p = (-\alpha,p)_p\cdot\epsilon_p.
\]
Here note that if $\eta = \alpha p \psi_p$, then the resulting character $\bar\eta$ of $\F_p$ is given by 
\[
\bar\eta(\bar a) = \psi_p(p^{-1}a) = e(-p^{-1}a).
\]
and $\gamma_{\F_p}(\bar\eta) = \left(\frac{-1}{p}\right)\cdot\epsilon_p$. 
Thus
\[
\gamma_p = \epsilon_p^{-n}\cdot(-D/p,p)_p^n\cdot(\det(V),p)_p, 
\]
as claimed.
\end{proof}


\subsection{A direct proof of Proposition~\ref{prop:promoted coefficients}}
\label{ss:appendix promotion}


The proof of Proposition~\ref{prop:promoted coefficients}, which expresses the Fourier coefficients of the vector valued form $\tilde f$ 
in terms of those of the scalar valued form $f\in M_{2-n}^!(D,\chi)$, appealed to the more general results of \cite{Scheithauer}.
In some respects, it is easier to prove Proposition~\ref{prop:promoted coefficients} from scratch than it is to extract it from [\emph{loc.~cit.}].  This is what we do here.

Recall that $\tilde f$ is defined from $f$ by the induction procedure of  (\ref{eq:vectorization}),
and that the coefficients $\tilde c(m,\mu)$ in its Fourier expansion (\ref{eq:tilde-f-Fourier}) 
are indexed by $m\in \Q$ and $\mu\in L'/L$.
Recall that, for $r\mid D$, $rs=D$, 
\[
W_r=\begin{pmatrix} r\alpha&\beta\\D\gamma&r \delta\end{pmatrix} = R_r\begin{pmatrix} r&{}\\{}&1\end{pmatrix}, \quad R_r= \begin{pmatrix} \alpha&\beta\\s\gamma&r \delta\end{pmatrix}\in \Gamma_0(s).
\]
Note that 
\begin{equation}\label{eq:chinese}
\Gamma_0(D)\backslash\SL_2(\Z) = \Gamma_0(D)\backslash\SL_2(\Z)/\Gamma(D) \simeq \prod_{p\mid D} B_p\backslash \SL_2(\F_p),
\end{equation}
so this set has order $\prod_{p\mid D}(p+1)$. A set of coset representatives is given by 
\[
\bigsqcup_{ \substack{  r\mid D  \\ c\, (\mathrm{mod}\, r)  }}  R_r \begin{pmatrix}1&c\\{}&1\end{pmatrix}.
\]
Now, using (\ref{other cusps}), we have 
\begin{align}
 \left(
 f\big|_{2-n}R_r \begin{pmatrix}1&c\\{}&1\end{pmatrix}\right)    (\tau)
 & =  
 \left(
 f\big|_{2-n}W_r \begin{pmatrix}r^{-1}&r^{-1}c\\{}&1\end{pmatrix}\right)  (\tau)  \nonumber\\
 \noalign{\smallskip}
 {}&=\chi_r(\beta)\chi_s(\alpha) \sum_{m\gg-\infty} r^{\frac{n}2-1} c_r(m) \cdot e^{ \frac{2\pi i m (\tau+c)}{r}  } .\label{first-contrib}
\end{align}
On the other hand, the image of the inverse of our coset representative on the right side of (\ref{eq:chinese}) has components 
\[
\begin{cases}\begin{pmatrix}1&-c\\{}&1\end{pmatrix}\begin{pmatrix} 0&-\beta\\-s\gamma&\alpha\end{pmatrix} &\text{if $p\mid r$} \\
\begin{pmatrix}1&-c\\{}&1\end{pmatrix}\begin{pmatrix} r\delta&-\beta\\0&\alpha\end{pmatrix}&\text{if $p\mid s$.}
\end{cases}
\]
Note that $r\alpha\delta - s\beta\gamma=1$.  Then, as elements of $\SL_2(\F_p)$, we have
\[
\begin{cases}\begin{pmatrix}1&-c\\{}&1\end{pmatrix}\begin{pmatrix} \beta&{}\\{}&\beta^{-1}\end{pmatrix}\begin{pmatrix}{}&-1\\1&{}\end{pmatrix}\begin{pmatrix}1&\alpha\beta\\{}&1\end{pmatrix} &\text{if $p\mid r$ }\\
\begin{pmatrix}1&-c\\{}&1\end{pmatrix}\begin{pmatrix} \alpha^{-1}&{}\\0&\alpha\end{pmatrix}\begin{pmatrix}1&-\alpha\beta\\{}&1\end{pmatrix}&\text{if $p\mid s$.}
\end{cases}
\]
The element on the second line just multiplies $\phi_{0,p}$ by $\chi_p(\alpha)$. 
For the element on the first line, the factor on the right fixes $\phi_0$ and 
\[
\omega_p\left(\begin{pmatrix}{}&-1\\1&{}\end{pmatrix} \right)\phi_0 = \gamma_p \,p^{-\frac{n}2}\,\sum_{\mu\in L'_p/L_p} \phi_\mu.
\]
Thus, the element on the first line carries $\phi_{0,p}$ to  
\[
\chi_p(\beta)\gamma_p \,p^{-\frac{n}2}\,\sum_{\mu\in L'_p/L_p}\psi_p(-c \,Q(\mu))\, \phi_\mu.
\]
Recall from (\ref{eq:r-mu})  that for $\mu\in L'/L$,  $r_\mu$ is the product of the primes $p\mid D$ such that $\mu_p\ne0$.
Thus
\begin{equation}\label{second-contrib}
\omega_L \left(   R_r \begin{pmatrix}1&c\\{}&1\end{pmatrix}   \right)^{-1} 
 \phi_0 = \chi_s(\alpha)\chi_r(\beta)\,\gamma_r\,r^{-\frac{n}2}
\sum_{\substack{\mu\in L'/L\\ r_\mu\mid r}} e^{ 2\pi i c Q(\mu)}  \phi_\mu.
\end{equation}
Taking the product of (\ref{first-contrib}) and (\ref{second-contrib}) and summing on $c$ and on $r$, we obtain
\begin{align*}
&\sum_{r\mid D}\gamma_r\cdot  r^{-1} \sum_{c\, (\text{mod}\,r)}  
\sum_{\substack{\mu\in L'/L\\ r_\mu\mid r}}     e^{2\pi i c Q(\mu) } 
\phi_\mu  \sum_{m\gg-\infty} c_r(m)
e^{  \frac{2\pi i  m(\tau+c) }{r}    } \\
\noalign{\smallskip}
{}&=\sum_{r\mid D} \gamma_r\,\sum_{\substack{\mu\in L'/L\\ r_\mu\mid r}}\,\phi_\mu  
\sum_{\substack{m\gg-\infty\\ \frac{m}{r}+Q(\mu)\in \Z}} c_r(m)\, q^{\frac{m}{r}}  \\
\noalign{\smallskip}
{}&=\sum_{\substack{m\in \Q\\ m\gg-\infty}}  \sum_{\substack{\mu\in L'/L\\  m+Q(\mu)\in \Z}}\   
\sum_{\substack{r \\  r_\mu\mid r \mid D}} \gamma_r  c_r(mr)  \phi_\mu\, q^{m}
\end{align*}
This gives the claimed general expression for $\tilde c(m,\mu)$ and completes the proof of Proposition~\ref{prop:promoted coefficients}.


\subsection{A more detailed proof of Proposition~\ref{prop:ortho FJ formula}}


In this section, we explain in more detail how to obtain the product formula of Proposition~\ref{prop:ortho FJ formula} from the general formula given in 
\cite{Ku:ABP}. 

For our weakly holomorphic $S_L$-valued modular form $\tilde f$ of weight $2-n$, with 
Fourier expansion given by (\ref{eq:tilde-f-Fourier}), the corresponding meromorphic Borcherds product $\Psi(\tilde f)$ on $\tilde{\mathcal{D}}^+$
 has a product formula \cite[Corollary 2.3]{Ku:ABP} in a neighborhood of the $1$-dimensional 
boundary component associated to $L_{-1}$. 
It is given as a product of $4$ factors, labeled (a), (b), (c) and (d). 
We note that, in our present case, there is a basic simplification in factor (b) due to the restriction on the support of the Fourier coefficients of $\tilde f$. 
More precisely, 
for $m>0$, $\tilde c(-m,\mu)=0$ for $\mu\notin L$, and $\tilde c(-m,0)=c(-m)$.  In particular, if $x\in L'$ with $[x,\eee_{-1}]=[x,\fff_{-1}]=0$, then $Q(x) = Q(x_0)$, where $x_0$ is the $(L_0)_\Q$
component of $x$. If $x_0\ne0$, then $Q(x)>0$, and $\tilde c(-Q(x),\mu)=0$ for $\mu\notin L$.   
The factors for $\Psi(\tilde f)$ are then given by:
\hfill\break
(a) 
\[
\prod_{\substack{x\in L'\\  [x,\fff_{-1}]=0\\  [x,\eee_{-1}]>0\\  \mod L\cap \,\Q\, \fff_{-1}}} \big(1-e^{ - 2\pi i [x,w ]}  \big)^{\tilde c(-Q(x),x)}.
\]
(b)
\[
P_1(w _0,\tau_1)\define
\prod_{\substack{x\in L_0\\ [x,W_0]>0}} 
\bigg(\ \frac{\vartheta_1(-[x,w ],\tau_1)}{\eta(\tau_1)}\ \bigg)^{c(-Q(x))},
\]
%
where $W_0$ is a Weyl chamber  in $V_0(\R)$, as in \cite[\S 2]{Ku:ABP}.
\hfill\break
(c)
\[
P_0(\tau_1)\define\prod_{\substack{x\in \mathfrak{d}^{-1}L_{-1}/ L_{-1}\\  x\ne 0}} 
\bigg(\ \frac{\vartheta_1(-[x,w ],\tau_1)}{\eta(\tau_1)} \,  e^{ \pi i [x,w ]  \cdot [x,\eee_1] }   \bigg)^{\tilde c(0,x)/2}
\]
(d) and 
\[
\kappa\, \eta(\tau_1)^{\tilde c(0,0)}\,q_2^{I_0},
\]
where $\kappa$ is a scalar of absolute value $1$, and 
\[
I_0 = -  \sum_{m} \sum_{\substack{x\in L'\cap (L_{-1})^\perp\\  \mod L_{-1}}}
 \tilde c(-m,x)\,\sigma_1(m-Q(x)).
 \]
 
The factors given in Proposition~\ref{prop:ortho FJ formula} are for the form 
\[
\tilde{\bm{\psi}}_g( f ) \define (2\pi i)^{ \tilde{c}(0,0) } \Psi(2 \tilde{f})
\]
The quantity $q_2$ in \cite{Ku:ABP} is our $e(\xi)$,  and $\tau_1$ there is our $\tau$. 

Recall from (\ref{def-Lpm}) that $\mathfrak{d}^{-1}L_{-1} =\Z \eee_{-1} + D^{-1}\Z \fff_{-1}$,  so that, in factor (c), the product runs over 
vectors $D^{-1} b \,\fff_{-1}$, with $b \pmod D$ nonzero.  For these vectors $[x,\eee_1]=0$. 
  In the formula for $I$, $x$ runs over vectors of the form 
\[
x = -\frac{b}{D} \fff_{-1} + x_0,
\]
with $x_0\in \mathfrak{d}^{-1}L_0$. But, again, if $x_0\ne 0$, $Q(x)=Q(x_0)>0$ and $\tilde c(-Q(x),x) =0$ unless $b=0$, 
and so the sum in that term runs over $x_0\in L_0$ $x_0\ne 0$ and over $-\frac{b}{D} \fff_{-1}$'s. 

Thus the factors for $\tilde{\bm{\psi}}_g( f )$ are given by:
\hfill\break
(a) 
\[
\prod_{\substack{x\in L'\\  [x,\fff_{-1}]=0\\  [x,\eee_{-1}]>0\\  \mod L\cap \,\Q\, \fff_{-1}}} \big(1-e^{ -2\pi i [x,w ]} \big)^{2\,\tilde c(-Q(x),x)},
\]
(b)
\[
P_1(w _0,\tau_1)\define
\prod_{\substack{x_0\in L_0\\ x_0\ne0}} 
\bigg(\ \frac{\vartheta_1(-[x_0,w ],\tau_1)}{\eta(\tau_1)}\ \bigg)^{c(-Q(x_0))},
\]
%
\hfill\break
(c)
\[
P_0(\tau_1)\define\prod_{\substack{b\in \Z/D\Z\\ b\ne 0}} 
\bigg(\ \frac{\vartheta_1(-[x,w ],\tau_1)}{\eta(\tau_1)} \,  \ \bigg)^{\tilde c(0,\frac{b}{D} f_{-1})},
\]
(d) and, setting $k=\tilde c(0,0)$,  
\[
\kappa^2\, (\,2\pi i\,\eta^2(\tau))^{k}\,q_2^{2I_0},
\]
where $\kappa$ is a scalar of absolute value $1$, and 
\[
I_0 = - 2\sum_{m>0} \sum_{\substack{ x_0\in L_0}} c(-m)\,\sigma_1(m-Q(x_0)) +\frac1{12}\sum_{\substack{b\in \Z/D\Z}}
 \tilde c(0,\frac{b}{D} \fff_{-1}).
 \]
Here note that for $\tilde{\bm{\psi}}_g( f )=(2\pi i)^{ \tilde{c}(0,0) } \Psi(2 \tilde{f})$ we have multiplied the previous expression by $2$. 

Finally recall
\[
w = -\xi  \eee_{-1} +(\tau \xi - Q(w_0)) \fff_{-1} + w_0 + \tau \eee_1 + f_1.
\]
If $[x,f_{-1}]=0$, then $x$ has the form
\[
x= -a \eee_{-1} - \frac{b}{D}  \fff_{-1} +x_0  + c \eee_1,
\]
so that 
\[
[x,w ] = - c\,\xi + [x_0,w_0]-a\tau - \frac{b}{D},
\]
and 
\[
Q(x) = -ac + Q(x_0).
\]
Using these values, the formulas given in Proposition~\ref{prop:ortho FJ formula} follow immediately.


\subsection{A direct proof of Proposition \ref{prop:other mult}}
\label{ss:direct mult}


Here we give a direct proof of Proposition \ref{prop:other mult}, which does not rely on Corollary \ref{cor:analytic BFJ}. 
We begin by recalling some general facts about derivatives of modular forms.

We let $q\frac{d}{dq}$ be the Ramanujan theta operator on $q$-series.
Recall that the image under $q\frac{d}{dq}$ of a holomorphic modular form 
$g$ of weight $k$ is in general not a modular form. However, the function
\begin{align}
\label{eq:defdg}
D(g)=q\frac{dg}{dq} -\frac{k}{12} g E_2
\end{align}
is a holomorphic modular form of weight $k+2$ (see \cite[\S 4.2]{BHY}).
Here 
\[ 
E_2(\tau)=-24\sum_{m\geq 0} \sigma_1(m)q^m
\] 
denotes the non-modular Eisenstein series of weight $2$ for $\SL_2(\Z)$. In particular $\sigma_1(0)=-\frac{1}{24}$. We extend $\sigma_1$ to rational arguments by putting $\sigma_1(r)=0$ if $r\notin\Z_{\geq 0}$.
If $R_k=2i\frac{\partial}{\partial\tau}+\frac{k}{v} $ denotes the Maass raising operator, and \[ E_2^*(\tau)=E_2(\tau)-\frac{3}{\pi v}\]  is the non-holomorphic (but modular) Eisenstein series of weight $2$, we also have
\begin{align*}
D(g)=-\frac{1}{4\pi }R_k(g)-\frac{k}{12} g E_2^*.
\end{align*}

\begin{proposition}
Let $f\in M_{2-n}^{!,\infty}(D,\chi)$ as in \eqref{input form}.
The integer 
\[
I= \frac{1}{12}\sum_{\substack{\alpha\in \mathfrak{d}^{-1}L_{-1}/L_{-1}}}
\tilde c(0,\alpha)-2\sum_{m>0} c(-m) \sum_{\substack{x\in L_0}} \sigma_1(m-Q(x)).
\]
defined in Proposition \ref{prop:ortho FJ formula} is equal to the integer 
\[
\mathrm{mult}_\Phi(f)=\frac{1}{n-2}\sum_{x\in L_0} c(-Q(x))Q(x)
\]
 defined by (\ref{f boundary mult}).
\end{proposition}

\begin{proof}
Consider the 
$S_{L_0}^\vee$-valued theta function 
\[
\Theta_0(\tau) = \sum_{x\in L_0'} q^{Q(x)} 
\chi_{x+L_0}^\vee\in M_{n-2}(\omega_{L_0}^\vee).
\]
Applying the above construction \eqref{eq:defdg} to $\Theta_0$ we obtain an $S_{L_0}^\vee$-valued modular form 
\begin{align*}
D(\Theta_0) = \sum_{x\in L_0'} Q(x)q^{Q(x)}\chi_{x+L_0}^\vee -\frac{n-2}{12} \Theta_0 E_2
\in M_{n}(\omega_{L_0}^\vee)
\end{align*}
of weight $n$.  For its Fourier coefficients we have
\begin{align*}
D(\Theta_0) &= \sum_{\nu\in L_0'/L_0}\sum_{m\geq 0} b(m,\nu)q^m\chi_{\nu}^\vee\\
b(m,\nu) &= \sum_{\substack{x\in \nu +L_0\\ Q(x)=m} }Q(x)
+2(n-2) \sum_{\substack{x\in \nu+L_0}} \sigma_1(m-Q(x)).
\end{align*}

As in \cite[(4.8)]{BHY},   an $S_L$-valued modular form $F$  induces an $S_{L_0}$-valued form $F_{L_0}$.
If we denote by $F_\mu$ the components of $F$ with respect to the standard basis $(\chi_\mu)$ of $S_L$, we have 
\begin{align}
\label{eq:deffd}
F_{L_0,\nu} =  \sum_{\substack{\alpha\in \mathfrak{d}^{-1}L_{-1}/L_{-1}}}F_{\nu+\alpha}
\end{align}
for $\nu\in L_0'/L_0$.

Let $\tilde f\in M_{2-n}^!(\omega_L)$ be the $S_L$-valued form corresponding to $f$, as in  \eqref{eq:vectorization}.
Using \eqref{eq:deffd} we obtain 
\begin{align*}
\tilde f_{L_0} \in M_{2-n}^!(\omega_{L_0})
\end{align*}
with Fourier expansion
\begin{align*}
\tilde f_{L_0} =\sum_{\nu,m} \sum_{\substack{\alpha\in \delta^{-1}I/I}}
\tilde c(m,\nu+\alpha) q^m \chi_{\nu+L_0}.
\end{align*}
We consider the natural pairing between the $S_{L_0}$-valued modular form $\tilde f_{L_0}$ of weight $2-n$ and the $S_{L_0}^\vee$-valued modular form $D(\Theta_0)$ of weight $n$,
\[
(\tilde f_{L_0}, D(\Theta_0) ) \in M_2^!(\SL_2(\Z)).
\]
By the residue theorem, the constant term of the $q$-expansion vanishes, and so  
\begin{align}
\label{eq:res}
\sum_{m\geq 0} \sum_{\substack{\nu\in L_0'/L_0\\
\alpha\in \delta^{-1}I/I}}
\tilde c(-m,\nu+\alpha)  b(m,\nu) =0.
\end{align}

We split this up in the sum over $m>0$ and the contribution from $m=0$.
Employing Proposition \ref{prop:promoted coefficients}, we obtain that the sum over $m>0$ is equal to 
\begin{align*}
\sum_{m>0} c(-m) b(m,0).
\end{align*}
For the contribution of $m=0$
we notice
\[
b(0,\nu) = \begin{cases}
-\frac{n-2}{12},& \nu=0\in L_0'/L_0,\\
0,& \nu\neq 0.
\end{cases}
\]
Hence this part is equal to 
\begin{align*}
-\frac{n-2}{12}\sum_{\substack{\alpha\in \mathfrak{d}^{-1}L_{-1}/L_{-1}}}
\tilde c(0,\alpha)  .
\end{align*}
Inserting the two contributions into \eqref{eq:res}, we obtain
\begin{align*}
0&=
\sum_{m>0} c(-m) b(m,0)-\frac{n-2}{12}\sum_{\substack{\alpha\in \mathfrak{d}^{-1}L_{-1}/L_{-1}}}
\tilde c(0,\alpha) \\
&=\sum_{m>0} c(-m) \bigg(
\sum_{\substack{x\in L_0\\ Q(x)=m} }Q(x)
+2(n-2) \sum_{\substack{x\in L_0}} \sigma_1(m-Q(x))\bigg)\\
&\phantom{=}{}
-\frac{n-2}{12}\sum_{\substack{\alpha\in \mathfrak{d}^{-1}L_{-1}/L_{-1}}}
\tilde c(0,\alpha)\\
&=\sum_{x\in L_0} c(-Q(x))Q(x)  +2(n-2)
\sum_{m>0} c(-m) \sum_{\substack{x\in L_0}} \sigma_1(m-Q(x))\\
&\phantom{=}{}-\frac{n-2}{12}\sum_{\substack{\alpha\in \mathfrak{d}^{-1}L_{-1}/L_{-1}}}
\tilde c(0,\alpha)\\
&= (n-2)\mathrm{mult}_\Phi(f)-(n-2)I.
\end{align*}
This concludes the proof of the proposition.
\end{proof}

Now we verify directly the other claim of Proposition \ref{prop:other mult}: the function
\[
P_1(\tau,w_0) = \prod_{m>0}\prod_{\substack{x\in L_0\\ Q(x)=m}} \Theta\big(\tau,\langle w_0,x\rangle\big)^{c(-m)}
\]
satisfies the transformation law (\ref{FJ-trans})  with respect to the translation action of $\mathfrak{b} L_0$ on the variable $w_0$. 

First recall that, for $a$, $b\in \Z$, 
\[
\Theta(\tau,z+a\tau+b) = \exp\big( - \pi i a^2\tau- 2\pi i az+ \pi i (b-a)\big) \cdot \Theta(\tau,z).
\]
If we write $\alpha=a\tau+b$ and $\tau=u+iv$, then
\[
a = \frac{\Im(\alpha)}{v} = \frac{\alpha-\bar\alpha}{2iv}, \qquad b= \Re(\alpha) - \frac{u}{v}\,\Im(\alpha).
\]
Thus
\[
\frac12a^2\tau+az+\frac12(a-b) = \frac1{4iv}(\alpha-\bar\alpha)\alpha + \frac1{2iv}(\alpha-\bar\alpha)z+\frac12(a-b-ab).
\]
For $z$ and $w$ in $\C$, write
\[
R(z,w)=R_\tau(z,w) = B_\tau(z,w)-H_\tau(z,w) = \frac{1}{v} z(w-\bar w).
\]
Then 
\[
\frac1{4v}(\alpha-\bar\alpha)\alpha + \frac1{2v}(\alpha-\bar\alpha)z = \frac12 R(z,\alpha) + \frac14 R(\alpha,\alpha),
\]
and we can write
\[
\Theta(\tau,z+\alpha) = \exp(-\pi R(z,\alpha) - \frac{\pi}2 R(\alpha,\alpha)) \cdot \exp(\pi i (a-b-ab))^{-1}\,\Theta(\tau,z).
\]

We will consider the contribution of the $\frac12(a-b-ab)$ term separately. 

For $\beta \in V_0$, we have $\langle w_0+\beta,x\rangle = \langle w_0,x\rangle + \langle \beta,x\rangle$. Suppose that for all $x\in L_0$, we have $\langle \beta,x\rangle =a\tau+b$
for $a$ and $b$ in $\Z$. Writing $\mathfrak{b} = \Z +\Z\tau$, this is precisely the condition that $\beta\in \mathfrak{b}\,L_0$.  Then we
obtain a factor
\[
\exp\left( -\pi  \sum_{m>0}\sum_{\substack{x\in L_0\\ Q(x)=m}}c(-m) 
\left[ R\big( \langle w_0,x\rangle, \langle \beta,x\rangle\big)+ \frac{   R\big( \langle \beta,x\rangle,\langle \beta,x\rangle \big)    }{2}\right] \right).
\]
Expanding the sum and using the hermitian version of Borcherds' quadratic identity from the proof of Proposition \ref{prop:quadratic bundles}, we have
\begin{align*}
&\sum_{\substack{x\in L_0}} \frac{  c(-Q(x))  }{v} 
\left[ \langle w_0,x\rangle  \langle \beta,x\rangle -   \langle w_0,x\rangle  \langle x,\beta\rangle
+\frac{ \langle \beta,x\rangle \langle \beta,x\rangle  }{2}  
-\frac{ \langle \beta,x\rangle\langle x,\beta\rangle}{2} \right]  \\
\noalign{\smallskip}{}
&= 
-\frac{1}{v} \left(  \langle w_0,\beta\rangle +\frac{1}{2} \langle \beta,\beta\rangle\right) \cdot  \frac{1}{2n-4} \cdot  \sum_{x\in L_0} c(-Q(x))\,[x,x]\\
\noalign{\smallskip}
{}
&= -\frac{1}{v} \left( \langle w_0,\beta\rangle +\frac{1}{2}  \langle \beta,\beta\rangle \right) \cdot \text{mult}_{\Phi}(f).
\end{align*}
Thus, using $I= \text{mult}_{\Phi}(f)$,  we have a contribution of
\[
\exp\Big(    \frac{\pi \langle w_0,\beta\rangle }{v}  + \frac{\pi \langle \beta,\beta\rangle  }{2 v}   \Big)^I
\]
to the transformation law.

Next we consider the quantity
\begin{eqnarray*}\lefteqn{
a-b-ab } \\
&=&  
\frac{\Im(\alpha)}{v} - \Re(\alpha) - \frac{u  \Im(\alpha)}{v}  - \frac{\Im(\alpha)}{v} \left(     \Re(\alpha) - \frac{u  \Im(\alpha) }{v}   \right)  \\
&=&  \frac{\alpha-\bar\alpha}{2iv} - \frac{ (\alpha+\bar\alpha) }{2}  
 - \frac{u (\alpha-\bar\alpha) }{2iv} - \frac{\alpha-\bar\alpha}{2iv} 
 \left(   \frac{  (\alpha+\bar\alpha) }{2}  - \frac{u(\alpha-\bar\alpha)}{2iv}     \right) .
\end{eqnarray*}
This will contribute $\exp(-\pi i A)$, where $A$ is defined as the sum
\[
 \sum_{x\ne0} c(-Q(x))\bigg[\  \frac{\alpha-\bar\alpha}{2iv} - \frac{ \alpha+\bar\alpha }{2} 
- \frac{u (\alpha-\bar\alpha) }{2iv} - \frac{\alpha-\bar\alpha}{2iv}\left(\frac{ (\alpha+\bar\alpha)}{2} - \frac{u (\alpha-\bar\alpha)}{2iv}   \right)  \bigg]
\]
where $\alpha = \langle \beta,x\rangle$.  Since $x$ and $-x$ both occur in the sum, the linear terms vanish and 
\[
A = \sum_{x\ne0} c(-Q(x))\bigg[    - \frac{\alpha-\bar\alpha}{2iv}\left(  \frac{(\alpha+\bar\alpha)}{2}   - \frac{u (\alpha-\bar\alpha)  }{2iv}   \right)  \bigg].
\]
Using the hermitian version of Borcherds quadratic identity, as in  the proof of Proposition \ref{prop:quadratic bundles}, we obtain
\[
A =   \frac{u I }{2v^2}  \cdot   \langle \beta,\beta\rangle .
\]
Thus we have
\begin{eqnarray*}\lefteqn{
P_1(\tau,w_0+\beta) }  \\
& = & 
 P_1(\tau,w_0)  \cdot  \exp \Big(   \frac{\pi}{v}  \langle w_0,\beta\rangle +\frac{\pi }{2 v} \langle \beta,\beta\rangle  \Big)^I
\cdot \exp\Big(  \frac{- 2\pi i u \langle \beta , \beta\rangle }{4v^2} \Big)^I.
\end{eqnarray*}

Finally, we recall the conjugate linear isomorphism $L_{-1} \iso \mathfrak{b}$ of (\ref{wee conj}) defined by 
$e_{-1} \mapsto \tau$ and    $f_{-1}\mapsto 1.$  As 
\[
\mathfrak{d}^{-1}L_{-1} = \Z e_{-1}+D^{-1}\Z f_{-1},
\] 
we have  $-\delta^{-1} \tau = a\tau + D^{-1}b$ for some $a, b\in \Z$,  and hence 
\[
\tau = -D^{-1}b(a+\delta^{-1})^{-1}.
\]
This gives $u/v= a\, D^{\frac12}.$ Also, using
\[
\delta e_{-1} =-Da e_{-1}-b \,f_{-1},
\]
we have
\[
\frac{1}{2} (1+\delta)\,e_{-1} = \frac{1}{2} (1-Da)\,e_{-1} -\frac{1}{2} b\,f_{-1}\in \Z e_{-1}+\Z f_{-1} = L_{-1}.
\]
Thus $a$ is odd and $b$ is even. Recall that $\mathrm{N}(\mathfrak{b} ) = 2 v  / \sqrt{ D }$. Thus
\[
\frac{u}{4v^2} = \frac{a D^{\frac12}}{2  \mathrm{N}(\mathfrak{b} )D^{\frac12}},
\]
and, since $\langle\beta,\beta\rangle\in \mathrm{N}(\mathfrak{b})$, we have 
\[
\exp\Big(- \frac{2\pi i u \langle\beta,\beta \rangle }{4v^2} \Big) = \exp\Big(-  \frac{\pi i  \langle\beta,\beta\rangle}{\mathrm{N}(\mathfrak{b})} \Big) = \pm1.
\]
The transformation law is then
\[
P_1(\tau,w_0+\beta) = 
\exp\Big(  \frac{\pi}{v} \langle w_0,\beta\rangle +\frac{\pi }{2 v}\langle \beta,\beta\rangle  - i\pi \,\frac{\langle\beta,\beta\rangle}{ \mathrm{N}(\mathfrak{b})}  \Big)^I  
\cdot P_1(\tau,w_0),
\]
as claimed in Proposition \ref{prop:other mult}.






\bibliographystyle{smfalpha}


\newcommand{\etalchar}[1]{$^{#1}$}
\providecommand{\bysame}{\leavevmode\hbox to3em{\hrulefill}\thinspace}
\providecommand{\MR}{\relax\ifhmode\unskip\space\fi MR }
\providecommand{\MRhref}[2]{%
  \href{http://www.ams.org/mathscinet-getitem?mr=#1}{#2}
}
\providecommand{\href}[2]{#2}

\end{document}